\documentclass[11pt]{amsart}
\pdfoutput=1
\usepackage{amsmath, amssymb, amsthm, amsbsy, bbm, enumitem, amsfonts, color, verbatim, array, floatflt}
\usepackage[margin=1.2in]{geometry}
\usepackage[english]{babel}

\usepackage{graphicx, color, overpic, xcolor}
\usepackage[pdfpagelabels]{hyperref}

\theoremstyle{plain}
\newtheorem{thm}{Theorem}[section]
\newtheorem*{thm*}{Theorem}
\newtheorem{thmi}{Theorem} 
\newtheorem{prop}[thm]{Proposition}
\newtheorem*{prop*}{Proposition}
\newtheorem{lemma}[thm]{Lemma}
\newtheorem*{lemma*}{Lemma}
\newtheorem{corollary}[thm]{Corollary}

\newtheorem{conjecture}[thm]{Conjecture}

\theoremstyle{definition}
\newtheorem{definition}[thm]{Definition}
\newtheorem{example}[thm]{Example}
\newtheorem{question}{Question}
\newtheorem{notation}[thm]{Notation}

\theoremstyle{remark}
\newtheorem{remark}[thm]{Remark}

\newcommand{\cC}{\mathcal{C}}
\newcommand{\cA}{\mathcal{A}}

\newcommand{\A}{\mathbb{A}}

\newcommand{\R}{\mathbb{R}}
\newcommand{\Z}{\mathbb{Z}}
\newcommand{\N}{\mathbb{N}}
\newcommand{\Q}{\mathbb{Q}}
\newcommand{\F}{\mathbb{F}}
\newcommand{\C}{\mathbb{C}}
\newcommand{\define}{\mathrel{\mathop:}=}
\newcommand{\ov}{\overline}

\newcommand{\cO}{\mathcal O}

\newcommand{\id}{\mathrm{1}} 
\newcommand{\App}{\mathcal{A}} 
\newcommand{\cham}{\mathrm{Cham}} 
\newcommand{\sW}{W_0} 
\newcommand{\aW}{W} 
\newcommand{\eW}{\widetilde W} 
\newcommand{\Cf}{\mathcal{{C}}_{f}} 
\newcommand{\fa}{{\bf{c}_f}} 
\newcommand{\Cfs}{\widetilde{\mathcal{{C}}}_{f}} 
\newcommand{\Cw}{\mathcal{{C}}} 
\newcommand{\Cu}{\mathcal{{C}}_u} 
\newcommand{\Cws}{\widetilde{\mathcal{{C}}}} 
\newcommand{\Cus}{\widetilde{\mathcal{{C}}}_u} 
\newcommand{\Cpos}{\operatorname{C}^+} 
\newcommand{\Cneg}{\operatorname{C}^-} 
\newcommand{\aH}{\mathcal{H}} 

\newcommand{\phipartial}{{\phi^{\mbox{\tiny{$\partial$}}}}} 
\newcommand{\phipartialg}{{\phi_g^{\mbox{\tiny{$\partial$}}}}}
\newcommand{\phipartialw}{{\phi_w^{\mbox{\tiny{$\partial$}}}}}
\newcommand{\phipartialvtlw}{{\phi_{v,t^\lambda w}^{\mbox{\tiny{$\partial$}}}}}
\newcommand{\phipartialvow}{{\phi_{v_0,w}^{\mbox{\tiny{$\partial$}}}}}
\newcommand{\phipartialwp}{{\phi_{w'}^{\mbox{\tiny{$\partial$}}}}}
\newcommand{\phipartialo}{{\phi_0^{\mbox{\tiny{$\partial$}}}}}

\newcommand{\phipartialu}{{\phi_{u}^{\mbox{\tiny{$\partial$}}}}}
\newcommand{\phipartialuw}{{\phi_{uw}^{\mbox{\tiny{$\partial$}}}}}
\newcommand{\phipartialsw}{{\phi_{sw}^{\mbox{\tiny{$\partial$}}}}}

\newcommand{\phipartialwow}{{\phi_{w_0w}^{\mbox{\tiny{$\partial$}}}}}
\newcommand{\phipartialwou}{{\phi_{w_0u}^{\mbox{\tiny{$\partial$}}}}}
\newcommand{\phipartialswow}{{\phi_{sw_0w}^{\mbox{\tiny{$\partial$}}}}}
\newcommand{\phipartialwowp}{{\phi_{w_0w'}^{\mbox{\tiny{$\partial$}}}}}

\newcommand{\load}{\operatorname{load}}

\newcommand{\type}{\mathrm{type}} 

\newcommand{\Inv}{\operatorname{Inv}} 

\newcommand{\conv}{\mathrm{conv}} 

\newcommand{\x}{\mathbf{x}}
\newcommand{\y}{\mathbf{y}}
\newcommand{\z}{\mathbf{z}}
\newcommand{\w}{\mathbf{w}}
\newcommand{\bb}{\mathbf{b}}

\newcommand{\Aut}{\operatorname{Aut}} 

\numberwithin{equation}{subsection}

\definecolor{amethyst}{rgb}{0.6, 0.4, 0.8}

\begin{document}

\hypersetup{pdfauthor={Milicevic, Schwer, Thomas},pdftitle={Dimensions of affine Deligne--Lusztig varieties}}

\title[Affine Deligne--Lusztig varieties and labeled folded alcove walks]{Dimensions of affine Deligne--Lusztig varieties: \\ a new approach via labeled folded alcove walks and root operators}

\author{Elizabeth Mili\'{c}evi\'{c}, Petra Schwer, and Anne Thomas}
\address{Elizabeth Mili\'{c}evi\'{c}, Department of Mathematics \& Statistics, Haverford College, 370 Lancaster Avenue, Haverford, PA, USA
\newline Petra Schwer, Department of Mathematics, Karlsruhe Institute of Technology, Englerstraße 2, 76133 Karlsruhe, Germany
\newline Anne Thomas, School of Mathematics \& Statistics, Carslaw Building F07,  University of Sydney NSW 2006, Australia}
\email{emilicevic@haverford.edu, petra.schwer@kit.edu, anne.thomas@sydney.edu.au}

\thanks{The first author was partially supported by Simons Collaboration Grant 318716 and NSF Grant 1600982.}
\thanks{The second author was supported by the DFG Project SCHW 1550/2-1.}
\thanks{This research of the third author was supported by ARC Grant No.\ DP110100440.  The third author was partially supported by an Australian Postdoctoral Fellowship.} 
\thanks{The first and third authors were also partially supported by DFG Grant SCHW 1550/2-1.}
\date{ \today }


\begin{abstract}
Let $G$ be a reductive group over the field $F=k((t))$, where $k$ is an algebraic closure of a finite field, and let $W$ be the (extended) affine Weyl group of $G$.  The associated affine Deligne--Lusztig varieties $X_x(b)$, which are indexed by elements $b \in G(F)$ and $x \in W$, were introduced by Rapoport \cite{RapSatake}.  
Basic questions about the varieties $X_x(b)$ which have remained largely open include when they are nonempty, and if nonempty, their dimension.  We use techniques inspired by geometric group theory and combinatorial representation theory to address these questions in the case that $b$ is a pure translation, and so prove much of a sharpened version of Conjecture 9.5.1 of G\"ortz, Haines, Kottwitz, and Reuman~\cite{GHKRadlvs}.  
Our approach is constructive and type-free, sheds new light on the reasons for existing results in the case that $b$ is basic, and reveals new patterns.  Since we work only in the standard apartment of the building for $G(F)$, our results also hold in the $p$-adic context, where we formulate a definition of the dimension of a $p$-adic Deligne--Lusztig set.  We present two immediate applications of our main results, to class polynomials of affine Hecke algebras and to affine reflection length. 
\end{abstract}

\maketitle

\tableofcontents


\section{Introduction}\label{sec:Intro}

Affine Deligne--Lusztig varieties were first defined by Rapoport in \cite{RapSatake}
as an application of certain positivity properties of the Satake isomorphism to proving the converse of Mazur's theorem from \cite{RR}.  They also arose implicitly beforehand in connection with the twisted orbital integrals appearing in the Fundamental Lemma, a crucial component of the Langlands program relating the study of Galois representations to automorphic forms. Affine Deligne--Lusztig varieties also appear in Dieudonn\'{e} theory, which classifies $p$-divisible groups over a perfect field in terms of isocrystals; see \cite{Fontaine}. As the name suggests, affine Deligne--Lusztig varieties can be thought of as generalizations to the affine setting of (classical) Deligne--Lusztig varieties, which were constructed by Deligne and Lusztig in order to study the representation theory of finite Chevalley groups; see \cite{DL} and \cite{LuszChev}. The terminology is motivated by the fact that one of the parameters which indexes an affine Deligne--Lusztig variety is an element of the affine analog of the Weyl group of a reductive group over a finite field.

Denote by $\widehat{\Q}_p^{\mbox{\tiny{ur}}}$ the maximal unramified extension of the field of $p$-adic numbers. In this $p$-adic setting one can define affine Deligne--Lusztig sets analogously to the actual affine Deligne--Lusztig varietes, although these sets fail to be varieties themselves.  Affine Deligne--Lusztig sets are closely related  to the reduction modulo the prime $p$ of Shimura varieties for connected reductive groups over $\Q_p$. The main tool to study the local structure of these Shimura varieties is the so-called local model, which describes the singularities of the special fiber in terms of linear algebra. In the special fiber of a moduli space of Rapoport--Zink spaces, the affine Deligne--Lusztig sets relate the intersections of two kinds of stratifications: the Newton stratification, whose strata are indexed by Frobenius-twisted conjugacy classes, and the Kottwitz--Rapoport stratification, whose strata are indexed by elements in the extended affine Weyl group.  See the surveys by Rapoport \cite{RapShimura} or Haines \cite{HainesSurvey} for a more detailed discussion of this theory of Shimura varieties.

\subsection{History of the problem}

Let $k = \ov{\mathbb{F}}_q$ be an algebraic closure of the finite field with $q$ elements. Consider the nonarchimedean local field $F = k((t))$ with ring of integers $\mathcal{O} = k[[t]]$. Denote by $\sigma$ the Frobenius automorphism on $k$, which can be extended to an automorphism on $F$ by acting on coefficients. Let $G$ be a split connected reductive group over $\mathbb{F}_q$, and fix a Borel subgroup $B$ of $G$ and a maximal torus $T$ in $B$. The Iwahori subgroup $I$ of $G(F)$ is the inverse image of $B$ under the projection $G(\cO) \rightarrow G(k)$. For an element $x$ in the extended affine Weyl group $ \eW \cong I\backslash G(F) /I$ and a fixed group element $b \in G(F)$, the associated affine Deligne--Lusztig variety is defined as \begin{equation*} X_x(b) = \{ g \in G(F)/I \mid g^{-1}b\sigma(g) \in IxI \}. \end{equation*}

In the classical context,  the Lang map $G(\ov{\mathbb{F}}_q) \rightarrow G(\ov{\mathbb{F}}_q)$ given by $x\mapsto x^{-1}\sigma (x)$ is surjective.  Together with the Bruhat decomposition $G(\ov{\mathbb{F}}_q) = B\sW B$, where $\sW$ is the finite Weyl group of $G$, the surjectivity of the Lang map implies that for every $w \in \sW$, the classical Deligne--Lusztig variety
 \begin{equation*} X_w = \{ g \in G(\ov{\mathbb{\F}}_q)/B \mid g^{-1}\sigma(g) \in BwB \} \end{equation*}  
 is non-empty.  In fact, Deligne--Lusztig varieties share many properties with Schubert varieties in $G(\C)/B$. For example, classical Deligne--Lusztig varieties are equidimensional of dimension $\ell(w)$, and their closures are determined by the Bruhat order: $\ov{X_w} = \bigsqcup\limits_{v\leq w} X_v$.
 
 By contrast, affine Deligne--Lusztig varieties inside the affine flag variety $G(F)/I$ frequently tend to be empty and are not necessarily equidimensional. There are thus two  foundational questions in the study of affine Deligne--Lusztig varieties:
\begin{enumerate}
\item Characterize the pairs $(x,b) \in \eW \times G(F)$ for which $X_x(b)$ is nonempty; and 
\item If $X_x(b)$ is nonempty, compute its dimension.
\end{enumerate}
Each of these has proven to be a deep and delicate combinatorial and algebro-geometric problem, and we outline the history of progress toward each of these problems below.

The first affine setting in which these questions were completely answered was that of the affine Grassmannian $G(F)/K$, where $K=G(\mathcal{O})$.  Here, affine Deligne--Lusztig varieties are indexed by an element $b \in G(F)$ and a dominant coweight $\mu \in X_*(T)$, and then defined using the Cartan decomposition of $G(F)$ by
\begin{equation*}
X_{\mu}(b) = \{ g \in G(F)/K \mid g^{-1}b\sigma(g) \in Kt^{\mu}K \}.
\end{equation*}
The nonemptiness question for $X_{\mu}(b)$ has been settled by Kottwitz and Rapoport \cite{KRFcrystals}, Lucarelli \cite{Luc}, and Gashi \cite{Gas,GasGLn}.  In the affine Grassmannian, the characterization for nonemptiness is phrased in terms of \emph{Mazur's inequality}, which is a group-theoretic generalization of the inequality between the Hodge and Newton vectors in crystalline cohomology; see \cite{Kat} and \cite{Maz}. Mazur's inequality relates the coweight $\mu$ and the \emph{Newton point} $\nu_b$ of the element $b$, stating that $\nu_b \leq \mu$ in dominance order.  In $G(F)/K$, Mazur's inequality is both a necessary and sufficient condition for guaranteeing nonemptiness of the affine Deligne--Lusztig variety $X_{\mu}(b)$.  In addition, there is a simple dimension formula originally conjectured by Rapoport, which has been proved by G\"{o}rtz, Haines, Kottwitz, and Reuman \cite{GHKR}, Viehmann \cite{VieDim}, and Hamacher \cite{Ham}:
\begin{equation*}
\dim X_{\mu}(b) = \langle \rho, \mu-\nu_b \rangle - \frac{1}{2}\operatorname{def}_G(b).
\end{equation*}
One key step in proving this dimension formula relies upon a calculation by Mirkovi\'{c} and Vilonen \cite{MV} for the dimensions of their MV--cycles, which form a basis for the cohomology of the standard sheaves on the Schubert varieties in $G(F)/K$.

For quite some time, these same foundational questions about nonemptiness and dimensions for affine Deligne--Lusztig varieties inside the affine flag variety remained largely open. In $G(F)/I$, there is no simple inequality that characterizes the nonemptiness pattern. If $x = t^{\lambda}w \in \eW$ where $\lambda$ is the coweight associated to the translation part of $x$, then Mazur's inequality only provides a necessary condition for nonemptiness.  Namely,  if $X_x(b) \neq \emptyset$, then $\nu_b \leq \lambda^+$, where $\lambda^+$ denotes the unique dominant coweight in the $\sW$-orbit of $\lambda$.  However, for most pairs $(x,b)$ this inequality does not suffice to yield nonemptiness. 

Complete answers for both the nonemptiness and dimension questions in $G(F)/I$ are now known in the special case in which the group element $b$ is \emph{basic}, such as when $b=1$.  Recall that two elements $b$ and $b'$ in $G(F)$ are \emph{$\sigma$-conjugate} if there exists $g \in G(F)$ such that $b' = gb\sigma(g)^{-1}$.  The $\sigma$-conjugacy classes in $G(F)$ are often denoted by $B(G)$, and this set was originally studied in \cite{KotIsoI} and \cite{KotIsoII}.  An element $b \in G(F)$ is basic if it is $\sigma$-conjugate to an element of length zero in the extended affine Weyl group.  Figure~\ref{fig:BasicElementsConjugacyA2} shows some basic elements $b \in \aW$ in type $\tilde A_2$.  (The online version of this paper has color figures, and we recommend viewing our figures in color.)

\begin{figure}[ht]
\centering
\includegraphics[width=0.6\textwidth]{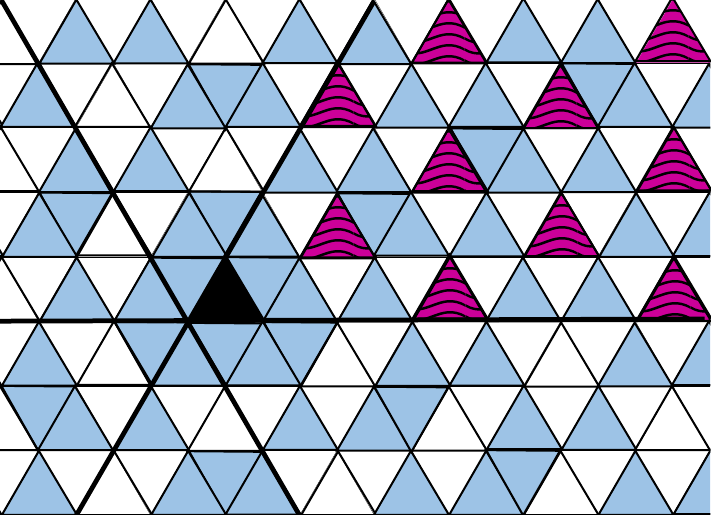}
\label{fig:BasicElementsConjugacyA2}
\caption{Some basic elements in $\aW$ of type $\tilde{A}_2$ are shaded (in blue), the identity $1$ is the black alcove, and the patterned (pink) alcoves are dominant pure translations. 
}
\end{figure}

For basic elements $b \in G(F)$, G\"{o}rtz, Haines, Kottwitz, and Reuman~\cite{GHKRadlvs} conjectured a complete characterization of the nonemptiness pattern and provided a conjectural dimension formula.  They also proved in~\cite{GHKRadlvs} that emptiness holds whenever their conjecture predicted it to do so.  A series of several papers throughout the last decade established the converse for $b$ basic.  In \cite{HeNoteOn}, He proved a nonemptiness pattern for $X_x(1)$ if the translation part of $x$ is quasi-regular, and in \cite{Be2} the first author proved a nonemptiness statement under a length additivity hypothesis on the pair of finite Weyl group elements associated to $x$. In \cite{GoertzHeDim}, G\"{o}rtz and He then proved the nonemptiness conjecture in \cite{GHKRadlvs}, although still under Reuman's original hypothesis that the alcoves lie in the \emph{shrunken} Weyl chambers; \textit{i.e.} far enough from the walls of the chambers. It remained a hard problem to characterize nonemptiness for alcoves which lie outside the shrunken Weyl chambers.  The initial insights into this problem occurred in the work of Reuman \cite{Reu} and the first author \cite{Be1} for groups of low rank. More recently, G\"{o}rtz, He, and Nie proved the full nonemptiness conjecture from \cite{GHKRadlvs} for all basic $b$ in \cite{GHN}.  A proof of the dimension formula for all basic $b$ was then provided by He in \cite{HeAnnals}.  

For non-basic elements $b \in G(F)$, comparatively little is known about the nonemptiness and dimension of $X_x(b)$ inside the affine flag variety.  Both Reuman \cite{Reu} and the first author \cite{Be1} have considered groups of low rank.  In \cite{HeAnnals}, He relates $\dim X_x(b)$ to the dimension of affine Deligne--Lusztig varieties inside the affine Grassmannian, obtaining an equality concerning these dimensions when $x=w_0t^\lambda$ with $\lambda$ dominant, and an upper bound on $\dim X_x(b)$ for any~$x$.  Yang \cite{Yang} has recently applied the results of~\cite{HeAnnals} to answer the questions of nonemptiness and dimension in the case that $\eW$ has type~$\tilde{A}_2$. 

The general picture for the varieties $X_x(b)$ inside the affine flag variety is predicted by Conjecture~9.5.1 in \cite{GHKRadlvs}, which generalizes Conjecture 7.5.1 in \cite{GHKR}.  This conjecture states that if the length of~$x$ is large enough, then emptiness of $X_x(b)$ is equivalent to  emptiness of $X_x(b_{\text{b}})$, where $b_{\text{b}}$ is a basic element naturally associated to $b$.  To be precise, $b_{\text{b}}$ is an element in the unique basic $\sigma$-conjugacy class with the same image as $b$ under the projection onto the quotient $\Lambda_G$ of $X_*(T)$ by the coroot lattice.  In addition, \cite{GHKRadlvs}  provides a conjectural dimension formula for $X_x(b)$ in terms of the dimension of $X_x(b_{\text{b}})$ and the \emph{defect} of $b$.  If we denote by $F^{\sigma}$ the Frobenius-fixed subfield of $F=\mathbb{F}_q((t))$, then the defect $\operatorname{def}_G(b)$ is defined as the $F^{\sigma}$-rank of $G$ minus the $F^{\sigma}$-rank of the Frobenius-twisted centralizer $J_b = \{ g \in G(F) \mid g^{-1}b\sigma(g) = b\}$. 

\begin{conjecture}[Conjecture 9.5.1 in \cite{GHKRadlvs}]\label{Conj_GHKR}
Let $b \in G(F)$.  Then there exists $N_b \in \N$ such that for all $x\in \eW$ with $\ell(x)>N_b$ we have
\begin{equation}X_x(b)\neq\emptyset \Longleftrightarrow X_x(b_{\operatorname{b}})\neq\emptyset.\end{equation}
Moreover, if these affine Deligne--Lusztig varieties are both nonempty, then
\begin{equation}\label{E:ConjDim}
\dim X_x(b) = \dim X_x(b_{\operatorname{b}}) - \frac{1}{2}\left( \langle 2\rho, \nu_b \rangle + \operatorname{def}_G(b) - \operatorname{def}_G(b_{\operatorname{b}}) \right).\end{equation}  
\end{conjecture} 

\noindent At the time Conjecture \ref{Conj_GHKR} was formulated, the primary evidence for it consisted of computer experiments.  It has since only been fully verified in type $\tilde A_2$ in the work of Yang~\cite{Yang}.  

One common approach in many of the aforementioned papers treating the basic case is a generalization of the classical Deligne--Lusztig theory of \cite{DL}. This permits a reduction to combinatorics on minimal length elements in conjugacy classes inside the extended affine Weyl group.  Although fruitful in the basic case, these methods have so far been less successful for proving nonemptiness or dimension formulas when $b$ is not basic.
To prove Conjecture~\ref{Conj_GHKR} in general, it seems that a new approach is required.  

Our primary results, which are described in detail in Section~\ref{Results}, represent a sharpening of much of Conjecture~\ref{Conj_GHKR} for the case in which $b$ is a pure translation in the affine Weyl group $\aW$.  We prove both nonemptiness statements and dimension formulas for affine Deligne--Lusztig varieties of the form $X_x(t^\mu)$.  The results in this paper represent some of the first concrete progress towards these foundational questions beyond the basic case in the affine flag variety, apart from several low rank examples and special cases.  


\subsection{Key ideas in this approach}\label{Tools}

Before reviewing the precise statements of our main results in Section~\ref{Results}, we provide an overview of the key ideas which inspired our approach.  A more technical overview of the proof techniques is provided in Section~\ref{Outline}, in which we outline both our strategy and the overall structure of the paper.

The springboard for the proof of our main theorems is a result of G\"{o}rtz, Haines, Kottwitz, and Reuman in \cite{GHKR} which expresses the dimensions of affine Deligne--Lusztig varieties associated to translation elements in terms of dimensions of the intersection of unipotent and Iwahori orbits inside the affine flag variety; we recall this result as Theorem~\ref{T:GHKRDim}. Although an exhaustive search algorithm for computing the dimensions of these orbits in terms of matrices is described in \cite{GHKR}, this algebraic algorithm is typically only feasible for verifying specific examples. 

The first key observation is that the dimensions of precisely these intersections of unipotent and Iwahori orbits can be calculated using the combinatorial model of the \emph{labeled folded alcoves walks} introduced by Parkinson, Ram, and C. Schwer in \cite{PRS}. Similarly, we could have used the folded galleries of  \cite{GaussentLittelmann}. Generally speaking, the dimension of the intersection of these orbits is calculated by counting the number of folds and positive crossings in a gallery which has been reflected according to rules governed by a choice of periodic orientation on the affine hyperplanes.  As such, the algebraic algorithm discussed in \cite{GHKR} can be replaced by an algorithm involving the geometry and combinatorics of alcove walks in the standard apartment of the associated Bruhat--Tits building.   In practice, however, as for the algorithm discussed in \cite{GHKR}, using labeled folded alcove walks to compute dimensions of affine Deligne--Lusztig varieties grows rapidly in combinatorial complexity as the length of $x \in \aW$ and/or the rank of the group increases.

The second crucial insight, which enables one to simultaneously efficiently calculate dimensions for infinite families of affine Deligne--Lusztig varieties, is to apply the \emph{root operators} of Gaussent and Littelmann \cite{GaussentLittelmann}.  The key to relating the dimensions of $X_x(b)$ and $X_x(b_{\operatorname{b}})$ as in Conjecture \ref{Conj_GHKR} is that one can track precisely how these root operators affect the dimension of a labeled folded alcove walk.  This work of Gaussent and Littelmann on \emph{combinatorial galleries} takes place inside the affine Grassmannian, and so one additional task performed in the present paper is the adaptation of results in \cite{GaussentLittelmann} to the more refined context of the affine flag variety.  This adaptation is quite delicate in places.

At the heart of the paper are many arguments which are inspired by work in the fields of geometric group theory and combinatorial representation theory, even though our primary applications are algebro-geometric.  For the sake of clarity, we reserve the term ``geometric'' for our arguments involving the affine building, while we use ``algebro-geometric'' to refer to applications to the associated family of varieties.  Our geometric approach is inspired by the work of Kapovich and Millson~\cite{KapMillson}, Littelmann \cite{LittelmannLR}, and also McCammond and Petersen~\cite{McCammondPetersen}, among others, in which affine Weyl groups are viewed as groups of isometries of Euclidean space, and affine buildings are regarded as metric spaces of nonpositive curvature.  For example, we make critical and repeated use of geometric arguments such as tracking the effect of translations, changing the identification of the alcoves in an apartment with the elements of the affine Weyl group $\aW$, and acting on galleries by elements of $\aW$.  

In particular, the work of Kapovich and Millson in \cite{KapMillson} on Hecke paths uses the framework of projections of geodesics in affine  buildings to Weyl chambers to prove the saturation conjecture for complex semisimple Lie groups. 
Such projections arise as images of retractions in the affine building centered at either an alcove or the boundary. Related methods using images of galleries under retractions appear in the convexity results of the second author \cite{Hitzel, Hitzel2}, the first of which makes use of the work of Gaussent and Littelmann \cite{GaussentLittelmann}.  

The driving force behind many of these geometric ideas is an underlying connection to the representation theory of complex symmetrizable Kac--Moody algebras, as established in the work on Littlewood--Richardson rules in terms of the Littelmann path model in \cite{LittelmannLR}.


\subsection{Summary of main results}\label{Results}

This section provides an overview of the main results  of the paper, which are stated as Theorems~\ref{Shrunken}--\ref{Padic} below, and are proved in Sections~\ref{sec:Xx1ShrunkenDominant}--\ref{sec:SymmetriesLabelings}.

Given $y \in G(F)$, the isomorphism $g \mapsto yg$ on $G(F)$ yields an isomorphism between $X_x(b)$ and $X_x(yb\sigma(y)^{-1})$.  It therefore suffices to consider affine Deligne--Lusztig varieties associated to $\sigma$-conjugacy classes $[b]$, rather than simply elements $b \in G(F)$.  In fact, every element in $G(F)$ is $\sigma$-conjugate to an element in the extended affine Weyl group $\widetilde{W}$; see Section~7.2 in \cite{GHKRadlvs}.  Thus it is enough to study affine Deligne--Lusztig varieties $X_x(b)$ where $b \in \widetilde{W}$.  In addition, since our arguments rely almost exclusively on combinatorics in the Bruhat--Tits building, in this paper we focus on elements $b \in \aW$ in a single connected component, since one can easily adapt our results to the context of extended alcoves. Note that every basic element $b \in \aW$ in the affine Weyl group belongs to the same $\sigma$-conjugacy class as $\id$, while the pure translations in $\eW$ represent a large proportion of $\sigma$-conjugacy classes in $G(F)$, namely those corresponding to integral Newton points.  In particular, the dominant pure translations in $\aW$ represent infinitely many pairwise distinct $\sigma$-conjugacy classes in $G(F)$; see the patterned (pink) alcoves in Figure~\ref{fig:BasicElementsConjugacyA2}.

For $b = t^\mu$ a pure translation in $\aW$, the natural associated basic element $b_{\text{b}}$ is the identity.  As such, Conjecture~\ref{Conj_GHKR} in the case of pure translations relates nonemptiness of the affine Deligne--Lusztig varieties $X_x(t^{\mu})$ and $X_x(1)$.   Now, for both the identity element and pure translations, the defect terms equal zero.  In addition, for $b = t^\mu$ the Newton point $\nu_b$ equals the unique dominant element in the $\sW$-orbit of $\mu$, which we denote by $\mu^+$.  In particular, if $\mu$ is dominant then $\nu_b = \mu$.  Thus, Conjecture~\ref{Conj_GHKR} predicts that the dimension of $X_x(t^\mu)$ equals $\dim X_x(1) - \langle \rho, \mu^+ \rangle$.

We first consider the alcoves $x$ which lie in the shrunken dominant Weyl chamber.  The following statement  combines Theorems~\ref{idDomDimw0},~\ref{idDomDim},~\ref{bDomDimw0}, and~\ref{bDomDim} in the body of the paper.

\begin{thmi}\label{Shrunken} Let $x = t^\lambda w \in W$ be such that every alcove at the vertex $\lambda$ lies in the shrunken dominant Weyl chamber. Let $a = t^{2\rho}w$.  Let $b = t^\mu$ be a dominant pure translation such that the alcove $b$ lies in the convex hull of $x$ and the base alcove, the alcove $t^{-\mu}x$ lies in the shrunken dominant Weyl chamber, and the vertex $\mu$ lies in the negative cone based at the dominant vertex $\lambda - 2\rho$.  Then
\begin{equation}
X_a(1) \neq \emptyset \iff X_x(1) \neq \emptyset 
\end{equation}
and 
\begin{equation}\label{E:ShrunkenDim}
X_x(1) \neq \emptyset \implies X_x(b) \neq \emptyset.
\end{equation}
If $w = w_0$ then $X_a(1) \neq \emptyset$, $X_x(1) \neq \emptyset$, and $X_x(b) \neq \emptyset$.

Moreover if the affine Deligne--Lusztig varieties $X_a(1)$, $X_x(1)$, and $X_x(b)$ are nonempty, then we have
\begin{equation}
\dim X_x(1) = \dim X_a(1) + \frac{1}{2}\ell(t^{\lambda - 2\rho}) = \frac{1}{2}\ell(t^\lambda)
\end{equation}
and
\begin{equation}\label{E:MainThmDimIneq}
\dim X_x(b) = \dim X_x(1) - \langle \rho, \mu \rangle = \langle \rho, \lambda - \mu \rangle. 
\end{equation}
\end{thmi}

For arbitrary $x \in \aW$, the following two theorems establish nonemptiness implications and dimension inequalities.   In these statements, $\rho_{B^-}$ is the half-sum of the roots which are positive for the opposite Borel $B^-$ and $\mu_{B^-}$ is the unique element in the $\sW$-orbit of $\mu$ which is dominant for $B^-$.  The next result is proved as Theorem~\ref{thm:relabel shift}.

\begin{thmi}\label{ForwardShift}  
Let $x\in\aW$. Let $b = t^\mu$ be a pure translation such that the alcove corresponding to~$b$ lies in the convex hull of $t^\mu x$ and the base alcove. Then 
\begin{equation}\label{E:ForwardShiftNonempty} 
X_x(1) \neq \emptyset \implies X_{t^\mu x}(t^\mu) \neq  \emptyset. 
\end{equation}
Moreover, if these affine Deligne--Lusztig varieties are both nonempty, then 
\begin{equation}\label{E:ForwardShiftDim} 
\dim  X_{t^\mu x}(t^\mu) \geq \dim X_x(1) - \langle \rho_{B^-}, \mu + \mu_{B^-} \rangle
\end{equation}
and if $\mu$ is dominant then 
\begin{equation}\label{E:ForwardShiftDimDom}
\dim  X_{t^\mu x}(t^\mu) \geq \dim X_x(1).
\end{equation}
\end{thmi}

\noindent In notes based on his ICCM 2013 talk, He announced a result which has some overlap with Theorem~\ref{ForwardShift} and thus provided some more evidence for Conjecture~\ref{Conj_GHKR} in the case of split $b$; compare Theorem 6.3 in \cite{HeNoteOn}.

We prove the following result as Theorem~\ref{ArbitraryNonemptyDim}.

\begin{thmi}\label{NonemptyDim} 
Let $x \in \aW$. Let $b = t^\mu$ be a pure translation and assume that the alcove corresponding to $b$ lies in the convex hull of $x$ and the base alcove, the alcoves $x$ and $t^{-\mu}x$ lie in the same Weyl chamber, and if $x$ is in a shrunken Weyl chamber then $t^{-\mu}x$ is in the same shrunken Weyl chamber.  Then
\begin{equation}\label{E:Nonempty}
X_x(1) \neq \emptyset \implies X_x(b) \neq \emptyset.
\end{equation}
Moreover,  if these affine Deligne--Lusztig varieties are both nonempty, then 
\begin{equation}\label{E:Dim}
\dim X_x(b) \geq \dim X_x(1) - \langle \rho, \mu^+ \rangle - \langle \rho_{B^-}, \mu + \mu_{B^-} \rangle
\end{equation}
and if $\mu = \mu^+$ is dominant then 
\begin{equation}\label{E:DimDom}
\dim  X_{x}(b) \geq \dim X_x(1)  - \langle \rho, \mu \rangle.
\end{equation}
\end{thmi}

The next theorem provides sharper results for alcoves outside of the dominant Weyl chamber, by conjugating $x$ in the shrunken dominant Weyl chamber by elements $u \in \sW$ to obtain information about the varieties $X_{u^{-1}xu}(t^\mu)$.  The following is proved as Theorem~\ref{T:Conjugation}.
 
\begin{thmi}\label{OtherWeyl}
Let $x = t^\lambda w \in \aW$ be such that every alcove at the vertex $\lambda$ lies in the shrunken dominant Weyl chamber.  Let $b = t^\mu$ be a dominant pure translation and $u$ any element of $\sW$. Then
\begin{equation}
X_x(b) \neq \emptyset \implies X_{u^{-1}xu}(b) \neq \emptyset.
\end{equation}
Moreover, if these affine Deligne--Lusztig varieties are both nonempty, then
\begin{equation}\label{E:correction}
\dim X_{u^{-1}xu}(b) \geq \dim X_x(b) + \frac{1}{2}(\ell(u^{-1}xu) - \ell(x)).
\end{equation}
\end{thmi}

\noindent In Theorem~\ref{T:Conjugationb=1} we consider the special case $b = 1$, and under similar hypotheses to Theorem~\ref{OtherWeyl}, we obtain that $X_x(1) \neq \emptyset \iff X_{u^{-1}xu}(1) \neq \emptyset$ and an equality concerning the dimensions of these varieties.

We emphasize that our Theorems~\ref{Shrunken}--\ref{OtherWeyl} apply to alcoves $x$ which lie close to the walls of one of the finite Weyl chamber, including alcoves $x$ which are close to the origin and (for Theorems~\ref{ForwardShift} and~\ref{NonemptyDim}) alcoves $x$ which are not in shrunken Weyl chambers.  Figure~\ref{fig:x_close_to_walls} gives some examples.

\begin{figure}[ht]
\begin{center}
\begin{overpic}[width=0.8\textwidth]{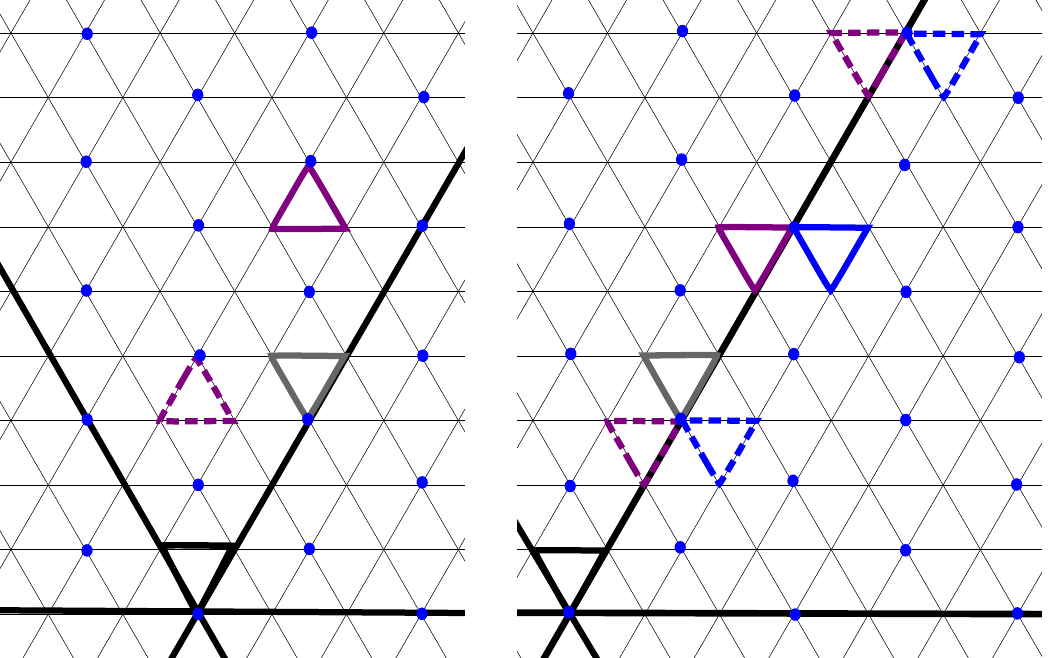}
\put(28.5,26){\footnotesize{$\color{black} \bb$}}
\put(29,20){\footnotesize{$\color{black} \mu$}}
\put(29,42.5){\footnotesize{$\color{magenta} \x$}}
\put(11,26){\footnotesize{$\color{magenta} t^{-\mu}\x$}}
\put(64,26){\footnotesize{$\color{black} \bb$}}
\put(64,19){\footnotesize{$\color{black} \mu$}}
\put(71,38){\footnotesize{$\color{magenta}  \x'$}}
\put(78,38){\footnotesize{$\color{blue} \x''$}}
\put(52,19){\footnotesize{$\color{magenta} t^{-\mu} \x'$}}
\put(72,19){\footnotesize{$\color{blue} t^{-\mu} \x''$}}
\put(75,55){\footnotesize{$\color{magenta} t^\mu \x'$}}
\put(92,54){\footnotesize{$\color{blue} t^\mu \x''$}}
\end{overpic}

\caption{We illustrate some alcoves $x$ and $b = t^\mu$ for which Theorems~\ref{Shrunken}--\ref{OtherWeyl} hold, where $x$ is close to a wall of a finite Weyl chamber.  Theorems~\ref{Shrunken} and~\ref{OtherWeyl} both hold for the pair $(x,b)$ as on the left, and Theorems~\ref{ForwardShift} and~\ref{NonemptyDim} both hold for the pairs $(x',b)$ and $(x'',b)$ on the right.}
\label{fig:x_close_to_walls}
\end{center}
\end{figure}
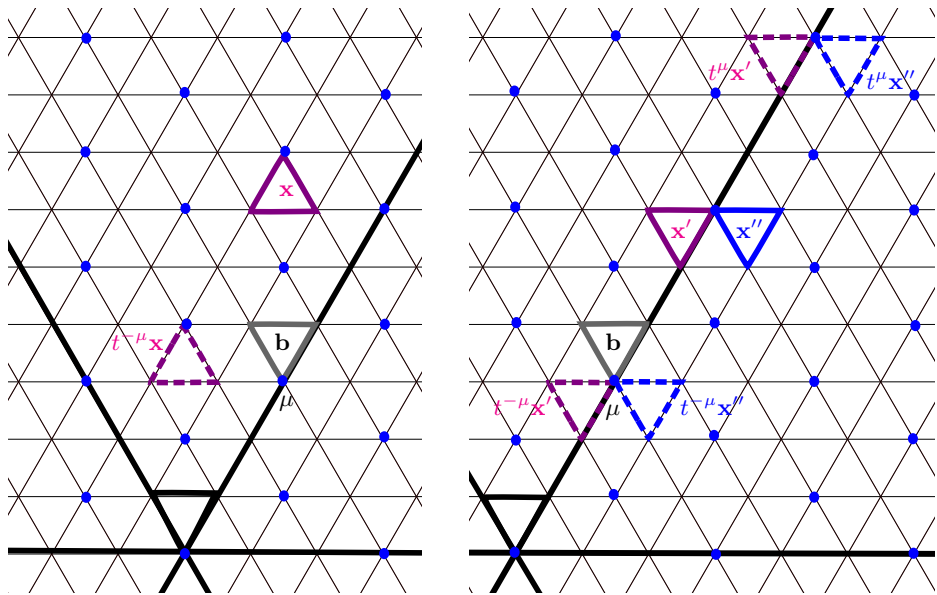

We also consider the effect of diagram automorphisms $g:\aW \to \aW$.  We prove geometrically that $X_x(b) \neq \emptyset$ if and only if $X_{g(x)}(g(b)) \neq \emptyset$, and that if both are nonempty then their dimensions agree.  This statement will not surprise experts, but since we are able to provide a short constructive proof, we include it in the body of the paper as Theorem~\ref{DiagramSymmetry}.

All of our techniques rely upon geometric and combinatorial algorithms in the standard apartment of the Bruhat--Tits building associated to $G(F)$.  This implies that our analysis extends beyond the function field context to the $p$-adic case, which is also sometimes called the mixed characteristic case (compare~\cite{GHN}).  Denote by $\widehat{\Q}_p^{\mbox{\tiny{ur}}}$ the maximal unramified extension of the field of $p$-adic numbers.  In this $p$-adic setting affine Deligne--Lusztig sets in $G(\widehat{\Q}_p^{\mbox{\tiny{ur}}})/I$ are defined analogously to the varieties $X_x(b)$, and are denoted by $X_x(b)_{\Q_p}$.  Recall Corollary 11.3.5 in \cite{GHKRadlvs} which says that $X_x(b) \neq \emptyset$ if and only if $X_x(b)_{\Q_p} \neq \emptyset$.  As predicted in \cite{GHKRadlvs}, our results prove that, subject to the formulation of a reasonable notion of the dimension of $X_x(b)_{\Q_p}$, the dimension of this affine Deligne--Lusztig set agrees with that of the variety $X_x(b)$ (for the $x, b \in \aW$ that we consider).  The following result summarizes Definition \ref{def:padic} and Theorem~\ref{T:padic} in the body of the paper, and implies that Theorems~\ref{Shrunken}--\ref{OtherWeyl} all hold in the $p$-adic setting.

\begin{thmi}\label{Padic}
Let $x \in\aW$, and let $b \in \aW$ be a pure translation.  There is a reasonable definition of the dimension of $X_x(b)_{\Q_p}$ in terms of combinatorics in the Bruhat--Tits building associated to $G(\widehat{\Q}_p^{\mbox{\tiny{ur}}})$, and using this definition
\begin{equation}
\dim X_x(b) = \dim X_x(b)_{\Q_p}. 
\end{equation}
\end{thmi}

Having stated our main results, we highlight that in Theorems~\ref{Shrunken}--\ref{OtherWeyl} we have replaced the existence of an (unspecified) lower bound on the length of $x$ in the original statement of Conjecture~\ref{Conj_GHKR} with precise convexity conditions on the relationship between $x$ and $b$.  In the discussion of our methods in Section~\ref{Outline}, we explain why these convexity conditions are natural in the context of our proofs. On the other hand, the experiments carried out in~\cite{GHKR} and other examples we have considered illustrate that our hypotheses in Theorems~\ref{Shrunken}--\ref{OtherWeyl} are not sharp. 

We also remark that the nonemptiness implications~\eqref{E:ShrunkenDim},~\eqref{E:ForwardShiftNonempty}, and~\eqref{E:Nonempty}, as well as all of our  dimension results, transfer information about the varieties $X_x(1)$, where nonemptiness and dimensions are known from the basic case, to the varieties $X_x(t^\mu)$.  In fact, as discussed further in Section~\ref{Outline}, we prove that for $x$ in the shrunken dominant Weyl chamber, it suffices to consider the questions of nonemptiness and dimensions for the finitely many varieties $X_a(1)$, where $a = t^{2\rho}w$, in order to obtain information about the infinite family $X_x(t^{\mu})$.  

The methods developed in this work also recover some nonemptiness and dimension results for the basic case in a manner which sheds new light on why those results are true.  The papers which adapt the classical arguments of Deligne and Lusztig from \cite{DL} to prove results in the basic case typically use an inductive argument on the length of an affine Weyl group element inside its conjugacy class.  Our constructive treatment of the $b=1$ case gives a new perspective on the nonemptiness criteria appearing in the original conjecture of Reuman \cite{Reu} which was generalized in \cite{GHKR}; see Section \ref{sec:Reuman} for more discussion of this comparison.  In fact, our approach also reveals previously unknown relationships between affine Deligne--Lusztig varieties in the basic case.

Our methods could feasibly be carried out to complete the picture for any specific $x$ and all pure translations $b$ with $\ell(b) \leq \ell(x)$; it is easy to show that this inequality is a necessary condition for $X_x(b)$ to be nonempty.  Theorems~\ref{ForwardShift},~\ref{NonemptyDim}, and~\ref{OtherWeyl} also provide simple and elegant means for treating the alcoves $x$ which lie outside the shrunken Weyl chambers, which is traditionally the most difficult situation to handle.  Our approach also naturally establishes both nonemptiness and bounds on dimension at the same time.  Finally, we point out that our methods are rank- and type-free.


\subsection{Outline of proof and organization of the paper}\label{Outline}

We now outline our strategy in more detail and describe the structure of the paper. In this section, $b=t^\mu$ always denotes a pure translation in the affine Weyl group $\aW$. 

\subsubsection*{Key steps in the proof.}\label{KeySteps}
The key steps of our approach are:
\begin{itemize}
\item[(1)] We relate nonemptiness of the variety $X_x(b)$ to the existence of certain positively folded galleries, such as those appearing in Figure \ref{fig:Xa1toXxbPropFigIntro}, and we relate $\dim X_x(b)$ to the number of folds and positive crossings of these galleries. 
\item[(2)] We then construct and manipulate such galleries using root operators, combinatorics in the associated Coxeter complex, and geometric transformations.
\end{itemize}

\begin{figure}[ht]
\centering
\includegraphics[width=0.5\textwidth]{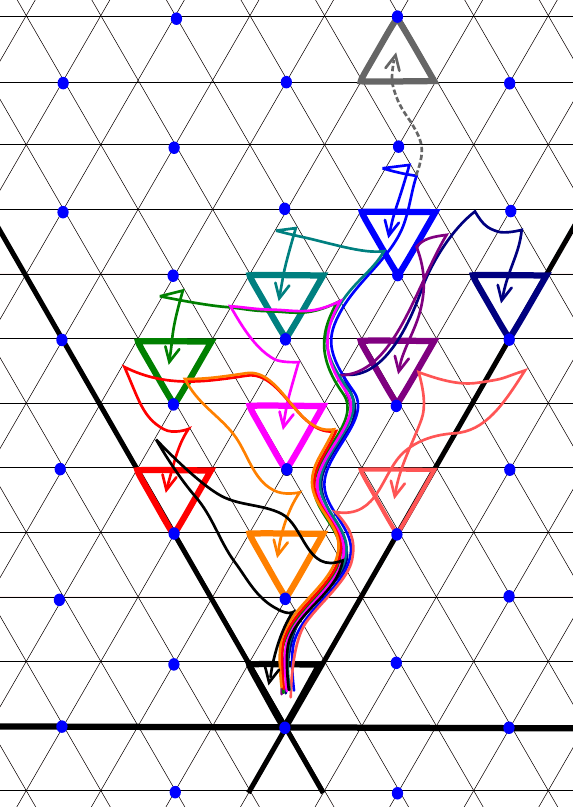}
\caption{An example of the families of galleries constructed in Section~\ref{sec:Construction}.}
\label{fig:Xa1toXxbPropFigIntro}
\end{figure}

Step (1) is essentially covered in Sections~\ref{sec:labelings}--\ref{sec:ADLVGalleries}. These two relationships (nonemptiness and dimension) go hand-in-hand, and they are established using results of  G\"{o}rtz, Haines, Kottwitz and Reuman \cite{GHKR}, and modified versions of the work of Gaussent and Littelmann \cite{GaussentLittelmann} and Parkinson, Ram, and C. Schwer \cite{PRS}. The key observation is the fact that the labeled folded alcove walks of \cite{PRS}, respectively the positively folded galleries of \cite{GaussentLittelmann},  index intersections of unipotent and Iwahori orbits in the affine flag variety; see Figure \ref{fig:Xa1toXxbPropFigIntro} for an illustration of these galleries.  In \cite{GHKR}, the dimensions of affine Deligne--Lusztig varieties corresponding to pure translations are expressed in terms of the intersections of intersections of Iwahori and unipotent orbits, which allows us to compute the dimensions of affine Deligne--Lusztig varieties using dimensions of certain folded galleries. The dimension of such galleries is computed by counting folds and positive crossings.

Having established (1), the problem is reduced to a problem of a combinatorial and geometric flavor in the affine Coxeter complex of the associated affine Weyl group: one needs to explicitly construct and manipulate positively folded galleries while keeping track of their type, their start and end alcoves, as well as their folds and crossings. Although we do use on the machinery Gaussent and Littelmann developed for the context of the affine Grassmannian, the adaptation of the corresponding results from \cite{GaussentLittelmann} to the context of the affine flag variety is quite delicate in places and thus comprises the main body of the paper; see in particular Sections~\ref{sec:Construction}--\ref{sec:Conjugation}.  

\subsubsection*{Organization of the paper}
We now provide more detailed comments on the organization of the paper. In Section~\ref{sec:BuildingsAlcoveWalks} we briefly recall standard definitions and results on reductive groups over local fields and their associated buildings, and proceed to discuss several key geometric tools in Section~\ref{sec:labelings}.  Of particular importance in Section~\ref{sec:labelings} is the notion of a periodic orientation on the affine hyperplanes which comes from a choice of labeling ``at infinity''; \textit{i.e.} at the boundary of the standard apartment of the building.  Section~\ref{sec:labelings} also reviews and connects the notions of the combinatorial galleries introduced by Gaussent and Littelmann \cite{GaussentLittelmann} and the labeled folded alcove walks defined by Parkinson, Ram, and C. Schwer \cite{PRS}.  This connection is critical for extending results in \cite{GaussentLittelmann} from the context of the affine Grassmannian to the affine flag variety.

To each labeled folded alcove walk, we associate a statistic which counts the number of folds and positive crossings, and we show in Section~\ref{sec:GalleryDim} that in certain situations this statistic coincides with Gaussent and Littelmann's definition of the dimension of a combinatorial gallery.  We then recall the root operators introduced by Gaussent and Littelmann~\cite{GaussentLittelmann} and discuss their effect on dimension.  The ability to relate dimensions of pairs of varieties in this paper stems from the application of the root operators of \cite{GaussentLittelmann} to labeled folded alcove walks, as illustrated in Figure~\ref{fig:Xa1toXxbPropFigIntro}.

Section~\ref{sec:GalleryDim} also includes closed formulas and inequalities concerning the dimension of a gallery, which may be of independent interest.   Each labeled folded alcove walk in the sense of \cite{PRS} starts by choosing a minimal gallery from the fundamental alcove to another alcove in the standard apartment.  In many places in our proofs, it is crucial that our notion of dimension is independent of this initial choice of minimal gallery, and we prove this independence (which is implicit in~\cite{PRS}) in Section~\ref{sec:GalleryDim} as well. 

The critical observation which permits the use of the combinatorics of labeled folded alcove walks to answer questions about the nonemptiness and dimensions of affine Deligne--Lusztig varieties is the fact that the labeled folded alcove walks of \cite{PRS} naturally index intersections of unipotent and Iwahori orbits in the affine flag variety.  In Section~\ref{sec:ADLVGalleries} we review a result of G\"{o}rtz, Haines, Kottwitz, and Reuman from \cite{GHKR} expressing the dimensions of affine Deligne--Lusztig varieties corresponding to pure translations in terms of the intersections of such orbits.  We are then able to recast the dimensions of these intersections as dimensions of certain folded galleries, completing Step (1) from page~\pageref{KeySteps}. 

We then proceed with Step (2), and thus the proofs of our main results.  Section~\ref{sec:Construction} provides explicit constructions of infinite families of positively folded galleries, which are crucial to the proof of Theorem~\ref{Shrunken}.  We begin with an explicit construction of a single gallery $\sigma_{a_0}$ which ends at the identity element $\id$.  We then apply geometric transformations and root operators to $\sigma_{a_0}$ to construct families of galleries ending at pure translations, as depicted in Figure~\ref{fig:Xa1toXxbPropFigIntro}.

Next, Section~\ref{sec:Xx1ShrunkenDominant} relates the varieties $X_x(1)$ and $X_a(1)$, where $x$ is in a shrunken Weyl chamber and $a$ is a single alcove based at a specific vertex close to the origin, as in the statement of Theorem~\ref{Shrunken}.  We remark that the results in Section~\ref{sec:Xx1ShrunkenDominant} can be deduced from the existing literature on the basic case, but that our approach uses the constructive machinery established in Section~\ref{sec:Construction}, revealing connections between pairs of varieties which were not previously known to be related.  

In Section~\ref{sec:RelabelingSymmetry}, we establish Theorems~\ref{ForwardShift} and~\ref{NonemptyDim}, and we then complete the proof  of Theorem~\ref{Shrunken}.  To prove Theorem~\ref{ForwardShift}, we transform positively folded galleries by applying translations $t^\mu$, which yields a relationship between the varieties $X_x(1)$ and $X_{t^\mu x}(t^\mu)$ for any $x$ under a mild convexity hypothesis.  This process is illustrated in Figure~\ref{fig: intro relabel shift}.  The reason for the convexity hypothesis is that our proof requires the existence of a minimal gallery from $\id$ to $t^\mu x$ which passes through the alcove $t^\mu$.  The proof of Theorem~\ref{NonemptyDim} then combines results of~\cite{GoertzHeDim},~\cite{GHN}, and~\cite{HeAnnals} for the case $b = 1$ with Theorem~\ref{ForwardShift}.

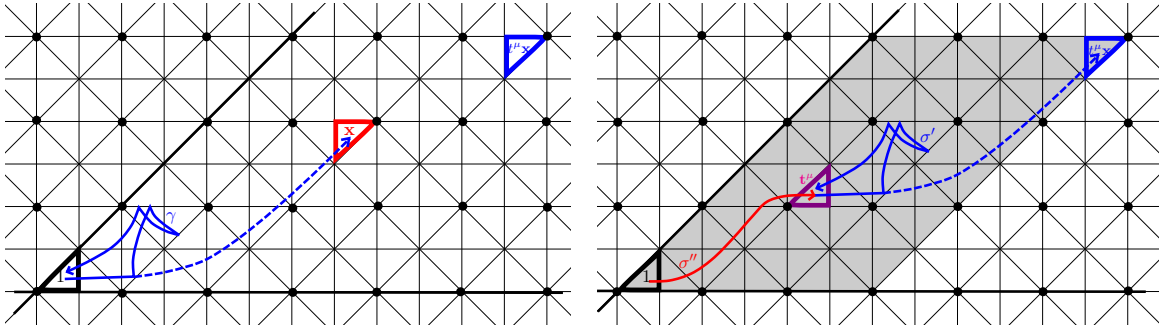
\begin{figure}[ht]
\begin{center}
\resizebox{\textwidth}{!}
{
\begin{overpic}{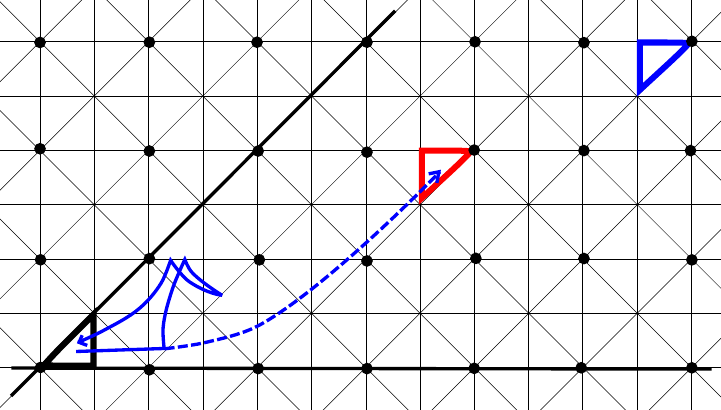}
\put(28.5,18.5){\bf $\color{blue} \gamma$}
\put(9,7.5){$\color{black} \id$}
\put(60,33.5){$\color{red} \x$}
\put(89,48){\footnotesize{$\color{blue} t^{\mu} \x$}}
\end{overpic}
\hspace{2ex}
\begin{overpic}{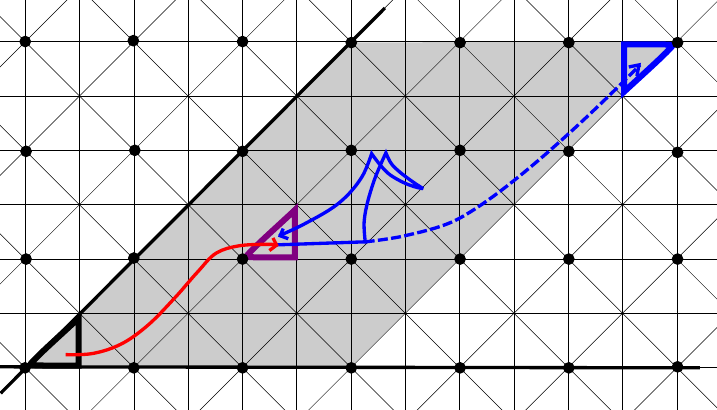}
\put(14.5,10){$\color{red} \sigma''$}
\put(57.5,32.5){\bf $\color{blue} \sigma'$}
\put(8, 7.5){$\color{black} \id$}
\put(36,25){\footnotesize{$\color{magenta} \mathbf{t}^\mu$}}
\put(87.5,48){\footnotesize{$\color{blue} t^{\mu} \x$}}
\end{overpic}

}
\caption{We illustrate the proof of Theorem~\ref{ForwardShift}. The positively folded gallery $\gamma$ on the left shows that $X_x(1)$ is nonempty and can be used to compute $\dim X_x(1)$.  We translate $\gamma$ by $t^\mu$ to obtain the gallery $\sigma'$ on the right.  The gallery obtained by concatenating $\sigma'$ with the gallery $\sigma''$ from $\id$ to $t^\mu$ shows that $X_{t^\mu x}(t^\mu)$ is nonempty and gives a lower bound on $\dim X_{t^\mu x}(t^\mu)$. The alcoves in the convex hull of $t^\mu x$ and the base alcove are shaded light gray on the right.}
\label{fig: intro relabel shift}
\end{center}
\end{figure}

Section~\ref{sec:RelabelingSymmetry} also contains a key step for the proof of Theorem~\ref{Shrunken}, which is to show that the galleries constructed in Section~\ref{sec:Construction} attain the dimension of $X_{x}(t^\mu)$ when $x = t^\lambda w_0$ is in the shrunken dominant Weyl chamber, $\mu$ is dominant, and suitable convexity hypotheses hold.  The argument here uses the LS-galleries of Gaussent and Littelmann~\cite{GaussentLittelmann}.  This completes the proof of Theorem~\ref{Shrunken} for $x$ with spherical direction $w_0$.  As with Theorem~\ref{ForwardShift}, our convexity hypotheses ensure the existence of minimal galleries which are required for our proofs.  

The work of Gaussent and Littelmann \cite{GaussentLittelmann} mostly takes place in the affine Grassmannian $G(F)/K$, which translates to questions about end-vertices of galleries, and the methods in \cite{GaussentLittelmann} are tailored to tracking images of such vertices under root operators and retractions.  We are interested in intersections of double cosets in the affine flag variety $G(F)/ I$, which translates to questions about end-alcoves.  Since the tracking of alcoves is more subtle, we are only able to give explicit constructions when $x = t^\lambda w_0$ and is in the shrunken dominant Weyl chamber.  We complete the proof of Theorem~\ref{Shrunken} for the other cases using the results of Sections~\ref{sec:Construction} and~\ref{sec:Xx1ShrunkenDominant}, as well as Theorem~\ref{NonemptyDim}.  Section~\ref{sec:RelabelingSymmetry} also contains a discussion of obstructions to extending our constructions, including the formulation of a question whose answer would represent significant progress, and a discussion of further expected relations between  \cite{GaussentLittelmann} and the present work, pertaining to crystals and LS-galleries.

In Section~\ref{sec:Conjugation}, we demonstrate how to take a gallery for the variety $X_x(t^\mu)$, ``conjugate" it by a simple reflection $s$, and extend it at the beginning and end to obtain a gallery which can be used to study the variety $X_{sxs}(t^{\mu})$.  This process is illustrated by Figure~\ref{fig:ConjugatedGalleryIntro}.  The proof of Theorem~\ref{OtherWeyl}  then follows by applying a sequence of such ``conjugations".  The arguments in Section~\ref{sec:Conjugation} represent constructive geometric methods which illuminate the inductive algebraic and combinatorial arguments that arise in the existing literature on the basic case.  In Section~\ref{sec:SymmetriesLabelings}, we prove Theorem~\ref{DiagramSymmetry}, which considers the effect of diagram automorphisms.  

\begin{figure}[ht]
\begin{center}
\begin{overpic}[width=0.8\textwidth]{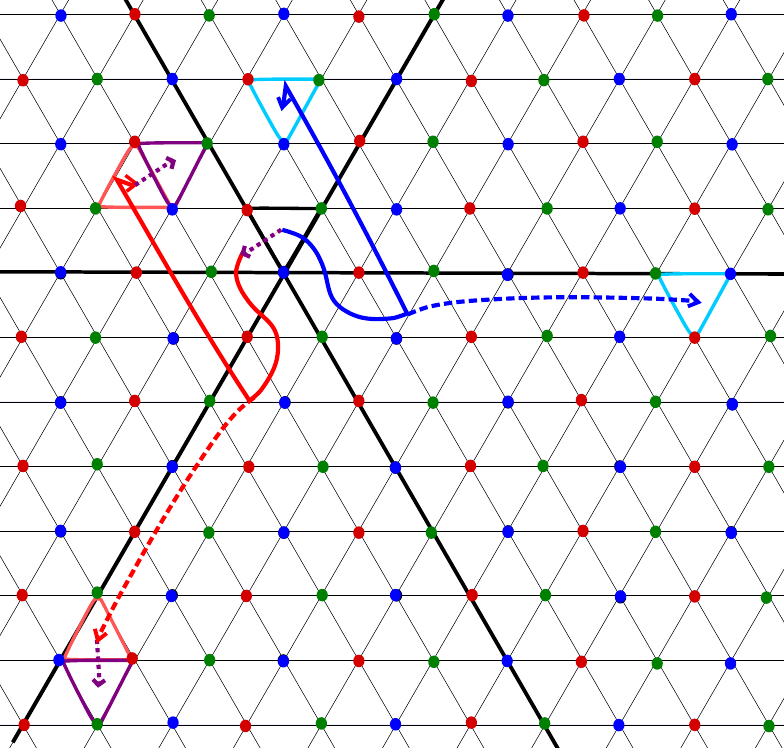}
\put(89,58){\footnotesize{$\color{blue} \x$}}
\put(45,70){\footnotesize{$\color{blue} \sigma$}}
\put(33,80){\footnotesize{$\color{blue} \bb$}}
\put(16.5,13){\footnotesize{$\color{red} s\x$}}
\put(14.9,4.5){\footnotesize{$\color{purple} \x^s$}}
\put(14.5,62){\footnotesize{$\color{red} \sigma^s$}}
\put(24.6,70){\footnotesize{$\color{purple} \bb^s$}}
\put(10,73){\footnotesize{$\color{red} s\bb$}}
\end{overpic}

\caption{We illustrate the proof of Theorem~\ref{OtherWeyl}.  This figure shows the effect of ``conjugating" a gallery $\sigma$ by a simple reflection $s$.}
\label{fig:ConjugatedGalleryIntro}
\end{center}
\end{figure}
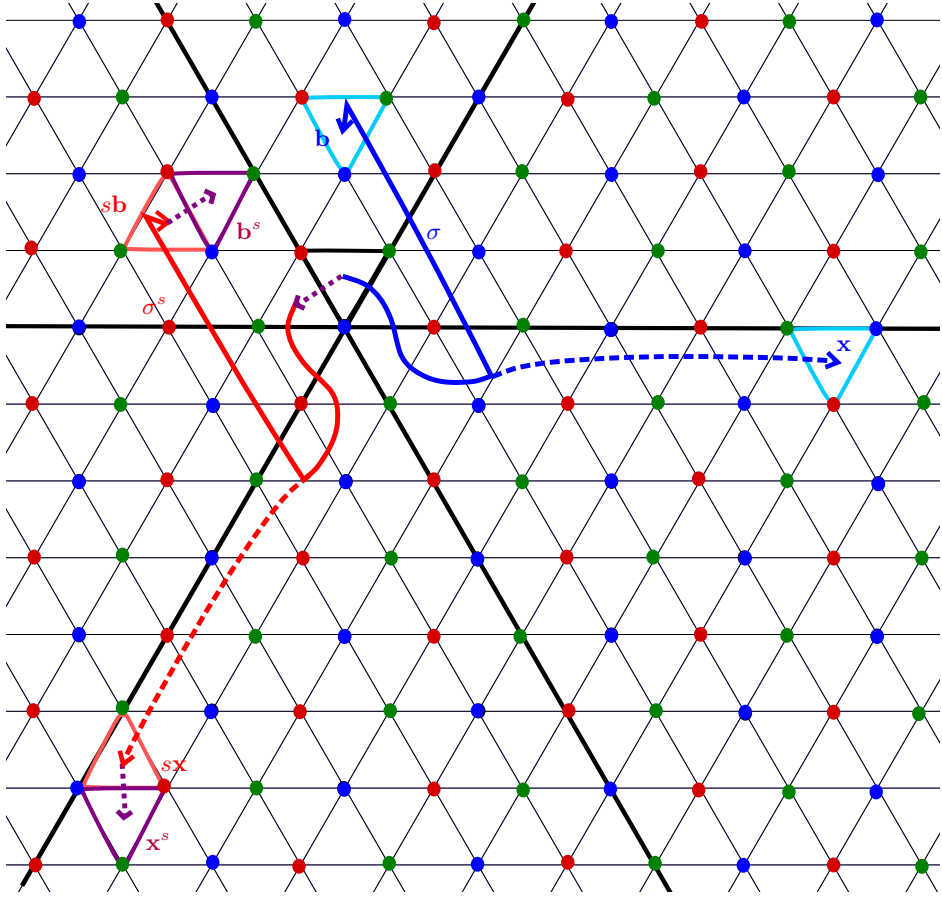


\subsection{Applications}\label{Applications}

We finish this work by providing two applications in Section~\ref{sec:Applications}.  We obtain two immediate applications, to class polynomials of the affine Hecke algebra and to affine reflection length.  The statement on class polynomials is a consequence of applying our main results to the work in \cite{HeAnnals}, while the application to reflection length directly involves our methods.  

\subsubsection{Class polynomials of the affine Hecke algebra}

A recent breakthrough of He establishes a connection between the dimensions of affine Deligne--Lusztig varieties and the degrees of class polynomials of affine Hecke algebras \cite{HeAnnals}.  Class polynomials for affine Hecke algebras were introduced by He and Nie in~\cite{HeNie12}, where the authors prove results analogous to those obtained by Geck and Pfeiffer for class polynomials of Hecke algebras associated to finite Weyl groups~\cite{GeckPfeiffer}.

The Hecke algebra $H$ for the affine Weyl group $\aW$ is a $\Z[v,v^{-1}]$-algebra with basis $\{ T_x : x \in \aW \}$ and certain defining relations.  Given $x \in \aW$ and a $\sigma$-conjugacy class $[b]$ where $b \in G(F)$, the class polynomial $f_{x,[b]}$ lies in $\Z[v,v^{-1}]$ and is a polynomial in $\Z[v- v^{-1}]$.  We review the precise definition in Section~\ref{sec:Applications}. Class polynomials can be constructed inductively, but in a way which makes it difficult to determine their degree.   

Yang~\cite{Yang} has recently computed the class polynomials in type $\tilde A_2$ and hence used He's results from~\cite{HeAnnals} to prove Conjecture~\ref{Conj_GHKR} in this case.  These computations are difficult even in type $\tilde A_2$, and it is well-known that the combinatorial complexity for such computations increases considerably as rank grows. 
However, by combining our main theorems with the work of He~\cite{HeAnnals}, we easily obtain the following conclusions about the degrees of certain class polynomials.  The next result is proved as Theorem~\ref{T:ClassShrunken}, and uses Theorem~\ref{Shrunken}.  

\begin{thmi}\label{Class}   Let $x = t^\lambda w \in W$ be such that every alcove at the vertex $\lambda$ lies in the shrunken dominant Weyl chamber $\Cfs$.  Let $b = t^\mu$ be a regular dominant pure translation.  Suppose that the alcove $b$ lies in the convex hull of the base alcove and $x$, the alcove $t^{-\mu}x$ lies in $\Cfs$, and $\mu$ lies in the negative cone based at $\lambda - 2\rho$.  Then if $X_x(1) \neq \emptyset$ we have 
\begin{equation} \deg(f_{x,[b]}) = \ell(w). \end{equation}
\end{thmi}

\noindent Recall that the emptiness pattern for $X_x(1)$ is completely known from the basic case.  We also obtain information about degrees of class polynomials from our Theorems~\ref{ForwardShift},~\ref{NonemptyDim}, and~\ref{OtherWeyl}.

\subsubsection{Reflection length in affine Weyl groups}

The affine Weyl group $\aW$, like any Coxeter group, has two natural generating sets: its Coxeter generating set, which we denote by $\tilde S$, and its set of reflections, denoted $\tilde R$.  By definition, $\tilde R$ consists of all $\aW$-conjugates of elements of~$\tilde S$.  We write $\ell_{\tilde R}(x)$ for the \emph{reflection length} of an element $x \in \aW$; that is, the smallest integer~$k$ so that $x = r_1 r_2 \cdots r_k$ with each $r_i \in \tilde R$.   

For general Coxeter groups, Dyer~\cite{Dyer} has shown that the reflection length of an element $x$ is equal to both the minimal number of simple reflections that must be deleted from a fixed reduced word for $x$ in order to get a trivial word, and to the minimal length of a path from the identity 1 to $x$ in the (directed) Bruhat graph.  In the same work, Dyer also related the reflection length of $x$ to the Kazhdan--Lusztig polynomial $R_{1,x}$.  

Reflection length for finite Weyl groups is well-understood; see, for example,~\cite{Carter}.  However, in the affine setting surprisingly little is known.  The results found in McCammond and Petersen's work \cite{McCammondPetersen} represent the state of the art.  For $\aW$ an irreducible affine Weyl group of rank $n$, they compute the exact reflection length of pure translations $t^\lambda \in \aW$, and show that $\ell_{\tilde R}(t^\lambda)$ can take any value $2k$ with $1 \leq k \leq n$.  They also prove that for $x = t^\lambda w$ with $\ell_{\tilde R}(t^\lambda) = 2k$ the following inequalities hold: $$k \leq \ell_{\tilde R}(x) \leq k + n.$$ Thus in particular, the reflection length of any element in $\aW$ is bounded above by $2n$.

 We prove the following result as Theorem~\ref{CoxeterElements}.  This statement includes the first exact calculation of reflection length in affine Weyl groups for elements other than pure translations. 

\begin{thmi}\label{Reflection}  Let $\aW$ be an irreducible affine Weyl group of rank $n$, with set of reflections~$\tilde R$, and let $\sW$ be the associated finite Weyl group, with set of reflections $R$.  Let $x = t^\lambda w \in \aW$, where $\lambda$ is a coroot and $w \in \sW$, and suppose that the alcove $x$ lies in the shrunken Weyl chamber corresponding to $u \in \sW$.  Assume that
\begin{equation}
u^{-1}wu \in \sW \backslash \bigcup\limits_{T \subsetneq S} W_T.
\end{equation}
Then for all $g \in \Aut(\aW)$ 
\begin{equation} \ell_R(w) \leq \ell_{\tilde R}(g(x)) \leq \ell(u^{-1}wu). \end{equation}
Moreover if $w$ is a Coxeter element of $\sW$, then for all $g \in \Aut(W)$ \begin{equation} \ell_{\tilde R}(g(x)) = n. \end{equation} 
\end{thmi}

After observing that reflection length is $\Aut(W)$-invariant, so that it suffices to consider $\ell_{\tilde R}(x)$, the proof of Theorem~\ref{Reflection} proceeds by considering a gallery which attains the dimension of the variety $X_x(1)$.  We then use a theorem of G\"ortz and He~\cite{GoertzHeDim} together with results from Sections~\ref{sec:GalleryDim} and~\ref{sec:ADLVGalleries} to conclude the proof.  We also obtain inequalities concerning reflection length when $x$ lies outside of the shrunken Weyl chambers.


\subsection{Acknowledgements}

We are grateful to Arun Ram for suggesting that the first and third authors would have mathematics to talk about, and for many helpful conversations throughout this collaboration.  In particular, the observation that the labeled folded alcove walk model would be applicable to this problem grew out of discussions with Arun Ram at the Spring 2013 Semester Program on ``Automorphic Forms, Combinatorial Representation Theory, and Multiple Dirichlet Series'' at the Institute for Computational and Experimental Research in Mathematics.  We also wish to thank the organizers of the June 2013 Oberwolfach workshop ``Geometric Structures in Group Theory'', during which the second and third authors realized that Gaussent and Littelmann's work on root operators was relevant to the problem.   We thank the University of Glasgow and DFG Project SCHW 1550/2-1 for travel and housing support, and Michael Cowling and the University of New South Wales for hosting a visit by the third author in August 2014.

We also thank Stephane Gaussent for illuminating remarks on root operators and  dimensions of folded galleries, Ulrich G\"ortz and Xuhua He for helpful email exchanges about their work in \cite{GoertzHeDim}, Robert Kottwitz for a conversation about the defect, Jon McCammond for discussions of reflection length, and James Parkinson for discussions of connections between labeled folded alcove walks and affine Weyl groups.

We are truly indebted to the authors of \cite{GHKR} and Ulrich G\"{o}rtz in particular for the beautiful pictures generated from the experimental evidence in support of their conjecture, and for keeping all of the pictures in the arXiv version of their paper.
 
Finally, we are grateful to Dan Bump, Benson Farb, Robert Kottwitz, Peter Littelmann, Michael Rapoport, and an anonymous referee for helpful comments on an earlier version of this manuscript.


\section{Preliminaries on Weyl groups, affine buildings, and related notions}\label{sec:BuildingsAlcoveWalks}

We begin by fixing notation and recalling standard definitions and results that will be used throughout the paper.  Section~\ref{sec:WeylGroupsRootSystems} discusses Weyl groups and root systems and Section~\ref{sec:HyperplanesAlcovesWeylChambers} recalls important features of affine buildings including hyperplanes, alcoves, and Weyl chambers.  We essentially follow \cite{Bourbaki4-6} and \cite{BruhatTits}.


\subsection{Weyl groups and root systems}\label{sec:WeylGroupsRootSystems}

Recall that $k$ is an algebraic closure of the finite field $\mathbb{F}_q$ with $q$ elements, and that $F$ is the field of Laurent series over $k$. Let $G$ be a split connected reductive group over $k$, and $T$ a split maximal torus of $G$. Fix a Borel subgroup $B = TU$, where $U$ is the unipotent radical.
We write $\sW$ for the (finite or spherical) Weyl group of $T$ in $G$, which equals $\sW = N_G(T)/T$.  Denote by $\eW = X_*(T)\rtimes \sW$ the extended affine Weyl group. For a coweight
$\lambda \in X_*(T)$, we will write $t^{\lambda}$ for the element in $T(F)$ that is the image of $t$ under the homomorphism $\lambda: \mathbb{G}_m \rightarrow T$.  We can uniquely express an element $x\in \eW$ as $x = t^{\lambda}w$, where $w \in \sW$ is the spherical direction of $x$ and $\lambda \in X_*(T)$.  Typically, elements of the spherical Weyl group will be denoted by the letters $u$, $v$, and $w$, and elements of the affine Weyl group by $x$, $y$, and $z$.  

Let $\Delta = \{\alpha_i \}_{i=1}^n$ 
be a basis of simple roots in the group of characters $X^*(T)$. Denote by $\Phi$ the set of roots of $T$ in $G$ and by $\Phi^+$ and $\Phi^-$ the corresponding set of positive and negative roots, respectively.  Denote by $\rho$ the half-sum of the positive roots.   We remark that the notion of positive and negative roots here depends upon the choice of $B$.  Throughout this paper, we assume that $\Phi$ is irreducible, and thus that the Dynkin diagram is connected.  Since $\Phi$ is irreducible, there is a unique highest root $\tilde{\alpha}$ which satisfies that $\tilde{\alpha} - \alpha$ is a sum of simple roots for all $\alpha \in \Phi^+$.  We denote by $\alpha^{\vee} = 2 \alpha / \langle \alpha, \alpha \rangle$ the coroot associated to $\alpha \in \Phi$ with respect to the inner product $\langle \cdot,  \cdot \rangle: X^*(T) \times X_*(T) \rightarrow \mathbb{Z}$.  
The basis of simple coroots is then denoted by $\Delta^{\vee}$, and we denote by $R^{\vee} = \oplus_{i=1}^n \Z \alpha_i^{\vee} \subset X_*(T)$ 
the coroot lattice. Elements of the coroot lattice will typically be denoted by Greek letters such as $\lambda$, $\mu$, and $\nu$.   Denote by $\Lambda_G$ the quotient of $X_*(T)$ by $R^{\vee}$, and by $\eta_G$ the surjection $\eta_G: G(F) \rightarrow \Lambda_G$.  The fundamental weights $\{ \varpi_i \}_{i=1}^n$
and coweights $\{ \varpi_i^{\vee} \}_{i=1}^n$ 
are dual bases to $\Delta^{\vee}$ and $\Delta$, respectively.  We denote the weight lattice by $P = \oplus_{i=1}^n \Z \varpi_i \subset X^*(T)$.

The finite Weyl group $\sW$ is a Coxeter group which is generated by the set of simple reflections $s_i$, where $s_i$ is the reflection across the hyperplane perpendicular to the $\alpha_i\in \Delta$ that passes through the origin, for $i = 1,\dots,n$. More generally, for any $\alpha \in \Phi$, we denote by $s_{\alpha}$ the reflection across the hyperplane perpendicular to $\alpha$ that passes through the origin.  
 If we put $S = \{ s_i \}_{i=1}^n$, then $(\sW, S)$ forms a Coxeter system. Denote the length of an element $w \in \sW$ by $\ell(w)$, and write $w_0$ for the longest word in $\sW$. We may also view $\sW$ as a finite reflection group, acting on Euclidean space $V = \R^n$ where $n$ is the rank of $G$, which we can identify with $X_*(T)\otimes_{\Z}\R$.  
 
 For each root $\alpha \in \Phi$ and each integer $k \in \Z$, we may consider the affine hyperplane 
$H_{\alpha, k} : = \{ v \in V \, | \, \langle \alpha, v  \rangle = k \},$
and we write $H_{\alpha, 0} = H_{\alpha}$ for convenience.  The affine reflection across $H_{\alpha,k}$ is given by
$s_{\alpha, k} (v):= v - ( \langle \alpha, v \rangle - k ) \alpha^{\vee}$
where we note that $s_{\alpha, 0} = s_{\alpha}$.   The (non-extended) affine Weyl group $\aW$ is then a Coxeter group which is generated by all affine reflections of the form $s_{\alpha, k}$.  For $\widetilde \alpha$ the highest root the set $\tilde{S} = S \cup \{ s_{\widetilde{ \alpha},1} \}$ is a Coxeter generating set for $\aW$, where we denote the additional affine generator as $s_0 =s_{\widetilde{ \alpha},1}$.  As with the finite Weyl group, we denote the length function by $\ell: \aW \rightarrow \Z$.  A minimal presentation for a group element $x \in \aW$ is a word for $x$ in the generating set $\tilde{S}$ which is of minimal length, that is, of length $\ell(x)$.  We also remark that $\aW \cong R^\vee \rtimes \sW$, and we will use both of these perspectives on the affine Weyl group throughout the paper.


\subsection{Hyperplanes, alcoves, and Weyl chambers}\label{sec:HyperplanesAlcovesWeylChambers}

The choice of split maximal torus corresponds to fixing an apartment $\App$ in a (thick) affine building $X$ of type $\aW$. We can (and will) identify this apartment with $X_*(T) {\otimes}_{\Z} \R$. The stabilizer $\aW_{v}$ of any special vertex $v$ in $\App$ is isomorphic to the spherical Weyl group $\sW$. 

Let $\aH$ be the collection of all affine hyperplanes or walls $\aH = \{ H_{\alpha, k} \mid \alpha \in \Phi, k \in \mathbb{Z}\}$, which we identify with a subset of $\App$.  The elements of $\aH$ are permuted naturally by the elements of $\sW$, and they are also permuted by translations by elements of the coroot lattice $R^{\vee}$.  Therefore the elements of the affine Weyl group $\aW$ also permute the hyperplanes in $\aH$, and so $\aW$ acts on the collection of connected components of $\App^{\circ} := \App \backslash \cup_{H \in \aH} H$.  Each such connected component is called an \emph{open alcove}.  Under the identification $\App = X_*(T)\otimes_{\Z} \R$, any open alcove in $\App$ consists of all points $x \in V$ satisfying the strict inequalities $k_{\alpha} < \langle \alpha, x \rangle < k_{\alpha} + 1$, where $\alpha$ runs through $\Phi^+$ and $k_{\alpha} \in \mathbb{Z}$ is some fixed integer depending upon the alcove and $\alpha$. 

We define a \emph{closed alcove}, or simply \emph{alcove}, to be the closure in $\App$ of an open alcove.  The elements of the affine Weyl group $\aW$ are in bijection with the set of alcoves in $\App$.  
The origin $v_0$ in $\App$ is the intersection $\cap_{i = 1}^n 
H_{\alpha_i}$, and we may choose the \emph{fundamental alcove} to be the alcove $\fa = \{ x \in \App \, \vert \, 0 \leq  \langle \alpha, x \rangle \leq 1 \text{ for all } \alpha \in \Phi^+ \},$ which contains the origin.  The codimension one faces of an alcove, that is, its maximal intersections with walls, are called \emph{panels}, and the \emph{walls of an alcove} $\mathbf{c}$ are defined to be the hyperplanes in $\aH$ which share a panel with $\mathbf{c}$.  For example, the walls of the fundamental alcove $\fa$ are $H_{\alpha_i}$ for $i = 1,\dots,n$ 
together with $H_{\tilde{\alpha}, 1}$.  If $p$ is a panel of an alcove $\mathbf{c}$ then the \emph{supporting wall of $p$} is the (unique) wall containing $p$.  

Let $I$ be the Iwahori subgroup of the group $G(F)$ associated to the origin, which fixes the base alcove $\fa$. With our adopted conventions, if $B = B^+$ is the standard Borel, then the Iwahori subgroup $I$ is the inverse image of the opposite Borel subgroup $B^-$ under the projection map $G(\mathcal{O}) \rightarrow G(k)$. The fundamental alcove $\fa$ is the basepoint of the affine flag variety $G(F)/I$, and is thus sometimes also referred to as the \emph{base alcove}. The Bruhat decomposition in this affine context then says that $G(F) = I\eW I$.

We define the \emph{fundamental} or \emph{dominant} Weyl chamber $\Cf$ in $\App$ to be the unique Weyl chamber based at the origin $v_0$ which contains the fundamental alcove $\fa$. That is, $\Cf$ is the set of points $x \in \App$ such that $\langle \alpha, x \rangle \geq 0$ for every $\alpha \in \Phi^+$. We denote by $-\Cf$ the \emph{antidominant} Weyl chamber, which is the unique chamber opposite $\Cf$ in $\App$, consisting of all points $x \in \App$ such that $\langle \alpha, x \rangle \leq 0$ for all $\alpha \in \Phi^+$. All chambers are labeled $\Cw_w$ for a unique element $w\in \sW$, with $\Cf = \Cw_\id$ and $-\Cf = \Cw_{w_0}$.  

Associated to each affine building $X$ there is a spherical building $\partial X$ at infinity whose chambers $C$ correspond to parallelism classes of Weyl chambers in $X$. Apartments of $\partial X$ are in one-to-one correspondence with apartments in $X$ and the chambers of an apartment $\partial \App$ in $\partial X$ are the parallelism classes of Weyl chambers in $\App$.  If $C$ is a chamber at infinity, $\mathbf{c}$ is an alcove, and $H$ is a wall of $\mathbf{c}$, we say $H$ \emph{separates $\mathbf{c}$ from $C$} if there is a representative Weyl chamber $\cC$ for $C$ such that $H$ separates $\mathbf{c}$ from $\cC$, that is $\mathbf{c}$ and $\cC$ are contained in different half-spaces determined by $H$.


\section{Labelings and orientations, galleries, and alcove walks}\label{sec:labelings}

This section discusses some of our key tools, and contains many of our most important definitions.  In Section~\ref{sec:LabelingsOrientations} we define labelings of the standard apartment and its boundary and describe the induced periodic orientation of hyperplanes.  Section~\ref{sec:Galleries} discusses two closely related types of combinatorial galleries, inspired by those introduced by Gaussent and Littelmann in~\cite{GaussentLittelmann}, and Section~\ref{sec:LFAW} recalls and slightly generalizes the labeled folded alcove walks which were introduced by Parkinson, Ram, and C.~Schwer in~\cite{PRS}. 


\subsection{Labelings and orientations of hyperplanes}\label{sec:LabelingsOrientations}

In the present section we will only be working inside a single apartment $\App$.  We write $\cham(\App)$ for the set of alcoves in $\App$. 
If we fix an origin $v_0$ and a fundamental alcove $\fa$ containing $v_0$, it is well known that the alcoves in $\App$ are in bijection with the elements of $\aW$. Fix once and for all one such $\aW$-equivariant bijection $\phi_0:\cham(\App)\to\aW$ that maps $\fa$ to the identity element $\id$ of $\aW$. We call this the \emph{standard labeling of $\App$}.  
Using the standard labeling, an alcove in 
$\App$ can be written as $x\fa$ for a unique element $x \in \aW$, and we will frequently use this correspondence between alcoves and elements in the affine Weyl group.  We will often write $\x$ for the alcove $x\fa$, where $x \in \aW$.

We now define more general labelings.

\begin{definition}
A \emph{labeling of $\App$} is a $\aW$-equivariant map from the set of alcoves of $\App$ to the affine Weyl group $\aW$. The \emph{basepoint} $v_\phi$ of a labeling $\phi$ is the intersection of all (closed) alcoves with label in $\sW$ under $\phi$. The \emph{base alcove of $\phi$} is the unique alcove $\x_\phi$ with $\phi(\x_\phi)=\id$.
\end{definition}

\noindent In particular, observe that the standard labeling $\phi_0$ has basepoint $v_{\phi_0}=v_0$ the origin, and base alcove $\x_{\phi_0}=\fa$ the fundamental alcove; that is, $x_{\phi_0} = 1 \in \aW$. 

\begin{definition}\label{def:InducedLabelingW}
Let $x\in\aW$ be an element of the affine Weyl group. The (type-preserving)  \emph{labeling induced by $x$} is the map $\phi_x:\cham(\App)\to  \aW$ defined by $$\phi_x(\y):= \phi_0(x^{-1}\y). $$ 
\end{definition}

\noindent Intuitively, Definition~\ref{def:InducedLabelingW} says that the alcove labeled $x$ by the standard labeling $\phi_0$ is labeled with the identity element by $\phi_x$, and all other alcoves are then labeled such that all panels and vertices keep their type.  Since $\aW$ acts on $\cA$ by simplicial bijections, Definition~\ref{def:InducedLabelingW} is a special case of the following: 

\begin{definition}\label{def:induced label}
Let $g:\App \to \App$ be any simplicial bijection. The \emph{labeling induced by $g$} is the map $\phi_g:\cham(\App)\to  \aW$ defined by \[\phi_g(\y)\define\phi_0(g^{-1}(\y)).\]
\end{definition}

\noindent The basepoint $v_{\phi_g}$ of $\phi_g$ equals $g(v_0)$ and the base alcove $\x_{\phi_g}$ equals $g(\fa)$. 
Note that the map $g:\App \to \App$ here need not be type-preserving, but it will always take special vertices to special vertices, and so the basepoint $v_{\phi_g}$ will always be a special vertex.

\begin{lemma}\label{lem:labels}
Each labeling is induced by a simplicial bijection $g:\App \to \App$. 
\end{lemma}
\begin{proof}
Let $\phi$ be a labeling.  We construct a simplicial bijection $g:\App \to \App$ which induces $\phi$ as follows.  For each alcove $\y$, since $\phi$ and $\phi_0$ are bijections from the set of alcoves of $\App$ to the affine Weyl group $\aW$, we may define $g(\y)$ to be the unique alcove such that $\phi(g(\y)) = \phi_0(\y)$.  Then as $\phi_0$ and $\phi$ are both $\aW$-equivariant, this bijection on the set of alcoves may be readily verified to induce a simplicial bijection $g:\App \to\App$.  Now if $\x = g(\y)$ we have $ \phi(\x)=\phi_0(g^{-1}(\x))$ and so $\phi = \phi_g$ as required.
\end{proof}

If a vertex $v$ and alcove $\x = x\fa$ are such that there exists a simplicial bijection $g:\App \to \App$ taking $v_0$ to $v$ and $\fa$ to $\x$, we may write $\phi_{v,x}$ for the induced labeling $\phi_g$.  In particular, for each vertex $v = t^\lambda$ with $\lambda \in R^\vee$, and each $w$ in the spherical Weyl group $\sW$, there is a labeling $\phi_{v,t^\lambda w}$.

We will also consider labelings of the chambers at infinity; that is, labelings of the chambers in $\partial \App$.

\begin{definition}
A \emph{labeling at infinity of $\App$} is a $\sW$-equivariant map from the set of chambers of $\partial \App$ to the spherical Weyl group $\sW$. 
\end{definition}

Any labeling $\phi$ induces a labeling at infinity $\phipartial$, as follows. 
Let $\phi$ be a labeling of $\App$ with basepoint $v$ (necessarily a special vertex). Then for each chamber $C$ in $\partial \App$, there exists a representative Weyl chamber $\cC \subset \App$ which is based at $v$.  We say that an alcove $\x$ is the \emph{tip of the chamber} $C$ if it is the unique alcove in $\cC$ which contains the basepoint $v$. 
Put 
\[\phipartial(C) = \phi(\x) \in \sW \text{ where $\x$ is the tip of $C$ at $v$.}\]
The \emph{standard labeling at infinity} is the labeling at infinity $\phipartialo$ induced by the standard labeling~$\phi_0$.  

\begin{lemma}\label{lem:induced infinity}  Let $\psi$ be a labeling at infinity.  Then there is a unique $w$ in the spherical Weyl group $\sW$ so that for every vertex $v = t^\lambda$ with $\lambda \in R^\vee$, the labeling at infinity $\psi$ is induced by the labeling $\phi_{v,t^\lambda w}$ of $\App$, that is, $\psi=\phipartialvtlw$.  In particular, each labeling at infinity is induced by a labeling $\phi_w$ of $\App$ for a unique $w \in \sW$.
\end{lemma}

\begin{proof}
Let $\psi$ be a labeling at infinity and let $C$ be the unique chamber at infinity such that $\psi(C) = 1$.  Then for any vertex $v = t^\lambda$ with $\lambda \in R^\vee$, the chamber $C$ is represented by a unique Weyl chamber $\cC$ based at $v$.  Suppose that this Weyl chamber $\cC$ is the Weyl chamber based at $v = t^\lambda$ which contains the alcove $t^\lambda w\fa$.  Then the labeling $\phi_{v,t^\lambda w}$ of $\App$ induces the labeling at infinity $\psi$.  In particular, taking $v = v_0$ we find that $\psi$ is induced by the labeling $\phi_{v_0,w} = \phi_w$.  The element $w\in \sW$ here is unique since if a Weyl chamber $\cC'$ is based at a vertex $v' = t^{\lambda'}$ then $\cC$ and $\cC'$ are parallel if and only if $\cC'$ contains the alcove $t^{\lambda'}w\fa$. 
\end{proof}

From now on, we will denote labelings at infinity by $\phipartial$ or sometimes by $\phi_{v,t^\lambda w}^{\mbox{\tiny{$\partial$}}}$ or $\phipartialw$, where $\phi_{v,t^\lambda w}$ or $\phi_w$ induces $\phipartial$ as in Lemma~\ref{lem:induced infinity}.  Note that $\phipartial = \phipartialw$ for $w \in \sW$ if and only if $\phipartial(C) = 1$ where $C$ is the chamber at infinity represented by the Weyl chamber $\cC_w$.

The affine Weyl group $\aW$ acts on the set of all labelings of $\App$, and the spherical Weyl group $\sW$ acts on the set of all labelings at infinity.  Note that if $x,y \in \aW$ then $x\phi_y = \phi_{xy}$ and if $u,w \in \sW$ then $u\phipartialw = \phipartialuw$.  We sometimes write $-\phi$ (respectively,~$-\phipartial$) for the opposite labeling $w_0\phi$ (respectively, $w_0 \phipartial$). 

Labelings at infinity induce orientations on the set of hyperplanes of an apartment, as follows.  

\begin{definition}\label{def:PeriodicOrientation}
For each $w \in \sW$, let $\phipartial = \phipartialw$ be the labeling at infinity induced by the labeling $\phi_w$ of $\App$.  The \emph{periodic orientation induced by $\phipartial$}  on the set of hyperplanes of $\App$ is the orientation of the hyperplanes such that for each $\alpha \in \Phi^+$:
\begin{enumerate}
\item the alcove $\w = w\fa$ is on the positive side of $H_{\alpha}$; and  
\item for all integers $k$, the hyperplanes $H_{\alpha}$ and $H_{\alpha,k}$ have the same orientation.
\end{enumerate}
\end{definition}

\noindent Note that the periodic orientation does not depend explicitly on a choice of Borel subgroup, since $H_{\alpha}= H_{-\alpha}$.  Given Definition~\ref{def:PeriodicOrientation}, we will often refer to a labeling at infinity $\phipartial$ as an \emph{orientation at infinity}.  In the case of the standard labeling at infinity $\phipartialo$ induced by $\phi_0$, we refer to the corresponding periodic orientation as the \emph{standard orientation at infinity}.


\subsection{Combinatorial galleries}\label{sec:Galleries}

We now introduce two closely related notions of combinatorial galleries, the first running from a vertex to a vertex of the same type and the second from an alcove to an alcove.   These are inspired by the combinatorial galleries introduced by Gaussent and Littelmann in \cite{GaussentLittelmann}, which are more general, and we refer the interested reader there for further details. To simplify notation, in this section we mostly denote alcoves by $c$ and $c_i$, rather than by $\x = x \fa$ with $x \in \aW$ as in Section~\ref{sec:LabelingsOrientations}.  We continue to work inside a single apartment $\App$ and to denote the standard labeling by $\phi_0$.

The first kind of combinatorial galleries that we consider, which run from a vertex to a vertex of the same type, are defined as follows.  These are a special case of the combinatorial galleries considered by Gaussent and Littelmann (see Definition 8 of \cite{GaussentLittelmann}).

\begin{definition}\label{def:CombGalleryVertex}
A \emph{(vertex-to-vertex) combinatorial gallery} is a sequence of alcoves $c_i$ and faces $p_i$ 
$$ \gamma=(p_0\subset c_0\supset p_1 \subset c_1 \supset p_2 \subset \dots \supset p_n \subset c_n \supset p_{n+1} ), $$
where the first and last faces $p_0$ and $p_{n+1}$ are vertices of the same type (\textit{e.g.} both in the coroot lattice), and the remaining $p_i$ are panels of both alcoves $c_i$ and $c_{i-1}$.    
\end{definition}

\noindent We remark that if $c_i \neq c_{i-1}$ then there is no choice for the panel $p_i$, and that once the vertex $p_0$ is specified there is no choice for the vertex $p_{n+1}$.  

We now define combinatorial galleries which run from alcoves to alcoves.

\begin{definition}\label{def:CombGalleryAlcove}
An \emph{(alcove-to-alcove) combinatorial gallery} is a sequence of alcoves $c_i$ and faces $p_i$ 
$$ \gamma=(c_0\supset p_1 \subset c_1 \supset p_2 \subset \dots \supset p_n \subset c_n), $$
where each $p_i$ is a panel of both alcoves $c_i$ and $c_{i-1}$.    
\end{definition}

\noindent Again, if $c_i \neq c_{i-1}$ there is no choice for the panel $p_i$.

Let $\gamma$ be a combinatorial gallery as in either Definition~\ref{def:CombGalleryVertex} or Definition~\ref{def:CombGalleryAlcove}.  The gallery $\gamma$ is said to be \emph{stuttering} if there is some $i$ so that $c_i = c_{i-1}$, and otherwise is \emph{non-stuttering}.  The \emph{length} of $\gamma$ is defined to be $n+1$, that is, the number of alcoves in $\gamma$ counted with multiplicity.  

When the context is clear, we will not specify whether a combinatorial gallery is vertex-to-vertex or alcove-to-alcove.  It may be helpful to  think of our vertex-to-vertex galleries as alcove-to-alcove galleries where both the first and last alcove come together with a fixed, marked vertex.  All of our combinatorial galleries will contain at least one alcove.

Obviously each vertex-to-vertex combinatorial gallery $\gamma$ as in Definition~\ref{def:CombGalleryVertex} can be truncated by removing the initial and final vertices in order to obtain a canonical associated alcove-to-alcove gallery, which we denote by $\gamma^\flat$.  To each alcove-to-alcove combinatorial gallery $\gamma$ as in Definition~\ref{def:CombGalleryAlcove} we will associate a canonical vertex-to-vertex gallery $$\gamma^\sharp := (p_0 \subset c_0\supset p_1 \subset c_1 \supset p_2 \subset \dots \supset p_n  \subset c_n \supset p_{n+1})$$ by defining $p_0$ to be the (unique) vertex of the first alcove $c_0$ which lies in the coroot lattice, hence $p_{n+1}$ is the (unique) vertex of the last alcove $c_n$ which lies in the coroot lattice.  Then $(\gamma^\sharp)^\flat = \gamma$ for each alcove-to-alcove combinatorial gallery, while if $\gamma$ is a vertex-to-vertex combinatorial gallery then $(\gamma^\flat)^\sharp = \gamma$ if and only if both the first and last faces of $\gamma$ lie in the coroot lattice.

Fix an orientation at infinity $\phipartial$ and write $C_\phi$ for the chamber at infinity with label $\phipartial(C_\phi)=\id$.  Then $C_{-\phi}$ is the chamber at infinity such that $\phipartial(C_{-\phi}) = w_0$.   
In~\cite{GaussentLittelmann} the only orientation considered is the standard orientation at infinity, but the definition of a positively (respectively, negatively) folded gallery given there naturally generalizes as follows.

\begin{definition}\label{def:PosFolded}
A combinatorial gallery $\gamma$ as in Definition~\ref{def:CombGalleryVertex} or Definition~\ref{def:CombGalleryAlcove} is \emph{positively (respectively, negatively) folded} with respect to $\phipartial$ if for all $1 \leq i \leq n$ the supporting wall $H_i$ of the panel $p_i$ either separates $c_{i-1}$ and $c_i$, or $H_i$ separates $c_i=c_{i-1}$ from the chamber $C_{-\phi}$ (respectively, $C_\phi$).  
\end{definition}

We mostly consider positively folded galleries.  We remark that if $\gamma$ is negatively folded with respect to the orientation $\phipartial$, then $\gamma$ is positively folded with respect to the opposite orientation $-\phipartial$.  

We will be applying various geometric transformations to galleries in this work.  The next two lemmas verify that these transformations take positively folded galleries to positively folded galleries.  

We first consider acting on the left by elements of $\sW$.

\begin{lemma}\label{lem:wpos}   Let $\phipartial$ be an orientation at infinity and let $\gamma$ be a combinatorial gallery which is positively folded with respect to $\phipartial$.  Let $w \in \sW$.  Then $\gamma$ is positively folded with respect to $\phipartial$ if and only if $w\gamma$ is positively folded with respect to the orientation $w\phipartial$.
\end{lemma}

\begin{proof}  A panel $p_i$ of $\gamma$ separates $c_{i-1}$ from $c_i$ if and only if the panel $wp_i$ of $w\gamma$ separates $wc_{i-1}$ from $wc_i$.  Now suppose that $c_i = c_{i-1}$ are alcoves of $\gamma$.  
The wall $H$ separates $c_i$ from $C_{-\phi}$ if and only if the wall $wH$ separates $wc_i$ from $wC_{-\phi}$.  As the labeling $\phipartial$ is $\sW$-equivariant, $\phipartial(wC_{-\phi}) = w_0$ if and only if $(w\phipartial)(C_{-\phi}) = w_0$.  Thus $wC_{-\phi}$ is the chamber at infinity $C_{-w\phi}$, and so $H$ separates $c_i$ from $C_{-\phi}$ if and only if $wH$ separates $wc_i$ from $C_{-w\phi}$.  This completes the proof.
\end{proof}

We will also be acting on positively folded galleries by translations.

\begin{lemma}\label{lem:transpos}    Let $\phipartial$ be an orientation at infinity and let $\gamma$ be a combinatorial gallery which is positively folded with respect to $\phipartial$.   Let $g:\App \to \App$ be a simplicial bijection which, viewed as a Euclidean isometry, is a translation.  Then the gallery $g(\gamma)$ is also positively folded with respect to $\phipartial$.
\end{lemma}

\begin{proof} Since $g$ is a translation $g$ fixes each chamber at infinity, and the wall $H$ is parallel to the wall $g(H)$.  The result follows.
\end{proof}

Now, using the standard labeling $\phi_0$ of the apartment $\App$, we may associate to a combinatorial gallery $\gamma$ as in either Definition~\ref{def:CombGalleryVertex} or Definition~\ref{def:CombGalleryAlcove} its \emph{type}, denoted by $\type(\gamma)$, which is the word in $\aW$ obtained as follows:
$$\type(\gamma)\define s_{j_1}s_{j_2}\dots s_{j_n},$$
where for $1 \leq i \leq n$ the panel $p_i$ of $\gamma$ has type $s_{j_i} \in \tilde{S}$.  Note that if $c_i \neq c_{i-1}$ then the alcove $c_i$ is obtained from $c_{i-1}$ by right-multiplication by the generator $s_{j_i}$, while if $c_i = c_{i-1}$ then the type can be thought of as recording one of the panels of $c_i$.  The type of a vertex-to-vertex gallery $\gamma$ does not depend upon its initial and final vertices.

 If $\gamma$ is an alcove-to-alcove  gallery as in Definition~\ref{def:CombGalleryAlcove} then $\gamma$ is \emph{minimal} if it has minimal length among all alcove-to-alcove combinatorial galleries from $c_0$ to $c_n$.  Such a gallery $\gamma$ is minimal if and only if the word $\type(\gamma)$ in $\aW$ is a minimal presentation for the corresponding group element in $\aW$.  

If $\gamma$ is a vertex-to-vertex gallery as in Definition~\ref{def:CombGalleryVertex}, then $\gamma$ is \emph{minimal} if has minimal length among all vertex-to-vertex combinatorial galleries from $p_0$ to $p_{n+1}$. Characterizing minimality of vertex-to-vertex galleries is subtle, and care is needed when comparing minimality for vertex-to-vertex and for alcove-to-alcove galleries.   In both cases, a minimal gallery must be non-stuttering.    If $\gamma$ is a minimal vertex-to-vertex gallery then the canonical associated alcove-to-alcove gallery $\gamma^\flat$ must be minimal, but if $\gamma$ is a minimal alcove-to-alcove gallery then the canonical associated vertex-to-vertex gallery $\gamma^\sharp$ is not in general minimal.  For instance, if $\gamma^\sharp$ contains more than one alcove in the link of its final vertex $p_{n+1}$ then $\gamma^\sharp$ is not minimal.  In our explicit constructions in Section~\ref{sec:Construction}, however, these two notions of minimality will coincide.

\begin{remark}\label{rem:GLminimal}  The relationship between our definition of minimality and that given by Gaussent and Littelmann in~\cite{GaussentLittelmann} is delicate.  The combinatorial galleries considered in~\cite{GaussentLittelmann} always start and end in a vertex, but their maximal faces need not be alcoves.  In particular, the galleries in~\cite{GaussentLittelmann} can be entirely contained in a wall.  Now the definition of minimality given in Definition 10 of \cite{GaussentLittelmann} implies that if the first vertex $p_0$ and the last vertex $p_{n+1}$ are both contained in some wall $H$, then all minimal galleries from $p_0$ to $p_{n+1}$ must be contained in~$H$.  Since no alcove is contained in a wall, it follows that there is no vertex-to-vertex combinatorial gallery as in Definition~\ref{def:CombGalleryVertex} which runs from $p_0$ to $p_{n+1}$ and is minimal in the sense defined in~\cite{GaussentLittelmann}. In fact in this situation, the intersection of all of our minimal vertex-to-vertex galleries from $p_0$ to $p_{n+1}$ will be the minimal vertex-to-vertex gallery in the sense of \cite{GaussentLittelmann}. Thus our definition of minimality for vertex-to-vertex galleries differs from that given in~\cite{GaussentLittelmann}.  However, if the first and last vertex are not contained in any common wall, then our definition of minimality for vertex-to-vertex galleries does agree with that in~\cite{GaussentLittelmann}.  In our constructions in Section~\ref{sec:Construction}, we will always be in  situations where these definitions of minimality coincide.
\end{remark}

\begin{notation}  If $\lambda$ is a vertex in the coroot lattice we use $\gamma_\lambda$ to denote a minimal vertex-to-vertex combinatorial gallery from the origin $v_0$ to $\lambda$.  Given such a minimal gallery $\gamma_\lambda$  and a vertex $\mu \in R^\vee$ we denote by $\Gamma^+(\gamma_\lambda,\mu)$ the set of all vertex-to-vertex galleries of the same type as $\gamma_\lambda$ which start at $v_0$, end in $\mu$, and are positively folded with respect to the standard orientation at infinity.   
 
We adopt similar notation for other orientations at infinity.  For $w \in \sW$, we denote by $\Gamma_w^+(\gamma_\lambda, \mu)$ the set of all galleries of the same type as $\gamma_\lambda$ which start at $v_0$, end in $\mu$, and are positively folded with respect to the orientation at infinity $\phipartialw$ induced by the labeling $\phi_w$.  

We also use similar notation for alcove-to-alcove galleries.  Given $x \in \aW$, we use $\gamma_x$ to denote a minimal alcove-to-alcove gallery from $c_0 = \fa$ to $c_n = x\fa$.  Then given $y \in \aW$ we define $\Gamma^+(\gamma_x, y)$ to be the set of all galleries of the same type as $\gamma_x$ which go from the fundamental alcove $\fa$ to the alcove $y\fa$ and are positively folded with respect to the standard orientation.  For $w \in \sW$ we define $\Gamma^+_w(\gamma_x,y)$  to be the set of all galleries of the same type as $\gamma_x$ which go from the fundamental alcove $\fa$ to the alcove $y\fa$ and are positively folded with respect to the orientation at infinity $\phipartialw$.  

It will not always be necessary to record each face $p_i$ and alcove $c_i$ of a combinatorial gallery $\gamma$.  In particular, for alcove-to-alcove galleries we often wish simply to indicate the first alcove $c_0$ and the last alcove $c_n$, and we use the notation $\rightsquigarrow$ in this case.  For example, if $\gamma$ is an alcove-to-alcove  combinatorial gallery such that $c_0 = \fa$ and $c_n = \y$ is an arbitrary alcove we denote this by $\gamma: \fa \rightsquigarrow \y$.  

We will also not always need to record the details of the type of combinatorial galleries.  In particular, if $\gamma$ is a minimal alcove-to-alcove gallery, and $x \in \aW$ is the group element defined by the word $\type(\gamma)$, we will say that $\gamma$ is of type $\vec{x}$ in cases where we do not need to specify the choice of minimal presentation for $x$.  
We then denote by $\Gamma_w^+(x,y)$ the set of all positively-folded galleries with respect to the orientation $\phipartialw$ which are of some type $\vec{x}$ and which go from $\fa$ to the alcove $y\fa$. 
 Note that the galleries in the set $\Gamma_w^+(x,y)$ all have type some minimal presentation for $x \in \aW$, but in general will not all have the same type.
\end{notation}


\subsection{Labeled folded alcove walks}\label{sec:LFAW}

In this section we recall and slightly generalize the labeled folded alcove walks which were introduced in~\cite{PRS}.  
We continue to use the notation for combinatorial galleries established in Section~\ref{sec:Galleries}.

Every combinatorial gallery we shall discuss in this paper (either vertex-to-vertex or alcove-to-alcove) can be naturally associated to an alcove walk in the sense of Ram's work \cite{Ram}, which gives rise to a convenient visualization.  An \emph{alcove walk} from an alcove $c$ to an alcove $c'$ is a path connecting a point in the interior of $c$ to a point in the interior of $c'$, such that the path does not pass through a vertex of any alcove.  A combinatorial gallery $ \gamma$ as in Definition~\ref{def:CombGalleryVertex} or Definition~\ref{def:CombGalleryAlcove} thus determines an alcove walk from $c_0$ to $c_n$, where we may choose the barycenter of $c_i$ as the interior point to identify with each alcove $c_i$.  

A \emph{step} in the alcove walk is visualized as a directed arrow with its tail and tip at the barycenter of two consecutive alcoves in the walk.  More specifically, the $i^{\text{th}}$ step of $\gamma$ is an arrow which meets $p_i$ orthogonally, and whose tail and tip lie at the barycenters of alcoves $c_{i-1}$ and $c_i$, respectively.  Given this correspondence, it is natural to visualize each step as either a fold or a crossing.  We say that $\gamma$ has a \emph{fold} at $p_i$ if $c_{i-1} = c_i$, and otherwise $\gamma$ has a \emph{crossing} at $p_i$.  We will often use terminology associated to a combinatorial gallery and the corresponding alcove walk interchangeably.  

To simplify our diagrams, we will typically ``smooth out" alcove walks, so that instead of depicting the alcove walk corresponding to $\gamma$ by a concatenation of directed arrows as in \cite{Ram}, and on the left of Figure~\ref{fig:smoothed}, we will instead as in the center of Figure~\ref{fig:smoothed} draw a single directed path which passes from the interior of $c_0$ to the interior of $c_n$ via the interiors of the other $c_i$, so that if $\gamma$ has a crossing at $p_i$ this path intersects $p_i$ transversally, and if $\gamma$ has as fold at $p_i$ then this path doubles back on itself at $p_i$.  The right of Figure~\ref{fig:smoothed} shows the ``smoothed out" vertex-to-vertex combinatorial gallery canonically associated to both the alcove walk on the left and the ``smoothed out" alcove-to-alcove gallery in the center.

\begin{figure}[ht]
\begin{center}
\resizebox{\textwidth}{!}
{
\begin{overpic}{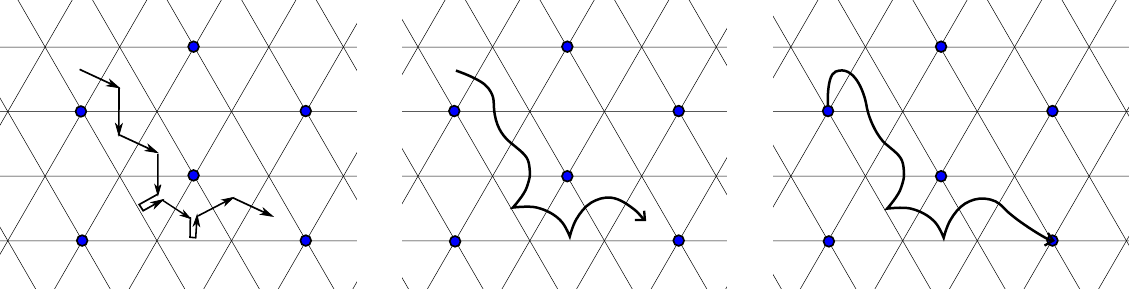}
\end{overpic}
}

\caption{``Smoothing out"  of alcove walks.  The heavy dots are the vertices which lie in the coroot lattice.}
\label{fig:smoothed}
\end{center}
\end{figure}

\begin{definition}
Given a combinatorial gallery $\gamma$ in $\App$ (not necessarily minimal) and a periodic orientation of the hyperplanes of $\App$ induced by a labeling at infinity $\phipartial$, a crossing in $\gamma$ is a \emph{positive crossing} if it crosses from the negative to the positive side of the corresponding hyperplane, and otherwise is a \emph{negative crossing}.  Similarly, $\gamma$ has a \emph{positive (respectively, negative) fold} if the fold occurs on the positive (respectively, negative) side of the hyperplane.
\end{definition}

\noindent We remark that $\gamma$ is positively folded as in Definition \ref{def:PosFolded} if and only if all folds are positive folds, and so these notions for alcove walks and combinatorial galleries coincide.

We may obtain new galleries of the same type by replacing crossings by folds, as follows.  Let $\gamma = (p_0\subset c_0\supset p_1 \subset c_1 \supset p_2 \subset \ldots \subset c_n \supset p_{n+1} )$ be a vertex-to-vertex gallery which has a crossing at face $p_i$.   Let $H_i$ be the supporting wall of $p_i$ and let $r_i$ be the (affine) reflection in the wall $H_i$.  Then the gallery $\gamma'$ obtained by \emph{folding $\gamma$ at $p_i$}, or equivalently \emph{folding $\gamma$ in $H_i$}, is the gallery obtained by applying the reflection $r_i$ to the portion of $\gamma$ after the face $p_i$.  That is, 
$$\gamma' = (p_0\subset c_0\supset p_1 \subset c_1 \supset p_2 \subset \ldots \subset c_{i-1} \supset p_i \subset c_i' \supset \ldots \subset c_n' \supset p'_{n+1} )$$
where $c_j' = r_i(c_j)$ for $i \leq j \leq n$ and $p_j' = r_i(p_{j})$ for $i+1 \leq j \leq n+1$.  In particular, $c_i' = c_{i-1}$ since the reflection $r_i$ interchanges $c_{i-1}$ and $c_i$, which means that the new gallery $\gamma'$ has a fold at $p_i$.  If $\gamma = (c_0\supset p_1 \subset c_1 \supset p_2 \subset \ldots \subset c_n)$ is an alcove-to-alcove gallery then the definition of folding $\gamma$ is similar.

The process described in the previous paragraph will be generally referred to as \emph{folding a gallery}. Where there is no risk of ambiguity, we may record a sequence of such foldings using only the sequence of reflections which have been applied.  In this case, we refer to the sequence of reflections applied in turn as the \emph{folding sequence}.

In order to make a precise connection between folded galleries and affine Deligne--Lusztig varieties, we will need to construct folding sequences in a very specific fashion. The galleries obtained by the following kind of folding sequence are those considered for the standard orientation at infinity  in~\cite{PRS}, which motivates the terminology.

\begin{definition}\label{def:PRS}  Let $\gamma = (p_0\subset c_0\supset p_1 \subset c_1 \supset p_2 \subset \ldots \subset c_n \supset p_{n+1} )$ be a minimal vertex-to-vertex gallery and $\phipartial$ an orientation at infinity.  A \emph{PRS folding sequence} is a sequence of galleries $\gamma^{(j)}$ for $0 \leq j \leq m$ so that $\gamma^{(0)} = \gamma$, $\gamma^{(1)}$ is obtained from $\gamma$ by positively folding at face $p_{i_1}$ with corresponding reflection $r_1$, and for $2 \leq j \leq m$ the gallery $\gamma^{(j)}$ is obtained from $\gamma^{(j-1)}$ by positively folding at face $r_{j-1} \cdots r_1 p_{i_j}$ with corresponding reflection $r_j$, such that $1 \leq i_1 < i_2 < \cdots < i_m \leq n$.  Similarly define a PRS folding sequence of minimal alcove-to-alcove galleries.
\end{definition}

Note that a PRS folding sequence always begins with a minimal gallery and that all folds are positive.  The idea  in a PRS folding sequence is that each fold after the first one occurs in the ``tail" of the gallery; that is, in the portion of the gallery which was just reflected by the previous fold.  In other words, each folding of the gallery does not reflect any previously folded parts.  By abuse of terminology we will also refer to the sequence of folds occurring in Definition~\ref{def:PRS} as a PRS folding sequence.  

We now recall the notion of a labeled folded alcove walk, which was introduced in \cite{PRS}.  We will see in Section~\ref{sec:ADLVGalleries} that these walks can be used to calculate dimensions of intersections of Iwahori and unipotent orbits, and thus dimensions of affine Deligne--Lusztig varieties.

\begin{definition}\label{def:LFAW}
Let $\gamma$ be a positively folded gallery (either vertex-to-vertex or alcove-to-alcove) obtained from a PRS folding sequence.  We obtain a \emph{labeled folded alcove walk} from $\gamma$ by labeling each positive crossing of $\gamma$ by an element of the field of coefficients $k$, each negative crossing by $0$, and each positive fold of $\gamma$ with an element of $k^{\times}$.  The \emph{type} of a labeled folded alcove walk is the type of the underlying (unlabeled) positively folded gallery. 
\end{definition}

\begin{remark}
A labeling of the apartment $\App$ as in Section~\ref{sec:LabelingsOrientations} should not be confused with a labeling of a gallery $\gamma$ as in Definition~\ref{def:LFAW}.  We will only refer to labelings of a gallery $\gamma$ again briefly in Section~\ref{sec:ADLVGalleries}, while we will be referring to labelings of $\App$ throughout this work.
\end{remark}

\begin{notation}\label{not:LFAW}
For each $x,y \in \aW$, and for each orientation at infinity $\phipartialw$, we denote by $\text{L}\Gamma^+_w(\gamma_x,y)$ the set of labeled folded alcove walks whose corresponding alcove-to-alcove galleries are in $\Gamma^+_w(\gamma_x,y)$.  That is, $\text{L}\Gamma^+_w(\gamma_x,y)$ consists of all labeled galleries $\gamma:\fa \rightsquigarrow y\fa$ of the same type as $\gamma_x$ which are positively folded with respect to $\phipartialw$ and obtained through a PRS folding sequence.  If $\phipartialw$ is the standard orientation, we denote this set by $\text{L}\Gamma^+(\gamma_x,y)$.  Similarly define $\text{L}\Gamma^+_w(x,y)$ and $\text{L}\Gamma^+(x,y)$. 
\end{notation} 


\section{Dimensions of galleries and root operators}\label{sec:GalleryDim}

This section makes precise connections between various notions of the dimension of a positively folded gallery, and recalls the root operators introduced by Gaussent and Littelmann in~\cite{GaussentLittelmann}.  In Section~\ref{sec:DimFolded} we define the dimension of the combinatorial galleries introduced in Section~\ref{sec:Galleries}, and show that in certain cases their dimension may be computed by counting the number of folds and positive crossings.  For vertex-to-vertex combinatorial galleries, our definition of dimension generalizes that in~\cite{GaussentLittelmann} to arbitrary orientations.  We then in Section~\ref{sec:RootOps} recall root operators and discuss their effect upon dimension.  Section~\ref{sec:Counting} establishes some useful formulas concerning the dimension of a gallery, and in Section~\ref{sec:BraidMoves} we prove that the dimension of a positively folded gallery does not depend upon the choice of initial minimal gallery.


\subsection{The dimension of a folded gallery}\label{sec:DimFolded}

In this section we define the dimension of a positively folded vertex-to-vertex gallery and of a positively folded alcove-to-alcove gallery, then discuss the relationship between these two notions of dimension.

We begin by generalizing Gaussent and Littelmann's original version of the dimension of a positively folded combinatorial gallery \cite{GaussentLittelmann} to vertex-to-vertex combinatorial galleries which are positively folded with respect to any orientation at infinity.

\begin{definition}\label{def:load-bearing}
Given an orientation at infinity $\phipartial$ and a vertex-to-vertex combinatorial gallery $\gamma=(p_0\subset c_0\supset p_1 \subset \dots \supset p_n \subset c_n\supset p_{n+1})$ which is positively folded with respect to $\phipartial$, we say that a wall $H$ is \emph{load-bearing for $\gamma$ with respect to $\phipartial$ at $p_i$} for $i\neq n+1$ if $H$ contains the face $p_i$ and $H$ separates $c_i$ from $C_{-\phi}$. 
\end{definition}

\noindent  Recall that $C_{-\phi}$ is the chamber at infinity such that $\phipartial(C_{-\phi}) = w_0 \in \sW$.  By definition, no wall will be load-bearing at the final vertex $p_{n+1}$.  For the standard orientation, Definition~\ref{def:load-bearing} agrees with the definition of load-bearing walls in Section 5 of~\cite{GaussentLittelmann}.

The next two lemmas consider the effect of various geometric transformations of galleries on load-bearing walls.  The following result considers the left-action of elements of $\sW$.  Its proof is similar to that of Lemma~\ref{lem:wpos}.

\begin{lemma}\label{lem:w0 load bearing}  Let $\phipartial$ be an orientation at infinity and let $\gamma$ be a vertex-to-vertex combinatorial gallery which is positively folded with respect to $\phipartial$.  Let $w \in \sW$.  Then a wall $H$ is load-bearing for $\gamma$ with respect to $\phipartial$ at $p_i$ if and only if the wall $wH$ is load-bearing for $w\gamma$ with respect to the orientation $w\phipartial$ at the face $wp_i$.  
\end{lemma}

We will also need that the property of being a load-bearing wall is invariant under translation.  The proof of the following result is similar to that of Lemma~\ref{lem:transpos}.

\begin{lemma}\label{lem:translations load bearing}  Let $\phipartial$ be an orientation at infinity and let $\gamma$ be a vertex-to-vertex combinatorial gallery which is positively folded with respect to $\phipartial$.   Let $g:\App \to \App$ be a simplicial bijection which, viewed as a Euclidean isometry, is a translation.   Then for any face $p_i$ of $\gamma$, a wall $H$ is load-bearing for $\gamma$ with respect to $\phipartial$ at $p_i$ if and only if the wall $g(H)$ is load-bearing for $g(\gamma)$ with respect to $\phipartial$ at the face $g(p_i)$.
\end{lemma}

We now define dimension of positively folded vertex-to-vertex combinatorial galleries.  For the standard orientation, this definition agrees with Definition 14 of~\cite{GaussentLittelmann}.

\begin{definition}\label{def:DimGalleryV}
Let $\phipartial$ be an orientation at infinity and let $\gamma$ be a vertex-to-vertex combinatorial gallery which is positively folded with respect to $\phipartial$.  The \emph{dimension of $\gamma$ with respect to $\phipartial$}, denoted $\dim_\phi(\gamma)$, is defined to be the number of load-bearing walls of $\gamma$. 
\end{definition}

In order to give a careful treatment of the relationship between the dimension of a vertex-to-vertex gallery $\gamma$ and the dimension of its canonical associated alcove-to-alcove gallery $\gamma^\flat$, we will need to keep track of the number of walls which are load-bearing at the first vertex $p_0$, and so make the following definition. 

\begin{definition}
Let $\phipartial$ be an orientation at infinity and let $\gamma$ be a vertex-to-vertex combinatorial gallery with first vertex $p_0$ which is positively folded with respect to $\phipartial$.  We define the \emph{load of $\gamma$ at $p_0$ with respect to $\phi$}, denoted $\load_\phi(p_0,\gamma)$, to be the number of walls which are load-bearing for $\gamma$ with respect to $\phipartial$ at $p_0$.  
\end{definition}
 
\noindent Note that $\load_\phi(p_0,\gamma)$ depends only upon the orientation $\phipartial$, the first vertex $p_0$, and the first alcove $c_0$ of $\gamma$.

\begin{example}\label{eg:load bearing}  Suppose a vertex-to-vertex gallery $\gamma$ has first vertex $p_0 = v_0$ the origin, and first alcove $c_0 = \fa$ the base alcove.  If we use the standard orientation at infinity $\phipartialo$, then $C_{-\phi_0}$ is represented by the Weyl chamber $\cC_{w_0}$ and so every wall containing $p_0 = v_0$ separates $c_0 = \fa$ from $C_{-\phi_0}$.  Thus $\load_{\phi_0}(p_0,\gamma)$ is equal to $|\Phi^+|$, the number of positive roots.  On the other hand, if we use the orientation at infinity $\phipartial = -\phipartialo$, then $C_{-\phi} = C_{\phi_0}$ is represented by the fundamental Weyl chamber $\Cf$.  No wall containing $p_0 = v_0$ separates $c_0 = \fa$ from $C_{\phi_0}$, and so $\load_{-\phi_0}(p_0,\gamma) = 0$.  For any orientation at infinity $\phipartial$ other than $\pm \phipartialo$, it will be the case that $0 < \load_\phi(p_0,\gamma) < |\Phi^+|$.  By Lemma~\ref{lem:translations load bearing} the same conclusions hold if $\gamma$ is any vertex-to-vertex gallery with first vertex $p_0 = \lambda$ in the coroot lattice and first alcove $c_0 = t^\lambda \fa$ a pure translation.
\end{example}

We now define the dimension of a combinatorial alcove-to-alcove gallery $\gamma$ which is positively folded with respect to some orientation at infinity $\phipartial$.  For this, we denote by $P_{\phi}(\gamma)$ the number of positive crossings in the gallery $\gamma$ and by $N_{\phi}(\gamma)$ the number of negative crossings in $\gamma$, counted with respect to the periodic orientation induced by $\phipartial$.  The number of folds in a gallery $\gamma$ which is positively folded with respect to $\phipartial$ will then be denoted by $F_{\phi}(\gamma) = \# \{ i \mid \gamma\ \text{has a fold at}\ p_i \}$.  When the choice of orientation is clear, we will occasionally abbreviate this notation by omitting the reference to the orientation as $P(\gamma)$, $N(\gamma)$, and $F(\gamma)$.

\begin{definition}\label{def:DimGalleryA}  Let $\phipartial$ be an orientation at infinity and let $\gamma$ be an alcove-to-alcove combinatorial gallery which is positively folded with respect to $\phipartial$.  Then the \emph{dimension of $\gamma$ with respect to $\phipartial$}, denoted $\dim_\phi(\gamma)$, is defined to be the number of positive crossings in $\gamma$ plus the number of folds in $\gamma$, that is,
$$
\dim_\phi(\gamma) = P_\phi(\gamma) + F_\phi(\gamma).
$$
\end{definition}

\noindent We give this definition of the dimension of a positively folded alcove-to-alcove gallery because, as we will see in Section~\ref{sec:ADLVGalleries}, it correctly records the dimension of the intersection of certain Iwahori and unipotent orbits in the affine flag variety.

We now note the effects on dimension of acting by $\sW$ on the left and of translations, for both vertex-to-vertex and alcove-to-alcove galleries.

\begin{lemma}\label{lem:w0 dim}  Let $w \in \sW$.  For any orientation at infinity $\phipartial$ and any combinatorial gallery $\gamma$ which is positively folded with respect to $\phipartial$, the gallery $w\gamma$ is positively folded with respect to $w\phipartial$ and $\dim_\phi(\gamma) = \dim_{w\phi}(w\gamma)$.  
\end{lemma}

\begin{proof}  If $\gamma$ is vertex-to-vertex this follows from  Lemmas~\ref{lem:wpos} and~\ref{lem:w0 load bearing} and Definition~\ref{def:DimGalleryV}.  If $\gamma$ is alcove-to-alcove and positively folded with respect to $\phipartial$, then $w\gamma$ is positively folded with respect to $w\phipartial$ and has the same number of (positive) folds as $\gamma$, that is, $F_\phi(\gamma) = F_{w\phi}(\gamma)$.  Also the positive crossings of $\gamma$ with respect to $\phipartial$ are exactly the positive crossings of $w\gamma$ with respect to $w\phipartial$, that is, $P_\phi(\gamma) = P_{w\phi}(w\gamma)$.  Thus by Definition~\ref{def:DimGalleryA} we have $\dim_\phi(\gamma) = \dim_{w\phi}(w\gamma)$ as required.
\end{proof}

\begin{lemma}\label{lem:translations dim}   Let $g:\App \to \App$ be a simplicial bijection which, viewed as a Euclidean isometry, is a translation.  Then for any orientation at infinity $\phipartial$ and any combinatorial gallery $\gamma$ which is positively folded with respect to $\phipartial$, the gallery $g(\gamma)$ is positively folded with respect to $\phipartial$ and $\dim_\phi(\gamma) = \dim_\phi(g(\gamma))$.
\end{lemma}

\begin{proof}   If $\gamma$ is vertex-to-vertex this follows from Lemmas~\ref{lem:transpos} and~\ref{lem:translations load bearing} and Definition~\ref{def:DimGalleryV}.   Now suppose that $\gamma$ is alcove-to-alcove and positively folded with respect to $\phipartial$.  Then since $g$ is a translation, the action of $g$ takes positive (respectively, negative) crossings to positive (respectively, negative) crossings, and takes positive folds to positive folds.  So $g(\gamma)$ is positively folded with respect to $\phipartial$, $P_\phi(\gamma) = P_{\phi}(g(\gamma))$, and $F_\phi(\gamma) = F_\phi(g(\gamma))$.  The result then follows from Definition~\ref{def:DimGalleryA}.
\end{proof}

The relationship between the two notions of dimension given by Definitions~\ref{def:DimGalleryV} and~\ref{def:DimGalleryA} is delicate, essentially due to the fact that if $\gamma$ is a minimal alcove-to-alcove gallery then the canonically associated vertex-to-vertex gallery $\gamma^\sharp$ need not be minimal.  These subtleties are known to experts and have been considered in other contexts; see for instance \cite{BaumannGaussent} and Remark 4.3 or 4.9 in \cite{CSchwer}.  For our purposes, it will suffice to show that for certain positively folded alcove-to-alcove galleries $\gamma$, the dimension of $\gamma$ is equal to the dimension of the canonically associated vertex-to-vertex gallery $\gamma^\sharp$.  In all of our explicit constructions of galleries in Section~\ref{sec:Construction}, we will be in this situation.  We first classify the mutually exclusive possibilities for load-bearing walls.

\begin{lemma}\label{load bearing}
Fix an orientation at infinity $\phipartial$.  If $\gamma$ is a vertex-to-vertex gallery that is positively folded with respect to $\phipartial$ and starts with the face $p_0$, and $H$ is load-bearing for $\gamma$ at $p_i$, then exactly one of the following is true:
\begin{enumerate}
 \item $\gamma$ crosses $H$ positively at $p_i$, which means that $c_{i-1}$ is on the negative side of $H$ but $c_i$ is on the positive side;
 \item $\gamma$ has a positive fold at $p_i$, which means that $c_{i-1}=c_i$ are both on the positive side of $H$; or
 \item $H$ contains the vertex $p_0$ and separates $c_0$ from $C_{-\phi}$. 
\end{enumerate}
\end{lemma}
\begin{proof}
Since $H$ is load-bearing, by definition $H$ separates $c_i$ from $C_{-\phi}$, and so the alcove $c_i$ must be on the positive side of $H$.  Thus if $i > 0$ then $\gamma$ has either a positive crossing or a positive fold at $p_i$, and these possibilities are mutually exclusive.  If $i = 0$ then $p_i = p_0$ so $H$ contains $p_0$ and separates $c_i = c_0$ from $C_{-\phi}$.
\end{proof}

We can now say when a vertex-to-vertex combinatorial gallery has load-bearing walls which contain its first vertex.  

\begin{lemma}\label{lem:folds+crossing}
Fix an orientation at infinity $\phipartial = \phipartialw$ and let $\gamma$ be a vertex-to-vertex gallery that is positively folded with respect to $\phipartial$, has first vertex $p_0$, and has first alcove $c_0$.  If $p_0 = \lambda$ is in the coroot lattice then $\load_\phi(p_0,\gamma) = 0$ if and only if $c_0$ is the alcove $t^\lambda w_0 w\fa$.
\end{lemma}
\begin{proof}
The orientation at infinity $\phipartial = \phipartialw$ is induced by the labeling $\phi_{p_0,x}$ where $x = t^\lambda w$.  The alcove containing $p_0$ with spherical direction $w_0$ times the spherical direction of $x$ is the alcove $t^{\lambda}w_0 w\fa$, and no wall containing $p_0$ separates this alcove from $C_{-\phi}$.  Thus if $c_0$ is the alcove $t^\lambda w_0 w \fa$ then $\load_\phi(p_0,\gamma) = 0$, since case (3) in Lemma~\ref{load bearing} will not occur.  Conversely, if $c_0$ is any alcove $t^\lambda u \fa$ with $u \neq w_0 w$ then there is at least one wall containing $p_0$ which separates $c_0$ from $C_{-\phi}$, hence $\load_\phi(p_0,\gamma) > 0$.
\end{proof}

The notions of dimension for vertex-to-vertex and alcove-to-alcove combinatorial galleries then coincide in the following situation.

\begin{corollary}\label{cor:folds+crossings}  Fix an orientation at infinity $\phipartial = \phipartialw$ and suppose $\gamma$ is an alcove-to-alcove combinatorial gallery which is positively folded with respect to $\phipartial$.  Let $c_0$ be the first alcove of $\gamma$.  If $\gamma^\sharp$ is the associated vertex-to-vertex gallery with first vertex $p_0 \in R^\vee$ then $\dim_\phi(\gamma) \leq \dim_{\phi}(\gamma^\sharp)$ with equality 
$$\dim_\phi(\gamma^\sharp) = \dim_\phi(\gamma) = P_\phi(\gamma) + F_\phi(\gamma)$$ if and only if $c_0$ has spherical direction $w_0w$.
\end{corollary}

\begin{proof}  By Lemma~\ref{load bearing}, the load-bearing walls of $\gamma^\sharp$ at the faces $p_i$ for $i > 1$ correspond to either the positive crossings in $\gamma$, in case (1), or the positive folds in $\gamma$, in case (2), and these possibilities are mutually exclusive.  Thus $\dim_\phi(\gamma) \leq \dim_\phi(\gamma^\sharp)$.  By  Lemma~\ref{lem:folds+crossing}, the gallery $\gamma^\sharp$ has no load-bearing walls at $p_0$ if and only if the spherical direction of $c_0$ is $w_0w$.  The result follows.
\end{proof}


\subsection{Root operators}\label{sec:RootOps}

We now recall the definition and key properties of the root operators $e_\alpha$ and $f_\alpha$, which were introduced by Gaussent and Littelmann in~\cite{GaussentLittelmann}.   A reader familiar with root operators could skip this section.  These operators act on the set of positively folded galleries of a fixed type, and will allow us to manipulate such galleries while keeping track of their dimensions.  We remark that root operators themselves, like positively folded galleries, can be expressed using retractions in Bruhat--Tits buildings. This connection is made precise in \cite{SchwerOperators} by the second author.

Recall that we denote by $\gamma_\lambda$ a minimal vertex-to-vertex combinatorial gallery from the origin to $\lambda \in R^\vee$, and by $\Gamma^+(\gamma_\lambda,\nu)$ the set of vertex-to-vertex combinatorial galleries which are of the same type as $\gamma_\lambda$, start at the origin and end in a cocharacter $\nu$, and are positively folded with respect to the standard orientation.  

We now define the root operators $e_\alpha$ and $f_\alpha$ for all simple roots $\alpha$.  The reason we assume below that $\lambda$ does not lie on any wall containing the origin is so that the minimal vertex-to-vertex gallery $\gamma_\lambda$ is minimal in the sense considered in~\cite{GaussentLittelmann}; compare Remark~\ref{rem:GLminimal}.
 
\begin{notation}\label{not:operators}  Let  $\lambda \in R^\vee$ be dominant and assume that $\lambda$ does not lie on any wall containing the origin.   
Let $\gamma=\left(p_0\subset c_0 \supset p_1 \subset c_1 \supset  \dots \subset c_n\supset p_{n+1} \right)$ 
be a vertex-to-vertex combinatorial gallery with $\gamma\in\Gamma^+(\gamma_\lambda,\nu)$.  Let $\alpha$ be a simple root, and define $m = m(\gamma,\alpha)\in\Z$ to be minimal such that there exists $i$ with $p_i$ contained in the hyperplane $H_{\alpha, m}$. Note that $m\leq 0$ as $p_0$ is the origin. 
 
There are the following cases:
\begin{itemize}
\item[(I)] $m\leq -1$. In this case let  $k = k(\gamma,\alpha,\mbox{I})$ be minimal with $p_k\subset H_{\alpha, m}$, and let $j = j(\gamma,\alpha,\mbox{I})$ be maximal such that $0\leq j\leq k$ and $p_j\subset H_{\alpha, m+1}$. 
\item[(II)] $m\leq \langle \alpha, \nu \rangle - 1$. In this case let $j = j(\gamma,\alpha,\mbox{II})$ be maximal with $p_j\subset H_{\alpha, m}$, and let $k = k(\gamma,\alpha,\mbox{II})$  be minimal such that $j\leq k\leq n+1$ and $p_k\subset H_{\alpha, m+1}$.
\end{itemize}
\end{notation}

\begin{definition}\label{def:GLoperators}
With all notation as in \ref{not:operators}, the \emph{root operators} $e_\alpha$ and $f_\alpha$ are defined as follows:
\begin{itemize}
\item In case (I) let $e_\alpha(\gamma)$ be the vertex-to-vertex combinatorial gallery defined by
\[ e_\alpha(\gamma)=(p_0\subset c'_0\supset p'_1 \subset c'_1 \supset  \dots \subset c'_n \supset p'_{n+1}), \]
where 
\[ c'_i = 
\left\{ \begin{array}{ll}
	  c_i & \text{for } i \leq j-1,  \\
	  s_{\alpha, m+1}(c_i) & \text{for } j\leq i \leq k-1, \\
	  t^{\alpha^\vee}(c_i) & \text{for } i\geq k. \\
        \end{array}
\right .
\]

\item  
In case (II) let $f_\alpha(\gamma)$ be the vertex-to-vertex combinatorial gallery defined by
\[f_\alpha(\gamma)=(p_0\subset c'_0\supset p'_1 \subset c'_1 \supset \dots \subset c'_n \supset p'_{n+1} ), \]
where
\[ c'_i = 
\left\{ \begin{array}{ll}
	  c_i & \text{for } i< j, \\
	  s_{\alpha, m}(c_i) & \text{for } j\leq i< k, \\
	  t^{-\alpha^\vee}(c_i) & \text{for } i\geq k. \\
        \end{array}
\right .
\]
\end{itemize} 
\end{definition}

\begin{remark}
Note that cases (I) and (II) are not mutually exclusive so that often both $e_\alpha$ and $f_\alpha$ will be defined. In \cite{GaussentLittelmann}, a third case is also considered, in which an operator $\tilde e_\alpha$ is defined. Since we will not make use of this operator in the present paper, we do not include its definition here. 
\end{remark}

\begin{example}\label{ex:root-operator} 
Figure~\ref{fig:e_and_f} shows all possible applications of the root operators $e_\alpha$ and $f_\alpha$ to the gallery $\gamma$ in the first frame, which runs from the origin $v_0$ to $\nu \in R^\vee$.   Let  $\gamma_\lambda$ denote the minimal gallery from~$v_0$ to $\lambda \in R^\vee$ shown in the second frame of the bottom row.  Each gallery depicted is of the same type as $\gamma_\lambda$,  and is positively folded with respect to the standard orientation. The hyperplanes $H_{\alpha,k}$ are the horizontal lines in this figure, with $H_{\alpha,0}$ the heavier horizontal line containing $v_0$.  \end{example}

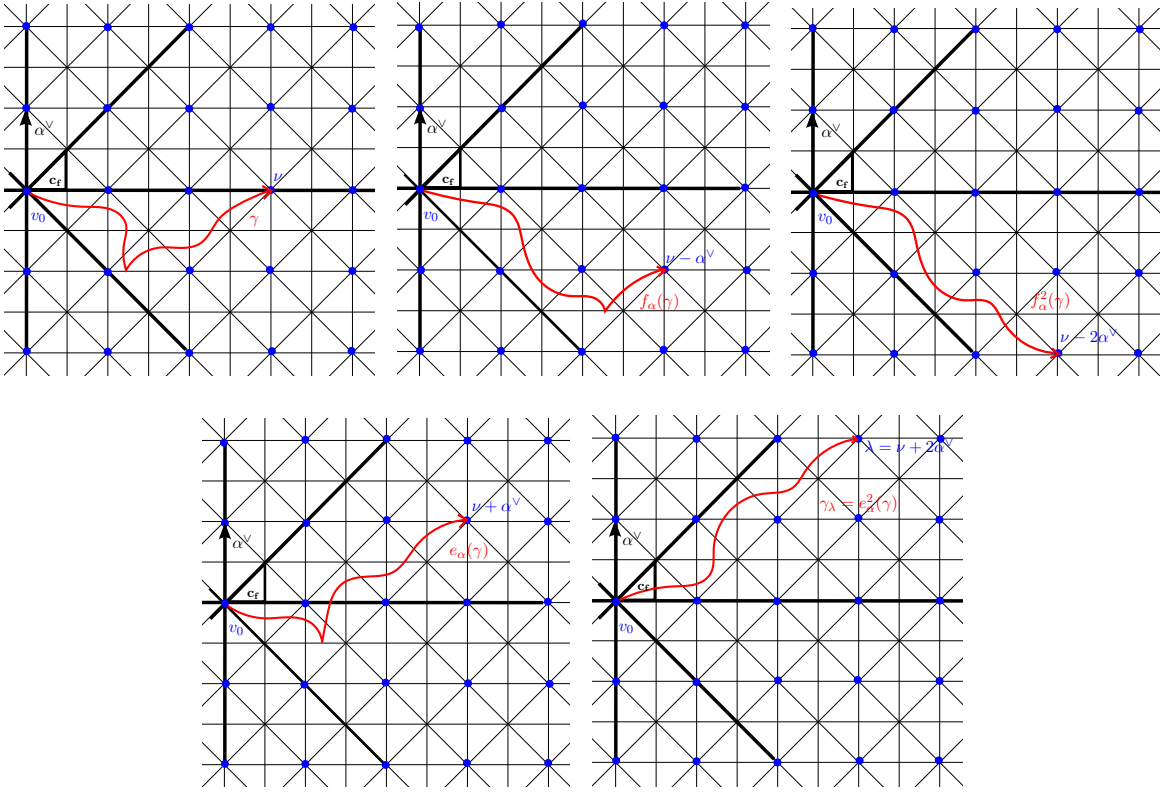
\begin{figure}[ht]
\begin{center}
\resizebox{\textwidth}{!}
{
\begin{overpic}{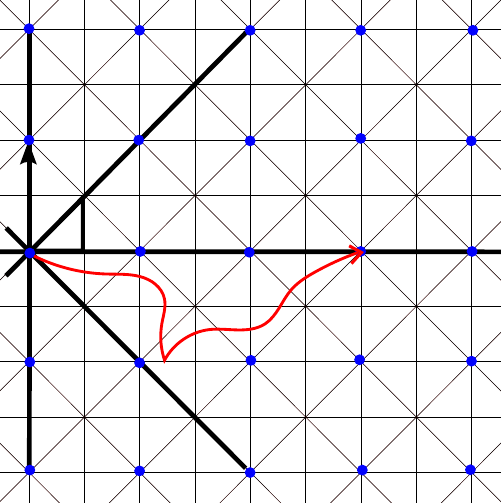}
\put(7,42){\textcolor{blue}{$v_0$}}
\put(8,65){\bf{$\alpha^\vee$}}
\put(12,51.5){\footnotesize{$\fa$}}
\put(72.5,52.5){\textcolor{blue}{$\nu$}}
\put(66,41){\textcolor{red}{\bf $\gamma$}}
\end{overpic}
\hspace{2ex}
\begin{overpic}{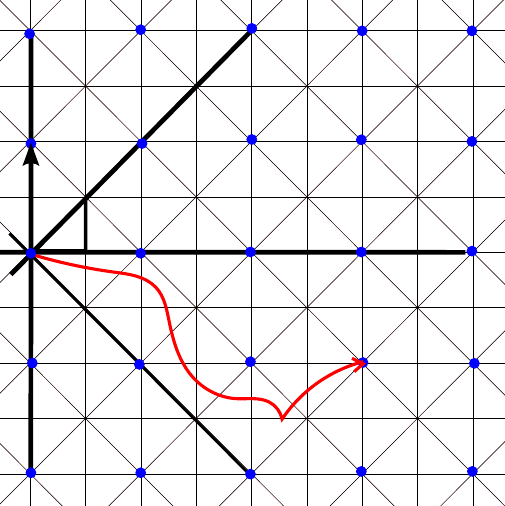}
\put(7,42){\textcolor{blue}{$v_0$}}
\put(8,65){\bf{$\alpha^\vee$}}
\put(12,51.5){\footnotesize{$\fa$}}
\put(72,30){\textcolor{blue}{$\nu-\alpha^\vee$}}
\put(65,18.5){\textcolor{red}{\bf $f_\alpha(\gamma)$}}
\end{overpic}
\hspace{2ex}
\begin{overpic}{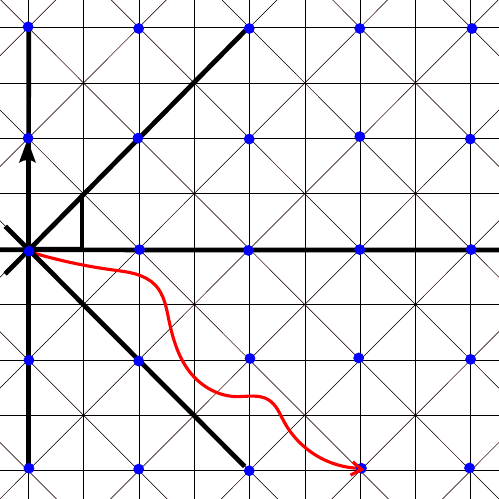}
\put(7,42){\textcolor{blue}{$v_0$}}
\put(8,65){\bf{$\alpha^\vee$}}
\put(12,51.5){\footnotesize{$\fa$}}
\put(73,9){\textcolor{blue}{$\nu-2\alpha^\vee$}}
\put(65,19){\textcolor{red}{\bf $f^2_\alpha(\gamma)$}}
\end{overpic}
}

\vspace{3ex}
\resizebox{0.66\textwidth}{!}
{
\begin{overpic}{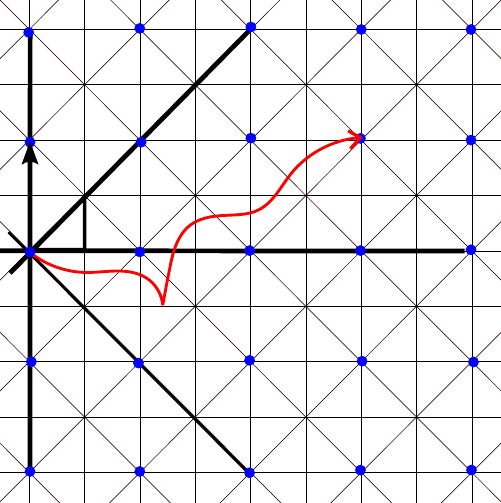}
\put(7,42){\textcolor{blue}{$v_0$}}
\put(8,65){\bf{$\alpha^\vee$}}
\put(12,51.5){\footnotesize{$\fa$}}
\put(73,75){\textcolor{blue}{$\nu+\alpha^\vee$}}
\put(67,63){\textcolor{red}{\bf $e_\alpha(\gamma)$}}
\end{overpic}
\hspace{2ex}
\begin{overpic}{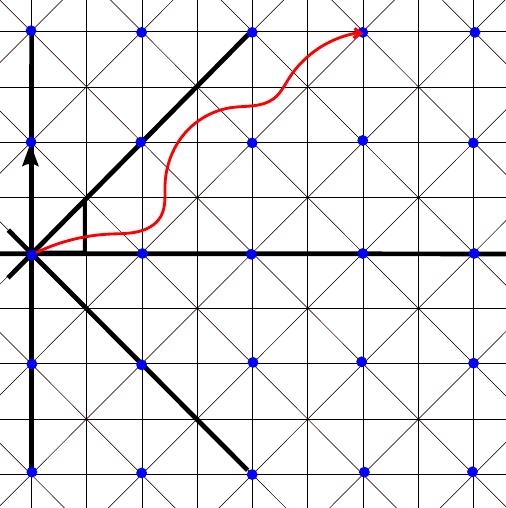}
\put(7,42){\textcolor{blue}{$v_0$}}
\put(8,65){\bf{$\alpha^\vee$}}
\put(12,53){\footnotesize{$\fa$}}
\put(73,90){\textcolor{blue}{$\lambda = \nu+2\alpha^\vee$}}
\put(61,75){\textcolor{red}{\bf $\gamma_\lambda = e^2_\alpha(\gamma)$}}
\end{overpic}
}

\caption{An illustration of the root operators $e_{\alpha}$ and $f_\alpha$ in type $\tilde C_2$.}
\label{fig:e_and_f}
\end{center}
\end{figure}

Lemmas 5, 6, and 7 in \cite{GaussentLittelmann} summarize several of the main properties of root operators.  Of these, the properties that we will need are gathered in the next result.  Note that we are only considering the standard orientation.

\begin{lemma}[Properties of Root Operators]\label{lem:posFoldsV} We continue all notation from~\ref{not:operators}. Recall that $\gamma\in\Gamma^+(\gamma_\lambda,\nu)$.  
\begin{enumerate}
\item\label{fold1} $e_{\alpha}$ is not defined if and only if $m=0$.  If it is defined, then $e_{\alpha}(\gamma)\in\Gamma^+(\gamma_\lambda, \nu+\alpha^\vee)$ and the dimension with respect to the standard orientation $\phipartialo$ increases by one, that is, $$\dim_{\phi_0}(e_{\alpha}(\gamma))=\dim_{\phi_0}(\gamma)+1.$$  
\item\label{fold2} $f_{\alpha}$ is not defined if and only if $m=\langle \alpha, \nu \rangle$. If it is defined then $f_{\alpha,\phi}(\gamma)\in\Gamma^+(\gamma_\lambda,\nu-\alpha^\vee)$ and the dimension with respect to the standard orientation $\phipartialo$ decreases by one, that is, $$\dim_{\phi_0}(f_{\alpha}(\gamma))=\dim_{\phi_0}(\gamma)-1.$$
\item\label{fold3} If $e_{\alpha}$ is defined, then $f_{\alpha}(e_{\alpha}(\gamma))$  is defined and equal to $\gamma$. 
If $f_{\alpha}$ is defined, then $e_{\alpha}(f_{\alpha}(\gamma))$  is defined and equal to $\gamma$. 
\item\label{fold4} If $p$ is maximal such that $f_{\alpha}^p(\gamma)$ is defined and $q$ is maximal such that $e_{\alpha}^q(\gamma)$ is defined, then $p-q = \langle \alpha, \nu \rangle$.
   
\end{enumerate}
\end{lemma}

\begin{remark}\label{rem:operators}
In our constructions in Section~\ref{sec:Construction} we will be paying close attention to the effects of root operators $e_\alpha$ on the first and last alcoves of a combinatorial gallery.  From Definition~\ref{def:GLoperators}, and as seen in Figure~\ref{fig:e_and_f}, it is possible for the first alcove to be either fixed or reflected by these operators, and for the last alcove to be either translated or reflected. 

In order for the first alcove to change under the application of a root operator $e_\alpha$ it needs to be in the part of the gallery that is reflected by that root operator. This is the case only if the root operator $e_\alpha$ is defined exactly once. In the proof of Lemma~\ref{lem:rootops1} we make use of this fact and essentially show that we never apply the last possible instance of a root operator $e_\alpha$.   

Now suppose the last alcove of a gallery is reflected under the operator $e_{\alpha_i}$.  Then its spherical direction gets shortened and multiplied on the left by $s_i$.  Heuristically, this final alcove will be brought closer to the chamber at infinity with standard label $\id$ by this root operator. This is essentially the reason why the number of folds is bounded by the length of  the longest word in $\sW$, compare Corollary~\ref{w0Folds}, and why in case the final alcove is in the identity position it will never be reflected by a root operator $e_\alpha$, compare Lemma~\ref{lem:rootops1}.  

Similar comments apply for the root operators $f_\alpha$.
 \end{remark}

We will also need the following result and definition from~\cite{GaussentLittelmann}.  Recall that we denote by $\rho$ the half-sum of the positive roots.  

\begin{prop}[Proposition 3 in~\cite{GaussentLittelmann}]\label{prop:LS}  
Let  $\lambda \in R^\vee$ be dominant and assume that $\lambda$ does not lie on any wall containing the origin.  If $\gamma\in\Gamma^+(\gamma_\lambda, \nu)$ then $\dim_{\phi_0}(\gamma)\leq \langle \rho,  \lambda + \nu \rangle $.
\end{prop}

\begin{definition}\label{def:LS}  With the same notation and assumptions as in Proposition~\ref{prop:LS}, a gallery $\gamma \in \Gamma^+(\gamma_\lambda, \nu)$ is called an \emph{LS-gallery} if its dimension with respect to the standard orientation achieves the maximum dimension $\langle \rho,  \lambda + \nu \rangle$  given by Proposition~\ref{prop:LS}. 
\end{definition}

\begin{remark}\label{rem:LS} 
It is shown in \cite[Corollary 2]{GaussentLittelmann} that any LS-gallery of type $\gamma_\lambda$ can be obtained from the minimal gallery $\gamma_\lambda$ via a finite number of applications of root operators $f_\alpha$, or equivalently $\gamma_\lambda$ can be obtained from $\gamma$ via finitely many applications of operators $e_\alpha$. 
\end{remark}

\begin{remark}  Various generalizations of root operators are known to experts; see for example J.E. Humphrey's Mathoverflow answer~\cite{HumphreysMO}.  In our constructions in Section~\ref{sec:Construction} we will be acting on the left by $\sW$ to transform galleries so that the root operators $e_\alpha$ and $f_\alpha$ can be applied.  We could instead have defined root operators with respect to arbitrary orientations and acted by these on the original galleries. 

Yet another (new) approach we discussed in an earlier version of this work was to  introduce root operators centered at a chamber in the spherical building at infinity. In this case one would apply them to infinite galleries that ``start'' at this chamber at infinity. These operators share a lot of properties with the ones introduced by Gaussent and Littelman, yet have the main advantage that infinitely many will be defined for a given gallery and given simple root at any time.   
However, we ended up simplifying the proof by not using this approach. 
\end{remark}

See also Section~\ref{sec:MV-GL-LS} for further remarks on LS-galleries and root operators. 


\subsection{Counting folds and crossings}\label{sec:Counting}

The purpose of this section is to collect some useful relations among the length and dimension of a combinatorial alcove-to-alcove gallery and its number of folds and crossings.  Many of these results are elementary and will be known to experts.

Write $C(\gamma)$ for the number of (unfolded) crossings in a combinatorial alcove-to-alcove gallery $\gamma$.  We first observe:

\begin{lemma}\label{LengthCrossings}
Let $\tilde\gamma$ be a minimal alcove-to-alcove gallery of type $\vec x$, and suppose that $\gamma$ is a gallery obtained by positively folding $\tilde\gamma$ with respect to the orientation $\phipartial$.  Then
\begin{equation}\label{E:LengthCrossings}
\ell(x) = C(\gamma) + F_\phi(\gamma) =  P_\phi(\gamma) + N_\phi(\gamma) + F_\phi(\gamma).  
\end{equation}
\end{lemma}

\begin{proof}
Since the original gallery $\tilde\gamma$ is minimal, we have $\ell(x) = C(\tilde\gamma)$, the number of crossings in $\tilde\gamma$.   After positively folding $\tilde\gamma$ to obtain $\gamma$, every crossing of the original gallery $\tilde\gamma$ has become either a positive crossing, a negative crossing, or a (positive) fold, with these three possibilities being mutually exclusive.  Thus 
$$
\ell(x) = C(\gamma) + F_\phi(\gamma) \quad \mbox{and} \quad C(\gamma) = P_\phi(\gamma) + N_\phi(\gamma).
$$
The result follows.
\end{proof}

The next results provide upper and lower bounds on the number of folds in a positively folded gallery.  
Its proof generalizes an observation made in Remark~4.4 of \cite{Ram}.
Let $\eta_1:\aW \to \sW$ be the function given by $\eta_1(x)=w$ where $x=t^\lambda w\in\aW$ with $\lambda \in R^\vee$ and $w \in \sW$.  That is, $\eta_1$ records the spherical direction of an affine Weyl group element $x$.

\begin{lemma}\label{wFolds2}
Let $\phipartial = \phipartialw$ be the orientation at infinity induced by the labeling $\phi = \phi_{v_0,w}$.  Then for all $x=t^\mu u$ and $y=t^\lambda v$  in $\aW$, all galleries $\gamma\in\Gamma^+_w(x, y)$ satisfy 
\begin{equation} \label{eq1}
F_{\phi}(\gamma) \leq \ell(\phi(\eta_1(y))) - \ell(\phi(\eta_1(x))) = \ell(w^{-1}v)-\ell(w^{-1}u).
\end{equation}
In particular, if $y=t^\lambda$ is a pure translation,  all galleries $\gamma\in\Gamma^+_{w}(x, y)$ satisfy 
\begin{equation} \label{eq2}
F_{\phi}(\gamma) \leq \ell(\eta_1(x))=\ell(u).
\end{equation}
\end{lemma}
\begin{proof}
Write $\gamma_x$ for the minimal gallery $\fa \rightsquigarrow x\fa=\x$ which was positively folded with respect to $\phipartial$ in order to obtain $\gamma$. Suppose $\gamma$ has $F_{\phi}(\gamma) = f$ folds.  For $0 \leq i \leq f$ let $\gamma_x^{(i)}$ be the gallery obtained from $\gamma_x$ by applying the same folds as in $\gamma$ up to the $i$th fold, and no more folds after that. 
We write $c^i$ for the final alcove of the gallery $\gamma_x^{(i)}$ and put $w^i=\eta_1(c^i)$.  
Then in particular, $\gamma_x^{(0)} = \gamma_x$, $c^0=\x$, and $w^0=u$, and $\gamma_x^{(f)} = \gamma$, $c^f=\y$, and $w^f=v$.
As all folds are positive, by applying each successive fold we reflect the last alcove to be closer to $-C_\phi$, the opposite chamber at infinity with respect to $\phi$. It follows that
$
\ell(\phi(w^{i-1}\fa)) < \ell(\phi(w^{i}\fa))
$
and thus
$
\ell(\phi(c^{0})) + f \leq \ell(\phi(c^{f})),  
$
which gives us the first part of the lemma. 

To see the rest suppose $y$ is a pure translation. By Definition~\ref{def:InducedLabelingW} of the induced labeling $\phi_w=\phi$ we get  
$ \phi(\eta_1(y))= \phi_0(w^{-1})=w^{-1}$ and $\phi(\eta_1(x))= \phi_0(w^{-1}u)=w^{-1}u$. 
Combining this with Equation~\eqref{eq1} we obtain
\[
F_{\phi}(\gamma) \leq \ell(\phi(\eta_1(y))) - \ell(\phi(\eta_1(x))) \leq \ell(w^{-1}) - \left(\ell(w^{-1}) - \ell(u)\right) = \ell(u).
\]  Hence the lemma. 
\end{proof}

As $w_0$ is the longest element of the spherical Weyl group we get the corollary. 

\begin{corollary}\label{w0Folds}
For any orientation $\phipartial$, any $x \in \aW$, any minimal gallery $\gamma_x:\fa \rightsquigarrow x\fa$, and any gallery $\gamma$ obtained by positively folding $\gamma_x$ with respect to $\phipartial$, we have 
\begin{equation*} F_{\phi}(\gamma) \leq \ell(w_0).\end{equation*}
\end{corollary}

In the following result we obtain a lower bound on the number of folds in terms of reflection length.  Recall that a \emph{reflection} in $\sW$ is a conjugate $ws_iw^{-1}$ where $w \in \sW$ and $s_i \in S$.  The set of reflections in $\sW$ is denoted by $R$.  Then $\sW$ is generated by $R$, and we denote by $\ell_R(w)$ the word length with respect to $R$ of an element $w \in \sW$.  Similarly, the set $\tilde R$ of \emph{reflections} in $\aW$ is the set of $W$-conjugates of elements of $\tilde S$, and we denote by $\ell_{\tilde R}(x)$ the word length with respect to $\tilde R$ of an element $x \in \aW$.

\begin{lemma}\label{FoldsLower}  For any orientation $\phipartial$, any minimal gallery $\gamma_x:\fa \rightsquigarrow x\fa$, and any gallery $\gamma:\fa \rightsquigarrow y\fa$ obtained by positively folding $\gamma_x$ with respect to $\phipartial$, we have 
\begin{equation*}\ell_{\tilde R}(xy^{-1}) \leq F_{\phi}(\gamma). \end{equation*}
\end{lemma}

\begin{proof}  Suppose $F_\phi(\gamma) = k$.  Then there are reflections $r_1, \dots, r_k$ (in $\aW$) which correspond to the PRS folding sequence applied to $\gamma_x$ in order to obtain $\gamma$.   Since this folding sequence takes the alcove $x\fa$ to the alcove $y\fa$, we have $r_k \dots r_1 x\fa = y\fa$ and so $xy^{-1} = r_1 \cdots r_k$.  Thus $\ell_R(xy^{-1}) \leq k = F_\phi(\gamma)$ as required.  
\end{proof}

We next consider all minimal galleries which start and end in particular alcoves.

\begin{lemma}\label{lem:defect}  For any orientation $\phipartial$ and any $x,y \in \aW$, any two minimal alcove-to-alcove galleries from $x\fa$ to $y\fa$ have the same dimension with respect to $\phipartial$.
\end{lemma}
\begin{proof}  Let $\gamma$ and $\gamma'$ be minimal alcove-to-alcove galleries from $x\fa$ to $y\fa$.  Since $\gamma$ and $\gamma'$ are minimal they have no folds, so we have by definition that $\dim_\phi(\gamma) = P_\phi(\gamma)$ and $\dim_\phi(\gamma') = P_\phi(\gamma')$.  Now $\gamma$ and $\gamma'$ cross the same hyperplanes in the same direction, just in a different order.  Hence they have the same number of positive crossings, that is, $P_\phi(\gamma) = P_\phi(\gamma')$.  So $\dim_\phi(\gamma) = \dim_\phi(\gamma')$ as required.
\end{proof}

Using Lemma~\ref{lem:defect} and its proof we may define the dimension of an alcove to be the dimension of any minimal alcove-to-alcove gallery from $\fa$ to that alcove.

\begin{definition}  Let $\phipartial$ be an orientation at infinity and let $x \in \aW$.  The \emph{dimension of $x$ with respect to $\phipartial$}, denoted $\dim_\phi(x)$, is given by
$$ \dim_\phi(x) = \dim_\phi(\gamma_x) = P_\phi(\gamma_x)$$
where $\gamma_x:\fa \rightsquigarrow x\fa$ is any minimal alcove-to-alcove gallery.
\end{definition}

We will need the following result concerning the length and dimension of translations.

\begin{lemma}\label{lem:length}  Let $b = t^\mu$ be a dominant pure translation.  For each $w \in \sW$ write $b^w$ for the translation $t^{w\mu}$.  Then $\ell(b) = \ell(b^{w})$ and $\dim_{-\phi_0}(b) = \dim_{\phi_{w_0w}}(b^w) = 0$, for all $w \in \sW$.
\end{lemma}

\begin{proof}  The equality $\ell(b) = \ell(b^w)$ is a standard result, a reference for which is, for example, (2.4.2) in \cite{Macdonald}.  Now $\dim_{-\phi_0}(b) = 0$ since $b$ is in the dominant Weyl chamber $\Cf$, so every crossing in a minimal gallery from $\fa$ to $b\fa$ is negative with respect to the orientation $-\phi_0$.  Similarly, since the vertex $w\mu$ (if not the alcove $t^{w\mu}\fa$) is contained in the Weyl chamber $w\Cf$, every crossing in a minimal gallery from $\fa$ to $b^w\fa$ is negative with respect to the orientation $\phi_{w_0w} = -\phi_w$, hence $\dim_{-\phi_w}(b^w) = 0$.
\end{proof}

We also consider the dimension of combinatorial alcove-to-alcove galleries $\gamma$ which begin and end in the same alcove. 

\begin{lemma}\label{lem:loop crossings}
 For any orientation at infinity $\phipartial$, any alcove $\x$, and any alcove-to-alcove gallery $\gamma:\x \rightsquigarrow \x$,
 \begin{equation}\label{E:LoopCrossings} P_{\phi}(\gamma) = N_{\phi}(\gamma) = \frac{1}{2}C(\gamma).
 \end{equation}
 \end{lemma}
\begin{proof}  For each hyperplane $H$, the complement of $H$ in the apartment $\App$ consists of two disjoint open half-spaces (which are interchanged by the affine reflection with fixed set $H$).  Since $\gamma$ has the same first and last alcove, it follows that each hyperplane $H$ crossed by $\gamma$ is crossed an even number of times, and the number of positive crossings of $H$ by $\gamma$ is equal to the number of negative crossings of $H$ by $\gamma$.  Now sum over all hyperplanes crossed by $\gamma$. This argument is independent of the orientation $\phipartial$.
\end{proof}

The following lemma will be very useful for computing dimension of folded galleries.

\begin{lemma}\label{lem:gallery dim}  Fix an orientation at infinity $\phipartial$.
Suppose that $\gamma_x:\fa \rightsquigarrow x\fa$ is a minimal alcove-to-alcove gallery and that $\gamma:\fa \rightsquigarrow y\fa$ is a gallery obtained by positively folding $\gamma_x$ with respect to the orientation $\phipartial$.  Then 
\begin{equation}
 \dim_\phi(\gamma) = \frac{1}{2}\left[\ell(x) + F_\phi(\gamma) - \ell(y) \right] + \dim_\phi(y). 
  \end{equation}
In particular, if $y = 1$, that is, the gallery $\gamma$ both starts and ends at $\fa$, then
\begin{equation}\label{LoopDimFormula}
 \dim_\phi(\gamma) = \frac{1}{2}\left[\ell(x) + F_\phi(\gamma) \right]. 
  \end{equation}
\end{lemma}
\begin{proof}  Let $\sigma:\fa \rightsquigarrow \fa$ be the alcove-to-alcove gallery obtained by concatenating $\gamma$ with the reverse of a minimal alcove-to-alcove gallery $\gamma_y:\fa \rightsquigarrow y\fa$.  Then $\sigma$ starts and ends at $\fa$, so by Lemma~\ref{lem:loop crossings}, we have 
\[ P_\phi(\sigma) = N_\phi(\sigma) = \frac{1}{2}C(\sigma).\]
The positive crossings in the reverse of $\gamma_y$ are exactly the negative crossings in $\gamma_y$, and since $\gamma_y$ has no folds, Lemma~\ref{LengthCrossings} gives $N_\phi(\gamma_y) = \ell(y) - P_\phi(\gamma_y)$.  Hence 
$$P_\phi(\sigma) = P_\phi(\gamma) + \ell(y) - P_\phi(\gamma_y) = P_\phi(\gamma) + \ell(y) - \dim_\phi(y).$$  Also, since $
\gamma_y$ has no folds, the number of crossings in $\sigma$ is given by $$C(\sigma) = C(\gamma) + C(\gamma_y) = C(\gamma) + \ell(y).$$  Then using Lemma~\ref{LengthCrossings} again it follows that 
\[
P_\phi(\gamma) = \frac{1}{2}\left[ C(\gamma) + \ell(y)\right] -\ell(y) + \dim_\phi(y) = \frac{1}{2}[\ell(x) - F_\phi(\gamma) - \ell(y)] + \dim_\phi(y).
\]
Thus
\begin{eqnarray*}
\dim_\phi(\gamma)  & = & P_\phi(\gamma) + F_\phi(\gamma) \\ & = & \frac{1}{2}[\ell(x) - F_\phi(\gamma)  - \ell(y)] + \dim_\phi(y) + F_\phi(\gamma) \\ & = & \frac{1}{2}[\ell(x) + F_\phi(\gamma) - \ell(y)] + \dim_\phi(y) \end{eqnarray*}
as required.
\end{proof}

The following corollary of Lemma~\ref{lem:gallery dim} gives a sufficient condition for a gallery $\fa \rightsquigarrow \fa$ to maximize dimension.
 
\begin{corollary}\label{wMaxDim}
Suppose that $\gamma:\fa \rightsquigarrow \fa$ is an alcove-to-alcove gallery obtained by positively folding a minimal gallery $\gamma_x:\fa \rightsquigarrow x\fa$ with respect to the orientation $\phipartial = \phipartialw$.  If the number of folds in $\gamma$ is equal to $\ell(w_0)$ then $\gamma$ maximizes dimension over all galleries in $\Gamma^+_u(x,1)$ for all $u \in \sW$, that is,
$$
\dim_\phi(\gamma) = \max_{u \in \sW}\{ \dim_{\phi_u}(\gamma') \mid \gamma' \in \Gamma^+_u(x,1)\},
$$  and $$\dim_\phi(\gamma) = \frac{1}{2}\left[ \ell(x) + \ell(w_0) \right].$$ 
\end{corollary}

\begin{proof}  If $\gamma$ has $\ell(w_0)$ folds then by Lemma~\ref{lem:gallery dim} in the case $y = 1$ we have $\dim_\phi(\gamma) = \frac{1}{2}\left[ \ell(x) + \ell(w_0) \right]$.  So it suffices to prove that for all $u \in \sW$ and all $\sigma \in \Gamma^+_u(x,1)$, $\dim_{\phi_u}(\sigma) \leq \frac{1}{2}[\ell(x) + \ell(w_0)]$.  For this, recall that $F_{\phi_u}(\sigma) \leq \ell(w_0)$ by Corollary \ref{w0Folds}.  Therefore by Lemma~\ref{lem:gallery dim} in the case $z = 1$
\[ \dim_{\phi_u}(\sigma) =   \frac{1}{2}[\ell(x) + F_{\phi_u}(\sigma)] \leq  \frac{1}{2}[\ell(x) + \ell(w_0)].\]
This completes the proof.
\end{proof}


\subsection{Independence of minimal gallery}\label{sec:BraidMoves}

We now show that the number of folds in and the dimension of a gallery do not depend upon the choice of an original minimal gallery, in the sense made precise in the statement of Proposition~\ref{prop:braids} below.  This result is implicit in~\cite{PRS} since Theorem 7.1 in that work (which we recall as Theorem~\ref{T:PRS} below) does not depend upon choices of reduced words for Weyl group elements.  In this section, we work only with alcove-to-alcove galleries.

\begin{prop}\label{prop:braids}  Let $\gamma_x:\fa \rightsquigarrow x\fa$ be a minimal gallery and suppose that $\gamma:\fa \rightsquigarrow y\fa$ is a gallery obtained by positively folding $\gamma_x$ with respect to the orientation $\phipartial$.   Then for every minimal gallery $\gamma'_x:\fa \rightsquigarrow x\fa$, there is a gallery $\gamma': \fa \rightsquigarrow y\fa$ obtained by positively folding $\gamma_x'$ with respect to the same orientation $\phipartial$, such that  $F_\phi(\gamma') = F_\phi(\gamma)$ and hence $\dim_\phi(\gamma') = \dim_\phi(\gamma)$.
\end{prop}

\begin{table}[htp]
\caption{Table for proof of Proposition~\ref{prop:braids}}
\begin{center}
\begin{tabular}{|l|l|l|l|l|}
\hline
Number & Value(s) of $m$ & Reflection &  Common & Reflection  \\
of folds & & product for $\tau$ & orientation & product for $\tau'$ \\ \hline
2 & 3, 4, 6 & $(t)(s)$     & $1$ & $(tst)(t)$ \\
& & $(sts)(s)$ & $t$  & $(t)(tst)$ \\ 
\cline{2-5}
& 4 & $(sts)(s)$ & $ts$  & $(tst)(sts)$ \\ 
 \cline{2-5}
& 6 & $(sts)(s)$ & $ts$ & $(tst)(tstst)$ \\ 
& & $(sts)(s)$ & $tst$  & $(tstst)(ststs)$ \\ 
& & $(sts)(s)$ & $tsts$  & $(ststs)(sts)$ \\ 
\hline
3 & 4 & $(s)(t)(s)$ & $1$ & $(tst)(s)(t)$ \\ 
& & $(t)(tst)(s)$ & $t$ & $(tst)(t)(tst)$ \\
\cline{2-5}
 & 6  &$(s)(t)(s)$ & $1$ & $(ststs)(s)(t)$ \\ 
& & $(sts)(t)(s)$ & $1$ & $(tstst)(s)(t)$ \\ 
& & $(t)(tst)(s)$ & $t$ & $(ststs)(t)(tst)$ \\ 
& & $(sts)(tst)(s)$ & $1$ & $(t)(s)(tst)$ \\
& & $(t)(tstst)(s)$ & $t$ & $(tstst)(t)(tst)$ \\
& & $(tst)(tstst)(s)$ & $ts$ & $(ststs)(tst)(tstst)$ \\
& & $(t)(tstst)(s)$ & $ts$ & $(tstst)(tst)(tstst)$ \\
& & $(tstst)(ststs)(s)$ & $tst$ & $(ststs)(tstst)(ststs)$ \\
& & $(t)(tst)(sts)$ & $t$ & $(tst)(t)(tstst)$ \\ 
\hline
4 & 6 & $(t)(s)(t)(s)$ & $1$ & $(ststs)(t)(s)(t)$ \\
& & $(tst)(s)(t)(s)$ & $1$ & $(s)(tst)(s)(t)$ \\
& & $(tst)(t)(tst)(s)$ & $t$ & $(s)(tst)(t)(tst)$ \\ 
& & $(tstst)(t)(tst)(s)$ & $t$ & $(t)(tst)(t)(tst)$ \\ 
& & $(t)(s)(tst)(s)$ & $1$ & $(s)(t)(s)(tst)$ \\ 
& & $(tstst)(tst)(tstst)(s)$ & $ts$ & $(tst)(tstst)(tst)(tstst)$ \\ 
\hline
5 & 6 & $(s)(t)(s)(t)(s)$ & $1$ & $(tst)(s)(t)(s)(t)$ \\ 
& & $(t)(tst)(t)(tst)(s)$ & $t$ & $(tst)(t)(tst)(t)(tst)$ \\ 
\hline
\end{tabular}
\end{center}
\label{table:dihedral}
\end{table}%

\begin{proof}  By Lemma~\ref{lem:gallery dim}, if $\gamma:\fa \rightsquigarrow y\fa$ and $\gamma':\fa \rightsquigarrow y\fa$ are both positively folded with respect to the same orientation $\phipartial$, and $F_\phi(\gamma) = F_\phi(\gamma')$, then $\dim_\phi(\gamma) = \dim_\phi(\gamma')$.  

Now if $\gamma_x$ and $\gamma_x'$ are both minimal galleries $\fa \rightsquigarrow x\fa$, the gallery $\gamma'_x$ can be obtained from $\gamma_x$ by applying a (finite) sequence of braid moves.  Hence it suffices to assume that $\gamma_x'$ can be obtained from $\gamma_x$ by applying a single braid move.  

We suppose that this braid move replaces a subgallery $\sigma$ of $\gamma_x$ of type $s_i s_j s_i \cdots $ ($m_{ij}$ letters) by a subgallery $\sigma'$ of type $s_j s_i s_j \cdots$ ($m_{ij}$ letters).  So $\gamma'_x$ is the same as $\gamma_x$ except for the subgallery $\sigma$ being replaced by $\sigma'$.  Equivalently, $\type(\gamma'_x)$ is the reduced word for $x$ obtained from $\type(\gamma_x)$ by applying this braid move.

The galleries $\gamma_x$ and $\gamma'_x$ are identical before their respective subgalleries $\sigma$ and $\sigma'$, so if the PRS folding sequence applied to $\gamma_x$ in order to obtain $\gamma$ involves any crossings occurring before $\sigma$, we begin the construction of $\gamma'$ by applying exactly the same folds at the crossings of $\gamma'_x$ occurring before $\sigma'$.  Let $\tilde\gamma$ be the result of folding $\gamma_x$ at all crossings before $\sigma$ and let $\tilde\gamma'$ be the result of folding $\gamma_x'$ at all crossings before $\sigma'$.  Then $\tilde\gamma$ and $\tilde\gamma'$ differ only by a single braid move on their subgalleries, say $\tilde\sigma$ and $\tilde\sigma'$, which are the images of $\sigma$ and $\sigma'$ respectively.  Note that $\tilde\sigma$ and $\tilde\sigma'$ do not contain any folds and have the same start and end alcoves. 

After relabeling, we may assume that $\tilde\sigma$ and $\tilde\sigma'$ both start in the base alcove $\fa$. Notice that the galleries $\tilde\sigma$ and $\tilde\sigma'$ now have types the two different reduced expressions for the longest word in the standard parabolic subgroup $W_{\{s_i, s_j\} }$ of $\sW$, which is a dihedral group of order $2m_{ij}$ generated by $s_i$ and $s_j$.  By abuse of notation we will continue to write $\phipartial$ for the orientation with respect to which $\gamma$, $\tilde\gamma$, and $\tilde\gamma'$ are positively folded. 

Let $\tau$ be the gallery obtained from $\tilde\sigma \subset \tilde\gamma$ after carrying out all folds in the PRS folding sequence for $\gamma$ that occur in crossings of $\tilde\sigma$.  Then $\tau$ is a positively folded gallery from $\fa$ to $w\fa$ for some $w \in W_{\{s_i, s_j \}}$.  

\begin{lemma}\label{lem:braids}  We can carry out a PRS folding sequence on $\tilde\sigma'$ which results in a gallery $\tau'$ from $\fa$ to $w\fa$ which is positively folded with respect to $\phipartial$ and has the same number of folds as $\tau$.  
\end{lemma}

Assuming the result of Lemma~\ref{lem:braids}, since $\tau$ and $\tau'$ have the same end alcoves, and the original minimal galleries $\gamma_x$ and $\gamma'_x$ are identical after their subgalleries $\sigma$ and $\sigma'$ which map onto $\tau$ and $\tau'$ respectively, we may complete the PRS folding sequence and obtain both $\gamma$ and $\gamma'$ by applying exactly the same folds to the crossings of $\gamma_x$ and $\gamma'_x$ which occur after $\sigma$ and $\sigma'$ respectively.  Then by construction, the same orientation $\phipartial$ has been used for both $\gamma$ and $\gamma'$, and $F_\phi(\gamma) = F_\phi(\gamma')$, so $\dim_\phi(\gamma) = \dim_\phi(\gamma')$.  In order to complete the proof of Proposition~\ref{prop:braids}, we now prove Lemma~\ref{lem:braids}.

\begin{proof}[Proof of Lemma~\ref{lem:braids}]
In this proof we will notate PRS folding sequences using products of reflections $r_k \cdots r_1$ in $W_{\{ s_i, s_j \}}$, where for $l = 1,\ldots, k$, the $l$th fold occurs in the hyperplane $H_l$ for the reflection $r_l$.  We observe that the galleries $\tilde\sigma$ and $\tilde\sigma'$ cross the same set of hyperplanes in the same direction for each hyperplane, but with the crossings of hyperplanes in the opposite order.

We consider several cases.  If $\tilde\sigma = \tau$, that is, no folds occur in $\tilde\sigma$, then the result is immediate.  

Next suppose that $\tau$ has just one fold, at the hyperplane $H_1$ with corresponding reflection $r_1$.  Since $\tilde\sigma'$  crosses $H_1$ in the same direction as $\tilde\sigma$, we may fold $\tilde\sigma'$ in $H_1$ as well.  Let $\tau'$ be the positively folded gallery so obtained from $\tilde\sigma'$.  Then the end alcove of both $\tau$ and $\tau'$ is $w\fa$ with $w = r_1(w_0)$, and $\tau$ and $\tau'$ have the same number of folds, as required.

Assume next that the folding sequence by which $\tau$ is obtained from $\tilde\sigma$ corresponds to a product of two reflections $r_2 r_1$, with $r_1$ commuting with $r_2$.  Then the gallery $\tilde\sigma$ crosses $H_1$ before it crosses $r_1 H_2 = H_2$, so the gallery $\tilde\sigma'$ crosses $H_2$ before it crosses $H_1$.  Therefore using the same orientation $\phipartial$ we may first fold $\tilde\sigma'$ in $H_2$.  The tail of the image of $\tilde\sigma'$ after this fold now crosses $r_2 H_1 = H_1$ in the same direction as $\tilde\sigma$ does, so we may do a second fold in $H_1$ to obtain $\tau'$ corresponding to the product of reflections $r_1 r_2$.  Both $\tau$ and $\tau'$ have end alcove $w\fa$ with $w = r_2r_1(w_0) = r_1r_2(w_0)$, so this completes the proof when $\tau$ has two folds corresponding to commuting reflections.

We next consider the case that the folding sequence by which $\tau$ is obtained corresponds to the product of $m_{ij}$ simple reflections $\cdots s_is_js_i$.  Since every crossing of $\tilde\sigma$ is folded, we have $\dim_\phi(\tau) = F_\phi(\tau) = m_{ij}$, and the  end alcove of $\tau$ is $\fa$.  There is then also a positively folded gallery $\tau'$ obtained using the same orientation $\phipartial$ by folding every crossing of $\tilde\sigma'$, and corresponding to the product of $m_{ij}$ simple reflections $\cdots s_j s_i s_j$.  The folded gallery $\tau'$ also has end alcove $\fa$ and $m_{ij}$ folds, so the result holds in this case.  

For the remaining cases, we use the fact that since $\aW$ is irreducible affine, $m_{ij} \in \{ 2,3,4,6\}$.   In order to simplify notation, put $m = m_{ij}$ and $s_i = s$, $s_j = t$.  Then $\tilde\sigma$ has type $sts\dots$ ($m$ letters) and $\tilde\sigma'$ has type $tst\dots$ ($m$ letters).   If $m = 2$ then the simple reflections $s$ and $t$ commute, and so this case has already been considered.  The rest of the proof now follows from the information in Table~\ref{table:dihedral}.  The column with heading ``Common orientation" records the $u \in \sW$ so that the folding sequences for both $\tau$ and $\tau'$ are positively folded with respect to $\phipartialu$; this is to enable easier checking of the calculations.  We omit products of two commuting reflections and products of $m$ reflections, and for $m = 6$ we also omit products of reflections which lie in the subgroup of the dihedral group of order $12$ which is isomorphic to the dihedral group of order $6$, since these are identical to the case $m = 3$.  All possible remaining folding sequences which result in $\tau$ and $\tau'$ are either shown in this table, or can be obtained by, in any given row, swapping the letters $s$ and $t$ and swapping columns $3$ and~$5$.
\end{proof}

This completes the proof of Proposition~\ref{prop:braids}.   
\end{proof}


\section{Affine Deligne--Lusztig varieties and folded galleries}\label{sec:ADLVGalleries}

This section establishes a precise connection between nonemptiness and dimensions of affine Deligne--Lusztig varieties and existence and dimension of certain folded galleries.

In Section~\ref{sec:DimADLVs}, we review a result from \cite{GHKR} which expresses the dimension of affine Deligne--Lusztig varieties associated to pure translations in terms of the dimensions of intersections of Iwahori and unipotent orbits in the affine flag variety. These double coset intersections have an interpretation as intersections of preimages of two types of retractions in the affine building and can hence be encoded using folded galleries or alcove walks; see  Section~\ref{sec:ConnectionFolded}.  
This reformulation depends only on the underlying type of the affine building and not on the field itself, which suggests the definition of the dimension of a $p$-adic affine Deligne--Lusztig set that we propose in Section~\ref{sec:padic}. 
Finally  we isolate a family of folded galleries in Section~\ref{sec:DLGalleries} whose dimension coincides with the dimension of the associated affine Deligne--Lusztig variety.


\subsection{Dimensions of affine Deligne--Lusztig varieties}\label{sec:DimADLVs}

Recall that given an $x \in \aW$ and an element $b \in G(F)$, the associated affine Deligne--Lusztig variety is defined as \begin{equation} X_x(b) = \{ g \in G(F)/I \mid g^{-1}b\sigma(g) \in IxI \}. \end{equation}
Here we review the aforementioned result from \cite{GHKR}.

For $w \in \sW$, recall that $\phipartialw = \phipartialvow$ denotes the orientation at infinity induced by the origin $v_0$ and the alcove $w\fa$. Theorem 6.3.1 in \cite{GHKR}, which we state below, depends on a choice of Borel, and so we fix $B^-$ to be the opposite Borel subgroup.  We write $^{w}U^- := w(U^-)w^{-1}$ for  the unipotent radical corresponding to the Borel subgroup $^{w}B^- := w(B^-)w^{-1}$. Recall that the choice of Borel and split maximal torus correspond to a choice of positive and negative roots.  Having fixed an apartment, and thus a split maximal torus, we can thus speak of roots which are positive or negative with respect to a specific choice of Borel.  For example, a root is positive with respect to $B^-$ if it is a negative root with respect to the standard Borel.  Similarly, a coroot $\lambda$ is dominant for $B^-$ if $\langle \alpha, \lambda \rangle \geq 0$ for all roots $\alpha$ which are positive with respect to $B^-$, meaning that $\lambda$ is antidominant with respect to the standard Borel.

\begin{thm}[Theorem 6.3.1 in \cite{GHKR}]\label{T:GHKRDim} Let $x \in \aW$ and $\lambda \in R^\vee$.  If $X_x(t^{\lambda})$ is nonempty, then
\begin{equation} \label{E:GHKRDim}
 \dim X_x(t^\lambda) =  \max_{w \in \sW} \left\{ \dim( ^{w}U^- t^{w\lambda}\fa \cap Ix\fa) \right\} - \langle \rho_{B^-}, \lambda + \lambda_{B^-} \rangle,
\end{equation}
where $\rho_{B^-}$ is the half-sum of the roots which are positive for $B^-$ and $\lambda_{B^-}$ is  the unique element in the $\sW$-orbit of $\lambda$ which is dominant for $B^-$.  We set $\dim (^{w}U^- y\fa \cap Ix\fa) = -\infty$ in the case where  $^{w}U^-y\fa \cap Ix\fa  = \emptyset$.    In particular, $X_x(t^\lambda) \neq \emptyset$ if and only if there exists a $w \in \sW$ such that $\displaystyle ^{w}U^-t^{w\lambda}\fa \cap Ix \fa \neq \emptyset$.  
 \end{thm}
 
 It will be useful to separately state two immediate corollaries corresponding to two different special cases of the above theorem.  First of all, the formula simplifies considerably in the special case in which the pure translation $t^{\lambda}$ is in fact just the identity element.
 
 \begin{corollary}[Theorem 6.3.1 in \cite{GHKR}]\label{GHKRid}
In the case in which $\lambda = 0$, we have that $X_x(1) \neq \emptyset$ if and only if there exists $w \in \sW$ such that $\displaystyle ^{w}U^-\fa \cap Ix \fa \neq \emptyset$, in which case 
\begin{equation}\label{E:GHKRid}
\dim X_x(1) = \max_{w \in \sW} \left\{ \dim ( ^{w}U^-\fa \cap Ix\fa) \right\} .
\end{equation}
 \end{corollary}

In addition, in the case in which $\lambda$ is in the fundamental Weyl chamber, then there is also a simplification to Theorem~\ref{T:GHKRDim}.

\begin{corollary}\label{GHKRNoCorrection}
If  $\lambda \in R^{\vee}$ is dominant and $X_x(t^{\lambda}) \neq \emptyset$, then 
\begin{equation} \label{E:DimNoCorrectionGHKR}
 \dim X_x(t^\lambda) =   \max_{w \in \sW}\left\{ \dim ( ^{w}U^-t^{w\lambda}\fa \cap Ix\fa)\right\} .
\end{equation}
\end{corollary}

\begin{proof}
If  $\lambda$ is dominant, then $\lambda_{B^-} = - \lambda$.  Therefore $\lambda + \lambda_{B^-} = 0$, and there is no correction term in Equation \eqref{E:GHKRDim}.
\end{proof}


\subsection{Connection to folded galleries}\label{sec:ConnectionFolded}

We now provide a connection between the folded galleries discussed in Section~\ref{sec:Galleries} and the problem of determining nonemptiness and dimensions of affine Deligne--Lusztig varieties. The crucial insight here is that, using a result of Parkinson, Ram, and C. Schwer \cite{PRS}, we can interpret the elements of the intersection $^{w}U^- y\fa \cap Ix\fa$ as certain positively folded alcove-to-alcove combinatorial galleries, which provides a way to explicitly compute the dimensions of affine Deligne--Lusztig varieties using the results of \cite{GHKR} above.  

Given $x,y \in \aW$, recall that we define $\Gamma_w^+(\gamma_x, y)$ to be the set of all galleries from the fundamental alcove $\fa$ to the alcove  $y\fa$ which are of type $\gamma_x$ and are positively folded with respect to the orientation $\phipartialw$ at infinity. The sets of labeled folded alcove walks $\operatorname{L}\Gamma^+(\gamma_x,y)$ and $\operatorname{L}\Gamma_w^+(\gamma_x, y)$ are similarly defined in Section~\ref{sec:LFAW}; see in particular Notation~\ref{not:LFAW}.

\begin{thm}[Theorem 7.1 of \cite{PRS}]\label{T:PRS} 
Let $x, y\in \aW$ and fix a minimal gallery $\gamma_x:\fa \rightsquigarrow x\fa$.  Then there is a natural bijection between the set of labeled folded alcove walks $\operatorname{L}\Gamma_{w_0}^+(\gamma_x,y)$ and points in the intersection $U^-yI \cap IxI$.  Moreover, each choice of a labeling corresponds to a distinct element $u \in U^-$ such that $uyI \cap IxI \neq \emptyset$.  
\end{thm}

As we see in Theorem~\ref{T:GHKRDim}, however, we will need to compute the dimensions of the intersections of unipotent and Iwahori orbits in the affine flag variety where we range over all possible choices of orientation at infinity.  We remark that, in order to match the conventions in \cite{GHKR}, the conventions used in \cite{PRS} are opposite of our chosen conventions.  In particular, the standard orientation at infinity $\phi_0$ coincides with information about the intersection of Iwahori and $U^+$-orbits, rather than $U^-$-orbits as in \cite{PRS}.  Therefore, information about intersections of Iwahori and $^{w}U^-$-orbits comes from using the orientation $\phi_{w_0w}$.  For example, the opposite standard orientation $\phi_{w_0} = -\phi_0$, which is the orientation we will use for the majority of our constructions in later sections, corresponds to intersections of the form $U^-yI \cap IxI$, which is the situation in Theorem~\ref{T:PRS}.

\begin{corollary}\label{T:PRSCor}
Let $x, y\in \aW$.  Fix a minimal gallery $\gamma_x:\fa \rightsquigarrow x\fa$ and an orientation $\phi_{w_0w}$.  Then there is a natural bijection between the set of labeled folded alcove walks $\operatorname{L}\Gamma_{w_0w}^+(\gamma_x, y)$ and points in the intersection $^{w}U^-yI \cap IxI$. Moreover, each choice of a labeling corresponds to a distinct element $u_w \in\ ^{w}U^-$ such that $u_wyI \cap IxI \neq \emptyset$.   
\end{corollary}

\begin{proof}
Changing the orientation at infinity corresponds to changing the choice of Borel subgroup.  The new Borel subgroup will be a conjugate of the $B^-$, namely $^{w}B^-$, having unipotent radical $^{w}U^-$.  Following the proof in \cite{PRS}, the set of elements in $^{w}U^-$ for which the intersection $^{w}U^-yI \cap IxI$ is nonempty can then be read off of the labelings of the alcove walks which are positively folded with respect to the orientation $\phi_w$.
\end{proof}

We now use Theorem~\ref{T:PRS} and Corollary \ref{T:PRSCor} to connect the dimensions of intersections $^{w}U^-yI \cap IxI$ to the dimensions of the corresponding labeled folded alcove walks. This same connection was already made implicitly in \cite{GaussentLittelmann}.

\begin{prop}\label{OrbitDims}
Let $w \in \sW$ and $x,y  \in \aW$.  Fix a minimal gallery $\gamma_x: \fa \rightsquigarrow x\fa$.  Then
\begin{equation}
\dim (^{w}U^-y \fa \cap Ix\fa) = \max \{ F(\gamma') + P_{\phi_{w_0w}}(\gamma') \mid \gamma' \in \Gamma^+_{w_0w}(\gamma_x,y) \}. 
\end{equation}
\end{prop}

\begin{proof}

Corollary \ref{T:PRSCor} indicates that the intersection $^{w}U^-yI \cap IxI$ has a stratification in which the set of points of $^{w}U^-yI \cap IxI$ lie in the same stratum if they have the same underlying unlabeled folded alcove walk; \textit{i.e.} the elements in $\Gamma_{w_0w}^+(\gamma_x, y)$ enumerate this natural stratification of $^{w}U^-yI \cap IxI$.  Following the proof of Theorem~\ref{T:PRS} in \cite{PRS}, we may thus compute the dimension of  $^{w}U^-yI \cap IxI$ as follows.  

Suppose that $\gamma_x$ is a minimal gallery from $\mathbf{c_f}$ to $x\fa$ and that $\gamma'$ is a positively folded gallery which ends at $y\fa$ and is obtained by folding $\gamma_x$ positively with respect to the periodic orientation induced by $\phi_{w_0w}$; that is, $\gamma' \in \Gamma_{w_0w}^+(\gamma_x,y)$.  Fixing a label for a positive crossing corresponds to making a choice of an element in the field $k$ as the coefficient of a unique power $t^m$ in one entry of the matrix $u_w \in$ $^{w}U^-$ such that $u_wyI\cap IxI \neq \emptyset$.  Similarly, fixing a label for a positive fold corresponds to making such a choice of coefficient among the elements in $k^{\times}$.  The negative crossings each indicate a coefficient of 0 for a certain power of the uniformizer $t$ in a particular entry of $u_w$.  Recall that $F(\gamma')$ equals the number of folds in $\gamma'$, and $P_{\phi_{w_0w}}(\gamma')$ is the number of crossings in $\gamma'$ which are positive with respect to the orientation $\phi_{w_0w}$. Then $\gamma'$ corresponds to a stratum of $^{w}U^-yI \cap IxI$ which is isomorphic to $\A^{P_{\phi_{w_0w}}(\gamma')} \times (\A\backslash \{0\})^{F(\gamma')}$ and thus has dimension equal to $F(\gamma') + P_{\phi_{w_0w}}(\gamma')$.  Ranging over all possible $\gamma' \in \Gamma_{w_0w}^+(\gamma_x,y)$ runs over all of the strata, each of which has the algebraic structure of the product of some number of copies of affine space and the affine line minus a point.  The maximum dimension over all of these strata therefore coincides with the dimension of the intersection $^{w}U^-y \fa \cap Ix\fa$.
\end{proof}

We are now able to express the problem of determining the nonemptiness of an affine Deligne--Lusztig variety of the form $X_x(t^{\lambda})$, in addition to computing its dimension, in terms of positively folded labeled alcove walks.  Recall that $\Gamma^+_w(x,1)$ is the set of all galleries of type $\vec x$ from $\fa$ to $\fa$ which are positively folded with respect to the orientation $\phipartial_w$.

\begin{prop}\label{ADLVvsGalleries}
Let $x \in \aW$ be an element of the affine Weyl group.  The variety $X_x(1)$ is nonempty if and only if there exists an orientation at infinity $\phipartial = \phipartialw$ and an alcove-to-alcove gallery $\gamma:\fa \rightsquigarrow \fa$ of type $\vec{x}$ which is positively folded with respect to $\phipartial$.  

Moreover, if $X_x(1)\neq \emptyset$ then $\dim X_x(1)$ equals the maximum value of $\dim_\phi(\gamma)$, ranging over all such galleries $\gamma: \fa \rightsquigarrow \fa$ and all orientations at infinity $\phipartialw$.  That is, 
\begin{equation}\dim X_x(1) = \max_{w\in \sW} \{ F(\gamma') + P_{\phi_w}(\gamma') \mid \gamma' \in \Gamma^+_w(x,1) \}. \end{equation}
\end{prop}

\begin{proof}  
The first observation follows directly from Corollaries \ref{GHKRid} and \ref{T:PRSCor}.  If we fix a minimal gallery $\gamma_x:\fa \rightsquigarrow x\fa$ then combining Equation \eqref{E:GHKRid} and Proposition~\ref{OrbitDims} yields
\begin{align*} \dim X_x(1) & =   \max_{w \in \sW} \dim (^{w}U^- \fa \cap Ix\fa) \\
& = \max_{w \in \sW} \{ F(\gamma') + P_{\phi_{w_0w}}(\gamma') \mid \gamma' \in \Gamma^+_{w_0w}(\gamma_x,1) \} \\
& = \max_{w \in \sW} \{ F(\gamma') + P_{\phi_w}(\gamma') \mid \gamma' \in \Gamma^+_w(\gamma_x,1) \}. 
\end{align*}
Now by Proposition~\ref{prop:braids} this maximum does not depend on the choice of minimal gallery $\gamma_x$, and so we may replace the set $\Gamma^+_w(\gamma_x,1)$ here by $\Gamma^+_w(x,1)$.
\end{proof}

We can also generalize Proposition~\ref{ADLVvsGalleries} to the case in which the gallery ends in an alcove corresponding to a pure translation, as opposed to just the fundamental alcove.  It will be useful for us later to separate out the special case in which the gallery both begins and ends at $\fa$, which is why we state and prove that case first.

\begin{thm}\label{ADLVvsGal2}
Suppose $b=t^\lambda$ for some $\lambda\in R^\vee$, and define $b^w=t^{w\lambda}$ for $w \in \sW$. The variety $X_x(b)$ is nonempty if and only if there exists an orientation at infinity $\phipartialwow$ for some $w \in \sW$ and a gallery $\gamma:\fa \rightsquigarrow {b^{w}\fa}$ of type $\vec{x}$ that is positively folded with respect to $\phipartialwow$.  Moreover, if $X_x(b)\neq \emptyset$, then
\begin{equation}\label{E:FoldsInnerProd}
\dim X_x(b) = \max_{w\in \sW} \{ F(\gamma') + P_{\phi_{w_0w}}(\gamma') \mid \gamma' \in \Gamma^+_{w_0w}(x, b^{w}) \} - \langle \rho_{B^-},\lambda+\lambda_{B^-} \rangle. \end{equation}
 \end{thm}

\begin{proof}
 Follows directly from Theorem~\ref{T:GHKRDim} combined with Proposition~\ref{OrbitDims}.
\end{proof}

\begin{remark}
In Section 6 of \cite{GHKR}, G\"{o}rtz, Haines, Kottwitz, and Reuman describe a version of this exhaustive search algorithm in terms of foldings of the entire building, which they used to verify Conjecture~\ref{Conj_GHKR} for several pairs $x$ and $b$ with $\ell(x)$ small.  Their algorithm recasts one of Dabrowski  \cite{Dabrowski}, which in turn is analogous to one of Deodhar in his study of Bruhat cells of $p$-adic Chevalley groups \cite{Deodhar}, in terms of retractions in the affine Bruhat--Tits building.
\end{remark}


\subsection{Dimension of a $p$-adic Deligne--Lusztig set}\label{sec:padic}

The proof of Theorem~\ref{ADLVvsGal2} is independent of the base field, suggesting that the notion of dimension for affine Deligne--Lusztig varieties in the function field and $p$-adic contexts should coincide.  We therefore propose the following definition for the dimension of the $p$-adic Deligne--Lusztig set associated to an affine Weyl group element and pure translation.  Denote by $\widehat{\Q}_p^{\mbox{\tiny{ur}}}$ the maximal unramified extension of the field of $p$-adic numbers. Let $X_x(b)_{\Q_p}$ denote the affine Deligne--Lusztig set inside the affine flag variety $G(\widehat{\Q}_p^{\mbox{\tiny{ur}}})/I$ associated to an element $b \in \widehat{\Q}_p^{\mbox{\tiny{ur}}}$ and an element $x$ of the corresponding affine Weyl group.

\begin{definition}\label{def:padic}
 Let $b = t^{\lambda}$ be a pure translation and $x$ any element in the affine Weyl group of $G(\widehat{\Q}_p^{\mbox{\tiny{ur}}})$. Define
\begin{equation}\dim X_x(b)_{\Q_p} = \max_{w\in \sW} \{ F(\gamma') + P_{\phi_{w_0w}}(\gamma') \mid \gamma' \in \Gamma^+_{w_0w}(x, b^{w}) \} - \langle \rho_{B^-},\lambda+\lambda_{B^-} \rangle. \end{equation}
\end{definition}

Corollary 11.3.5 in \cite{GHKRadlvs} states that $X_x(b) \neq \emptyset$ if and only if $X_x(b)_{\Q_p} \neq \emptyset$.  By combining this with our Theorem~\ref{ADLVvsGal2}, we obtain the following.

\begin{thm}\label{T:padic}
Let $x$ be any element of the affine Weyl group of $G(\widehat{\Q}_p^{\mbox{\tiny{ur}}})$ and $b$ any pure translation.  Then
\begin{equation}
X_x(b) \neq \emptyset \iff X_x(b)_{\Q_p} \neq \emptyset.
\end{equation} Moreover, with $\dim X_x(b)_{\Q_p}$ as in Definition \ref{def:padic}, if both Deligne--Lusztig sets are nonempty, then
\begin{equation}
\dim X_x(b) = \dim X_x(b)_{\Q_p}. 
\end{equation}
\end{thm}

Theorem~\ref{ADLVvsGal2} is the primary springboard from which we prove our main results about dimensions of affine Deligne-Lusztig varieties.  This result permits a translation of the question of dimension into the language of labeled folded alcove walks in a standard apartment of the affine building, and this will be our primary framework for the remainder of the paper.  Note that none of the results which follow in subsequent sections rely upon the base field, and hence hold equally for the $p$-adic context.


\subsection{Deligne--Lusztig galleries}\label{sec:DLGalleries}

We now define the set of Deligne--Lusztig galleries, which are galleries whose dimension is equal to the dimension of an associated affine Deligne--Lusztig variety.  

Recall from Definition~\ref{def:LS} that an LS-gallery maximizes dimension among certain vertex-to-vertex combinatorial galleries which are positively folded with respect to the standard orientation.  As we see from Theorem~\ref{ADLVvsGal2}, we will be maximizing dimension of alcove-to-alcove galleries over all possible orientations, so we first introduce the notion of an optimal LS-gallery.

\begin{definition}\label{def:optimalLS}  Let $x,y\in \aW$ and let $\phipartial$ be an orientation at infinity.  An alcove-to-alcove combinatorial gallery $\gamma:\fa \rightsquigarrow y\fa$ of type $\vec x$ which is positively folded with respect to $\phipartial$ is called an \emph{optimal LS-gallery} if it maximizes dimension among all positively folded galleries from $\fa$ to $y\fa$ of type $\vec x$ over all possible orientations, that is, if
$$ \dim_\phi(\gamma) = \max_{w \in \sW} \{ \dim_{\phi_w}(\gamma') \mid \gamma' \in \Gamma^+_w(x,y)\}. $$
\end{definition}

We now define Deligne--Lusztig galleries.

\begin{definition}\label{def:DLGallery}
Let $x \in \aW$ and let $b = t^\lambda$ be a pure translation.  An alcove-to-alcove combinatorial gallery $\gamma$ of type $\vec{x}$ is called a \emph{Deligne--Lusztig gallery for $X_x(b)$}, abbreviated \emph{DL-gallery}, if there exists an orientation $\phipartial = \phipartialwow$ so that $\gamma$ runs from $\fa$ to $b^w\fa$, is positively folded with respect to $\phipartialwow$, and
$$\dim X_x(b)  =  \dim_{\phi_{w_0w}}(\gamma),$$ or equivalently, $$\dim X_x(b) = F(\gamma) + P_{\phi_{w_0w}}(\gamma).$$  
\end{definition}

\noindent We will discuss the relationship between LS-galleries and DL-galleries in Section~\ref{sec:MV-GL-LS}.

Due to the inner product correction factor appearing in Theorem~\ref{ADLVvsGal2}, an optimal LS-gallery may not have dimension equal to the dimension of the affine Deligne--Lusztig variety for which it proves nonemptiness, and thus may not be a DL-gallery.  This is illustrated by the following example.

\begin{example}\label{Ex:ConvexityHypNotSharp}
Consider $x= t^{2\rho}w_0$ in type $\tilde A_2$.  The variety $X_x(b)$ for $b = t^{\rho}$ is nonempty, since there exists a gallery $\sigma: \fa \rightsquigarrow \bb = b\fa$ which is positively folded with respect to $-\phipartialo$, illustrated in Figure~\ref{fig:ConvexityHypNotSharp}.  In addition, the gallery $\sigma$ is both an optimal LS-gallery and a DL-gallery, since $\dim X_x(b) = \dim_{-{\phi_0}}(\sigma) = 1$. 

On the other hand, if we instead consider the translation $c = t^{-\rho}$, then we have a gallery $\gamma: \fa \rightsquigarrow \mathbf{c} = c\fa$ which is positively folded with respect to $-\phipartialo$, proving that $X_x(c)\neq \emptyset$.  The gallery $\gamma$ is an optimal LS-gallery, but not a DL-gallery, since $\dim_{-{\phi_0}}(\gamma) = 5$, but $\dim X_x(c) = \dim_{-{\phi_0}}(\gamma) - \langle \rho_{B^-}, -\rho-\rho \rangle = 5-4=1$ by Theorem~\ref{ADLVvsGal2}.
\end{example}

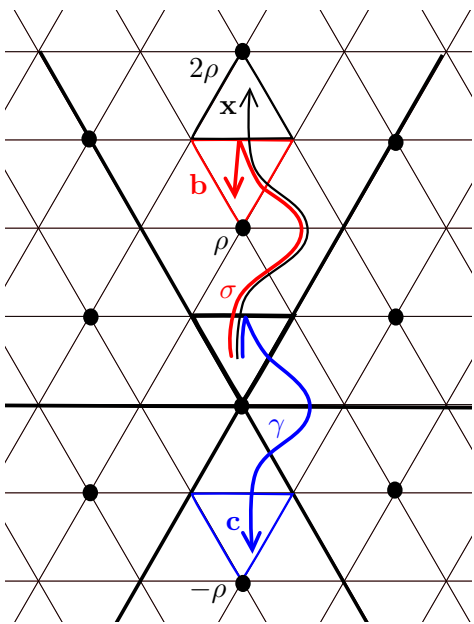
\begin{figure}[ht]
\begin{center}
\begin{overpic}[width=0.4\textwidth]{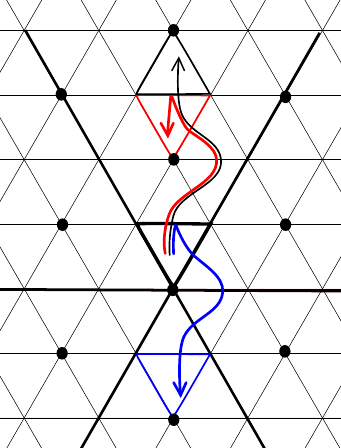}
\put(30,89){$2\rho$}
\put(35,83){$\x$}
\put(30,70){\color{red}{$\bb$}}
\put(34,61){$\rho$}
\put(35,53){\color{red}{$\sigma$}}
\put(43,31){\color{blue}{$\gamma$}}
\put(36,15){\color{blue}{$\mathbf{c}$}}
\put(30,4){$-\rho$}
\end{overpic}
\caption[shrunken]{The gallery $\sigma$ is both an optimal LS-gallery and a DL-gallery.  The gallery $\gamma$ is an optimal LS-gallery but not a DL-gallery.}
\label{fig:ConvexityHypNotSharp}
\end{center}
\end{figure}

However, there are situations in which the notions of optimal LS-galleries and DL-galleries do coincide.  We will be in these situations in our explicit constructions in Section~\ref{sec:Construction}.  The next result shows that if $b$ is a dominant pure translation and the maximum in Equation \eqref{E:DimNoCorrection}  occurs with respect to the opposite standard orientation, then any gallery of maximal dimension over all possible orientations is automatically a DL-gallery, and vice versa.  

\begin{lemma}\label{DomADLVDim}
Suppose $b=t^\lambda$ for some $\lambda\in\Cf\cap R^\vee$, and $x \in \aW$.  If $X_x(b) \neq \emptyset$, then 
\begin{equation} \label{E:DimNoCorrection}
 \dim X_x(b) =   \max_{w \in \sW}\left\{ \dim_{\phi_{w_0w}}(\gamma) \mid \gamma \in \Gamma^+_{w_0w}(x,b^w) \right\}. 
\end{equation}
Moreover, if the maximum in Equation \eqref{E:DimNoCorrection} occurs with respect to the opposite standard orientation, then any optimal LS-gallery of type $\vec x$ from $\fa$ to $b\fa$ is a DL-gallery for $X_x(b)$.

In particular, if $X_x(1) \neq \emptyset$, then 
\begin{equation} \label{E:DimNoCorrection1}
 \dim X_x(1) =   \max_{w \in \sW}\left\{ \dim_{\phi_{w_0w}}(\gamma) \mid \gamma \in \Gamma^+_{w_0w}(x,1) \right\} 
\end{equation}
and  if the maximum in Equation \eqref{E:DimNoCorrection1} occurs with respect to the opposite standard orientation, then any optimal LS-gallery of type $\vec x$ from $\fa$ to $\fa$ is a DL-gallery for $X_x(1)$.
\end{lemma}

\begin{proof}
Since $\lambda$ is dominant, $\lambda_{B^-} = -\lambda$ thus $\lambda + \lambda_{B^-} = 0$ and there is no correction term in Equation~\eqref{E:FoldsInnerProd}.  The first statement then follows from Theorem~\ref{ADLVvsGal2}.  The second statement follows because the opposite standard orientation coincides with the case in which the maximum occurs for $w=1$.
\end{proof}

We will use the following observation in Section~\ref{sec:Applications}.

\begin{lemma}\label{lem:ConstantFolds}  Let $x \in \aW$ and let $b = t^\mu$ be a pure translation, with $\mu$ dominant.  If $X_x(b)$ is nonempty then any two DL-galleries for $X_x(b)$ have the same number of folds.
\end{lemma}

\begin{proof}  Suppose that $\gamma$ and $\gamma'$ are DL-galleries for $X_x(b)$.  Then $\gamma$ and $\gamma'$ are of type $\vec x$ and there are $w, w' \in \sW$ such that $\gamma:\fa \rightsquigarrow b^w \fa$ is positively folded with respect to the orientation $\phi_{w_0w}$ and $\gamma':\fa \rightsquigarrow b^{w'}\fa$ is positively folded with respect to the orientation $\phi_{w_0w'}$. 

By the definition of a DL-gallery and Lemma~\ref{lem:gallery dim} we have 
\[  \dim X_x(b) = \dim_{\phi_{w_0w}}(\gamma) = \frac{1}{2}\left[\ell(x) + F_{\phi_{w_0w}}(\gamma) - \ell(b^w) \right] + \dim_{\phi_{w_0w}}(b^w) \]
and
\[  \dim X_x(b) =  \dim_{\phi_{w_0w'}}(\gamma') = \frac{1}{2}\left[\ell(x) + F_{\phi_{w_0w'}}(\gamma) - \ell(b^{w'}) \right] + \dim_{\phi_{w_0w'}}(b^{w'}). \]
Now Lemma~\ref{lem:length} implies that $\ell(b^w) = \ell(b^{w'})$ and $\dim_{\phi_{w_0w}}(b^w) = \dim_{\phi_{w_0w'}}(b^{w'}) = 0$. The result follows.
\end{proof}


\section{Explicit constructions of positively folded galleries}\label{sec:Construction}

Our goal in this section is to provide explicit constructions of positively folded galleries from which all of the nonemptiness and dimension patterns for affine Deligne--Lusztig varieties associated to pure translations can be read.   In particular, in Sections~\ref{sec:Xx1ShrunkenDominant} and~\ref{sec:RelabelingSymmetry} we will show that the galleries constructed in this section are DL-galleries.  We begin by providing some motivation for our choices in Section~\ref{sec:shrunken}.  We then give an explicit construction of a single positively folded gallery $\sigma_{a_0}$ of type $\vec a_0$, where $a_0\fa$ is a particular alcove close to the origin, in Section~\ref{sec:construction_a0}.  In Section~\ref{sec:more constructions}, we build upon this construction, applying root operators to obtain an infinite family of galleries which will be used in our proofs. 

\subsection{Motivation: the shrunken Weyl chambers}\label{sec:shrunken}

In this section we provide some additional context which sheds some light on our particular choice of the gallery to construct in Section~\ref{sec:construction_a0}.  In some later parts of this paper, we will need to restrict to working inside the shrunken Weyl chambers, as in \cite{GHKR}. 

\begin{definition}\label{def:shrunken}
The \emph{shrunken Weyl chamber corresponding to $w\in \sW$} is defined as follows:
$$ \Cws_w : = \Cw_w + \sum\limits_{i: w\alpha_i >0} w \varpi_i^{\vee}.$$ That is, the shrunken Weyl chambers exclude all of the alcoves which lie between the hyperplanes $H_{\alpha,0}$ and $H_{\alpha,1}$ for each $\alpha \in \Phi^+$. We sometimes denote the shrunken dominant Weyl chamber by $\Cfs$. Note that in general $w\Cfs\neq \Cws_w$. 
\end{definition}

As a point of reference, we remark that the complement of the shrunken Weyl chambers is a Kazhdan--Lusztig two-sided cell and that it is the lowest one with respect the left-right order.   The union of the ``strips'' which comprise the complement of the shrunken Weyl chambers is therefore also often referred to as the \emph{lowest two-sided cell} in the affine Weyl group; see \cite{Bedard} and \cite{Shi} for details on this perspective.

Recall that $\rho$ is the half-sum of the positive roots.  The quantity $\langle \rho, \lambda \rangle$ equals the height of the coroot $\lambda \in R^{\vee}$.  Therefore, if $\alpha_i$ is a simple root, we have $\langle \rho, \alpha^{\vee}_i \rangle = 1$, and so we can also characterize $\rho$ as $\rho = \cap_{i=1}^n H_{\alpha_i,1}$, where $\alpha_1,\dots,\alpha_n$ are the simple roots.  Said another way, $\rho \in \Cf$  is the special vertex which serves as the basepoint for the shrunken fundamental Weyl chamber $\Cfs$.  Similarly, note that $ 2\rho = \cap_{i=1}^n H_{\alpha_i,2} \in \Cfs,$ and so $2\rho$ is the vertex in $\Cf$ closest to the origin such that all alcoves of the form $t^{2\rho}w\fa$ for any $w \in \sW$ are contained in $\Cfs$.  That is, the local copy of $\sW$ based at $2\rho$ lies completely inside of the shrunken fundamental Weyl chamber, and this is the ``first'' (smallest with respect to dominance order on $R^{\vee}$) vertex for which this is true.
Further recall that since $\rho$ is special, its link is a spherical building of the same type as the link of the origin $v_0$. Moreover the non-type-preserving translation $t^\rho$ in the extended affine Weyl group isometrically maps the star of $v_0$, which is by definition the union of alcoves containing $v_0$, to the star of $\rho$. Given $w\in\sW$ it thus makes sense to speak of the alcove in the $w$-position at $\rho$, which is the image under the (non-type-preserving) translation $t^\rho$ of the alcove in $w$-position at the origin. 

Our main results leverage the nonemptiness pattern and dimensions of the varieties $X_a(1)$ where $a =  t^{2\rho}w$ to determine nonemptiness and dimension for all $X_x(t^{\mu})$.  We thus begin by devoting our attention to a careful construction of a positively folded gallery of type $\vec a_0$ where $a_0 = t^{2\rho}w_0$.


\subsection{Constructing one positively folded gallery}\label{sec:construction_a0}

In the following let $a_0 = t^{2\rho}w_0$ and let $\x_{\pm\rho,w} = x_{\pm\rho,w}\fa$ be the alcove in the $w$-position in the link of $\pm\rho$, for $w \in \sW$. In terms of the extended affine Weyl group one has $x_{\pm\rho,w}=t^{\pm\rho}w$.  In this section we first construct a minimal gallery $\sigma:\x_{-\rho,1} \rightsquigarrow \x_{\rho,w_0}$ by concatenating smaller galleries, then fold $\sigma$ to obtain a gallery $\tilde\sigma:\x_{-\rho,1} \rightsquigarrow \x_{-\rho,1}$, and then translate this gallery by $t^\rho$ to obtain a positively folded gallery $\sigma_{a_0}:\fa\rightsquigarrow \fa$ of type $\vec a_0$.   We conclude by computing the dimension of $\sigma_{a_0}$.

 First put $\w_0\define w_0\fa$ and choose minimal galleries 
\[
\gamma : \fa \rightsquigarrow \x_{\rho, w_0} \quad \mbox{and} \quad \tau: \w_0 \rightsquigarrow \fa. 
\]
Let $\bar\gamma$ be the minimal gallery from $\x_{-\rho, 1}$ to $\w_0$ that is obtained by translating $\gamma$ by $t^{-\rho}$.  We obtain a gallery $\sigma:\x_{-\rho,1} \rightsquigarrow \x_{\rho,w_0}$ by concatenating $\bar\gamma$, $\tau$ and $\gamma$:
 \[ 
 \sigma \define \bar\gamma  \ast \tau \ast \gamma.
 \]  
 See the top left of Figure~\ref{fig:w0_fold}  for an illustration of these galleries in type $\tilde C_2$.

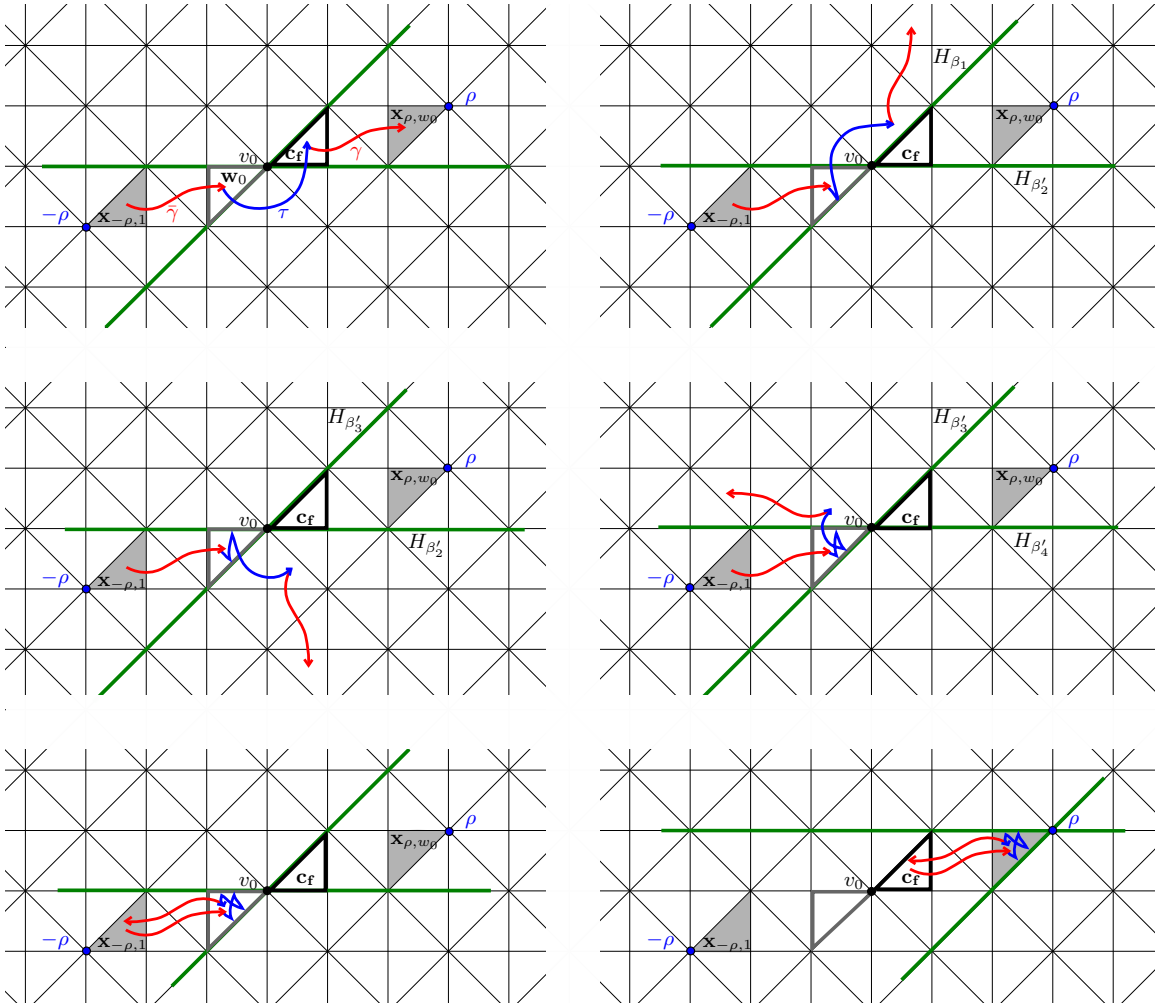
\begin{figure}[ht]
\begin{center}
\resizebox{\textwidth}{!}
{
\begin{overpic}{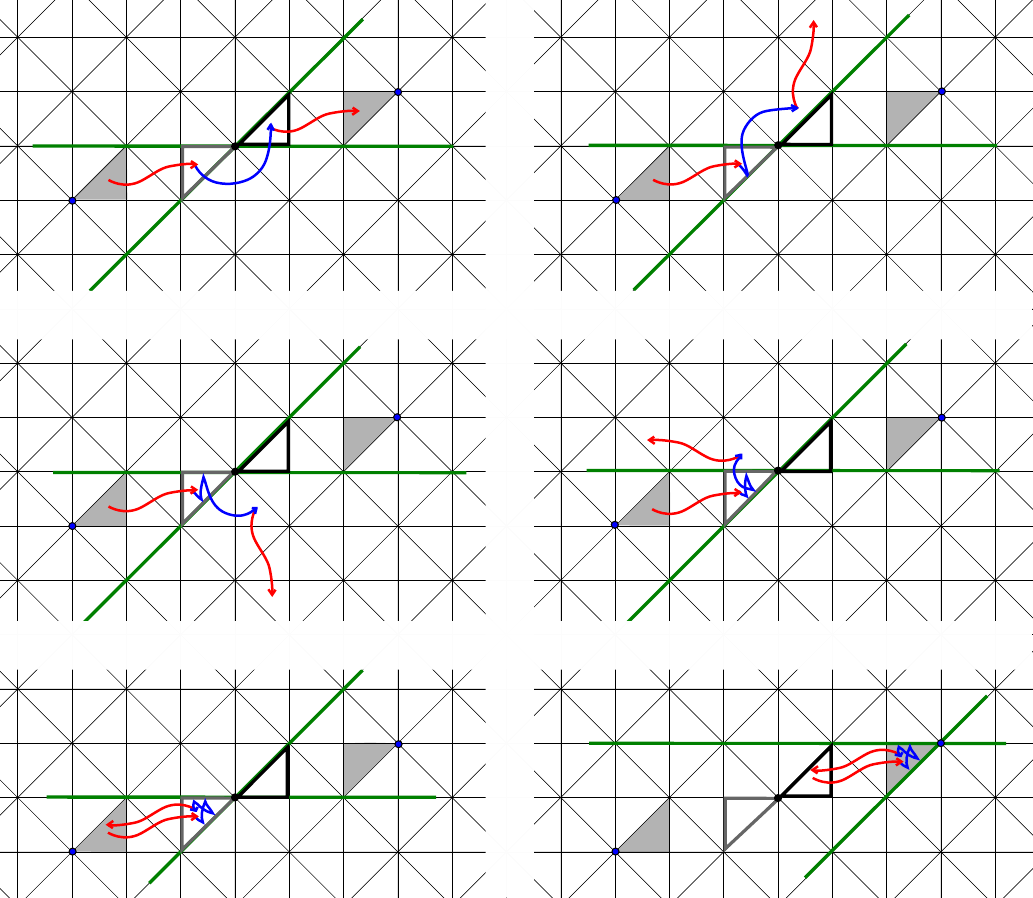}
\put(23.7,68.3){\footnotesize{\textcolor{blue}{$\tau$}}}
\put(30,73.7){\footnotesize{\textcolor{red}{$\gamma$}}}
\put(14,68.3){\footnotesize{\textcolor{red}{$\bar\gamma$}}}
\put(20.3,73.2){\footnotesize{\textcolor{black}{$v_0$}}}
\put(3.1,68.2){\footnotesize{\textcolor{blue}{$-\rho$}}}
\put(40,78.5){\footnotesize{\textcolor{blue}{$\rho$}}}
\put(73,73.2){\footnotesize{\textcolor{black}{$v_0$}}}
\put(55.8,68.2){\footnotesize{\textcolor{blue}{$-\rho$}}}
\put(92.5,78.5){\footnotesize{\textcolor{blue}{$\rho$}}}
\put(20.3,41.5){\footnotesize{\textcolor{black}{$v_0$}}}
\put(3.1,36.5){\footnotesize{\textcolor{blue}{$-\rho$}}}
\put(40,47){\footnotesize{\textcolor{blue}{$\rho$}}}
\put(73,41.6){\footnotesize{\textcolor{black}{$v_0$}}}
\put(55.8,36.5){\footnotesize{\textcolor{blue}{$-\rho$}}}
\put(92.5,47){\footnotesize{\textcolor{blue}{$\rho$}}}
\put(20.3,10.1){\footnotesize{\textcolor{black}{$v_0$}}}
\put(3.1,5){\footnotesize{\textcolor{blue}{$-\rho$}}}
\put(40,15.5){\footnotesize{\textcolor{blue}{$\rho$}}}
\put(73,10.1){\footnotesize{\textcolor{black}{$v_0$}}}
\put(55.8,5){\footnotesize{\textcolor{blue}{$-\rho$}}}
\put(92.5,15.5){\footnotesize{\textcolor{blue}{$\rho$}}}
\put(24.3,73.5){\footnotesize{\textcolor{black}{$\fa$}}}
\put(8,68){\footnotesize{\textcolor{black}{$\x_{-\rho,1}$}}}
\put(33.5,77){\footnotesize{\textcolor{black}{$\x_{\rho,w_0}$}}}
\put(78,73.5){\footnotesize{\textcolor{black}{$\fa$}}}
\put(60.7,68){\footnotesize{\textcolor{black}{$\x_{-\rho,1}$}}}
\put(86.2,77){\footnotesize{\textcolor{black}{$\x_{\rho,w_0}$}}}
\put(18.7,71.5){\footnotesize{\textcolor{black}{$\w_0$}}}
\put(25.3,41.8){\footnotesize{\textcolor{black}{$\fa$}}}
\put(8,36.3){\footnotesize{\textcolor{black}{$\x_{-\rho,1}$}}}
\put(33.5,45.5){\footnotesize{\textcolor{black}{$\x_{\rho,w_0}$}}}
\put(78,41.8){\footnotesize{\textcolor{black}{$\fa$}}}
\put(60.7,36.3){\footnotesize{\textcolor{black}{$\x_{-\rho,1}$}}}
\put(86.2,45.5){\footnotesize{\textcolor{black}{$\x_{\rho,w_0}$}}}
\put(25.3,10.4){\footnotesize{\textcolor{black}{$\fa$}}}
\put(8,4.8){\footnotesize{\textcolor{black}{$\x_{-\rho,1}$}}}
\put(33.5,13.8){\footnotesize{\textcolor{black}{$\x_{\rho,w_0}$}}}
\put(78,10.4){\footnotesize{\textcolor{black}{$\fa$}}}
\put(60.7,4.8){\footnotesize{\textcolor{black}{$\x_{-\rho,1}$}}}
\put(80.7,81.8){\footnotesize{\textcolor{black}{$H_{\beta_1}$}}}
\put(87.7,71.2){\footnotesize{\textcolor{black}{$H_{\beta_2'}$}}}
\put(28,50.5){\footnotesize{\textcolor{black}{$H_{\beta_3'}$}}}
\put(35,39.6){\footnotesize{\textcolor{black}{$H_{\beta_2'}$}}}
\put(80.7,50.5){\footnotesize{\textcolor{black}{$H_{\beta_3'}$}}}
\put(87.7,39.7){\footnotesize{\textcolor{black}{$H_{\beta_4'}$}}}
\end{overpic}
}
\vspace{0.5cm}
\caption[w0_fold]{Construction of a positively folded gallery from $\fa$ to $\fa$ of type $\vec{a}_0$ where $a_0 = t^{2\rho}w_0$ in type $\tilde C_2$. We start in the top left with a minimal gallery from $\x_{-\rho,1}$ to $\x_{\rho,w_0}$ and apply a four-step PRS folding sequence, with all folds occurring in the heavily shaded (green) hyperplanes. The final step, a translation by $t^\rho$ illustrated in the bottom row of the figure, yields the desired gallery.}
\label{fig:w0_fold}
\end{center}
\end{figure}

\begin{lemma}\label{lem_w0-folding-sequence}
The gallery $\sigma: \x_{-\rho,1} \rightsquigarrow \x_{\rho,w_0}$ is minimal.
\end{lemma}
\begin{proof}
It suffices to show that $\sigma$ crosses each hyperplane at most once.  Note first that since $w_0$ is the longest word in $W$, the set of hyperplanes crossed by $\tau$, denoted by $\Inv(\tau)$, is 
\[
\Inv(\tau) = \{ H_{\alpha,0} \mid \alpha \in R^+ \} = \{ H_\alpha \mid \alpha \in R^+ \},
\]
 where we put $H_\alpha\define H_{\alpha,0}$. Now the alcove $\fa$ is on the positive side of $H_{\alpha}$ for all $\alpha \in R^+$. Thus the minimal gallery $\gamma: \fa \rightsquigarrow x_{\rho,w_0}\fa $ does not cross any of the hyperplanes in $\Inv(\tau)$, and so
$
\Inv(\gamma) \cap \Inv(\tau) = \emptyset. 
$
Hence the concatenation $\tau\ast\gamma$ is minimal.  

To see that the concatenation of $\bar\gamma$ with $\tau\ast\gamma$ is minimal, note that the alcove $\w_0 $ is on the negative side of all hyperplanes in $\Inv(\tau)$.  Thus the minimal gallery $\bar\gamma$ that ends in $w_0\fa$ does not cross any of the hyperplanes $H_{\alpha}$, hence $\Inv(\bar\gamma)$ is disjoint from $\Inv(\tau)$ and from $\Inv(\gamma)$.  This completes the proof.
\end{proof}

We now apply a PRS folding sequence to the minimal gallery $\sigma:  \x_{-\rho,1} \rightsquigarrow \x_{\rho,w_0}$ defined above to obtain a gallery $\tilde\sigma$ which is positively folded with respect to $-\phipartialo$, such that each fold occurs in a hyperplane orthogonal to a simple root.  Recall that $\phipartialo$ denotes the standard orientation.  Our method will provide an explicit algorithm for folding $\sigma$ a total of $\ell(w_0)$ times.

To construct this PRS folding sequence, we first enumerate the hyperplanes crossed by the gallery $\tau$ as $H_{\beta_1}, \ldots, H_{\beta_l}$, in that order, and write $s_{\beta_i}$ for the reflection $s_{\beta_i, 0}$. Note that $\beta_1$ must be a simple root, since $H_{\beta_1}$ is a wall of the alcove $\w_0$.  We introduce a first fold in the hyperplane $H_{\beta_1}$, which corresponds to applying the reflection $s_{\beta_1}$ to the portion of the gallery $\sigma$ after $H_{\beta_1}$, including the alcove $\fa$.  Put $\beta_1' = \beta_1$ and define $H_{\beta_2'} = s_{\beta_1} H_{\beta_2}$.   The top right of Figure~\ref{fig:w0_fold} shows the situation after this first fold, in type $\tilde C_2$. 

Now the hyperplane $H_{\beta_2'}$ is equal to $s_{\beta_1} H_{\beta_2}$ which is a wall of the alcove $\w_0$.  Thus $\beta_2'$ is also a simple root.  The second fold is in the hyperplane $H_{\beta_2'}$, and so corresponds to applying the reflection $s_{\beta'_2}$.  Inductively, the $j^{\text{th}}$ fold will be in the hyperplane $H_{\beta_j'}$ which is the image of $H_{\beta_j}$ under the product of reflections $s_{\beta'_{j-1}} \cdots s_{\beta'_2}s_{\beta_1'}$, and the root $\beta'_j$ is simple.   See the middle row and bottom left of Figure~\ref{fig:w0_fold} for the completion of this folding sequence in type $\tilde C_2$, where there are $l = \ell(w_0) = 4$ folds.

\begin{lemma}\label{lem:lem_w0-folding-endpoint}  Applying the product of reflections $s_{\beta'_{l}} \cdots s_{\beta'_2}s_{\beta_1'}$ takes $\x_{\rho, w_0}$ to $\x_{-\rho, 1}$.
\end{lemma}

\begin{proof}
We first note that all folds occur within the gallery $\tau$ by construction.  Also, as observed above, each of the roots $\beta_j'$ is simple and $0\in H_{\beta'_j}$ for $1 \leq j \leq l$.  

Folding $\tau$ in each of the hyperplanes $H_{\beta_1'}, \ldots, H_{\beta_l'}$ in turn corresponds to the product of simple reflections $s_{\beta_l'} \cdots s_{\beta_1'}$.  Now $s_{\beta_1} = s_{\beta_1'}$ and $s_{\beta_2'} = s_{\beta_1} s_{\beta_2} s_{\beta_1}$, and hence $s_{\beta_2'}s_{\beta_1'} = s_{\beta_1} s_{\beta_2}$.  Inductively, we find that
\[ 
s_{\beta_l'} \cdots s_{\beta_1'} = s_{\beta_1} \cdots s_{\beta_l}. 
 \]
But the hyperplanes crossed by $\tau:\w_0 \rightsquigarrow \fa$ are exactly $H_{\beta_1}, \ldots, H_{\beta_l}$, in that order, and so 
\[ 
(s_{\beta_l} \cdots s_{\beta_1})^{-1} = s_{\beta_1} \cdots s_{\beta_l}
\]
is a word for $w_0$.  Thus the result of applying this sequence of reflections takes the end alcove $\fa$ of $\tau$ to the alcove $\w_0$, and takes the minimal gallery $\gamma:\fa \rightsquigarrow \x_{\rho,w_0}$ to the minimal gallery $w_0\gamma:\w_0 \rightsquigarrow w_0\x_{\rho,w_0}$.  

Now observe that $$w_0x_{\rho,w_0} = w_0 t^\rho w_0 = t^{-\rho} = x_{-\rho,1},$$ and so applying the product of reflections $s_{\beta'_{l}} \cdots s_{\beta'_2}s_{\beta_1'}$ takes $\x_{\rho, w_0}$ to $\x_{-\rho, 1}$ as required.
\end{proof}

\begin{lemma}\label{lem:lem_w0-folding-orientation}  The resulting folded gallery $\tilde\sigma: \x_{-\rho,1} \rightsquigarrow \x_{-\rho,1}$ is positively folded with respect to $-\phipartialo$.
\end{lemma}

\begin{proof}  Since each $\beta_j'$ is simple, all folds of $\tau$ corresponding to the sequence of reflections $s_{\beta_l'} \cdots s_{\beta_1'}$ occur in walls of the alcove $\w_0$, and so are negative folds with respect to the standard orientation, hence positive with respect to the orientation $-\phipartialo$.  \end{proof}

Recall that we are aiming for a gallery $\sigma_{a_0}:\fa\rightsquigarrow \fa$ of type $\vec a_0$ some reduced expression for $a_0 = t^{2\rho}w_0$.
We put 
\[
\sigma_{a_0}\define t^{\rho}\tilde\sigma: \fa\rightsquigarrow \fa .
\]
Here, the translation $t^{\rho}$ is an element of the extended affine Weyl group, which naturally acts on the set of galleries.  This action in general permutes types of panels (so does not preserve types of galleries).  Since $\tilde\sigma:\x_{-\rho,1} \rightsquigarrow \x_{-\rho,1}$ has the same type as the minimal gallery $\sigma:\x_{-\rho,1} \rightsquigarrow \x_{\rho,w_0}$, and $x_{-\rho,1}a_0 = x_{\rho,w_0}$, we obtain that $\sigma_{a_0}$ has type $\vec{a_0}$.  The gallery $\sigma_{a_0}$ can also be described as follows: first let $\sigma'$ be the image of the minimal gallery $\sigma$ under $t^\rho$, so that $\sigma'$ has type $\vec{a_0}$.  Then apply to $\sigma'$ the shifted folding sequence at $t^\rho$, that is, the folding sequence we obtain from the one constructed above by taking as folding hyperplanes the images of the previous folding hyperplanes under $t^\rho$.  The gallery $\sigma_{a_0}$ in type $\tilde C_2$ is shown on the bottom right of Figure~\ref{fig:w0_fold}.

Finally, we compute the dimension of $\sigma_{a_0}$.

\begin{lemma}\label{lem:dim sigma}  The gallery $\sigma_{a_0}:\fa \rightsquigarrow \fa$ satisfies $\dim_{-\phi_0}(\sigma_{a_0}) = \ell(t^{\rho})$.
\end{lemma}

\begin{proof}
Recall that by construction the gallery $\sigma_{a_0}$ has $\ell(w_0)$ folds.  Hence by Lemma~\ref{lem:gallery dim} in the case $y = 1$
\[
\dim_{-\phi_0}(\sigma_{a_0}) = \frac{1}{2}\left(\ell(a_0) + \ell(w_0) \right).
\]
Now, recall that the alcove $\mathbf{a}_0 = a_0 \fa$ is in the dominant Weyl chamber $\Cf$, and so $a_0 = t^{2\rho}w_0$ is the minimal length coset representative in $\aW/\sW$, while $t^{2\rho}$ is the maximal length coset representative.  Therefore $\ell(a_0) + \ell(w_0) = \ell(t^{2\rho})$, and so 
\begin{equation*}
\dim_{-\phi_0}(\sigma_{a_0}) = \frac{1}{2}\left( \ell(a_0) + \ell(w_0) \right) = \frac{1}{2}\left( \ell(t^{2\rho}) \right) = \ell(t^{\rho}),
\end{equation*}
as desired.
\end{proof}

\begin{remark}
Since all folds of $\sigma$ occur within the subgallery $\tau$, it is easy to see that this construction could be extended to obtain larger families of folded galleries, if we replace the subgalleries $\gamma$ and $\bar\gamma$ of $\sigma$ by other galleries (one not necessarily of the same type as the other).
\end{remark}


\subsection{An infinite family of positively folded galleries}\label{sec:more constructions}

Building upon Section~\ref{sec:construction_a0}, we now use root operators to construct an infinite family of positively folded galleries.  The main result of this section is Proposition~\ref{prop:constr-elts1} below.   Figure \ref{fig:Xa1toXxbPropFig} depicts a family of galleries obtained by the construction given in Proposition~\ref{prop:constr-elts1}. 

\begin{figure}[ht]
\begin{center}
\begin{overpic}[width=0.6\textwidth]{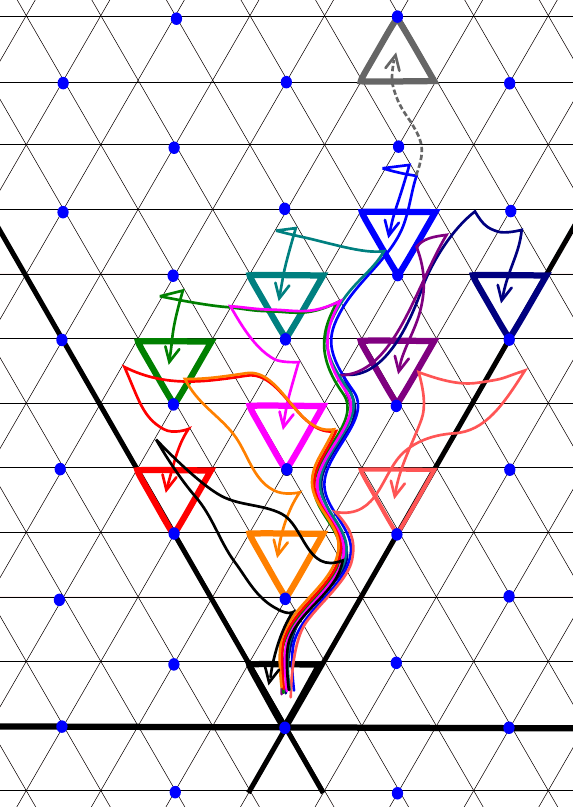}
\put(46,98.5){\footnotesize{\color{blue}{$\lambda$}}}
\put(52,95){\color{gray}{$\x_0$}}
\end{overpic}
\caption{Each gallery in this figure is an instance of a gallery $\gamma_0:\fa \rightsquigarrow \bb$ of type $\vec{x}_0$ obtained by the construction given in Proposition~\ref{prop:constr-elts1}.}
\label{fig:Xa1toXxbPropFig}
\end{center}
\end{figure}

We define the \emph{positive cone} $\Cpos$ to be the positive span of all coroots that are both simple and positive.  Analogously, we define the \emph{negative cone} $\Cneg$ to be the positive span of all coroots that are both simple and negative; compare~\cite{Bourbaki4-6}.  Given any two alcoves $\x$ and $\y$, we define the \emph{convex hull of $\x$ and $\y$}, denoted by $\conv(\x,\y)$, to be the set of all alcoves which lie in some minimal alcove-to-alcove gallery from $\x$ to $\y$.  

\begin{prop}\label{prop:constr-elts1}
Let $x_0 = t^\lambda w_0 \in \aW$ have spherical direction $w_0$ and be such that the alcove $\x_0 =x_0 \fa$ is in the shrunken dominant Weyl chamber $\Cfs$. 
Suppose that $\mu \in R^\vee$ is in the intersection of $\Cf$ with the negative cone based at the vertex $\lambda - 2\rho$, that is, $$\mu\in \Cf\cap((\lambda - 2\rho)+\Cneg),$$ and let $\bb = t^\mu\fa$. 
Then there exists a gallery $\gamma_0:\fa\rightsquigarrow \bb$ of type $\vec x_0$ which is positively folded with respect to $-\phipartialo$ and has $\ell(w_0)$ folds.  
Moreover the dimension of $\gamma_0$ satisfies
\begin{equation} 
\dim_{-\phi_0}(\gamma_0)= \dim_{-\phi_0}(\sigma_{a_0})  + \langle \rho, \lambda - 2\rho - \mu \rangle =  \langle \rho, \lambda - \mu \rangle.
\end{equation}
\end{prop}

\begin{remark}\label{rem:gamma0}  Since $\x_0$ is in the $w_0$ position in the shrunken dominant Weyl chamber, the vertex $\lambda$ does not lie on any wall containing the origin.  Moreover $\lambda - 2\rho$ is dominant and so in particular the intersection $\Cf\cap((\lambda - 2\rho)+\Cneg)$ contains the origin, thus this intersection is nonempty.  
\end{remark}

\begin{proof}  Let $\y$ be the pure translation alcove $t^{\lambda-2\rho} \fa$ and observe that since $\lambda - 2\rho$ is dominant, we have $\y\in \conv(\fa,\x_0)$.   Now choose $\gamma_y:\fa\rightsquigarrow \y$ a minimal gallery.  Then $\dim_{-\phi_0}(\gamma_y)=0$ since all crossings in $\gamma_y$ are negative with respect to $-\phipartialo$ and $\gamma_y$ has no folds.  See the first frame in Figure~\ref{fig:gamma0sequence1}.

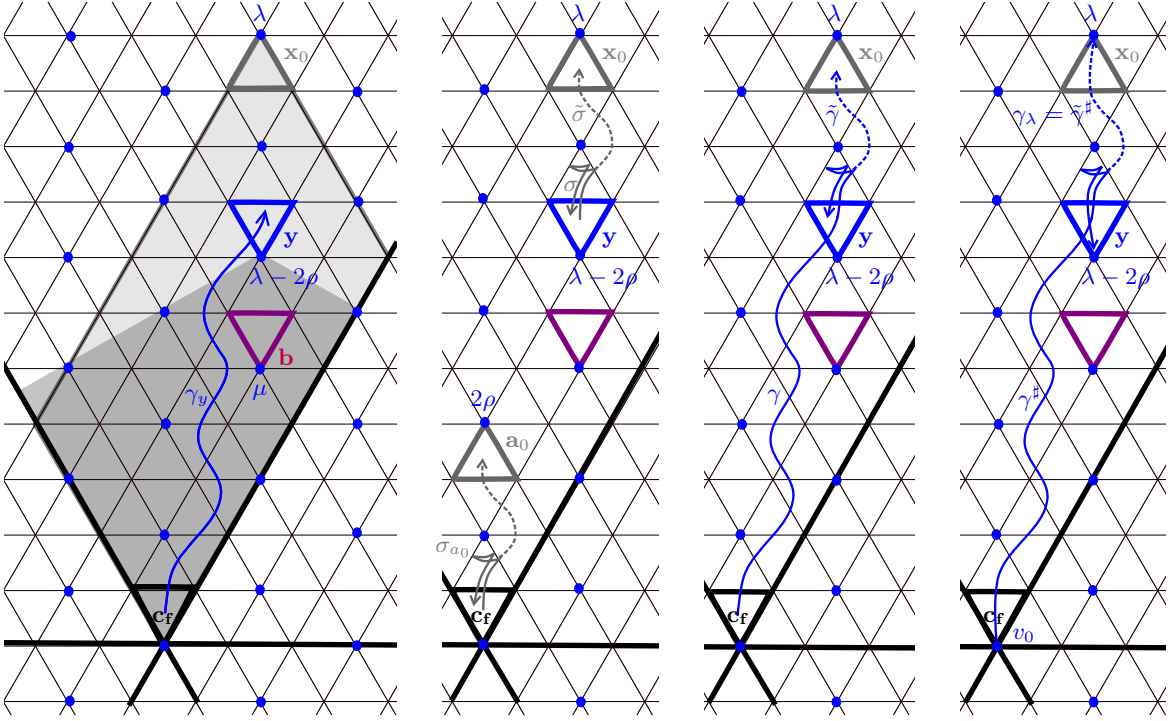
\begin{figure}[ht]
\begin{center}
\begin{overpic}[width=\textwidth]{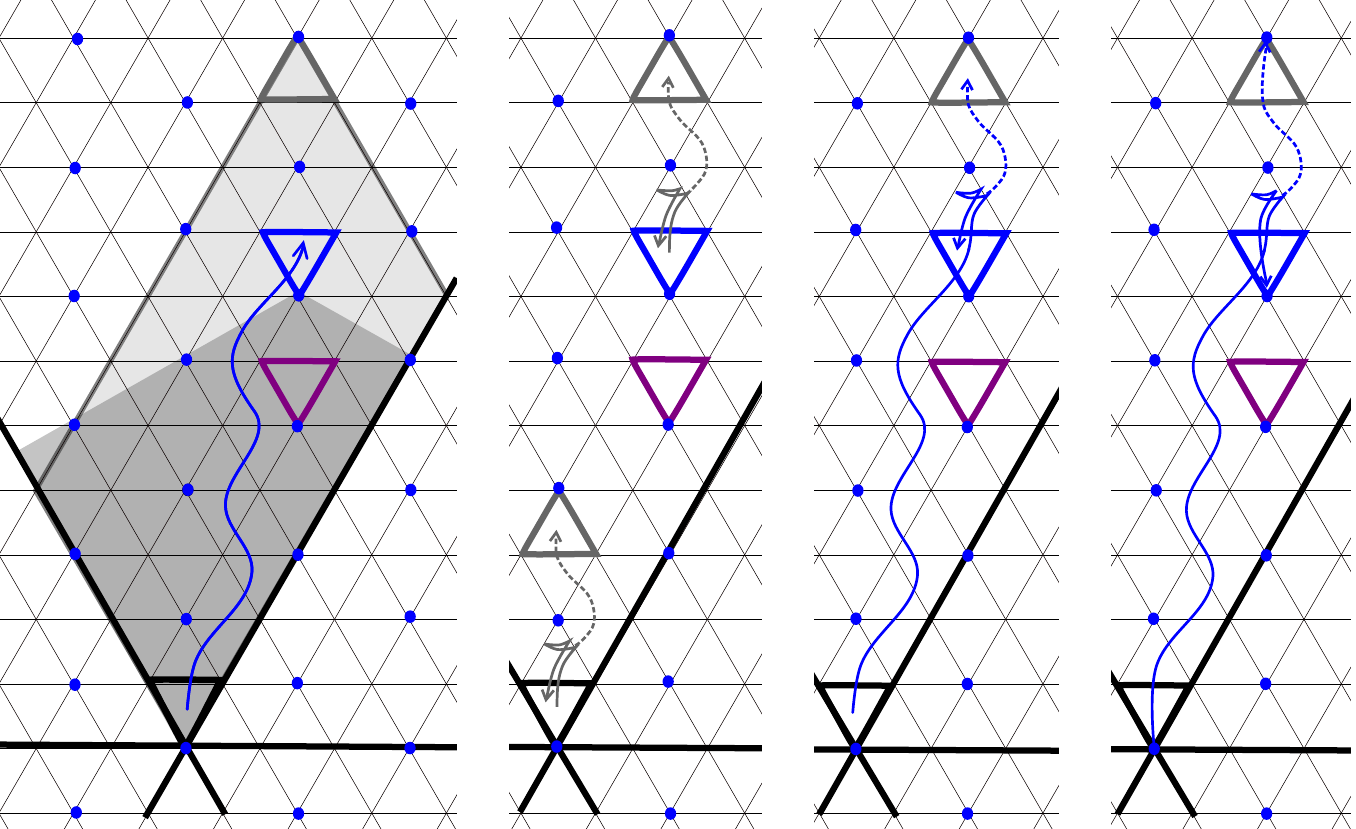}
\put(15.5,27){\footnotesize{\color{blue}{$\gamma_y$}}}
\put(12.7,8){\footnotesize{\color{black}{$\fa$}}}
\put(24,56.3){\footnotesize{\color{gray}{$\x_0$}}}
\put(24,40.5){\footnotesize{\color{blue}{$\y$}}}
\put(23.5,30){\footnotesize{\color{purple}{$\bb$}}}
\put(21.3,59.6){\footnotesize{\color{blue}{$\lambda$}}}
\put(21,37){\footnotesize{\color{blue}{$\lambda - 2\rho$}}}
\put(21.3,27.5){\footnotesize{\color{blue}{$\mu$}}}
\put(37,14){\footnotesize{\color{gray}{$\sigma_{a_0}$}}}
\put(48,45){\footnotesize{\color{gray}{$\sigma$}}}
\put(48.7,51){\footnotesize{\color{gray}{$\tilde\sigma$}}}
\put(43,23){\footnotesize{\color{gray}{$\mathbf{a}_0$}}}
\put(40,8){\footnotesize{\color{black}{$\fa$}}}
\put(51.3,56.3){\footnotesize{\color{gray}{$\x_0$}}}
\put(51.3,40.5){\footnotesize{\color{blue}{$\y$}}}
\put(48.6,59.6){\footnotesize{\color{blue}{$\lambda$}}}
\put(48.4,37){\footnotesize{\color{blue}{$\lambda - 2\rho$}}}
\put(40,26.5){\footnotesize{\color{blue}{$2\rho$}}}
\put(65.5,27){\footnotesize{\color{blue}{$\gamma$}}}
\put(70.5,51){\footnotesize{\color{blue}{$\tilde\gamma$}}}
\put(73.3,56.3){\footnotesize{\color{gray}{$\x_0$}}}
\put(73.3,40.5){\footnotesize{\color{blue}{$\y$}}}
\put(70.6,59.6){\footnotesize{\color{blue}{$\lambda$}}}
\put(70.4,37){\footnotesize{\color{blue}{$\lambda - 2\rho$}}}
\put(62,8){\footnotesize{\color{black}{$\fa$}}}
\put(87,26.5){\footnotesize{\color{blue}{$\gamma^\sharp$}}}
\put(86.5,51){\footnotesize{\color{blue}{$\gamma_\lambda = \tilde\gamma^\sharp$}}}
\put(95.3,56.3){\footnotesize{\color{gray}{$\x_0$}}}
\put(95.3,40.5){\footnotesize{\color{blue}{$\y$}}}
\put(92.6,59.6){\footnotesize{\color{blue}{$\lambda$}}}
\put(92.4,37){\footnotesize{\color{blue}{$\lambda - 2\rho$}}}
\put(86.5,6.5){\footnotesize{\color{blue}{$v_0$}}}
\put(84,8){\footnotesize{\color{black}{$\fa$}}}
\end{overpic}
\caption{We illustrate the proof of Proposition~\ref{prop:constr-elts1}.   In the first frame, the intersection of the negative cone based at the vertex $\lambda - 2\rho$ with the dominant Weyl chamber $\Cf$ is shaded dark gray and the convex hull of the fundamental alcove $\fa$ and the alcove $\x_0$ is shaded light gray.}
\label{fig:gamma0sequence1}
\end{center}
\end{figure}

By the construction given in Section~\ref{sec:construction_a0} above and Lemma~\ref{lem:dim sigma}, there exists $\sigma_{a_0}:\fa \rightsquigarrow \fa$ of type $\vec a_0$ with $a_0 = t^{2\rho}w_0$, so that $\sigma_{a_0}$ is positively folded with respect to $-\phipartialo$ and $\dim_{-\phi_0}(\sigma_{a_0}) = \ell(t^\rho)$.  We now translate $\sigma_{a_0}$ by $t^{\lambda-2\rho}$ to obtain a gallery $\sigma: \y\rightsquigarrow \y$ of type $\vec a_0$.  By Lemma~\ref{lem:translations dim} we have $\dim_{-\phi_0}(\sigma) = \dim_{-\phi_0}(\sigma_{a_0})$.  Note also that $\sigma$ has $\ell(w_0)$ folds.  Let $\tilde\sigma$ be the minimal alcove-to-alcove gallery from $\y$ to $\x_0$ which is of the same type $\vec a_0$ as $\sigma$, that is, $\tilde\sigma$ is the minimal gallery obtained by unfolding $\sigma$.  See the second frame in Figure~\ref{fig:gamma0sequence1}.

Now define $\tilde\gamma$ to be the concatenation $\tilde\gamma\define \gamma_y \ast \tilde\sigma$ and $\gamma$ to be the concatenation $\gamma\define \gamma_y \ast\sigma$.  See the third frame of Figure~\ref{fig:gamma0sequence1}.  These galleries have the same type $\vec x_0$ since $x_0 = t^\lambda w_0 = t^{\lambda - 2\rho}t^{2\rho}w_0 = t^{\lambda - 2\rho}a_0$.  Notice that the gallery $\tilde\gamma$ is minimal.  Also $\tilde\gamma$ is contained in the dominant Weyl chamber, with first alcove $\fa$ a pure translation and last alcove $\x_0$ in the $w_0$ position.  Thus the canonically associated vertex-to-vertex gallery $\tilde\gamma^\sharp$ is minimal as well.  

We observe that the gallery $\gamma$ is positively folded with respect to $-\phipartialo$, all of its $\ell(w_0)$ folds appear outside of $\gamma_y$, and since $\dim_{-\phi_0}(\gamma_y) = 0$ we have 
\[ 
\dim_{-\phi_0}(\gamma) = \dim_{-\phi_0}(\sigma) = \ell(t^\rho).
\]  
Let $\gamma^\sharp$ be the canonical associated vertex-to-vertex gallery, with initial vertex $v_0$ the origin and final vertex $\lambda - 2\rho \in R^\vee$.  As the first alcove of $\gamma$ and thus of $\gamma^\sharp$ is $\fa = w_0w_0\fa$,  Corollary~\ref{cor:folds+crossings} implies that $\dim_{-\phi_0}(\gamma^\sharp) = \ell(t^\rho)$ as well.  Moreover, since as argued above the vertex-to-vertex gallery $\tilde\gamma^\sharp$ is minimal, it follows that $\gamma^\sharp$ has the type of a minimal gallery $\gamma_\lambda = \tilde\gamma^\sharp$ from the origin to $\lambda$.  As noted in Remark~\ref{rem:gamma0} above, the end-vertex $\lambda$ does not lie on any wall through the origin.  It follows that $\gamma_\lambda$ is also a minimal gallery in the sense defined in~\cite{GaussentLittelmann} (compare Remark~\ref{rem:GLminimal}).  See the final frame of Figure~\ref{fig:gamma0sequence1}.

We now act on the left by $w_0$.  See Figure~\ref{fig:gamma0sequence2}.  The gallery $\gamma^\sharp$ is positively folded with respect to the orientation $-\phipartialo$ and has $\ell(w_0)$ folds, therefore using Lemma~\ref{lem:wpos} the gallery $w_0\gamma^\sharp$ is positively folded with respect to the standard orientation and still has $\ell(w_0)$ folds.  By Lemma~\ref{lem:w0 dim}, we have $\dim_{\phi_0}(w_0\gamma^\sharp) = \dim_{-\phi_0}(\gamma^\sharp)$, so the dimension of $w_0\gamma^\sharp$ with respect to the standard orientation is equal to $\ell(t^\rho)$.  As the action of $\sW$ is type-preserving, the gallery $w_0\gamma^\sharp$ is also of the same type as $\gamma_{\lambda}$.   The final vertex of $w_0\gamma^\sharp$ is $2\rho - \lambda = w_0(\lambda - 2\rho)$.  Hence $w_0\gamma^\sharp \in \Gamma^+(\gamma_{\lambda},2\rho - \lambda)$, and we can consider the application of root operators to the gallery $w_0\gamma^\sharp$.  

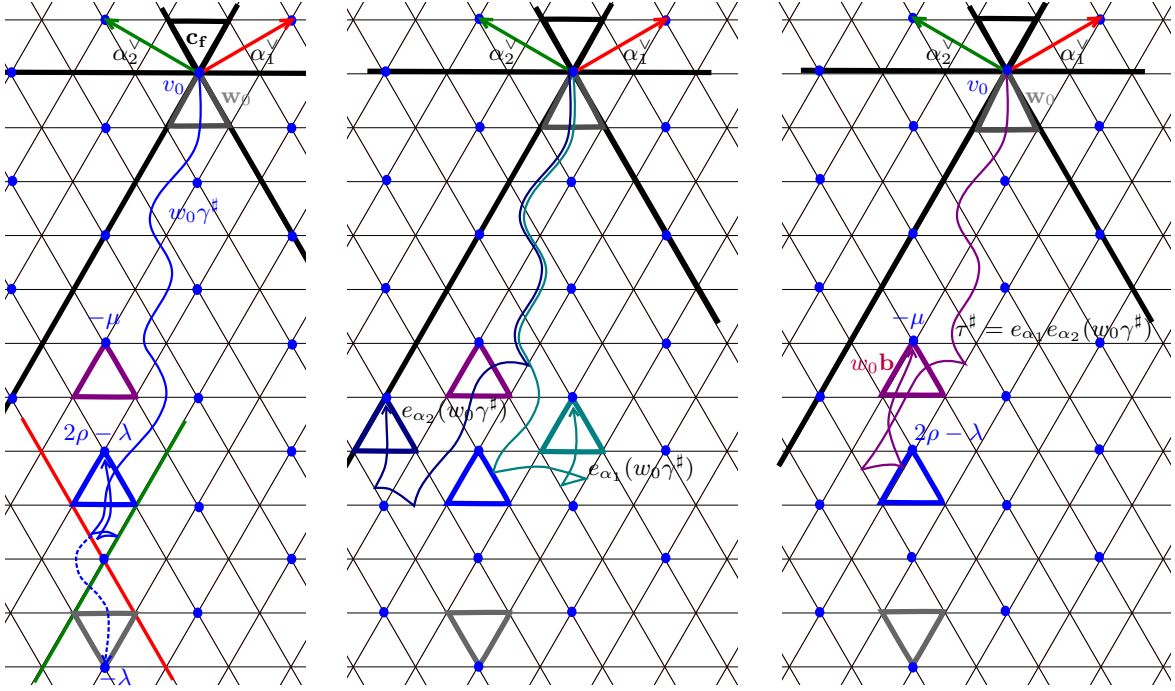
\begin{figure}[ht]
\begin{center}
\begin{overpic}[width=\textwidth]{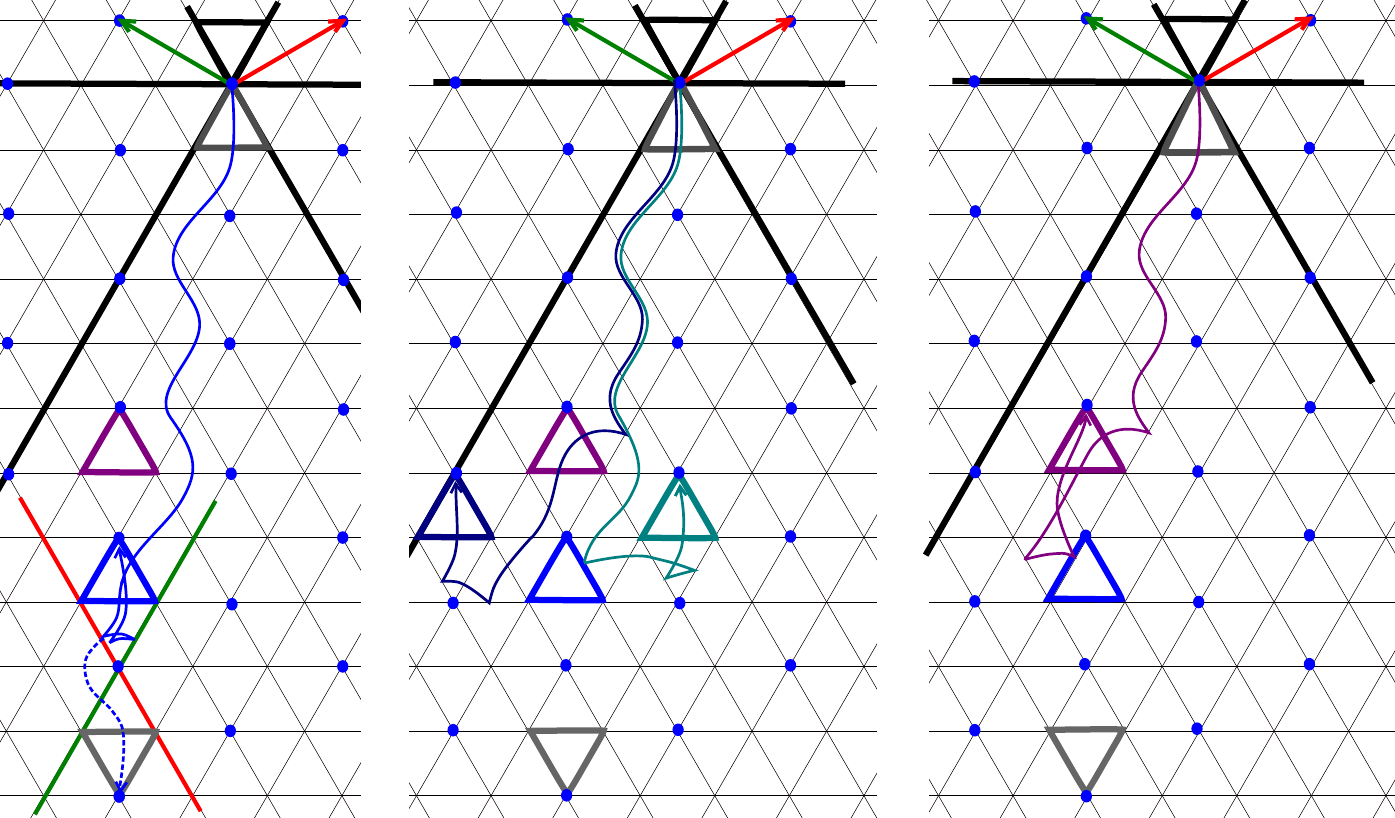}
\put(14,40){\footnotesize{\color{blue}{$w_0\gamma^\sharp$}}}
\put(15.5,55){\footnotesize{\color{black}{$\fa$}}}
\put(18.5,50){\footnotesize{\color{gray}{$\mathbf{w}_0$}}}
\put(8,0){\footnotesize{\color{blue}{$-\lambda$}}}
\put(5,21){\footnotesize{\color{blue}{$2\rho-\lambda$}}}
\put(7,31){\footnotesize{\color{blue}{$-\mu$}}}
\put(13.5,51){\footnotesize{\color{blue}{$v_0$}}}
\put(9.2,53.7){\footnotesize{\color{black}{$\alpha_2^\vee$}}}
\put(21,53.7){\footnotesize{\color{black}{$\alpha_1^\vee$}}}
\put(34,23){\footnotesize{\color{black}{$e_{\alpha_2}(w_0\gamma^\sharp)$}}}
\put(50,18){\footnotesize{\color{black}{$e_{\alpha_1}(w_0\gamma^\sharp)$}}}
\put(41.5,53.7){\footnotesize{\color{black}{$\alpha_2^\vee$}}}
\put(53,53.7){\footnotesize{\color{black}{$\alpha_1^\vee$}}}
\put(81.5,30){\footnotesize{\color{black}{$\tau^\sharp = e_{\alpha_1}e_{\alpha_2}(w_0\gamma^\sharp)$}}}
\put(79,53.7){\footnotesize{\color{black}{$\alpha_2^\vee$}}}
\put(90.5,53.7){\footnotesize{\color{black}{$\alpha_1^\vee$}}}
\put(76,31){\footnotesize{\color{blue}{$-\mu$}}}
\put(78,21){\footnotesize{\color{blue}{$2\rho-\lambda$}}}
\put(87.5,50){\footnotesize{\color{gray}{$\mathbf{w}_0$}}}
\put(72.5,27){\footnotesize{\color{purple}{$w_0\bb$}}}
\put(82.5,51){\footnotesize{\color{blue}{$v_0$}}}
\end{overpic}
\caption{We illustrate the proof of Proposition~\ref{prop:constr-elts1}, showing the application of root operators.}
\label{fig:gamma0sequence2}
\end{center}
\end{figure}

We first consider the indexing of the hyperplanes in which the gallery $w_0\gamma^\sharp$ has folds.  In the first frame of Figure~\ref{fig:gamma0sequence2}, the folds of the gallery $w_0\gamma^\sharp$ are in the hyperplanes $H_{\alpha_1,-6}$ and $H_{\alpha_2,-3}$ (and these hyperplanes are colored according to the corresponding simple root).

\begin{lemma}\label{lem:rootops0}  The folds in the gallery $w_0\gamma^\sharp$ occur in hyperplanes $H_{\alpha,M(\alpha)}$, where $M(\alpha) \in \Z$ is such that $2\rho - \lambda$ is contained in $H_{\alpha,M(\alpha)+1}$. For each simple root $\alpha$ there is at least one such fold.  Moreover $M(\alpha) \leq -1$ and 
$
M(\alpha) = m(w_0\gamma^\sharp, \alpha),
$ 
where $m(w_0\gamma^\sharp, \alpha) \in \Z$ is as defined in Notation~\ref{not:operators}.
\end{lemma}   

\begin{proof} Recall that by the construction given in Section~\ref{sec:construction_a0} the $l = \ell(w_0)$ folds of $\sigma_{a_0}$ occur in the hyperplanes $H_{\beta_1', 1}$, \ldots, $H_{\beta_l',1}$, where each $\beta_j'$ is a simple root and $s_{\beta_l'}\cdots s_{\beta_1'}$ is a word for $w_0$.  Hence each simple root has to appear at least once among $\beta_1',\dots,\beta_l'$.  

Since $\sigma$ is the result of translating $\sigma_{a_0}$ by $t^{\lambda - 2\rho}$, all folds of $\sigma$ and thus of $\gamma$ and $\gamma^\sharp$ appear in hyperplanes of the form $H_{\alpha,m'(\alpha)}$ where $m'(\alpha) \in\Z$ is such that $\lambda - 2\rho \in H_{\alpha,m'(\alpha) - 1}$, and for each simple root $\alpha$ there is at least one such fold.  Therefore all folds of $w_0\gamma^\sharp$ appear in hyperplanes of the form $H_{\alpha,-m'(\alpha)}$ where $2\rho - \lambda \in H_{\alpha,-m'(\alpha)+1}$.  Observe that $M(\alpha) = -m'(\alpha)$ and we have that all folds of $w_0\gamma^\sharp$ appear in hyperplanes of the form described in the assertion. As each simple root $\alpha$ appears at least once among $\beta_1',\dots,\beta_l'$  there is at least one fold in $H_{\alpha,M(\alpha)}$ for all $\alpha$.

As $2\rho-\lambda$ is antidominant, we have that $M(\alpha)\leq -1$.  It follows by the construction of $w_0\gamma^\sharp$ that  the integer $m(w_0\gamma^\sharp,\alpha)$ defined in Notation \ref{not:operators} equals $M(\alpha)$ defined above.   
\end{proof}

We now consider the application of a single root operator $e_\alpha$.  The proof of the second sentence of the following lemma is in essence a formalization of the heuristics described in Remark~\ref{rem:operators}. We write $e^{m}_{\alpha}$ for an $m$-fold application of the same root operator $e_\alpha$.  The second frame in Figure~\ref{fig:gamma0sequence2} illustrates single applications of the two different possible root operators in type $\tilde A_2$. 

\begin{lemma}\label{lem:rootops1} 
For each simple root $\alpha$, the root operator $e^{\vert M(\alpha)\vert}_{\alpha}$ is defined for $w_0\gamma^\sharp$.  Moreover the image of $w_0\gamma^\sharp$ under $e_\alpha^{\vert M(\alpha)\vert - 1}$ still has first alcove $w_0\fa$, the number of folds remains $\ell(w_0)$, and the spherical direction of its final alcove is the same as the one of $w_0\gamma^\sharp$.
\end{lemma}

\begin{proof}
We obtain from Lemma~\ref{lem:posFoldsV}\eqref{fold1} and Lemma~\ref{lem:rootops0} that $e_\alpha$ can be applied at least once to $w_0\gamma^\sharp$.
Let $m = m(w_0\gamma^\sharp, \alpha)$.
The operator $e_\alpha$ has the effect of either reflecting or translating all parts of $w_0\gamma^\sharp$ which lie between the walls $H_{\alpha,m}$ and $H_{\alpha,m + 1}$ so that their image is on the same side of $H_{\alpha,m+1}$ as the origin, and of fixing the first panel of $w_0\gamma^\sharp$ which lies in $H_{\alpha,m+1}$.  Hence 
$$
m(e_\alpha(w_0\gamma^\sharp),\alpha) = m(w_0\gamma^\sharp,\alpha)+1 = m + 1 = M(\alpha) + 1.
$$  
By induction, if $1 \leq i \leq |M(\alpha)|$ then 
$$m(e^j_\alpha(w_0\gamma^\sharp),\alpha) = M(\alpha) + i \leq 0.$$
It follows by Lemma~\ref{lem:posFoldsV}\eqref{fold1} and \eqref{fold3} that a single operator $e_\alpha$ is defined exactly $|M(\alpha)|$-many times.

Suppose now that $1 \leq j \leq |M(\alpha)| - 1$.  By Definition~\ref{def:GLoperators}, the root operator $e_\alpha$ leaves the first alcove of the gallery $e_\alpha^{j-1}(w_0\gamma^\sharp)$ fixed unless $m(e_\alpha^{j-1}(w_0\gamma^\sharp),\alpha) = -1$.  However
\[
m(e_\alpha^{j-1}(w_0\gamma^\sharp),\alpha) = M(\alpha) + (j-1) \leq -2,
\]
and so the first alcove $w_0\fa$ is fixed by the first $(|M(\alpha)|-1)$-many applications of $e_\alpha$.  Each of these applications has the effect of 
shifting a fold in the wall perpendicular to $\alpha$ of index $m(e_\alpha^{j-1}(w_0\gamma^\sharp),\alpha) \leq -2$ to a fold in the wall perpendicular to $\alpha$ of index one greater, and of translating all other folds, so the number of folds is preserved.

As for the final alcove, Definition~\ref{def:GLoperators} implies that the final alcove of $e_\alpha^{j-1}(w_0\gamma^\sharp)$ is translated by $e_{\alpha}$ unless $k(e_\alpha^{j-1}(w_0\gamma^\sharp), \alpha)=n+1$ and the only face of $e_\alpha^{j-1}(w_0\gamma^\sharp)$ to be contained in $H_{\alpha,M(\alpha) + (j-1)}$ is its final vertex.
Since $M(\alpha) + (j-1) \leq -2$, as argued above the gallery $e_\alpha^{j-1}(w_0\gamma^\sharp)$ has a fold in the hyperplane $H_{\alpha,M(\alpha) + (j-1)}$, and so has at least one panel contained in this hyperplane.  Hence the final alcove of $e_\alpha^{j-1}(w_0\gamma^\sharp)$ is in the part of the gallery that had been translated by $e_{\alpha}$.  Thus the spherical direction of the final alcove of $w_0\gamma^\sharp$ is not changed by  $e_\alpha^{|M(\alpha)|-1}$.
\end{proof}

We next consider the application of root operators for distinct roots.  The final frame of Figure~\ref{fig:gamma0sequence2} illustrates the result of applying first $e_{\alpha_2}$ then $e_{\alpha_1}$ to the gallery $w_0\gamma^\sharp$.

\begin{lemma}\label{lem:rootops2}  Let $\alpha$ and $\beta$ be distinct simple roots.  Then for all $1 \leq j \leq |M(\alpha)| - 1$ the operator $e_\beta^{\vert M(\beta)\vert - 1}$ can be applied to  $e_\alpha^j(w_0\gamma^\sharp)$, so that the first alcove $w_0\fa$ stays fixed, the spherical direction of the final alcove never changes, and the number of folds remains $\ell(w_0)$.
\end{lemma}

\begin{proof}
Lemma~\ref{lem:posFoldsV} item (\ref{fold4}) says that 
\[
q(w_0\gamma^\sharp,\beta)
= p(w_0\gamma^\sharp,\beta) - \langle \beta, 2\rho - \lambda \rangle,
\]
where $p=p(w_0\gamma^\sharp,\beta)$ is maximal such that $f_\beta^p$ is defined for $w_0\gamma^\sharp$ and $q=q(w_0\gamma^\sharp,\beta)$ is maximal such that $e_\beta^q$ is defined for $w_0\gamma^\sharp$. 

Now for the gallery $e^j_\alpha(w_0\gamma^\sharp)$ with end vertex $2\rho - \lambda +j\alpha^{\vee}$ we want to show that  $q(e^j_\alpha(w_0\gamma^\sharp),\beta)\geq q(w_0\gamma^\sharp,\beta)$.   
We have by Lemma~\ref{lem:posFoldsV} item \eqref{fold3} and \eqref{fold4}
  and the fact that $\langle \beta, \alpha^\vee \rangle \leq 0$ that 

\begin{eqnarray*}
q(e^j_\alpha(w_0\gamma^\sharp),\beta) 
& = & p(e^j_\alpha(w_0\gamma^\sharp),\beta) - \langle \beta,  2\rho - \lambda+j\alpha^\vee \rangle \\ 
& = & p(w_0\gamma^\sharp,\beta) + j - \langle \beta,  2\rho - \lambda+j\alpha^\vee \rangle  \\ 
&  = & q(w_0\gamma^\sharp,\beta) + j(1-\langle\beta, \alpha^\vee\rangle) \\
& \geq & q(w_0\gamma^\sharp,\beta). 
\end{eqnarray*}

Thus the operator $e_\beta$ is applicable at least as often for $e_\alpha(w_0\gamma^\sharp)$  as it was for $w_0\gamma^\sharp$.  

The claims about the first alcove, the number of folds, and the spherical direction of the final alcove are proved similarly to the corresponding claims in Lemma~\ref{lem:rootops1}. 
\end{proof}

Using induction and similar arguments one can show the following corollary.  

\begin{corollary}  Let $\alpha_1, \dots, \alpha_n$ be any enumeration of the simple roots.  Let $1 \leq i \leq n$ and suppose for all $1 \leq j \leq i$ that $c_j$ is an integer satisfying $1 \leq c_j \leq |M(\alpha_j)| - 1$.  Then for all galleries $\sigma$ of the form $\sigma=e^{c_i}_{\alpha_i}(e^{c_{i-1}}_{\alpha_{i-1}}(\cdots (e^{c_1}_{\alpha_1}(w_0\gamma^\sharp)))$ one has that 
$q(\sigma, \alpha_k) \geq q(w_0\gamma^\sharp, \alpha_k)$ whenever $i < k \leq n$. 
\end{corollary}

To complete the proof of Proposition~\ref{prop:constr-elts1}, for any $\mu\in \Cf\cap((\lambda - 2\rho)+\Cneg)$ there are integers $c_\alpha\geq 0$ such that $\sum_{\alpha}c_\alpha\alpha^\vee = (\lambda - 2\rho) - \mu$, where the sum runs over the simple roots~$\alpha$. Equivalently, $-\mu = (2\rho - \lambda) + \sum_{\alpha}c_\alpha\alpha^\vee $.  (For example, in Figure~\ref{fig:gamma0sequence2} we have $-\mu = (2\rho - \lambda) + \alpha_1 + \alpha_2$.)  It is then easy to check that $c_\alpha < \vert M(\alpha)\vert$ for all  $\alpha$: recall that $ M(\alpha) = \langle \alpha, 2\rho-\lambda\rangle -1\leq -1$ and observe that 
\begin{eqnarray*}
c_\alpha\leq \langle \alpha, c_\alpha\alpha^\vee\rangle  
 & = & \langle \alpha,   \lambda-2\rho-\mu-\sum_{\alpha_i\neq\alpha}c_{\alpha_i}\alpha_i^\vee \rangle \\
 & = & \vert M(\alpha) \vert  - 1 - \langle \alpha,   \mu+\sum_{\alpha_i\neq\alpha}c_{\alpha_i}\alpha_i^\vee \rangle \\
 & \leq &   \vert M(\alpha) \vert  - 1.\\
\end{eqnarray*}

Enumerate the simple roots by $\alpha_1, \dots, \alpha_n$ and put $c_i\define c_{\alpha_i}$ for all $i$.  (For example, in Figure~\ref{fig:gamma0sequence2}, we have $c_1 = c_2 = 1$.)  By Lemma~\ref{lem:rootops1}, each root operator $e_{\alpha_i}$ is defined at least $c_i$-many times for the gallery $w_0\gamma^\sharp$. The order in which we apply any of these operators does not matter, as by Lemma~\ref{lem:rootops2} for each $i \neq j$ an application of $e_{\alpha_i}$ does not decrease the number of times $e_{\alpha_j}$ can be defined.  

Consider the gallery 
\[
\tau^\sharp \define e^{c_1}_{\alpha_1}(e^{c_2}_{\alpha_2}(\cdots (e^{c_n}_{\alpha_n}(w_0\gamma^\sharp)))).
\]
See the final frame of Figure~\ref{fig:gamma0sequence2}.  By definition of $\tau^\sharp$ and repeated application of Lemma~\ref{lem:posFoldsV} item (\ref{fold1}), the gallery  $\tau^\sharp$ has first vertex the origin and final vertex $-\mu$, is of the same type as $\gamma_{\lambda}$, is positively folded with respect to the standard orientation $\phipartialo$, and satisfies 
\[
\dim_{\phi_0}(\tau^\sharp)= \ell(t^\rho) + \sum_{\alpha} c_\alpha. 
\]
Now by definition of the coefficients $c_\alpha$, and since $\lambda - 2\rho - \mu$ is dominant, we have 
\[
 \sum_{\alpha} c_\alpha = \langle \rho, \lambda - 2\rho -\mu \rangle
\]
and so as $\ell(t^\rho) = \langle \rho, 2\rho \rangle$ we obtain 
\begin{equation}\label{E:dimtau}
\dim_{\phi_0}(\tau^\sharp)= \ell(t^\rho) + \langle \rho, \lambda - 2\rho -\mu \rangle =  \langle \rho, \lambda - \mu \rangle. 
\end{equation}
Moreover, by Lemmas~\ref{lem:rootops1} and~\ref{lem:rootops2}, since we have applied each root operator $e_\alpha$ at most $|M(\alpha)| - 1$ times, the first alcove of $\tau^\sharp$ is $w_0\fa$, the final alcove of $\tau^\sharp$ is $w_0\bb$, and $\tau^\sharp$ contains $\ell(w_0)$ folds.  For later use, we record the following lemma.

\begin{lemma}\label{lem:tauDL}  
The gallery $\tau^\sharp \in \Gamma^+(\gamma_{\lambda},-\mu)$ is an LS-gallery.
\end{lemma}

\begin{proof}  By Definition~\ref{def:LS}, it suffices to show that $\dim_{\phi_0}(\tau^\sharp) = \langle \rho, \lambda - \mu \rangle$.  This is given by Equation~\eqref{E:dimtau} above.
\end{proof}

In order to complete the proof of Proposition~\ref{prop:constr-elts1}, we now define the gallery $\gamma_0$ that we are seeking to be the alcove-to-alcove gallery obtained from $w_0\tau^\sharp$ by removing its first and last vertex.  Figure~\ref{fig:Xa1toXxbPropFig} illustrates the galleries $\gamma_0$ obtained in this way, for all $\mu$ satisfying the hypotheses Proposition~\ref{prop:constr-elts1} when $\lambda$ is as depicted.  We have thus obtained a gallery $\gamma_0: \fa \rightsquigarrow \bb$ which is of type $\vec x_0$, is positively folded with respect to $-\phipartialo$, and contains $\ell(w_0)$ folds.   By definition of $\gamma_0$ the canonically associated vertex-to-vertex gallery $\gamma_0^\sharp$ is equal to $w_0\tau^\sharp$, so we have by Lemma~\ref{lem:w0 dim} that $\dim_{-\phi_0}(\gamma_0^\sharp) = \dim_{\phi_0}(\tau_0^\sharp)$.  But since the spherical direction of the first alcove $\fa$ of $\gamma_0$ is $\id = w_0w_0$, Corollary~\ref{cor:folds+crossings} implies that $\dim_{-\phi_0}(\gamma_0) = \dim_{-\phi_0}(\gamma_0^\sharp)$ and hence 
\[
\dim_{-\phi_0}(\gamma_0)= \dim X_{a_0}(1) + \langle \rho, \lambda - 2\rho -\mu \rangle =  \langle \rho, \lambda - \mu \rangle. 
\] 
This completes the proof of Proposition~\ref{prop:constr-elts1}.
\end{proof}

\begin{remark}
The proof of Proposition~\ref{prop:constr-elts1} provides an explicit construction (via root operators) of the gallery $\gamma_0$ whose existence is stated in the proposition. Moreover, as one can easily see from the proof and Figure~\ref{fig:Xa1toXxbPropFig}, the set of $\mu$ for which we can construct a gallery that is of the type stated in the proposition is in fact larger than $\Cf\cap((\lambda - 2\rho)+\Cneg)$. 
\end{remark}


\section{The varieties $X_x(1)$ in the shrunken dominant Weyl chamber}\label{sec:Xx1ShrunkenDominant}

In this section we prove Theorems~\ref{idDomDimw0} and~\ref{idDomDim} below, which consider the varieties $X_x(1)$ for $x \in W$ such that the alcove $\x=x\fa$ is in the shrunken dominant Weyl chamber $\Cfs$. (See~\ref{def:shrunken} for the definition of $\Cfs$.)  Theorems~\ref{idDomDimw0} and~\ref{idDomDim} are key steps towards establishing Theorem~\ref{Shrunken} in the introduction.  The proof of Theorem~\ref{Shrunken} will be completed in Section~\ref{sec:RelabelingSymmetry}.

The first main result in this section considers alcoves in $\Cfs$ with spherical direction $w_0$.   We prove Theorem~\ref{idDomDimw0} in Section~\ref{sec:w0} using our constructions from Section~\ref{sec:Construction}. 

\begin{thm}\label{idDomDimw0}
Let $a_0 = t^{2\rho}w_0$, and let $x_0 = t^{\lambda}w_0$ be such that the alcove $\x_0 = x_0\fa$ is in the shrunken dominant Weyl chamber $\Cfs$.  Then $X_{a_0}(1) \neq \emptyset$, $X_{x_0}(1) \neq \emptyset$, and 
\begin{equation}\label{E:idDimCalcw0}
\dim X_{x_0}(1) =\dim X_{a_0}(1) + \frac{1}{2}\ell\left(t^{\lambda-2\rho}\right) = \frac{1}{2} \ell(t^{\lambda}).
\end{equation}
\end{thm}

The next theorem, which is the second main result in this section, and which we prove in Section~\ref{sec:Xx1LocallyConstant}, considers arbitrary alcoves in $\Cfs$. 

\begin{thm}\label{idDomDim}
Let $a = t^{2\rho}w$ for some $w \in \sW$, and let $x = t^{\lambda}w$ have the same spherical direction as $a$ and be such that the alcove $\x = x\fa$ is in $\Cfs$.  Then
\begin{equation}\label{E:idNonEmptya}
X_a(1) \neq \emptyset \iff X_x(1) \neq \emptyset.
\end{equation}
Moreover if both of the varieties $X_a(1)$ and $X_x(1)$ are nonempty, their dimensions relate via
\begin{equation}\label{E:idDimCalc}
\dim X_x(1) =\dim X_a(1) + \frac{1}{2}\ell\left(t^{\lambda-2\rho}\right) = \frac{1}{2} \ell(t^{\lambda}).
\end{equation}
\end{thm}

Obviously Theorem~\ref{idDomDim} implies Theorem~\ref{idDomDimw0}, once it is known that $X_{a_0}(1)$ and $X_{x_0}(1)$ are nonempty.  The reason we have separated these statements is that our proof of Theorem~\ref{idDomDimw0} is completely constructive, while that of Theorem~\ref{idDomDim} logically depends on Theorem 1.1.2 from \cite{GoertzHeDim}, which we recall as Theorem~\ref{idADLV} below.  In Section~\ref{sec:Reuman} we discuss this dependence further, and also make a comparison with the nonemptiness criterion found in \cite{GHKR}.


\subsection{The $w_0$ position}\label{sec:w0}

In this section we prove Theorem~\ref{idDomDimw0}, which considers alcoves in the shrunken dominant Weyl chamber $\Cfs$ which have spherical direction $w_0$ the longest word.

We first show that the gallery $\sigma_{a_0}$ constructed in Section~\ref{sec:construction_a0} is a DL-gallery for $X_{a_0}(1)$.  Recall that $a_0 = t^{2\rho}w_0$.

\begin{prop}\label{w0Construction}
The variety $X_{a_0}(1)$ is nonempty and the gallery $\sigma_{a_0}: \fa \rightsquigarrow \fa$ is a DL-gallery for $X_{a_0}(1)$, hence $\dim X_{a_0}(1) = \ell(t^{\rho})$.
\end{prop}
\begin{proof}
Recall Corollary~\ref{wMaxDim}, which says that since the number of folds in $\sigma_{a_0}$ equals $\ell(w_0)$, the gallery $\sigma_{a_0}$ achieves the maximum dimension among all labeled folded galleries from $\fa$ to $\fa$ of type $\vec a_0$ with respect to any orientation.  Therefore, Proposition~\ref{ADLVvsGalleries} and Lemma~\ref{lem:dim sigma} say that $X_{a_0}(1) \neq \emptyset$ and $\dim X_{a_0}(1) = \dim_{-\phi_0}(\sigma_{a_0})=\ell(t^\rho)$, as required.
\end{proof}

Now let $x_0 = t^\lambda w_0$ be such that the alcove $\x_0 = x_0 \fa$ is in the shrunken dominant Weyl chamber.  It follows that $\lambda - 2\rho$ is dominant.  Then as $\lambda$, $\rho$, and $\lambda - 2\rho$ are all dominant, we have by Proposition~\ref{w0Construction} that 
$$
\frac{1}{2}\ell(t^\lambda) = \ell(t^\rho) + \frac{1}{2}\ell(t^\lambda) - \frac{1}{2}\ell(t^{2\rho})= \dim X_{a_0}(1) + \frac{1}{2}\ell(t^{\lambda - 2\rho}).
$$
Hence to complete the proof of Theorem~\ref{idDomDimw0}, it suffices to establish the following proposition.

\begin{prop}\label{w0Dim}
Let $x_0 = t^\lambda w_0$ be such that the alcove $\x_0 = x_0\fa$ is in the shrunken dominant Weyl chamber $\Cfs$.  Then $X_{x_0}(1) \neq \emptyset$ and $\dim X_{x_0}(1) = \frac{1}{2}\ell(t^\lambda)$.
\end{prop}

\begin{proof}
We will apply Proposition~\ref{prop:constr-elts1} in the special case that $\bb = t^\mu\fa = \fa$, that is, $\mu$ is the origin.  Since $\lambda - 2\rho$ is dominant, the origin is in the intersection of the fundamental Weyl chamber $\Cf$ with the negative cone based at $\lambda - 2\rho$.  Then by Proposition~\ref{prop:constr-elts1} there exists a gallery $\gamma_0:\fa\rightsquigarrow \fa$ of type $\vec x_0$ which is positively folded with respect to $-\phipartialo$.  Moreover $\gamma_0$ has dimension $\dim_{-\phi_0}(\gamma_0) =  \langle \rho, \lambda \rangle = \frac{1}{2}\ell(t^\lambda)$.  We then have by Theorem~\ref{ADLVvsGal2} and Lemma~\ref{DomADLVDim} that $X_{x_0}(1)$ is nonempty and that $\dim X_{x_0}(1) \geq \frac{1}{2}\ell(t^\lambda)$.  It remains to show that $\gamma$ is a DL-gallery for $X_{x_0}(1)$.  Since $\gamma$ has $\ell(w_0)$ folds and both starts and ends in the alcove $\fa$, this follows from Lemma~\ref{DomADLVDim} again together with Corollary~\ref{wMaxDim}.
\end{proof}


\subsection{Arbitrary spherical directions}\label{sec:Xx1LocallyConstant}

We will now prove Theorem~\ref{idDomDim}, which considers all alcoves in the shrunken dominant Weyl chamber.  The results in this section depend on Theorem 1.1.2 of \cite{GoertzHeDim}, which we  recall below.  Given any $x \in \aW$, let $\eta_1: \aW \rightarrow \sW$ be the natural projection onto the finite part of $x$ and $\eta_2:\aW \rightarrow \sW$ the map which identifies the Weyl chamber $\mathcal{C}_w$ in which the alcove $\x$ lies.  The following first appeared as Conjecture 7.2.2 in \cite{GHKR}.  The original conjecture was based on the observations made by Reuman in \cite{Reu}, and we therefore informally refer to this nonemptiness condition as \emph{Reuman's criterion}.

\begin{thm}[Theorem 1.1.2 of \cite{GoertzHeDim}]\label{idADLV}
Suppose that $x \in \aW$ is such that $x$ lies in the identity component of $G(F)$ and the alcove $\x = x\fa$ lies in one of the shrunken Weyl chambers.  Then $X_x(1) \neq \emptyset$ if and only if
\begin{equation}\label{E:idNonEmpty}
\eta_2(x)^{-1}\eta_1(x)\eta_2(x) \in \sW \backslash \bigcup\limits_{T \subsetneq S} W_T.
\end{equation}
In the case in which $X_x(1) \neq \emptyset$, 
\begin{equation}\label{E:idADLVDim}
\dim X_x(1) = \frac{1}{2}\left( \ell(x)+\ell(\eta_2(x)^{-1}\eta_1(x)\eta_2(x)) \right).
\end{equation}
\end{thm}

We now prove Theorem~\ref{idDomDim}.  In the special case that $a = a_0 = t^{2\rho}w_0$ and $x = x_0 = t^\lambda w_0$ this result follows from Theorem~\ref{idDomDimw0} above.  For other spherical directions, we use Theorem~\ref{idADLV}.

\begin{proof}[Proof of Theorem~\ref{idDomDim}]
Let $a = t^{2\rho}w$, and note that the corresponding alcove $\mathbf{a} = a\fa$ is in the shrunken dominant Weyl chamber $\Cfs$ for any choice of $w \in \sW$.  Now let $x = t^{\lambda}w$ have the same finite part as $a$.  To see that $X_a(1) \neq \emptyset \iff X_x(1) \neq \emptyset,$ note that for a fixed Weyl chamber, the nonemptiness condition \eqref{E:idNonEmpty} of Theorem~\ref{idADLV} only depends on the finite part of $x$.  In particular, in the dominant Weyl chamber, in which $\eta_2(x) = 1$, this condition reduces to $ w \in \sW \backslash \bigcup_{T \subsetneq S} W_T.$  Therefore Theorem~\ref{idADLV} directly implies \eqref{E:idNonEmptya}.

For the equality comparing the dimensions of $X_a(1)$ and $X_x(1)$ in \eqref{E:idDimCalc}, we also use Theorem~\ref{idADLV}.  
Since the alcoves $\x = x\fa$ and $\mathbf{a} = a\fa$ are in the shrunken dominant Weyl chamber, at the vertices $\lambda$ and $2\rho$ respectively, we have $\ell(x) + \ell(w) = \ell(t^\lambda)$ and $\ell(a) + \ell(w) = \ell(t^{2\rho})$.  Together with Equation~\eqref{E:idADLVDim} this implies that
\begin{align}
\dim X_x(1) 
& = \frac{1}{2}\ell(t^\lambda) \\
& =  \frac{1}{2}\ell(t^{2\rho})+ \frac{1}{2}\ell \left( t^{\lambda-2\rho} \right)\\
& =  \dim X_a(1) + \frac{1}{2}\ell \left( t^{\lambda-2\rho} \right),
\end{align}
as desired.
\end{proof}

Theorem~\ref{idADLV} can also be used to show that the varieties $X_x(1)$ have ``locally constant"  dimension, in the sense made precise in the statement of the following result.  

\begin{prop}\label{LocConstant} 
Suppose $x = t^\lambda w$ and $x' = t^\lambda w'$ in  $\aW$ are such that the alcoves $\x = x\fa$ and $\x' = x\fa$ are contained in the shrunken dominant Weyl chamber $\Cfs$ and are based at the same vertex $\lambda \in R^{\vee}$.  If both $X_x(1)$ and $X_{x'}(1)$ are nonempty, then $\dim X_x(1) =\dim X_{x'}(1).$
\end{prop}

\begin{proof}
We have $\eta_1(x) = w$, $\eta_1(x') = w'$, and $\eta_2(x) = \eta_2(x')= 1$.  Assume that both $X_x(1)$ and $X_{x'}(1)$ are nonempty.  Then Theorem~\ref{idADLV} says that 
\begin{equation}
\dim X_x(1)-\dim X_{x'}(1) = \frac{1}{2} \left( (\ell(x) + \ell(w)) - (\ell(x') + \ell(w')) \right).
\end{equation}
Since $x$ and $x'$ share a common vertex $\lambda$ and the alcoves $\x$ and $\x'$ are in the shrunken dominant Weyl chamber, we have $\ell(x) + \ell(w) = \ell(x') + \ell(w') = \ell(t^\lambda)$.  Therefore $\dim X_x(1)-\dim X_{x'}(1) = 0$, and so the dimensions are the same.
\end{proof}


\subsection{Dependence upon Theorem~\ref{idADLV} and comparison with Reuman's criterion}\label{sec:Reuman}

In the case that $x$ is in the shrunken Weyl chambers and $b\in G(F)$ is basic, in particular for $b=1$, both the nonemptiness patterns and dimensions of affine Deligne--Lusztig varieties $X_x(b)$ were already known, having been proved by G\"{o}rtz and He in \cite{GoertzHeDim}.  We already remarked that our proofs in Section~\ref{sec:w0} are entirely constructive and are logically independent of Theorem~\ref{idADLV} from \cite{GoertzHeDim}, but that in Section~\ref{sec:Xx1LocallyConstant} we depend upon Theorem~\ref{idADLV}.  In this section we discuss the relationship between our proofs and Theorem~\ref{idADLV} further, and compare our approach with Reuman's criterion.

We first compare the formulas from Theorem~\ref{idADLV} directly with our conclusions in Propositions~\ref{w0Construction} and \ref{w0Dim} for alcoves in the $w_0$ position.   In the case of $x=a_0 = t^{2\rho}w_0$, the alcove $a_0\fa$ lies in the shrunken fundamental Weyl chamber $\Cfs$, and so $\eta_2(a_0) = 1$.  On the other hand, $a_0$ is in the $w_0$ position, and so $\eta_1(a_0) = w_0$.  Therefore $\eta_2(a_0)^{-1}\eta_1(a_0)\eta_2(a_0) = w_0 \notin W_T$ for any $T \subsetneq S$, which means that $X_{a_0}(1) \neq \emptyset$ by \eqref{E:idNonEmpty}.  In addition, Equation \eqref{E:idADLVDim} predicts that $X_{a_0}(1) = \frac{1}{2}\left( \ell(a_0) + \ell(w_0) \right)$, which matches our conclusion from Proposition~\ref{w0Construction} since $\frac{1}{2}(\ell(a_0) + \ell(w_0)) = \frac{1}{2}\ell(t^{2\rho}) = \ell(t^\rho)$.  More generally, if $x_0 = t^\lambda w_0$ is such that $x_0\fa$ lies in the shrunken fundamental Weyl chamber $\Cfs$, then again $\eta_2(x_0) = 1$ and $\eta_1(x_0) = w_0$ which means that $X_{x_0}(1) \neq \emptyset$ by \eqref{E:idNonEmpty}.  Also Equation \eqref{E:idADLVDim} predicts that $X_{x_0}(1) = \frac{1}{2}\left( \ell(x_0) + \ell(w_0) \right)$, which matches our conclusion from Proposition~\ref{w0Dim}.   

Of course, these conclusions should match, and the previous paragraph shows that we could have obtained both of Theorem~\ref{idDomDimw0} and~\ref{idDomDim} as corollaries of Theorem~\ref{idADLV}.  On the other hand, while the proof of Theorem~\ref{idADLV} in \cite{GoertzHeDim} uses an inductive argument, we highlight a direct relationship between $X_a(1)$ and $X_x(1)$ in Theorems~\ref{idDomDimw0} and~\ref{idDomDim}, and our constructive approach gives a new explanation as to why relationships between the nonemptiness and dimensions of $X_a(1)$ and $X_x(1)$ should exist.  

More precisely, the constructions presented in Section~\ref{sec:Construction} provide a heuristic argument as to why the conditions on nonemptiness given in Reumann's criterion are natural, as follows.  In our constructions, in order for the varieties $X_x(1)$ to be nonempty for all $x = x_0 = t^\lambda w_0$ in the shrunken dominant Weyl chamber, all of the simple root operators have to be available, so that we can obtain a gallery of type $\vec{x}_0$ ending in $\fa$ and so construct a DL-gallery for $X_{x_0}(1)$.  This availability of all root operators in turn depends on the DL-gallery for $X_{a_0}(1)$ having at least one fold in a hyperplane perpendicular to each of the simple roots, which in turn depends on any minimal word for $w_0$ containing each simple reflection at least once.  That is, the key fact is that Reuman's criterion holds for alcoves in the $w_0$ position in $\Cfs$.  Such a connection between Reuman's criterion and availability of root operators has not appeared in the literature before.  Our approach then allows us to  compare the dimensions of $X_{a_0}(1)$ and $X_{x_0}(1)$ directly, since we can track the effect of Gaussent and Littelmann's root operators on dimension.  

Recall from Remark~\ref{rem:LS} that any LS-gallery can be obtained from a minimal gallery by a finite number of applications of root operators.  This, together with the observations above for the $w_0$ position, and other examples we have considered, suggests the following.  

\begin{question}\label{Q:RootOps}  Let $x = t^\lambda w \in \aW$ be such that the alcove $\x = x\fa$ is in the shrunken dominant Weyl chamber and the variety $X_x(1)$ is nonempty. Let $\gamma:\fa \rightsquigarrow \fa$ be a DL-gallery for $X_x(1)$, that is, $\gamma$ is of type some minimal gallery $\gamma_x:\fa \rightsquigarrow \x$ and $\gamma$ is positively folded with respect to some orientation at infinity $\phipartial$.  
\begin{enumerate}
\item[(a)] Can $\gamma$ be obtained from $\gamma_x$ by applying a finite sequence of root operators, for roots which are simple and positive with respect to $\phipartial$?  
\item[(b)] If the answer to (a) is yes, must every root which is simple and positive with respect to $\phipartial$ appear at least once in this sequence of root operators?
\end{enumerate}
\end{question}

\noindent We are abusing terminology here by discussing the application of root operators to alcove-to-alcove galleries; formally, it may be necessary to extend to vertex-to-vertex galleries, apply root operators, and then truncate, as in our proof of Proposition~\ref{prop:constr-elts1}.

In order to prove our Theorem~\ref{idDomDim} without relying upon Theorem~\ref{idADLV}, we would either need to find an alternative proof of the fact that dimensions of the varieties $X_x(1)$ for $x$ in the shrunken dominant Weyl chamber are ``locally constant", or to mimic the constructions in Section~\ref{sec:Construction} for other spherical directions.  We discuss the obstructions to extending our constructions to other spherical directions in Section~\ref{sec:obstructions} below, since similar obstructions arise when considering the affine Deligne--Lusztig varieties $X_x(b)$ for $b$ a pure translation.


\section{The varieties $X_x(1)$ and $X_x(b)$}\label{sec:RelabelingSymmetry}

In this section we complete the proof of Theorem~\ref{Shrunken} in the introduction.  We also prove Theorem~\ref{ForwardShift} as Theorem~\ref{thm:relabel shift}, and prove Theorem~\ref{NonemptyDim} as Theorem~\ref{ArbitraryNonemptyDim}.  
Theorem~\ref{Shrunken} is obtained by combining the main results from Section~\ref{sec:Xx1ShrunkenDominant} with the main results in this section, which are Theorems~\ref{thm:relabel shift},~\ref{ArbitraryNonemptyDim}, ~\ref{bDomDimw0}, and~\ref{bDomDim}.  As in Section~\ref{sec:Xx1ShrunkenDominant}, we provide separate statements for alcoves in the $w_0$ position.  We remark that many of our proofs in this section use the explicit constructions of Section~\ref{sec:Construction}.

We first in Section~\ref{sec:Translations} use ``forward-shifting"  of galleries to establish Theorem~\ref{thm:relabel shift}.  We then apply  Theorem~\ref{thm:relabel shift} to prove Theorem~\ref{ArbitraryNonemptyDim} in Section~\ref{sec:Xxb}.  After this, we restrict attention to the shrunken dominant Weyl chamber $\Cfs$.  We prove Theorem~\ref{bDomDimw0}, which pertains to alcoves in the $w_0$ position in $\Cfs$, in Section~\ref{sec:Xxbw0}. Our proofs of this result are entirely constructive and we obtain an equality concerning dimensions.  
Theorem~\ref{bDomDim}, which gives an upper bound on dimension for arbitrary alcoves in the shrunken Weyl chamber, is established in Section~\ref{sec:DimShrunkenDom}.  
In Section~\ref{sec:obstructions} we describe obstructions we have found to obtaining constructive proofs for general alcoves, and state a question whose answer would be a significant step towards obtaining further constructions.  Finally, in Section~\ref{sec:MV-GL-LS}, we discuss some relationships between galleries, root operators, crystals, and MV-cycles.

\subsection{Forward-shifting galleries}\label{sec:Translations}

The main result in this section is Theorem~\ref{thm:relabel shift}, which involves ``forward-shifting" a gallery by applying a pure translation, and is illustrated in Figure~\ref{fig:relabel shift}.   Instead of shifting galleries in this way, we could have relabeled the apartment.  Theorem~\ref{thm:relabel shift} establishes a nonemptiness implication and a dimension inequality for certain pairs of affine Deligne--Lusztig varieties.  We note that this result applies to arbitrary alcoves $\x = x\fa$.

\begin{figure}[ht]
\begin{center}
\resizebox{\textwidth}{!}
{
\begin{overpic}{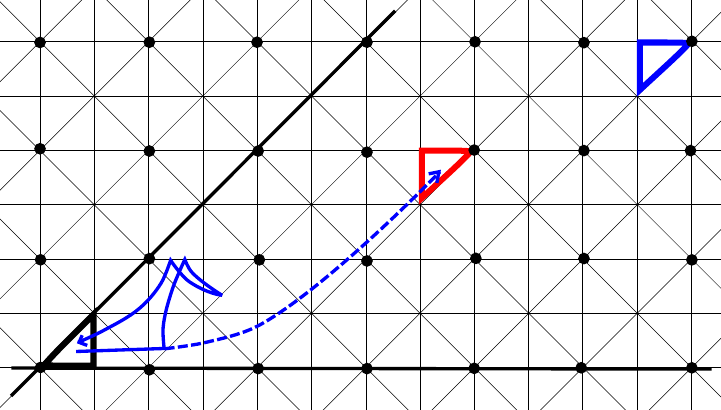}
\put(29,18.5){$\color{blue} \gamma$}
\put(7,10){$\color{black} \fa$}
\put(60,37){$\color{red} \x$}
\put(89,48){\footnotesize{$\color{blue} t^{\mu} \x$}}
\end{overpic}
\hspace{2ex}
\begin{overpic}{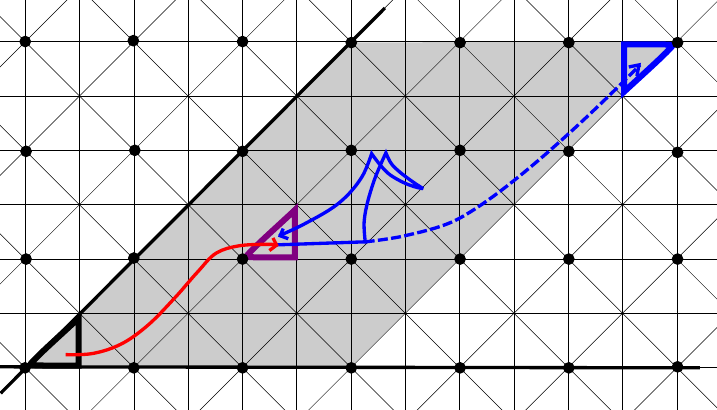}
\put(14.5,10){$\color{red} \sigma''$}
\put(58,32){$\color{blue} \sigma'$}
\put(7.5, 7.5){$\color{black} \fa$}
\put(35,25){\footnotesize{$\color{magenta} \bb$}}
\put(88,48){\footnotesize{$\color{blue} t^{\mu} \x$}}
\end{overpic}

}
\caption{We illustrate the proof of Theorem~\ref{thm:relabel shift}. The DL-gallery $\gamma:\fa \rightsquigarrow \fa$ of type $\vec x$ on the left may be  translated by $t^\mu$ to obtain $\sigma':\bb \rightsquigarrow \bb$ on the right.  The concatenation of  $\sigma'$ with $\sigma'':\fa \rightsquigarrow \bb$ is a  gallery $\sigma:\fa \rightsquigarrow \bb$ of type $\vec z$ where $z = t^\mu x$.  The alcoves in $\conv(\fa, t^\mu \x)$ are shaded light gray on the right. }
\label{fig:relabel shift}
\end{center}
\end{figure}

\begin{thm}\label{thm:relabel shift}
Let $x \in W$ and let $b = t^\mu$ be a pure translation.  Assume that $\bb \in\conv(\fa, t^\mu \x)$.  Then  
\begin{equation}
 X_x(1)\neq \emptyset \Longrightarrow X_{t^\mu x}(t^\mu)\neq\emptyset.
\end{equation}
Moreover, if these varieties are nonempty, then
\begin{equation}\label{E:DimShift}
 \dim X_{t^\mu x}(t^\mu) \geq \dim X_x(1) - \langle \rho_{B^-}, \mu + \mu_{B^-} \rangle
\end{equation}
and if $b = t^\mu$ is dominant then $ \dim X_{t^\mu x}(t^\mu) \geq \dim X_x(1) $.
\end{thm}

\begin{proof}
Suppose $ X_x(1)\neq \emptyset $.  Then Lemma~\ref{DomADLVDim} implies that there is an orientation at infinity $\phipartial$ and a gallery $\gamma:\fa\rightsquigarrow \fa$ of type $\vec x$ that is positively folded with respect to $\phipartial$ and is such that $\dim_\phi(\gamma) = \dim X_x(1)$.  That is, $\gamma:\fa \rightsquigarrow \fa$ is a DL-gallery for $X_x(1)$.  

Now translate $\gamma$ by $t^\mu$ to obtain a gallery $\sigma' = t^\mu\gamma$ of the same type as $\gamma$ that runs from $\bb$ to $\bb$.  By Lemma~\ref{lem:translations dim}, $\dim_\phi(\sigma') = \dim_\phi(\gamma)$.  Let  $\sigma'':\fa\rightsquigarrow \bb$ be minimal.  Since we assumed that $\bb\in\conv(\fa, t^\mu \x)$, the concatenation $\sigma=\sigma''\ast\sigma':\fa \rightsquigarrow \bb = t^\mu\fa$ is a gallery of type $\vec{z}$, where $z = t^\mu x$.  By construction, $\sigma$ is positively folded with respect to $\phipartial$.  Thus by Theorem~\ref{ADLVvsGal2}, the existence of $\sigma$ establishes that $X_{t^\mu x}(t^\mu) \neq \emptyset$.  

We now consider dimension.  Using Theorem~\ref{ADLVvsGal2} again, we obtain that $$\dim X_{t^\mu x}(t^\mu) \geq \dim_\phi(\sigma) - \langle \rho_{B^-}, \mu + \mu_{B^-} \rangle.$$ By Lemma~\ref{DomADLVDim} again, we have that if $b = t^\mu$ is dominant then $\dim X_{t^\mu x}(t^\mu) \geq \dim_\phi(\sigma)$.  

Now recall that for any $b \in W$ and any orientation $\phipartial$ at infinity, the dimension of $b$ with respect to $\phipartial$ is defined by $ \dim_\phi(b)= \dim_\phi(\gamma_b)$ for any minimal gallery $\gamma_b:\fa\rightsquigarrow \bb$, where $\bb = b\fa$.   Since $\sigma'':\fa \rightsquigarrow \bb$ is minimal, for $b = t^\mu$ as in the statement we have that $\dim_\phi(b) = \dim_\phi(\sigma'')$.  Hence
\[
\dim_{\phi}(\sigma)=\dim_{\phi}(\sigma'') + \dim_{\phi}(\sigma')=\dim_{\phi}(b) + \dim_{\phi}(\gamma)
\]
and since $\dim_\phi(b) \geq 0$, it follows that $\dim_\phi(\sigma) \geq \dim_\phi(\gamma)$.  But $\dim_\phi(\gamma) = \dim X_x(1)$ and so the result follows.
\end{proof}

Note that for any $u \in \sW$, $\mu = u\rho$, and arbitrary $\x$ in the Weyl chamber $\Cu$, the convexity condition on $\x$ and $\bb$ in Theorem~\ref{thm:relabel shift}
is satisfied, which leads to the following emptiness criterion for alcoves in the strips. 

\begin{corollary}\label{cor:strip}
Let $x \in W$ be such that $\x \in \Cu\setminus\Cus$. Then  
$X_{t^{u\rho} x}(t^{u\rho})=\emptyset \Longrightarrow  X_x(1)= \emptyset.$
\end{corollary}


\subsection{Nonemptiness and dimension for arbitrary alcoves}\label{sec:Xxb}

This section contains Theorem~\ref{ArbitraryNonemptyDim}, which establishes a nonemptiness implication and a dimension inequality for the varieties $X_x(1)$ and $X_x(b)$, where $x \in \aW$ is arbitrary and $b = t^\mu$ is a pure translation satisfying certain convexity hypotheses.  

\begin{thm}\label{ArbitraryNonemptyDim}
Let $x = t^\lambda w \in \aW$ and let $b = t^\mu$ be a pure translation, with $\mu \in R^\vee$.  Assume that $\bb \in \conv(\fa,\x)$, the alcoves $\x$ and $t^{-\mu}\x$ lie in the same Weyl chamber $\Cu$ where $u \in \sW$, and if $\x$ is in the shrunken Weyl chamber $\Cus$ then $t^{-\mu}\x$ is also in $\Cus$.  Then
\begin{equation}\label{E:1NonEmptyb}
X_x(1) \neq \emptyset \implies X_x(b) \neq \emptyset.
\end{equation}
Moreover, if these varieties are nonempty, then
\begin{equation}
 \dim X_{x}(b) \geq \dim X_x(1) - \langle \rho, \mu^+ \rangle - \langle \rho_{B^-}, \mu + \mu_{B^-} \rangle
\end{equation}
and if $\mu = \mu^+$ is dominant then 
\begin{equation}
 \dim X_{x}(b) \geq \dim X_x(1) - \langle \rho, \mu \rangle.
\end{equation}
\end{thm}

\begin{proof}  We first make the following observations.

\begin{lemma}\label{xby}  Let $x$ and $b = t^\mu$ be as in the statement of Theorem~\ref{ArbitraryNonemptyDim} and let $y = t^{-\mu}x$.  Then $\ell(x) - \ell(y) = \ell(b)$ and $\y \in \conv(\fa,\x)$.
\end{lemma}

\begin{proof}   Since $\bb$ is in $\conv(\fa,\x)$, there exists a minimal presentation say $s_{i_1}\cdots s_{i_k}$ for $x$ that starts with a minimal presentation $s_{i_1}\cdots s_{i_j}$ for $b$.  The remaining subword $s_{i_{j+1}}\cdots s_{i_k}$ of this minimal presentation for $x$ must then be a minimal presentation for $b^{-1}x = t^{-\mu}x = y$.  Hence $$\ell(x) - \ell(y) = k - (k-j) = j = \ell(b).$$  

Now by Lemma~3.5 of~\cite{McCammondPetersen}, since $s_{i_1}\cdots s_{i_k}$ is a product of (simple) reflections, there is also a presentation for $x$ as a product of $k$ reflections which begins  with the minimal presentation $s_{i_{j+1}}\cdots s_{i_k}$ for $y$.  That is, $x = s_{i_{j+1}}\cdots s_{i_k}r_1\cdots r_j$ where $r_1,\ldots,r_j$ are reflections.  Let $z$ be the group element $r_1 \cdots r_j$, in other words $z = y^{-1}x$.  If $\ell(z) \neq j$ we contradict $\ell(x) = k$.  Thus $\ell(z) = j$, and there is a minimal presentation for $z$ containing $j$ elements of $\widetilde S$ such that the concatenation of the minimal presentation $s_{i_{j+1}}\cdots s_{i_k}$ for $y$ with this minimal presentation for $z$ is a minimal presentation for $x$.  It follows that $\y$ is in $\conv(\fa,\x)$.
\end{proof}

We will also use the following proposition, which depends on results from~\cite{GoertzHeDim},~\cite{GHN}, and~\cite{HeAnnals}.

\begin{prop}\label{Xx1Xy1}  Let $x$ and $b = t^\mu$ be as in the statement of Theorem~\ref{ArbitraryNonemptyDim}, and let $y = t^{-\mu}x$.  Then \begin{equation}X_x(1) \neq \emptyset \iff X_y(1) \neq \emptyset.\end{equation}  Moreover, if these varieties are nonempty, then 
\begin{equation}\label{E:DimXx1Xy1}
\dim X_y(1) = \dim X_x(1) - \frac{1}{2}(\ell(x) - \ell(y)).
\end{equation}
\end{prop}

\begin{proof}  Suppose first that $\x$ is in the shrunken Weyl chamber $\Cus$.  Then by assumption $\y = y\fa = t^{-\mu}\x$ is also in $\Cus$.  Since $x$ and $y$ have the same spherical direction $w$, we may apply Reuman's criterion from Theorem~\ref{idADLV} to conclude that $X_x(1) \neq \emptyset$ if and only if $X_y(1) \neq \emptyset$.  For dimension, Theorem~\ref{idADLV} says that if $X_x(1)$ and $X_y(1)$ are nonempty then 
$$ \dim X_x(1) = \frac{1}{2}(\ell(x) + \ell(u^{-1}wu)) \quad \mbox{and} \quad \dim X_y(1) = \frac{1}{2}(\ell(y) + \ell(u^{-1}wu)). $$
Equation~\eqref{E:DimXx1Xy1} follows immediately from this.

Now assume that $\x$ is in $\Cu \setminus \Cus$.  For nonemptiness, we use the language of $P$-alcoves, which was introduced in~\cite{GHKRadlvs}.  We will only need the special case $b = 1$ of Conjecture~1.1.1 in~\cite{GHKRadlvs}.  The full conjecture was proved as Theorem A in~\cite{GHN}.  The relevant statement for us is that for any $z \in \aW$, the variety $X_z(1)$ is nonempty if and only if, for every semistandard parabolic subgroup $P = MN$ for which $\z = z\fa$ is a $P$-alcove, the element $1$ is $\sigma$-conjugate to an element $b' \in M(F)$ and $z$ and $b'$ have the same image under the Kottwitz homomorphism $\eta_M:M(F) \to \Lambda_M$.  We refer the reader to \cite{GHKRadlvs} for definitions of the terms in this statement.  

We are in the situation that $x$ and $y$ differ by a pure translation, hence have the same spherical direction, that $\x \in \Cu \setminus \Cus$, and that $\y \in \conv(\fa,\x)$ by Lemma~\ref{xby}.  It is then not difficult to verify that for any semistandard parabolic $P$, the alcove $\x$ is a $P$-alcove if and only if the alcove $\y$ is a $P$-alcove.  Now if $X_x(1)$ is nonempty and $\x$ is a $P$-alcove for some semistandard $P=MN$, let $b' \in M(F)$ be an element $\sigma$-conjugate to $1$ such that $\eta_M(x) = \eta_M(b')$.  Since we are working in the affine Weyl group $\aW$, we have $\eta_M(x) = \eta_M(y)$,  as both $x$ and $y$ are in the same connected component of the $k$-ind-scheme $M(F)$.  It follows that $X_y(1)$ is also nonempty.  The converse $X_y(1) \neq \emptyset \implies X_x(1) \neq \emptyset$ is similar.

For dimension, we use the special case $b = 1$ of Corollary 12.2 of~\cite{HeAnnals}.   Suppose that $X_x(1)$ and $X_y(1)$ are nonempty.  Then by Corollary 12.2 of~\cite{HeAnnals} we have that $ \dim X_x(1) = d_x(1)$ and $\dim X_y(1) = d_y(1)$, where $d_x(1)$ and $d_y(1)$ are the virtual dimensions, defined at 10.1 of \cite{HeAnnals}.  Now $x$, $y$, and $1$ are all in $\aW$, so these elements have the same image under the Kottwitz map.  Also the identity element $1$ has defect zero and its Newton point is the origin, thus  
\begin{equation}\label{E:virtualdim} \dim X_x(1) = \frac{1}{2}(\ell(x) + \ell(\eta_\sigma(x)) ) \quad \mbox{and} \quad  \dim X_y(1) = \frac{1}{2}(\ell(y) + \ell(\eta_\sigma(y)) ) \end{equation}
where $\eta_\sigma:\aW \to \sW$ is the map defined as follows.  Enumerate the simple roots as $\alpha_1, \ldots, \alpha_n$, and for any $\nu \in P_+$, define $I(\nu) = \{ i \mid \langle \alpha_i, \nu \rangle  = 0 \}$.  Let $W_{I(\nu)} = \langle s_i \mid i \in I(\nu) \rangle \leq \sW$.  Then $\eta_\sigma(z) = \sigma^{-1}(v')v$, where $z = vt^\nu v'$ with $\nu \in P_+$, $v \in \sW$, and $v'$ a minimal length element in the cosets $W_{I(\nu)} \backslash \sW$.   

We claim that $\eta_\sigma(x) = \eta_\sigma(y)$.  Write $x = v_xt^{\nu_x} v'_x$ and $y = v_y t^{\nu_y}v'_y$ where $\nu_x,\nu_y \in P_+$, $v_x,v_y \in \sW$, and $v'_x$ (respectively, $v'_y$) is a minimal length element in the cosets $W_{I(\nu_x)} \backslash \sW$ (respectively, $W_{I(\nu_y)} \backslash \sW$).   By assumption, $\bb \in \conv(\fa,\x)$.  Since $\x$ is in $\Cu \setminus \Cus$ it follows that the pure translation alcove $\bb$ is also in $\Cu \setminus \Cus$.   Thus the vertex $\mu$ lies in a $0$-hyperplane.  Now $x = t^\lambda w$ and $y = t^{-\mu}x = t^{\lambda - \mu}w$ differ by the pure translation $t^\mu$, and since $\mu$ lies in a $0$-hyperplane and $\y \in \conv(\fa,\x)$ it follows that $I(\nu_x) = I(\nu_y)$, $v_x = v_y$, and $v'_x = v'_y$.  Thus $\eta_\sigma(x) = \eta_\sigma(y)$ as claimed.  We then obtain from Equation~\eqref{E:virtualdim} that  $\dim X_y(1) = \dim X_x(1) - \frac{1}{2}(\ell(x) - \ell(y))$ as required.
\end{proof}

To complete the proof of Theorem~\ref{ArbitraryNonemptyDim}, let $y = t^{-\mu}x$ and suppose that $X_x(1) \neq \emptyset$.   Then by Proposition~\ref{Xx1Xy1}, $X_y(1) \neq \emptyset$.  Since $\x = t^\mu \y$ we have $\bb \in \conv(\fa,t^\mu \y)$.  Hence we may apply Theorem~\ref{thm:relabel shift} to the pair $y$ and $b$ to conclude that $X_{t^\mu y}(t^\mu)  = X_{t^\mu t^{-\mu}x}(t^\mu) = X_{x}(b)$ is nonempty, as desired.

For dimension, assume that $X_x(1)$ and $X_x(b)$ are both nonempty.  Combining the statement $\ell(x) - \ell(y) = \ell(b) = \ell(t^\mu)$ from Lemma~\ref{xby} with Equation~\eqref{E:DimXx1Xy1} from Proposition~\ref{Xx1Xy1}, we obtain 
$$  \dim X_y(1) = \dim X_x(1) - \frac{1}{2}\ell(t^{\mu}).$$
Now $\ell(t^\mu) = \ell(t^{\mu^+})$ by Lemma~\ref{lem:length}.  Thus
$$ \dim X_y(1) = \dim X_x(1) - \frac{1}{2}\ell(t^{\mu^+}) = \dim X_x(1) - \langle \rho, \mu^+ \rangle.$$
By the dimension inequality~\eqref{E:DimShift} in Theorem~\ref{thm:relabel shift}, it follows that
$$ \dim X_x(b) \geq \dim X_y(1) - \langle \rho_{B^-}, \mu + \mu_{B^-} \rangle =  \dim X_x(1) - \langle \rho, \mu^+ \rangle - \langle \rho_{B^-}, \mu + \mu_{B^-} \rangle$$
with $$ \dim X_x(b) \geq \dim X_y(1) = \dim X_x(1) - \langle \rho, \mu \rangle$$ if $\mu$ is dominant.  This completes the proof of Theorem~\ref{ArbitraryNonemptyDim}.
\end{proof}


\subsection{The $w_0$ position in the shrunken dominant Weyl chamber}\label{sec:Xxbw0}

In the remainder of Section~\ref{sec:RelabelingSymmetry} we consider $x \in \aW$ such that the alcove $\x = x\fa$ is in the shrunken dominant Weyl chamber $\Cfs$.  The current section is devoted to the proof of the following result, which establishes nonemptiness and an equality concerning dimension for the varieties $X_{x_0}(1)$ and $X_{x_0}(b)$, where $\x_0$ is an alcove in $\Cfs$ with spherical direction $w_0$ and $b$ is a pure translation.

\begin{thm}\label{bDomDimw0}
Let $x_0 = t^{\lambda}w_0$ have spherical direction $w_0$ and be such that the alcove $\x_0 = x_0\fa$ lies in the shrunken dominant Weyl chamber~$\Cfs$.  Suppose that $\mu \in R^\vee$ is in the intersection of $\Cf$ with the negative cone based at the vertex $\lambda - 2\rho$  and let $b = t^\mu$. 
Then 
$$
X_{x_0}(1) \neq \emptyset  \text{  and } X_{x_0}(b) \neq \emptyset
$$
and their dimensions relate via 
\begin{equation}\label{E:bDimCalcw0}
\dim X_{x_0}(b)=\dim X_{x_0}(1) - \langle \rho, \mu \rangle .
\end{equation}
\end{thm}

\begin{proof}  We have by Theorem~\ref{idDomDimw0} that $X_{x_0}(1) \neq \emptyset$.  For nonemptiness and dimension of the variety $X_{x_0}(b)$, we use the following result.

\begin{prop}\label{prop:DLXx0b}  Let $x_0 = t^\lambda w_0$ and $\mu \in R^\vee$ be as in Theorem~\ref{bDomDimw0} and let $b = t^\mu$.  Then the gallery $\gamma_0$ constructed in Proposition~\ref{prop:constr-elts1} is a DL-gallery for $X_{x_0}(b)$, and so $X_{x_0}(b) \neq \emptyset$ and 
\[
\dim X_{x_0}(b) = \langle \rho, \lambda - \mu \rangle.
\]  
Moreover, every DL-gallery for $X_{x_0}(b)$ is positively folded with respect to the orientation $-\phipartialo$.
\end{prop}

\begin{proof}  By Theorem~\ref{ADLVvsGal2} and Lemma~\ref{DomADLVDim}, the existence of $\gamma_0$ shows that $X_{x_0}(b) \neq \emptyset$ and that $\dim X_{x_0}(b) \geq \dim_{-\phi_0}(\gamma_0)$.  Now Proposition~\ref{prop:constr-elts1} says that $\dim_{-\phi_0}(\gamma_0) = \langle \rho, \lambda - \mu \rangle$.  So it suffices to prove that $\dim X_{x_0}(b) \leq \langle \rho, \lambda - \mu \rangle$.

Consider the canonical associated vertex-to-vertex gallery $\gamma_0^\sharp$.  As observed in the proof of Proposition~\ref{prop:constr-elts1}, we have $\dim_{-\phi_0}(\gamma_0) = \dim_{-\phi_0}(\gamma_0^\sharp)$.   Now by the construction given in the proof of Proposition~\ref{prop:constr-elts1} and Lemma~\ref{lem:tauDL}, $\gamma_0^\sharp = w_0 \tau^\sharp$ where $\tau^\sharp \in \Gamma^+(\gamma_{\lambda},-\mu)$ is an LS-gallery.  That is, $\tau^\sharp$ achieves the maximum dimension with respect to the standard orientation among all galleries in $\Gamma^+(\gamma_{\lambda},-\mu)$.  Note also that $\dim_{\phi_0}(\tau^\sharp) =  \langle \rho, \lambda - \mu \rangle$.

Now suppose that $\gamma_w$ is a DL-gallery for $X_{x_0}(b)$, which is positively folded with respect to the orientation $\phipartialwow$ and ends in $b^w\fa = t^{w\mu}\fa$.  By Proposition~\ref{prop:braids} we may assume that $\gamma_w$ is of the same type as $\gamma_0$.  Thus $\gamma_w^\sharp$ is of the same type $\gamma_\lambda$ as $\gamma_0^\sharp$, and so $\gamma_w^\sharp \in \Gamma^+_{w_0w}(\gamma_\lambda, w\mu)$.  Since $\gamma_w$ is positively folded with respect to the orientation $\phi_{w_0w} = -\phi_w$, the gallery $\gamma_w$ is negatively folded with respect to the orientation $\phi_w$, and so $\gamma_w^\sharp \in \Gamma^-_{w}(\gamma_\lambda,w\mu)$.

We now act by first $w^{-1}$ then $w_0$ on the left.  The gallery $w^{-1}\gamma_w^\sharp$ is negatively folded with respect to the standard orientation and ends in $w^{-1}w\mu = \mu$.  Now acting on the left by $w_0$, we obtain that $w_0 w^{-1}\gamma_w^\sharp$ is positively folded with respect to the standard orientation and ends in $w_0\mu = -\mu$.  Since the left-action of any element of $\sW$ is type-preserving, the gallery $w_0w^{-1}\gamma_w^\sharp$ is also of the same  type as  $\gamma_{\lambda}$.  Therefore $w_0w^{-1}\gamma_w^\sharp \in \Gamma^+(\gamma_{\lambda},-\mu)$.  Since $\tau^\sharp$ is an LS-gallery, this implies that $$\dim_{\phi_0}(w_0w^{-1}\gamma_w^\sharp) \leq \dim_{\phi_0}(\tau^\sharp) =  \langle \rho, \lambda - \mu \rangle.$$ Thus by Lemma~\ref{lem:w0 dim}, $\dim_{\phi_{w_0w}}(\gamma_w^\sharp) \leq  \langle \rho, \lambda - \mu \rangle$.  

By Corollary~\ref{cor:folds+crossings} we have $\dim_{\phi_{w_0w}}(\gamma_w) \leq \dim_{\phi_{w_0w}}(\gamma_w^\sharp)$ with equality if and only if $w = \id$.  Since $\gamma_w$ is a DL-gallery for $X_{x_0}(b)$, it follows that
$$\dim X_{x_0}(b) = \dim_{\phi_{w_0w}}(\gamma_w) \leq \langle \rho, \lambda - \mu \rangle.$$
Hence $\gamma_0$ is a DL-gallery for $X_{x_0}(b)$, and so $\dim X_{x_0}(b) = \langle \rho, \lambda - \mu \rangle$.  This proof also shows that $w$ must equal $\id$, that is, that every DL-gallery for $X_{x_0}(b)$ is positively folded with respect to the orientation $-\phipartialo$.
\end{proof}

We now prove Equation~\eqref{E:bDimCalcw0} of Theorem~\ref{bDomDimw0}.

\begin{corollary}   Let $x_0 = t^\lambda w_0$ and $\mu \in R^\vee$ be as in Theorem~\ref{bDomDimw0} and let $b = t^\mu$.   Then
\[
\dim X_{x_0}(b) =\dim X_{x_0}(1) - \langle \rho, \mu \rangle.
\]
\end{corollary}

\begin{proof}  We can apply Proposition~\ref{prop:DLXx0b} to obtain that $\dim X_{x_0}(b) = \langle \rho, \lambda - \mu \rangle$.  By Theorem~\ref{idDomDimw0}, we have $\dim X_{x_0}(1) = \frac{1}{2} \ell(t^\lambda)$.  Now $\frac{1}{2}\ell(t^\lambda) = \langle \rho, \lambda \rangle$, and the result follows.  
\end{proof}

This completes the proof of Theorem~\ref{bDomDimw0}.  As an alternative to using Theorem~\ref{idDomDimw0}, we could have applied Proposition~\ref{prop:DLXx0b} in the special case that $\mu$ is the origin to get that $X_{x_0}(1) \neq \emptyset$ and $\dim X_{x_0}(1) = \langle \rho, \lambda \rangle$.
\end{proof}

\subsection{Dimension in the shrunken dominant Weyl chamber}\label{sec:DimShrunkenDom}

We now establish an upper bound on the dimension of the variety $X_x(b)$ in terms of the dimension of $X_x(1)$, where the alcove $\x = x\fa$ is in the shrunken dominant Weyl chamber and has arbitrary spherical direction, and $b$ is a pure translation.  The main result in this section is Theorem~\ref{bDomDim}, which completes the proof of Theorem~\ref{Shrunken} in the introduction.  

\begin{thm}\label{bDomDim}
Let $x = t^\lambda w \in \aW$ be such that every alcove at the vertex $\lambda$ is contained in the shrunken dominant Weyl chamber $\Cfs$.  Suppose that $\mu \in R^\vee$ lies in the intersection of $\Cf$ with the negative cone based at the dominant vertex $\lambda - 2\rho$, and let $b = t^\mu$.  Then if both of the varieties $X_x(1)$ and $X_x(b)$ are nonempty, their dimensions relate via 
\begin{equation}\label{E:bDimCalc}
\dim X_x(b) \leq \dim X_x(1) - \langle \rho, \mu \rangle.
\end{equation}
\end{thm}

The proof of Theorem~\ref{bDomDim} is contained in the next two statements.  

\begin{prop}\label{lem:LocallyConstant}
Let $x$ and $b = t^\mu$ be as in the statement of Theorem~\ref{bDomDim}.  Let $x_0 = t^\lambda w_0$.  
 If the varieties $X_x(b)$ and $X_{x_0}(b)$ are both nonempty, then
\begin{equation}\label{E:LocallyConstant} 
\dim X_{x}(b)   \leq \dim X_{x_0}(b). 
\end{equation}
\end{prop}

\begin{proof}  Recall that for $v \in \sW$ we denote by $b^v$ the translation $t^{v\mu}$.  

As $b$ is dominant and $X_x(b)$ is nonempty, by Lemma~\ref{DomADLVDim} there is a DL-gallery $\sigma$ for $X_x(b)$ which is positively folded with respect to $\phi_{w_0v}$ and ends at $b^v\fa$, for some $v \in \sW$.  By assumption all alcoves containing $\lambda$ are in $\Cfs$, and so the vertex $\lambda - 2\rho$ is dominant.  We may thus apply Proposition~\ref{prop:constr-elts1} together with Proposition~\ref{prop:DLXx0b} to obtain a DL-gallery $\gamma_0$ for $X_{x_0}(b)$ which is positively folded with respect to the orientation $-\phipartialo$, ends at $b\fa = \bb$, and contains $\ell(w_0)$ folds. 

Since $\sigma$ and $\gamma_0$ are DL-galleries for respectively $X_x(b)$ and $X_{x_0}(b)$, it suffices to compare the dimensions of these galleries.
 From Lemma~\ref{lem:gallery dim} we have that
$$
\dim_{\phi_{w_0v}}(\sigma)=\frac{1}{2}\left[\ell(x)+F_{\phi_{w_0v}}(\sigma) -\ell(b^v)  \right] + \dim_{\phi_{w_0v}}(b^v)
$$
and
$$
\dim_{-\phi_0}(\gamma_0)=\frac{1}{2}\left[\ell(x_0)+F_{-\phi_0}(\gamma_0) -\ell(b)  \right] + \dim_{-\phi_0}(b).
$$
Now as $b$ is dominant, Lemma~\ref{lem:length} implies that $\dim_{-\phi_0}(b) = \dim_{\phi_{w_0v}}(b^v) = 0$.  Also since both $\x$ and $\x_0$ are in $\Cfs$ we get that $k\define\ell(x)-\ell(x_0)=\ell(w_0)-\ell(w)$, and by Lemma~\ref{lem:length} again we have $\ell(b) = \ell(b^v)$.  Hence as $\gamma_0$ has $\ell(w_0)$ folds 
\begin{eqnarray*}
\dim_{\phi_{w_0v}}(\sigma) 
& = &\dim_{-\phi_0}(\gamma_0)- \frac{1}{2}\left(F_{-\phi_0}(\gamma_0)- F_{\phi_{w_0v}}(\sigma)-k\right) \\ 
&= &\dim_{-\phi_0}(\gamma_0)- \frac{1}{2}\left(\ell(w_0)- F_{\phi_{w_0v}}(\sigma)-k\right) \\ 
& = &\dim_{-\phi_0}(\gamma_0)+ \frac{1}{2}\left(F_{\phi_{w_0v}}(\sigma) - \ell(w)\right).
\end{eqnarray*}
By Lemma~\ref{wFolds2}, we have $F_{\phi_{w_0v}}(\sigma) \leq   \ell(w)$, and so we obtain
\[ \dim X_x(b) = \dim_{\phi_{w_0v}}(\sigma) \leq \dim_{-\phi_0}(\gamma_0) = \dim X_{x_0}(b) \]
as desired.
\end{proof}

The following result completes the proof of Theorem~\ref{bDomDim}.

\begin{corollary}\label{cor:bDomDim}  Let $x$ and $b = t^\mu$ be as in the statement of Theorem~\ref{bDomDim}.  If the varieties $X_x(b)$ and $X_{x}(1)$ are both nonempty, then
\[ \dim X_{x}(b) \leq \dim X_{x}(1) - \langle \rho, \mu \rangle. \]
\end{corollary}

\begin{proof}
Let $x_0 = t^\lambda w_0$.  Then Theorem~\ref{bDomDimw0} says that $\dim X_{x_0}(b) = \dim X_{x_0}(1) - \langle \rho, \mu \rangle$, and by Proposition~\ref{LocConstant}, we get $\dim X_{x_0}(1) = \dim X_x(1)$.  
The result then follows from Proposition~\ref{lem:LocallyConstant}.
\end{proof}


\subsection{Obstructions to further constructive proofs}\label{sec:obstructions}

We now discuss several obstructions we have found to extending our constructions in Section~\ref{sec:Construction} to spherical directions besides $w_0$ in the shrunken dominant Weyl chamber $\Cfs$, and so obtaining entirely constructive proofs of several of our main results in Sections~\ref{sec:Xx1ShrunkenDominant} and~\ref{sec:RelabelingSymmetry}.  Our discussion is limited to $\Cfs$ since we expect that overcoming these obstructions in $\Cfs$ will be a necessary first step towards obtaining constructive proofs of all of our results.

In Section~\ref{sec:Construction} we constructed DL-galleries for varieties $X_{x_0}(b)$, where $x_0 = t^\lambda w_0$ is such that the alcove $\x_0 = x_0 \fa$ lies in $\Cfs$, and $b = t^\mu$ is a dominant pure translation satisfying certain convexity hypotheses.  In all of these constructions, we used the orientation $-\phipartialo$, and in fact we proved in Proposition~\ref{prop:DLXx0b} that this is the only orientation possible.  For other spherical directions, we have not been able to find a closed formula for the orientation(s) to use for DL-galleries, even in the special case $X_x(1)$, where we know nonemptiness and dimension by Theorem~\ref{idADLV}.  Thus we pose the following question.

\begin{question}\label{Q:Orientation}  
Let $x = t^\lambda w$ be such that the alcove $\x = x \fa$ lies in the shrunken dominant Weyl chamber $\Cfs$, and let $b = t^\mu$ be a dominant pure translation.  If $X_x(b) \neq \emptyset$, for which $u \in \sW$ does there exist a DL-gallery $\gamma_u:\fa \rightsquigarrow b^u \fa$ for $X_x(b)$ which is positively folded with respect to the orientation $\phipartialwou$?
\end{question}

Even knowing the answer to Question~\ref{Q:Orientation}, in arbitrary type it seems a nontrivial combinatorial problem to provide an explicit construction of any family of DL-galleries.  Indeed, providing a construction of DL-galleries for just the finitely many varieties $X_a(1)$ where $a = t^{2\rho}w$ and $w$ satisfies Reuman's criterion appears very difficult in arbitrary type.

\subsection{Galleries, root operators, crystals, and MV-cycles}\label{sec:MV-GL-LS}
We conclude this section by discussing some relationships between LS-galleries, DL-galleries, root operators, crystals, and MV-cycles.  

In \cite{LittelmannLR}, Littelmann developed an algorithm using Lakshmibai--Seshadri paths, called LS-paths for short,  which gives both character formulas for complex symmetrizable Kac--Moody algebras, and the decomposition of the tensor product of two highest weight representations. These LS-paths (and the root operators defined to manipulate them) took their inspiration from the standard monomial theory of Lakshmibai and Seshadri \cite{LakshmibaiSeshadri}, and can be seen as a generalization of Young tableaux to the setting of symmetrizable Kac--Moody algebras.
Although differently defined, the crystal basis introduced by Kashiwara \cite{Kashiwara} and equivalently the canonical basis of Lusztig \cite{Lusztig}, were later shown to be equivalent to Littelmann's path model.  
Gaussent and Littelmann then developed the theory of LS-paths into a folded gallery model in the standard apartment of the associated Bruhat--Tits building~\cite{GaussentLittelmann}.   The LS-galleries play the same role in \cite{GaussentLittelmann} as Lakshmibai--Seshadri paths do in the path model for crystals; see \cite{LittelmannPaths} and \cite{LittelmannLR}. 

Recall from Remark~\ref{rem:LS} that any LS-gallery can be obtained from a minimal gallery by applying root operators.   We believe that an analogous result should hold for arbitrary DL-galleries, even allowing for varying the orientation, which has the effect of changing the set of simple roots and thus the available root operators (see also Question~\ref{Q:RootOps}, which discusses a special case).  For $x_0 = t^\lambda w_0$ in the shrunken dominant Weyl chamber, we obtain the same crystal structure on the corresponding family of DL-galleries as appears in \cite{GaussentLittelmann}.  We expect that there is a deeper connection to crystals even for other families of DL-galleries.

Furthermore, LS-galleries coincide with the Mirkovi\'{c}--Vilonen cycles of \cite{MirkovicVilonen} inside the affine Grassmannian.  The crucial perspective on various alcove walk models for the study of crystal graphs is that labeled LS-galleries are in bijection with the intersections of unipotent and Iwahori-orbits in both partial and complete affine flag varieties.   In \cite{GaussentLittelmann}, the authors identify the irreducible components of these intersections of unipotent and $G(\mathcal{O})$-orbits with their LS-galleries, and these irreducible components are precisely the MV-cycles of \cite{MirkovicVilonen}.  As such, the alcove walk model for these crystal graphs can be used to index cells in a generalized version of the MV-cycles in the affine Grassmannian.  Deligne--Lusztig galleries play the same role in the affine flag variety as LS-galleries do in the affine Grassmannian, and so we expect that DL-galleries are indexing some generalized family of MV-cycles.


\section{Conjugating to other Weyl chambers}\label{sec:Conjugation}

The aim of this section is to obtain additional statements concerning nonemptiness and dimension in the other (shrunken) Weyl chambers.  The main result is Theorem~\ref{T:Conjugation}, which establishes Theorem~\ref{OtherWeyl} of the introduction.  This relates the affine Deligne--Lusztig variety associated to an element $x$ in the shrunken dominant Weyl chamber to that of any conjugate of $x$ by an element of the finite Weyl group.  

We begin in Section~\ref{sec:ConjGalleries} by recalling some results concerning alcoves and galleries in affine Coxeter complexes, and in Section~\ref{sec:ConjSimple} we convert these into statements about the corresponding affine Deligne--Lusztig varieties.  In Section~\ref{sec:ConjProofs} we consider the special case $b = 1$ in Theorem~\ref{T:Conjugationb=1} and then prove Theorem~\ref{T:Conjugation}.

\subsection{Conjugating galleries}\label{sec:ConjGalleries}

We start by reviewing some facts about alcoves and galleries in an affine Coxeter complex.  The main result in this section is Proposition~\ref{prop:transformation}.

First recall that, given a fixed labeling of the standard apartment $\App$, the panels in $\App$ can be colored by the generators $s_0, s_1, \ldots, s_n \in \tilde{S}$ of the affine Weyl group $\aW$, where $s_1,\dots,s_n$ generate the spherical Weyl group $\sW$.  In Figure~\ref{fig:conjugation}, we have instead colored the vertices of $\App$, and the color of each panel is the color of its opposite vertices.

\begin{figure}[ht]
  \begin{center}
    	\includegraphics[width=0.8\textwidth]{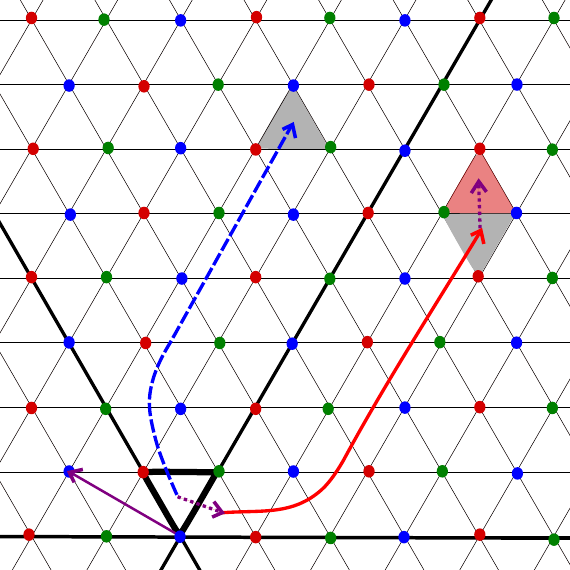}
\caption[conjugation]{Conjugation of a dashed (blue) minimal gallery by a dotted (purple) generator.}

\vspace{2ex}
\label{fig:conjugation}
\end{center}
\end{figure}

For a fixed alcove $\x = x\fa$ and any $y \in \aW$, every (minimal) word $s_{j_1}\cdots s_{j_k}$ for $y$ corresponds to a unique (minimal) combinatorial alcove-to-alcove gallery  
$$
 \gamma_{x,y}= (c_0 \subset p_1 \supset  c_1 \subset \dots \supset p_k \subset c_k)
$$
such that $c_0 = \x = x\fa$, $c_k = xy\fa$, and the alcoves $c_i$ and $c_{i-1}$ intersect in their panel $p_i$ of type $s_{j_i}$. 
That is, we define the first alcove $c_0$ of $\gamma_{x,y}$ to be $\x$ and then, reading the word for $y$ from the left, for $1 \leq i \leq k$ we define the alcove $c_i$ in $\gamma_{x,y}$ to be the unique $s_{j_i}$-neighbor of the alcove $c_{i-1}$.

If we start with a minimal word the gallery $\gamma_{x,y}$ will, by definition, be non-stuttering, that is, $c_i \neq c_{i-1}$ for $1 \leq i \leq k$.  If not mentioned otherwise, we will always choose the fundamental alcove  $\fa=c_0$ as the start alcove of the gallery in this construction. 

Similarly we can associate to $\bf x$ and $y$ a gallery $\gamma_{y,x}$ from $y^{-1}\x$ to $\bf x$ by reading a word for $y$ from the right and carrying out multiplication on the left. Combining left and right multiplication by $s_i \in \tilde{S}$, the geometric effect of conjugation by a generator on a minimal gallery is illustrated in Figure~\ref{fig:conjugation}. Observe that these elongations at the end, respectively at the start, can also be done for stuttering galleries. 

Propositions~\ref{prop:transformation} relies on this elongation procedure and will serve as the model case for Theorem~\ref{T:Conjugation}.  Before we can get started we require two more technical lemmas.

\begin{lemma}\label{le:texlemma2}
Suppose $x = t^\lambda v \in \aW$ is such that the alcove $\x = x\fa$ is in $u\Cf$ with $u\in\sW$.  Then  there exists a minimal gallery $\gamma:\fa \rightsquigarrow \x$ that is the concatenation of two minimal galleries as follows: a gallery connecting $\fa$ with $u\fa$ and a minimal gallery from $u\fa$ to $\x$.  
\end{lemma}

\begin{proof}
Let $ \hat\x = u^{-1}\x$.  Then $\hat \x$ is the unique alcove in the $\sW$-orbit of $\x$ that is contained in $\Cf$. 
Choose a minimal gallery $\gamma':\fa \rightsquigarrow \hat \x$.
Take $\gamma'' := u\gamma'$ to be the image of $\gamma'$ under $u$, that is, the unique gallery of the same type as $\gamma'$ that starts in $u\fa$. Observe that this gallery does not cross any hyperplane that contains the origin. Moreover it connects $u\fa$ with $\x$. 
Let $\tau$ be a minimal gallery from $\fa$ to $u\fa$. Then all hyperplanes crossed by $\tau$ contain the origin.  Thus the concatenation $\tau\ast\gamma''$ is a minimal gallery $\gamma:\fa \rightsquigarrow \x$ as desired.
\end{proof}

\begin{lemma}\label{le:words}
Let $x\in\aW$ and let $s\in \tilde{S}$ be a generator.
\begin{enumerate}
 \item Either $\ell(sxs)=\ell(x)$ or $\ell(sxs)=\ell(x)\pm 2$. 
 \item If $\ell(sxs)=\ell(x) = \ell$ then there exists a minimal presentation $s_{i_1}\cdots s_{i_\ell}$ of $x$ satisfying either $s_{i_\ell}=s$ or $s_{i_1}=s$. In particular we then have either $sxs=ss_{i_1}\cdots s_{i_{\ell-1}}$ or $sxs=s_{i_2}\cdots s_{i_\ell}s$, respectively. 
 \item If $\ell(sxs)=\ell(x)+2$ then no minimal presentation of $x$ starts or ends with $s$. 
 \item If $\ell(sxs)=\ell(x)-2$ then there is a  minimal presentation of $x$ that starts and ends with $s$. 
\end{enumerate}
\end{lemma}
\begin{proof}
 This follows from the deletion condition and Corollary 1.4.6 in \cite{BjoernerBrenti}.
\end{proof}

Proposition~\ref{prop:transformation} will tell us how to transform folded galleries for alcoves in some Weyl chamber $\cC$ to folded galleries corresponding to alcoves in Weyl chambers $s\cC$ where $s \in S$ in such a way that we can control the change of dimension.  Recall that $S = \{s_1,\dots,s_n\}$ is the generating set of simple reflections for the finite Weyl group $\sW$ and that $\tilde S = S \cup \{ s_0 \}$ is the generating set for the affine Weyl group $\aW$.  

\begin{prop}\label{prop:transformation}
Let $x \in \aW$ be such that the alcove $\x = x\fa$ is in $u\Cf$, where $u \in \sW$. Let $s_{i_1}\dots s_{i_\ell}$ be a minimal presentation for $x$, where each $s_{i_j} \in \tilde S$ and $\ell = \ell(x)$ and let $b = t^\mu$ be a pure translation.   Suppose there exists a gallery $\sigma:\fa\rightsquigarrow \bb = b\fa$ of type $s_{i_1}\dots s_{i_\ell}$ that is positively folded and of dimension $d$ with respect to an orientation at infinity $\phipartialw$.

If $s \in S$ is such that $\ell(su) > \ell(u)$, then there exists a gallery $\sigma^{s}:\fa\rightsquigarrow \bb^{s}:= t^{s\mu}\fa$ that is positively folded with respect to the orientation $\phipartialsw$ and has the following properties: 
\begin{enumerate}
\item If $\ell(sxs)=\ell(x)+2$ then $\sigma^{s}$ is of type the minimal presentation $ss_{i_1}\dots s_{i_\ell}s$ for $sxs$, and has dimension $d+1$ with respect to $\phipartialsw$.
\item If $\ell(sxs)=\ell(x)$ the gallery $\sigma^{s}$ is of dimension $d$ with respect to $\phipartialsw$.  Moreover, the minimal presentation for $x$ can be chosen to end with $s_{i_\ell} = s$, and:
\begin{enumerate}
\item If $\sigma$ has a crossing at its last panel (which has type $s$), then $\sigma^{s}$ is of type the minimal presentation $ss_{i_1}\cdots s_{i_{\ell-1}}$ for $sxs$. 
\item If $\sigma$ has a fold at its last panel (which has type $s$), then $\sigma^{s}$ is of type the minimal presentation  $ss_{i_1}\cdots s_{i_\ell}$ for $sx$. 
\end{enumerate}
\end{enumerate}
\end{prop}

\begin{proof}
We may assume, by Proposition~\ref{prop:braids} and Lemma~\ref{le:texlemma2}, that the minimal word $s_{i_1}\cdots s_{i_\ell}$ presenting $x$ starts with a word $s_{i_1}\cdots s_{i_k}$ for $0 \leq k \leq \ell$ presenting $u\in\sW$. In case $u=\id$ the index $k=0$ and the first letter of the word equals $s_{i_1}=s_0$. 
As $\ell(su)>\ell(u)$, no minimal presentation of $x$ starts with $s$.  Hence, $\ell(sxs)\neq \ell(x)-2$. By Lemma~\ref{le:words}, the remaining two cases are $\ell(sxs)=\ell(x)+2$ and $\ell(sxs)=\ell(x)$.  

\begin{figure}[ht]
\begin{center}
\begin{overpic}[width=0.8\textwidth]{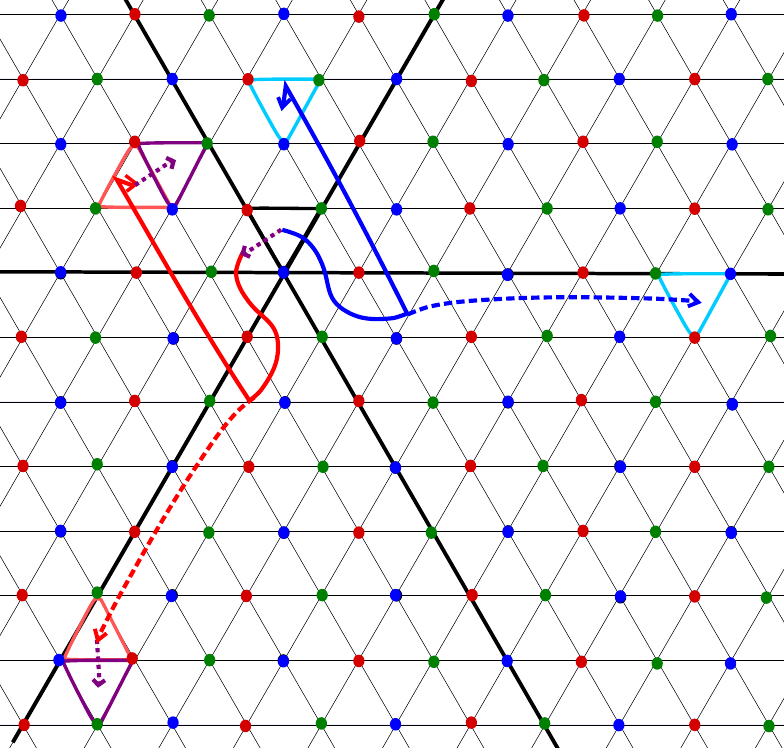}
\put(89,58){\footnotesize{$\color{blue} \x$}}
\put(45,70){\footnotesize{$\color{blue} \sigma$}}
\put(33,80){\footnotesize{$\color{blue} \bb$}}
\put(16.5,13){\footnotesize{$\color{red} s\x$}}
\put(14.9,4.5){\footnotesize{$\color{purple} \x^s$}}
\put(14.5,62){\footnotesize{$\color{red} \sigma^s$}}
\put(24.6,70){\footnotesize{$\color{purple} \bb^s$}}
\put(10,73){\footnotesize{$\color{red} s\bb$}}
\end{overpic}

\caption{Conjugating a gallery $\sigma:\fa \rightsquigarrow \bb$ of type $\vec x$ by a finite simple reflection $s$ such that $\ell(sxs) = \ell(x)+2$ produces a gallery $\sigma^s:\fa \rightsquigarrow \bb^s$ of type $\vec{s}\vec{x}\vec{s}$ with dimension of the gallery increased by 1.} \label{fig:ConjugatedGallery}
\end{center}
\end{figure}

Suppose first that $\ell(sxs)=\ell(x)+2$. Then the word $s s_{i_1}\cdots s_{i_\ell}s$ is minimal, and in particular, $s_{i_1}\neq s$ and $s_{i_\ell} \neq s$.  We may thus obtain a gallery $\sigma'$ of type the minimal word $s s_{i_1}\cdots s_{i_\ell}s$ by concatenating the gallery from $\fa$ to $s\fa$ which is just the crossing in the $s$-panel of $\fa$, with the image $s\sigma:s\fa \rightsquigarrow s\bb$ of $\sigma$ under the left-action of $s$, and then the gallery from $s\bb$ to its unique $s$-neighbor, which is $\bb^s = sbs\fa = t^{s\mu}\fa$ itself.  Thus conjugating $\sigma$ by $s$ we are left with a gallery $\sigma'$ from $\fa$ to $\bb^{s}$ that is of type $s s_{i_1}\cdots s_{i_\ell}s$.

Passing from $\sigma$ to $\sigma'$ no new folds are added, and the subgallery $s\sigma$ is positively folded with respect to the orientation~$s\phipartialw = \phipartialsw$.  Now by construction both extensions of $s\sigma$ are crossings in an $s$-panel.  One of these crossings goes from $\fa$ to $s\fa$ while the other goes from $s\bb$ back to the pure translation alcove $\bb^{s}$.  In the case of $\bb = \fa$ these crossings are in the exact same $s$-panel, and in general, these crossings are in parallel walls since $\fa$ and $\bb$ are both pure translations.  Therefore, one of these crossings is negative with respect to the orientation $\phipartialsw$, and one is positive. Hence $\sigma'$ is positively folded with respect to $\phipartialsw$ and of dimension $d+1$; see Figure~\ref{fig:ConjugatedGallery}. Put $\sigma^{s}\define \sigma'$ and we are done with case (1).   

\begin{figure}[ht]
\begin{center}
\begin{overpic}[width=0.8\textwidth]{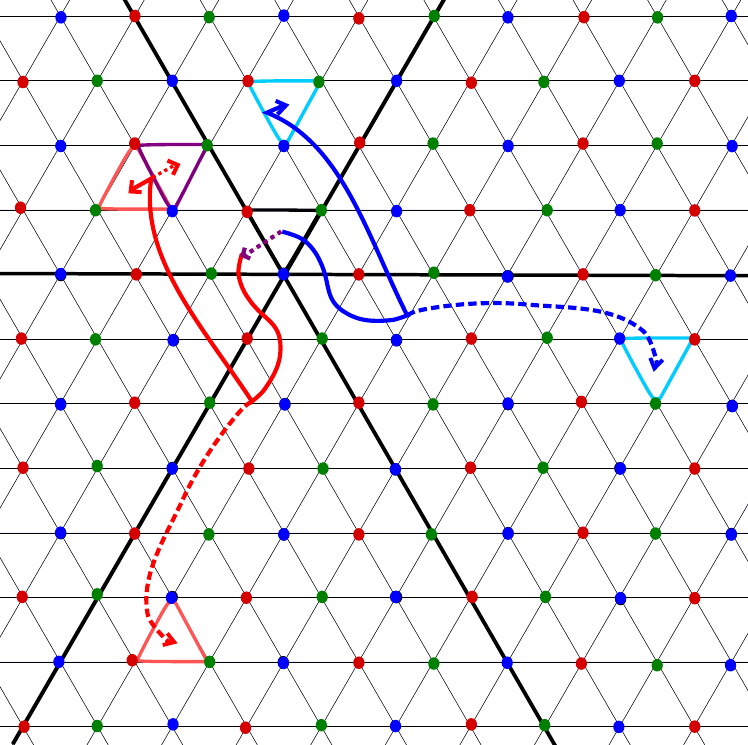}
\put(89,50){$\color{blue} \x$}
\put(48,73){$\color{blue} \sigma$}
\put(36,87){$\color{blue} \bb$}
\put(12,76.5){$\color{red} s\bb$}
\put(27,14.5){$\color{red} s\x$}
\put(21,57){$\color{red} \sigma^s$}
\put(21.5,81){$\color{purple} \bb^s$}
\end{overpic}
\caption{Conjugating a gallery $\sigma:\fa \rightsquigarrow \bb$ by a finite simple reflection $s$ such that $\ell(su) > \ell(u)$ and $\ell(sxs) = \ell(x)$ results in a gallery $\sigma^s:\fa \rightsquigarrow \bb^s$ of the same dimension. }
\label{fig:ConjugatedGallery2}
\end{center}
\end{figure}
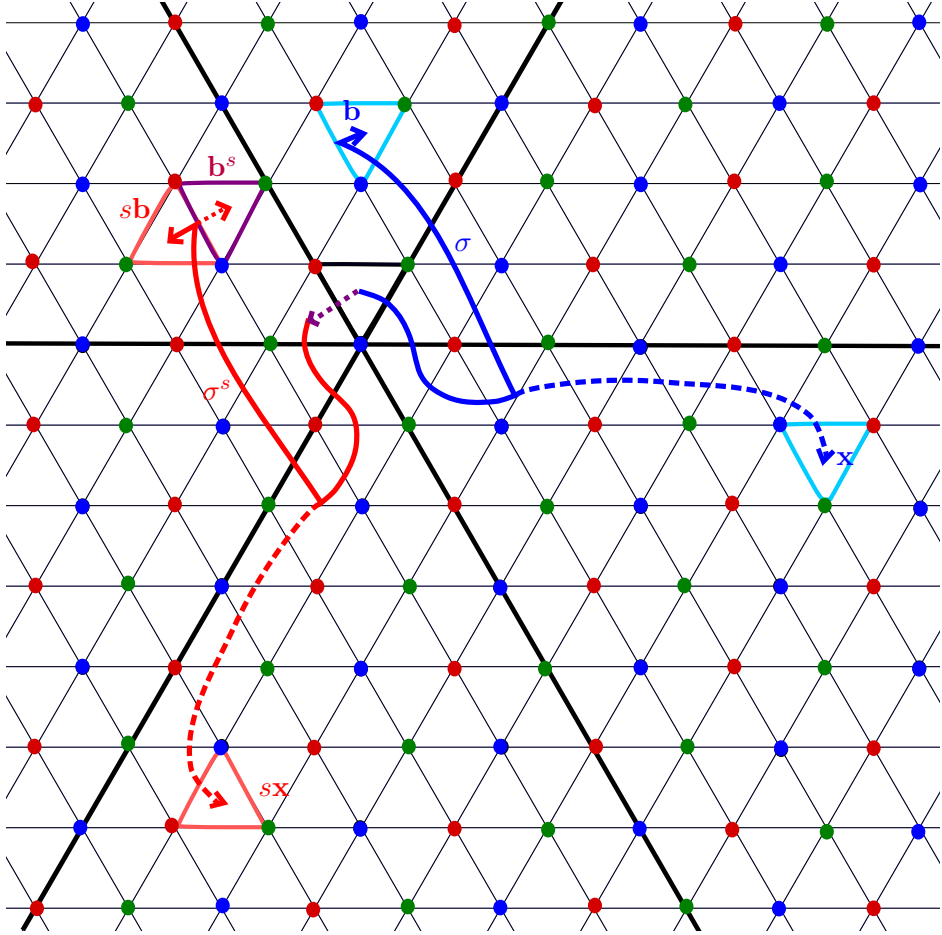

Suppose now that we are in case (2), so that $\ell(sxs)=\ell(x)$. Then since $\ell(su) > \ell(u)$ and we are assuming that the minimal presentation of $x$ begins with a word for $u$, there exists a presentation of $x$ with final letter $s_{i_\ell} = s$.  We may assume without loss of generality by Proposition~\ref{prop:braids} that $\sigma$ is of this type. Now the last panel in $\sigma$ is a face of $\bb$ of type $s$.
There are two cases for $\sigma$: at this last panel there is either a (positive or negative) crossing or a positive fold with respect to $\phipartialw$. 

In case (2)(a), the gallery $\sigma$ has a crossing at its last panel.  Apply $s$ on the left to $\sigma$ to obtain a gallery $s\sigma:s\fa \rightsquigarrow s\bb = sb\fa$ of the same type as $\sigma$, and which is positively folded with respect to the orientation $\phipartialsw$.  Next, delete the last crossing of $s\sigma$ to obtain $\sigma':s\fa\rightsquigarrow sbs\fa = \bb^s$ of type $s_{i_1}\cdots s_{i_{\ell-1}}$. The dimension of $\sigma'$ is $d$ if the deleted crossing was negative and is $d-1$ if the deleted crossing was positive. Now elongate $\sigma'$ at the start by adding a crossing from $\fa$ to the alcove $s\fa$.  We then obtain a gallery $\sigma'':\fa\rightsquigarrow \bb^s$ of type the minimal presentation $ss_{i_1}\cdots s_{i_{\ell-1}}$ for $sxs$ and which is positively folded with respect to $\phipartialsw$.  As we deleted a crossing from the pure translation alcove $\bb^s = sbs\fa$ to its $s$-neighbor $sb\fa$, and we added back a crossing from the fundamental alcove $\fa$ to its $s$-neighbor $s\fa$, the dimension of $\sigma''$ is again $d$.   Put $\sigma^s = \sigma''$ and this completes case (2)(a).

For case (2)(b), we have a positive fold at the last panel of $\sigma$, and this panel is of type $s$.  We may unfold this fold to obtain a gallery $\sigma':\fa\rightsquigarrow bs\fa$ of the same type as $\sigma$ but whose dimension has changed by $-1$ as we removed a positive fold and replaced it by a negative crossing.  Now elongate $\sigma'$ at the start by a crossing from the alcove $s\fa$ to $\fa$. This elongation adds a positive crossing to the gallery (with respect to $\phipartialw$), resulting in a gallery of type $ss_{i_1}\cdots s_{i_{\ell}}$ that starts in $s\fa$ and ends in $bs\fa$. The image of this gallery under the left-action of $s$ is the desired gallery $\sigma^{s}:\fa \rightsquigarrow \bb^s$.  See Figure \ref{fig:ConjugatedGallery2} for an illustration.  
\end{proof}

The reverse statement of Proposition~\ref{prop:transformation} is contained in Proposition~\ref{prop:transformation-reverse}. 
As we have not been able to deduce a clean statement about nonemptiness of affine Deligne-Lusztig varieties from \ref{prop:transformation-reverse}  we leave its proof, which can be done using similar cutting and unfolding techniques to the proof of Proposition~\ref{prop:transformation}, to the reader. 
  
\begin{prop}\label{prop:transformation-reverse} 
Let $x \in \aW$ be such that $\x = x\fa$ is in $u\Cf$, where $u \in \sW$, and let $s_{i_1}\cdots s_{i_\ell}$ be a minimal presentation for $x$, where each $s_{i_j} \in \tilde S$ and $\ell = \ell(x)$.  Let $b = t^\mu$ be a pure translation.   Suppose there exists a gallery $\sigma:\fa\rightsquigarrow \bb = b\fa$ of type $s_{i_1}\cdots s_{i_\ell}$ that is positively folded and of dimension $d$ with respect to an orientation at infinity $\phipartialw$. 

If $s \in S$ is such that $\ell(su) < \ell(u)$, then there exists a gallery $\sigma^{s}:\fa\rightsquigarrow \bb^{s}:= t^{s\mu}\fa$ that is positively folded with respect to the orientation $\phipartialsw$ and has the following properties: 
\begin{enumerate}
\item If $\ell(sxs)=\ell(x)-2$, then $\sigma^{s}$ is of dimension $d-1$.  Moreover, the minimal presentation for $x$ can be chosen to start with $s_{i_1} = s$ and end with $s_{i_\ell} = s$, and: 
\begin{enumerate}
\item If $\sigma$ has a crossing at its last and first panel, then $\sigma^{s}$ is of type the minimal presentation $s_{i_{2}}\cdots s_{i_{\ell-1}}$ for $sxs$.
\item If $\sigma$ has a crossing at its last panel and a fold at its first, then $\sigma^{s}$ is of type the minimal presentation $s_{i_{1}}\cdots s_{i_{\ell-1}}$ for $sxs$.
\item If $\sigma$ has a fold at its last panel and a crossing at its first, then $\sigma^{s}$ is of type the minimal word $s_{i_{2}}\cdots s_{i_\ell}$ for $sx$.
\item If $\sigma$ has a fold at its last and first panel, then $\sigma^{s}$ is of type the minimal word $s_{i_{1}}\cdots s_{i_\ell}$ for $sx$.
\end{enumerate}
\item If $\ell(sxs)=\ell(x)$, then the gallery $\sigma^{s}$ is of dimension $d$.  Moreover, the minimal presentation for $x$ can be chosen to start with $s_{i_1} = s$, and  
\begin{enumerate}
\item If $\sigma$ has a crossing at its first panel (which has type $s$), then $\sigma^s$ is of type the minimal presentation  $s_{i_2}\cdots s_{i_\ell}s$ for $sxs$. 
\item If $\sigma$ has a fold at its first panel (which has type $s$), then $\sigma^s$ is of type the minimal presentation  $s_{i_1}\cdots s_{i_\ell}s$ for $xs$. 
\end{enumerate}
\end{enumerate}
\end{prop}


\subsection{Conjugating by simple reflections}\label{sec:ConjSimple}

We will now translate the geometric statements of the previous section into algebro-geometric results for the corresponding affine Deligne--Lusztig varieties.  The main results in this section are Propositions~\ref{T:ConjugateXx1} and~\ref{prop:bConjugateOneDirection}, which consider conjugates of $x$ by finite simple reflections $s \in S$, and so relate alcoves in Weyl chambers.  Proposition~\ref{T:ConjugateXx1} considers the case $b = 1$ and Proposition~\ref{prop:bConjugateOneDirection} considers general translations $b = t^\mu$.  For the proof of Proposition~\ref{prop:bConjugateOneDirection}, we will need a result from~\cite{GoertzHeDim}, which we review as Theorem~\ref{T:GoertzHe}. 

We first consider the case $b = 1$.  The proof of the following result relies on Theorem~\ref{idADLV}.

\begin{prop}\label{T:ConjugateXx1}
Let $x$ be an alcove in $u\Cfs$, and write $x = ut^{\lambda}w$, where $t^\lambda w$ lies in the shrunken dominant chamber $\Cfs$ and $u,w \in \sW$.   
Then for any $s \in S$ such that $\ell(su) > \ell(u)$,
$$
X_x(1)\neq \emptyset  \Longleftrightarrow X_{sxs}(1)\neq \emptyset.
$$ 
If both varieties are nonempty, then their dimensions relate as follows:  
\begin{enumerate}
\item if $\ell(sxs)=\ell(x)+2$, then $\dim X_{sxs}(1)=\dim X_x(1)+1$;  and 
\item if $\ell(sxs)=\ell(x)$, then $\dim X_{sxs}(1)=\dim X_x(1)$.
\end{enumerate}
\end{prop}

\begin{proof}
Given any $x \in \aW$, recall that $\eta_1: \aW \rightarrow \sW$ is the natural projection onto the finite part of $x$ and $\eta_2:\aW \rightarrow \sW$ is the map which identifies the Weyl chamber in which the alcove $\x$ lies.  By Reuman's criterion, $X_x(1) \neq \emptyset$ if and only if $\eta_2(x)^{-1}\eta_1(x) \eta_2(x)$ has full support in $S$.  If $x = ut^{\lambda}w$ with $t^{\lambda}w$ in the dominant chamber, then $\eta_1(x) = uw$ and $\eta_2(x) =  u$.  Therefore, $X_x(1) \neq \emptyset$ if and only if $u^{-1}(uw)u = wu \in \sW \backslash \bigcup\limits_{T \subsetneq S} W_T$.

Now consider $sxs = sut^{\lambda}ws$.  Then $\eta_1(sxs) = suws$ and $\eta_2(sxs) = su$ so that \[\eta_2(sxs)^{-1}\eta_1(sxs)\eta_2(sxs) = (su)^{-1}(suws)(su) = wu,\] which is precisely the same as for $x$.  Therefore, $X_x(1) \neq \emptyset $ if and only if $X_{sxs}(1) \neq \emptyset$ by Reuman's condition in Theorem \ref{idADLV}.

Recall from \eqref{E:idADLVDim} that $\dim X_x(1) = \frac{1}{2}\left(\ell(x) + \ell(\eta_2(x)^{-1}\eta_1(x)\eta_2(x)\right) = \frac{1}{2}\left( \ell(x) + \ell(wu) \right)$, which means that $\dim X_{sxs}(1) = \frac{1}{2} \left(\ell(sxs) + \ell(wu)\right)$ by the previous discussion.  Since by hypothesis $\ell(su) > \ell(u)$, the case that $\ell(sxs) = \ell(x) -2$ cannot arise.  Clearly, if $\ell(sxs) = \ell(x)$, then these dimension formulas yield the same result.  On the other hand, if $\ell(sxs) = \ell(x) + 2$, then $\dim X_{sxs}(1) = \frac{1}{2}\left( \ell(x) + 2 + \ell(wu)\right) = \dim X_x(1) + 1$.
\end{proof}

To treat the case in which $b\neq 1$, we need the following basic algebraic fact, which uses that affine Deligne-Lusztig varieties associated to $\sigma$-conjugate elements are isomorphic.

\begin{lemma}\label{T:FrobConjIso}
Let $t^\lambda$ be a pure translation in $\aW$ and let $w \in \sW$.  Then for all $x \in \aW$ 
\begin{equation} X_x(t^\lambda) \neq \emptyset \iff X_x(t^{w\lambda}) \neq \emptyset,
\end{equation}
and moreover 
\begin{equation}
\dim X_x(t^\lambda)  = \dim X_x(t^{w\lambda}).
\end{equation}
\end{lemma}

 \begin{proof}
 For $y \in G(F)$, the isomorphism $g \mapsto yg$ on $G(F)$ yields an isomorphism between $X_x(b)$ and $X_x(yb\sigma(y)^{-1})$.  All elements in the $\sW$-orbit of $t^{\lambda}$ are conjugate to $t^{\lambda}$ by an element of $\sW$.  Since elements in $\sW$ are $\sigma$-fixed, the elements $t^{\lambda}$ and $t^{w\lambda} = wt^\lambda w^{-1}$ are also $\sigma$-conjugate, and the result follows.
 \end{proof}

We next review a result of G\"{o}rtz and He about the structure of affine Deligne--Lusztig varieties associated to affine Weyl group elements which are related by conjugation by simple reflections.  The following theorem is obtained by a generalization of the reduction method of Deligne and Lusztig in \cite{DL}.

\begin{thm}[Corollary 2.5.3 in \cite{GoertzHeDim}]\label{T:GoertzHe}
Let $x \in \aW$ and $s \in \tilde{S}$ be an affine simple reflection, and $b \in G(F)$.
\begin{enumerate}
\item If $\ell(sxs) = \ell(x)$, then there exists a universal homeomorphism $X_x(b) \rightarrow X_{sxs}(b)$.
\item If $\ell(sxs) = \ell(x)-2$, then $X_x(b)$ can be written as a disjoint union $X_x(b) = X_1 \sqcup X_2$, where $X_1 \rightarrow X_{sxs}(b)$ and $X_2 \rightarrow X_{sx}(b)$ are compositions of a Zariski-locally trivial fiber bundle with one dimensional fibers and a universal homeomorphism.  Here, $X_1$ is closed and $X_2$ is open in $X_x(b)$.
\end{enumerate}
\end{thm}

We now apply Theorem~\ref{T:GoertzHe} to obtain an analog of Proposition~\ref{T:ConjugateXx1} for arbitrary pure translations $b = t^\mu$.

\begin{prop}\label{prop:bConjugateOneDirection}
Let $\x$ be an alcove in $u\Cf$ where $u \in \sW$, and let $s \in S$ be such that $\ell(su)>\ell(u)$.  Let $b = t^{\mu}$ be a pure translation with $\mu \in R^{\vee}$.  Then 
$$
X_x(b)\neq \emptyset  \Longrightarrow X_{sxs}(b)\neq \emptyset,
$$ 
and if both varieties are nonempty, then the dimensions relate as follows:  
\begin{enumerate}
\item if $\ell(sxs)=\ell(x)+2$, then $\dim X_{sxs}(b) \geq \dim X_x(b)+1$;  and 
\item if $\ell(sxs)=\ell(x)$, then $\dim X_{sxs}(b) \geq \dim X_x(b)$.
\end{enumerate}
\end{prop}

\begin{proof}  Suppose that $X_x(b) \neq \emptyset$.  By Lemma~\ref{T:FrobConjIso} we may assume without loss of generality that $b$ is dominant.  Then by Theorem~\ref{ADLVvsGal2} and Lemma~\ref{DomADLVDim}, there exists an orientation at infinity $\phipartialwow$ and a gallery $\sigma: \fa \rightsquigarrow b^w\fa$ of type $\vec{x}$ 
which is positively folded with respect to the orientation $\phipartialwow$, such that $\dim X_x(b) = \dim_{\phi_{w_0w}}(\sigma)$.  That is, $\sigma$ is a DL-gallery for $X_x(b)$.

By the same arguments as in Proposition~\ref{prop:transformation} the hypothesis that $\ell(su)>\ell(u)$ implies that either $\ell(sxs) = \ell(x) + 2$ or $\ell(sxs) = \ell(x)$.  Assume first that $\ell(sxs) = \ell(x)+2$.  Then Proposition~\ref{prop:transformation}(1) implies that the gallery $\sigma^s: \fa \rightsquigarrow sb^ws\fa = b^{sw}\fa = \bb^{sw}$ is positively folded and of dimension $\dim_{\phi_{w_0w}}(\sigma) + 1$ with respect to the orientation $\phipartialswow$.   Also, $\sigma^s$ is of type a minimal presentation $s s_{i_1} \cdots s_{i_\ell} s$ for $sxs$, where $s_{i_1} \cdots s_{i_\ell}$ is a minimal presentation for $x$.

Now $\sigma^s$ is positively folded with respect to the orientation associated to the element $sw_0w$ of $\sW$, but the gallery $\sigma^s$ runs from $\fa$ to $\bb^{sw}$.  Writing $sw_0w = (w_0sw)(w^{-1}sw_0sw_0w)$, we then have that $\sigma^s$ is positively folded with respect to the orientation associated to the element $w_0swv^{-1}$ where $v^{-1}=w^{-1}sw_0sw_0w$, and that $\sigma^s$ ends in $(b^v)^{swv^{-1}}\fa$.  Thus Theorem~\ref{ADLVvsGal2} implies that $X_{sxs}(b^v) \neq \emptyset$ and that $$\dim X_{sxs}(b^v) \geq \dim_{\phi_{sw_0w}}(\sigma^s) - \langle \rho_{B^-}, v\mu + (v\mu)_{B^-} \rangle.$$ We assumed that $b$ is dominant, equivalently, $\mu$ is dominant, and so $(v\mu)_{B^-} = -\mu$ is less than or equal to $v\mu$ in the dominance order for all $v \in \sW$.  Thus $\langle \rho_{B^-}, v\mu + (v\mu)_{B^-} \rangle \leq 0$ and so $$\dim X_{sxs}(b^v) \geq \dim_{\phi_{sw_0w}}(\sigma^s) =  \dim_{\phi_{w_0w}}(\sigma) + 1 = \dim X_x(b) + 1.$$   Lemma~\ref{T:FrobConjIso} then says that $X_{sxs}(b) \neq \emptyset$ and that $\dim X_{sxs}(b) = \dim X_{sxs}(b^v)$.  Therefore $\dim X_{sxs}(b) \geq \dim X_x(b) + 1$, proving (1).

Now assume that $\ell(sxs) = \ell(x)$.  Then Theorem~\ref{ADLVvsGal2} and Proposition~\ref{prop:transformation}(2) imply that the gallery $\sigma^s: \fa \rightsquigarrow \bb^{sw}$ is positively folded with respect to the orientation $\phipartialswow$.  In addition, $\sigma^s$ has either type $\vec s\vec x\vec s$ or $\vec s \vec x$, according to whether or not $\sigma$ has a crossing or fold at its last panel, respectively.  In the case in which $\sigma$ has a crossing at its last panel, by the same arguments as in the previous paragraph we have that $X_{sxs}(b^v) \neq \emptyset$  and $\dim X_{sxs}(b^v) \geq \dim X_x(b)$, where as before $v^{-1} = w^{-1}sw_0sw_0w$.  Therefore, applying Lemma~\ref{T:FrobConjIso} resolves (2) of Proposition~\ref{prop:bConjugateOneDirection} in this case.  

In the case in which $\sigma$ has a fold at its last panel, we appeal to Theorem \ref{T:GoertzHe} in order to say that the varieties $X_x(b)$ and $X_{sxs}(b)$ are isomorphic.  Therefore, clearly they are nonempty and empty together and, if nonempty, have the same dimension.
\end{proof}

Observe that we used Theorem~\ref{T:GoertzHe} only in the last step of the proof, which is the case of a gallery with a fold at its last panel. To obtain an independent and purely constructive proof it would be enough to be able to deal with this case geometrically.


\subsection{Conjugate affine Deligne--Lusztig varieties}\label{sec:ConjProofs}

In this section we establish Theorems~\ref{T:Conjugationb=1} and~\ref{T:Conjugation}, which consider conjugates of $x$ by arbitrary elements $u \in \sW$, for $b = 1$ and general translations $b = t^\mu$, respectively.  

In passing to Weyl chambers that are not adjacent, the book-keeping of the dimension gets slightly more complicated. We thus  introduce a correction term that depends on the spherical part $w$ of the alcove $\x\in\Cf$ and the label $u$ of the new Weyl chamber $u\Cf$. Ultimately, this correction term will describe how the dimensions of $X_x(1)$ and $X_{u^{-1}xu}(1)$ are related for any $u \in \sW$.

\begin{definition}\label{def:correctionterm}
Let $u,w \in\sW$ and suppose $u$ has a reduced expression $u=s_{i_1}\cdots  s_{i_k}$ with generators $s_{i_j}\in S$ corresponding to simple roots $\alpha_{i_j}$. Denote by $u^j\define s_{i_1}\cdot\ldots\cdot s_{i_j} $, where we put $u^0 = \id$. That is, $u^j$ is the product of the first $j$ letters of the word presenting $u$. Define the \emph{correction term} $c(w,u)$ by
$$
c(w,u)=\sum_{j=1}^k \max\left\{0,\ell(wu^{j-1})-\ell(wu^j)\right\}.
$$
\end{definition}

\begin{remark}\label{R:CorrectionTerm}
We can make several immediate observations about the definition of this correction term.  First, the term $\ell(wu^{j-1})-\ell(wu^j)$ equals $1$ if the element $wu^j \in\sW$ is shorter than $wu^{j-1}$, and it equals $-1$ if $wu^j \in\sW$ is longer than $wu^{j-1}$.  These are the only two possibilities; in particular, equality never happens.  Note also that $wu^j$ is shorter than $wu^{j-1}$ if and only if the crossing from the alcove $wu^{j-1}$ to the alcove $wu^{j}$ is a positive crossing with respect to the standard orientation.  Since any minimal gallery from $w$ to $wu$ crosses the same set of hyperplanes in the same directions, it follows that $c(w,u)$ does not depend on the choice of minimal presentation for $u$.  By definition $c(w,u)$ is non-negative for all $u, w \in \sW$ and nonconstant on each set of alcoves based at a same vertex in $R^\vee$.
\end{remark}

We will need the following technical lemma.  Recall that the \emph{star} of a vertex $v$ in a simplicial complex is the union of all simplices containing $v$. 

\begin{lemma}\label{le:texlemma} 
Suppose $x = t^\lambda w \in \aW$ is such that the alcove $\x = x\fa$ is in $\Cfs$. Let $\gamma$ be a minimal gallery $\gamma:\fa\rightsquigarrow \x$ of type $\vec{x}$.
For a fixed minimal word $\vec{u}$ representing some $u\in\sW$ let $\sigma_u$ be the unique gallery of type $\vec{u}$ that runs from $x\fa$ to $xu\fa$ in the star of the vertex $\lambda$.
Then:
\begin{enumerate}
 \item\label{item1} $c(w,u)$ equals the number of positive crossings of $\sigma_u$ with respect to the standard orientation;
 \item\label{item2} if $l(wu^{j-1})-l(wu^j)=1$ for all $j$, that is, $c(w,u)=\ell(u)$, then $\gamma\star\sigma_u$ is a minimal gallery from $\fa$ to $xu\fa$ of type $\vec{x}\vec{u}$; and
 \item\label{item3} if $l(wu^{j-1})-l(wu^j)=-1$ for some $j$ then $\gamma\star\sigma_u$ is a stuttering gallery. It can be shortened within the star of $\lambda$ if $\gamma$ runs via $t^\lambda w_0$. 
\end{enumerate}
\end{lemma}

\begin{proof} 
As noticed in Remark~\ref{R:CorrectionTerm} above, for $1 \leq j \leq k$ the value of $c(w,u)$ increases by one if and only if the crossing from $wu^{j-1}$ to $wu^j$ is positive with respect to the standard orientation.  Item (1) now follows.  

For (2), since $\gamma$ is a minimal gallery from $\fa$ to an alcove $\x$ in $\Cfs$, the gallery $\gamma$ contains only positive crossings (with respect to the standard orientation).  By (1), if $c(w,u) = \ell(u)$ then the gallery $\sigma_u$ also contains only positive crossings.  The concatenation $\gamma\star\sigma_u$ then crosses each hyperplane at most once, since if a gallery crosses a hyperplane more than once at least one crossing must be negative.  Hence $\gamma\star\sigma_u$ is a minimal gallery from $\fa$ to $xu\fa$ of type $\vec{x}\vec{u}$.

Item (3) follows from the deletion condition for Coxeter groups. 
\end{proof}

We now relate the varieties $X_x(1)$ and $X_{u^{-1}xu}(1)$, where $\x$ is in $u\Cfs$ and $u \in \sW$.

\begin{thm}\label{T:Conjugationb=1}
Let $\x$ be an alcove in $u\Cfs$ for some $u \in \sW$, and write $x = ut^{\lambda}w$.  
Suppose that the entire star of $\lambda$ lies in the shrunken dominant Weyl chamber $\Cfs$.
Then 
\begin{equation}
X_x(1)\neq \emptyset  \Longleftrightarrow X_{u^{-1}xu}(1)\neq \emptyset.
\end{equation}
Here, $u^{-1}x u\fa$ lies in the dominant Weyl chamber, and $X_{u^{-1}xu}(1) \neq \emptyset $ if and only if $wu \in \sW \backslash \bigcup\limits_{T \subsetneq S} W_T$.
If both varieties are nonempty, then their dimensions are related by 
\begin{equation}
\dim X_x(1) = \dim X_{u^{-1}xu}(1) + \frac{1}{2}(\ell(u^{-1}xu) - \ell(x)).
\end{equation}
\end{thm}

\begin{proof}
As $x = ut^\lambda w$, where $t^\lambda w\fa$ is in the dominant Weyl chamber $\Cf$, the element $u$ indexes the Weyl chamber which contains the alcove $\x$.  Consider a reduced expression $u = s_{i_1}\cdots s_{i_k}$.  We now define a sequence of pairs of the form $(s_{i_j}, u_j)$, where $u_j$ is the element of $\sW$ defined by $u_j = s_{i_{j+1}}\cdots s_{i_k}$ for $1 \leq j \leq k$.  That is, $u_j$ is the result of removing the first $j$ letters from the chosen reduced expression for $u$, and $u_k = \id$.  Notice that, by definition, $\ell(s_{i_j}u_j) > \ell(u_j)$ for all $1 \leq j \leq k$.  Therefore, Proposition~\ref{T:ConjugateXx1} applies to each successive pair in the sequence, proving that
\[ 
X_x(1) \neq \emptyset \iff X_{s_{i_1}xs_{i_1}}(1) \neq \emptyset  \iff X_{s_{i_2}s_{i_1}xs_{i_1}s_{i_2}}(1) \neq \emptyset \iff \cdots  \iff X_{u^{-1}xu}(1) \neq \emptyset.
\]

Since we assume that the entire star based at $\lambda$ lies in the shrunken fundamental Weyl chamber, then the alcove corresponding to each of these conjugates of $x$ also lies in the union of the shrunken Weyl chambers.  In addition, since we write $x = ut^\lambda w$ and assume that $\lambda$ is dominant, then $u^{-1}xu\fa = t^{\lambda}wu\fa$ lies in the dominant chamber.
 Moreover, the nonemptiness criterion from Theorem~\ref{idADLV} for alcoves in the shrunken
  dominant chamber requires that each simple reflection occurs in every reduced expression for the finite part. Therefore, under the convexity hypotheses in Theorem~\ref{bDomDim},  $X_{u^{-1}xu}(b) \neq \emptyset$ if and only if the finite part $wu$ lies in $\sW$, but not in any parabolic subgroup $W_T$ for $T \subsetneq S$.

Note that according to our conventions, the minimal length element in any coset $\aW / \sW$ is of the form $t^{\nu}w_0$ when $\nu$ is dominant.  Write $x = ut^{\lambda}w = ut^\lambda w_0w_0w$, and note that any reduced expression for $t^{\lambda}w_0$ necessarily begins and ends with $s_0$ by our hypothesis that $t^\lambda w\fa$ is in $\Cfs$.  Since we are assuming that some word for $x$ begins with the given reduced expression for $u$, then left multiplication by the simple reflections $s_{i_1}, \dots, s_{i_k}$ in that order decreases the length at each step.  Therefore, conjugation of $x$ by the finite simple reflections $s_{i_1}, s_{i_2}, \dots, s_{i_k}$ in that order either keeps the length the same at each step or decreases the length by 2.  Moreover, since $u, w_0w \in \sW$, the length of the entire product decreases at step $1 \leq j \leq k$ if and only if 
\[
 \ell( w_0wu^{j-1}) = \ell(w_0ws_{i_1}\cdots s_{i_j}) > \ell(w_0ws_{i_1}\cdots s_{i_j}s_{i_{j+1}}) = \ell( w_0wu^{j}) .
 \]
 At each step in the conjugation, Proposition~\ref{T:ConjugateXx1} says that the dimension decreases by 1 each time we have $\ell( w_0wu^{j-1}) > \ell( w_0wu^{j})$ and stays the same each time $\ell( w_0wu^{j-1}) < \ell( w_0wu^{j})$.  

Adding up all of these changes until we obtain a folded gallery of type $\vec{u}\vec{x}\vec{u}^{-1}$ in $\Cf$ gives precisely the correction factor $c(w_0w,u)$, which shows that 
$$\dim X_x(1)=\dim X_{u^{-1}xu}(1) + c(w_0w,u).$$  Now by the correspondence between word length staying the same (respectively, decreasing by $2$) and dimension staying the same (respectively, decreasing by $1$), we have that $2c(w_0 w, u) = \ell(u^{-1}xu) - \ell(x)$, and the result follows. 
\end{proof}

In the case where $b\neq 1$ we can prove one of the two directions of the previous theorem. 

\begin{thm}\label{T:Conjugation}
Let $\x$ be an alcove in $\Cfs$, and write $x = t^{\lambda}w$, where $w \in \sW$ and the entire star based at $\lambda$ lies in the shrunken dominant Weyl chamber $\Cfs$. 
Then for every  $u \in \sW$ we have 
\begin{equation}
X_x(b)\neq \emptyset  \Longrightarrow X_{u^{-1}xu}(b)\neq \emptyset.
\end{equation}
Moreover, if both varieties are nonempty, then the dimensions are related by 
\begin{equation}
\dim X_{u^{-1}xu}(b) \geq \dim X_{x}(b) + \frac{1}{2}(\ell(u^{-1}xu) - \ell(x)).
\end{equation}
\end{thm}

\begin{proof}
Consider a reduced expression $u^{-1} = s_{i_1}\cdots s_{i_k}$.  We now define a sequence of pairs of the form $(s_{i_{k-j}}, u_j)$, where $u_j$ is the element of $\sW$ defined as $u_j = s_{i_{k-j+1}}\cdots s_{i_k}$ for $1 \leq j \leq k$.  That is, $u_j$ is the product of the last $j$ letters from the chosen reduced expression for $u^{-1}$.  Notice that, by definition, $\ell(s_{i_{k-j}}u_j) > \ell(u_j)$ for all $1 \leq j \leq k$.  Since the entire star based at $\lambda$ lies in the shrunken dominant Weyl chamber, we have $s_{i_k}xs_{i_k} \in u_1 \Cf$, $s_{i_{k-1}}s_{i_k}xs_{i_k}s_{i_{k-1}} \in u_2\Cf$, and so on.  Therefore, Theorem~\ref{prop:bConjugateOneDirection} applies to each successive pair in the sequence, proving that
\[
 X_x(b) \neq \emptyset \Longrightarrow X_{s_{i_k}xs_{i_k}}(b) \neq \emptyset  \Longrightarrow X_{s_{i_{k-1}}s_{i_k}xs_{i_k}s_{i_{k-1}}}(b) \neq \emptyset \Longrightarrow \cdots  \Longrightarrow X_{u^{-1}xu}(b) \neq \emptyset.
 \]

Note that according to our conventions, the minimal length element in any coset $\aW / \sW$ is of the form $t^{\nu}w_0$ when $\nu$ is dominant.  By our hypothesis that $t^\lambda w\fa$ is in $\Cfs$ we can write $x = t^{\lambda}w = t^\lambda w_0w_0w$, and note that any reduced expression for $t^{\lambda}w_0$ necessarily begins and ends with $s_0$.  Conjugation of $x$ by the finite simple reflections $s_{i_k}, s_{i_{k-1}}, \dots, s_{i_1}$ in that order either keeps the length the same at each step or increases the length by 2.  Moreover, since $u, w_0w \in \sW$, the length of the entire product increases at step $1 \leq j \leq k$ if and only if
\[ 
\ell( w_0w u^{j-1}) = \ell(w_0ws_{i_k}\cdots s_{i_{k-j+2}}) < \ell(w_0ws_{i_k}\cdots s_{i_{k-j+1}}) = \ell( w_0w u^{j}) .
\]  
Here $u^{j}$ is the product of the first $j$ letters of the minimal presentation for $u$ obtained from the one for $u^{-1}$. At each step in the conjugation, Theorem~\ref{prop:bConjugateOneDirection} says that the dimension increases by 1 each time we have $\ell( w_0w u^{j-1}) < \ell( w_0w u^{j})$ and stays the same each time $\ell( w_0w u^{j-1}) > \ell( w_0w u^{j})$.  Adding up all of these changes until we obtain a folded gallery of type $\vec{u}^{-1}\vec{x}\vec{u}$ in $\Cf$ gives precisely the correction factor $c(w_0w,u)$, which shows that 
$$\dim X_{u^{-1}xu}(b) \geq \dim X_x(b)+ c(w_0w,u).$$
As in the proof of Theorem~\ref{T:Conjugationb=1}, we have that  $2c(w_0w,u)$ equals $\ell(u^{-1}xu) - \ell(u)$, and the result follows.
\end{proof}


\section{Diagram automorphisms}
\label{sec:SymmetriesLabelings}

The main result in this section is Theorem~\ref{DiagramSymmetry} below, which relates the varieties $X_x(b)$ and $X_{g(x)}(g(b))$, where $g$ is a diagram automorphism of $\aW$ and $b$ is a pure translation.  The results in this section hold for arbitrary $x \in W$ and for arbitrary pure translations $b$. 

We first consider the effect on galleries of relabeling an apartment.

\begin{prop}\label{prop:change label}
Let $x \in W$ and let $\gamma_x : \fa \rightsquigarrow x\fa$ be a minimal gallery. Let $g:\App\to\App$ be a simplicial bijection, not necessarily type-preserving.  Fix a labeling $\phi$ and let $\phi_g$ be the induced labeling given by $\phi_g(\y)\define \phi(g^{-1}(\y))$. Then the following are equivalent:
\begin{enumerate}
\item There exists a gallery $\gamma: \fa\rightsquigarrow y\fa$ of the same type as $\gamma_{x}$ that is positively folded with respect to $-\phipartial$ and satisfies $\dim_{-\phi}(\gamma) = d$.
\item There exists a gallery $\gamma': g(\fa) \rightsquigarrow g(y\fa)$ of the same type as type $g(\gamma_{x})$ that is positively folded with respect to $-\phipartialg$ and satisfies $\dim_{-\phi_g}(\gamma') = d$.  
\end{enumerate}
Moreover, $F(\gamma) = F(\gamma')$ and $P_{-\phi}(\gamma) = P_{-\phi_g}(\gamma')$. 
\end{prop}
\begin{proof}
By precomposing $g$ with a (not necessarily type-preserving) translation we may assume that $g$ fixes the origin.  By Lemma~\ref{lem:translations dim} and its proof, any translation of $\App$ preserves the orientation with respect to which galleries are positively folded, their dimension, and their number of folds and positive crossings.  Thus it suffices to prove the statement for $g$ fixing the origin.

Since $g$ fixes the origin, it is easy to see that $g$ acts on the set of chambers at infinity, and that a gallery $\gamma$ will be positively folded with respect to the orientation $-\phipartial$ if and only if the gallery $\gamma' := g(\gamma)$ is positively folded with respect to the orientation $-\phipartialg$.  The action of $g$ preserves the number of folds, and the positive crossings of $\gamma$ with respect to $-\phipartial$ are exactly the positive crossings of $\gamma'$ with respect to $-\phipartialg$.  Thus $\dim_{-\phi}(\gamma) = \dim_{-\phi_g}(\gamma')$.  This completes the proof.
\end{proof}

We will use Proposition~\ref{prop:change label} to draw several conclusions about nonemptiness and dimensions of pairs of affine Deligne--Lusztig varieties which are related by certain simplicial isometries of the apartment $\App$, namely those induced by diagram automorphisms of $\aW$.  Recall that a \emph{diagram automorphism} of $\aW$ is an automorphism of $\aW$ induced by an automorphism of the Dynkin diagram for $(\aW,\tilde S)$, equivalently an an automorphism of $\aW$ induced by a permutation of the generating set $\tilde S = \{ s_1,\dots,s_n\} \cup \{s_0\}$.  

We will need the following elementary lemma, which we were not able to find in the literature.  For the purposes of this lemma, it is important to make a distinction between the coroot lattice $R^\vee$ and the subgroup $\{ t^\lambda \mid \lambda \in R^\vee \}$ of pure translations in $\aW$. 

\begin{lemma}\label{lem:diagram}  Let $g : \aW \to \aW$ be a diagram automorphism.  Then $g$ preserves the subgroup of pure translations in $\aW$.   
\end{lemma}

\begin{proof}  The subgroup of pure translations $\{ t^\lambda \mid \lambda \in R^\vee\}$ is a (normal) subgroup of $\aW$ of index $|\sW|$, hence its image under $g$ is a (normal) subgroup of $\aW$ of index $|\sW|$.  Thus it suffices to prove that $g(t^\lambda)$ is a translation for each $\lambda \in R^\vee$.  A basis for $R^\vee$ is the set of translations by simple coroots $\{ t^{\alpha_i^\vee} \}_{i=1}^n$, and so it is enough to show that $g(t^{\alpha_i^\vee})$ is a translation for each $i = 1,\dots,n$.

Recall that the translation $t^{\alpha_i^\vee}$ is equal to the product of reflections $s_{\alpha_i,1}s_{\alpha_i}$, where $s_{\alpha_i} = s_{\alpha_i,0}$ and for all $k \in \Z$ the fixed set of the reflection $s_{\alpha_i,k}$ is the hyperplane $H_{\alpha_i,k}$.  Since $g$ is a diagram automorphism, $g$ takes reflections to reflections, and so $g(t^{\alpha_i^\vee})$ is the product of reflections $g(s_{\alpha_i,1})g(s_{\alpha_i})$.  This product of reflections will be a translation if and only if the fixed sets of these reflections are parallel hyperplanes.

As $g$ is a diagram automorphism, $g$ induces a simplicial bijection $\App \to \App$, which by abuse of notation we also denote by $g$.  This induced map takes parallel hyperplanes to parallel hyperplanes, and if $r$ is the reflection in $\aW$ with fixed set the hyperplane $H$, then $g(r)$ is the reflection in $\aW$ with fixed set the hyperplane $g(H)$.  It follows that the reflections $g(s_{\alpha_i})$ and $g(s_{\alpha_i,1})$ have fixed sets the parallel hyperplanes $g(H_{\alpha_i})$ and $g(H_{\alpha_i,1})$, respectively.  Thus $g(t^{\alpha_i^\vee})$ is a translation, as required.  This completes the proof.
\end{proof}

We are now ready to prove the main result.

\begin{thm}\label{DiagramSymmetry}
Let $x \in W$ and let $b = t^\mu$ be a pure translation, $\mu \in R^\vee$.  Then for any diagram automorphism $g:\aW \to \aW$ we have
$$ X_x(b) \neq\emptyset \Longleftrightarrow X_{g(x)}(g(b)) \neq\emptyset.$$ 
Moreover if both these varieties are nonempty, then
$$ \dim X_x(b)  = \dim X_{g(x)}(g(b)).$$
\end{thm}

\begin{proof}
Any diagram automorphism $g$ induces a simplicial bijection $\App \to \App$, which by abuse of notation we also denote by $g$.  We remark that the induced simplicial bijection $g:\App \to \App$ will not be type-preserving unless $g$ is trivial, and that $g$ does not in general preserve the coroot lattice.  However $g$ does preserve the base alcove $\fa$, and by Lemma~\ref{lem:diagram}, the map $g$ does preserve the set of alcoves which are obtained from the base alcove $\fa$ by applying a pure translation.

Suppose first that $X_x(b)\neq\emptyset$.   By Lemma~\ref{T:FrobConjIso}, we may assume that $b$ is dominant.  Then as $b$ is dominant, Lemma~\ref{DomADLVDim} says that for some $w \in \sW$, there exists a gallery $\gamma:\fa \rightsquigarrow b^w\fa$ which is positively folded with respect to the orientation $\phipartialwow$ and satisfies $\dim X_x(b) = \dim_{\phi_{w_0w}}(\gamma)$.  Let $\gamma_x:\fa \rightsquigarrow x\fa$ be the minimal gallery of the same type as $\gamma$.

Since $g:\App \to \App$ preserves $\fa$, the map $g$ takes the minimal gallery $\gamma_x:\fa \rightsquigarrow x\fa$ to a minimal gallery, say $\gamma_{g(x)}$, from $\fa$ to the alcove $g(x)\fa$.  Then if we apply $g$ to the folded gallery $\gamma$ we obtain a gallery $\gamma' := g(\gamma):\fa\rightsquigarrow g(b^w)\fa$ of the same type as $\gamma_{g(x)}$.  Write $g(w) = t^\lambda w'$ where $\lambda \in R^\vee$ and $w' \in \sW$.  Note that $g$ need not fix the origin, so we could have $\lambda \neq 0$.   Then 
 $$g(b^w) = g(t^{w\mu}) = g(wt^\mu w^{-1}) = g(w)g(b)g(w)^{-1}= t^\lambda w' t^{\nu}(w')^{-1}t^{-\lambda} = t^\lambda t^{w'\nu}t^{-\lambda} = t^{w'\nu}$$ where $g(b) = g(t^\mu)$ is the translation $t^{\nu}$, $
\nu \in R^\vee$. 
Also, putting $\phi = \phi_w$ in Proposition~\ref{prop:change label}, we have that as $\gamma$ is positively folded with respect to the orientation $-\phipartial = -\phipartialw = \phipartialwow$, the gallery $\gamma' = g(\gamma)$ is positively folded with respect to the orientation $-\phipartialg$, and $\dim_{-\phi_g}(\gamma') = \dim_{-\phi_w}(\gamma)$.  

We claim that the orientation at infinity induced by the labeling $\phi_g$ is the same as the orientation at infinity induced by the labeling $\phi_{w'}$.  To see this, observe that $$\phi_g(t^\lambda w') = \phi( g^{-1}(t^\lambda w')) = \phi_w(w) = 1$$
hence $\phi_g(t^\lambda w') = \phi_{w'}(w') = 1$.  It follows that both the orientations at infinity $\phipartialg$ and $\phipartialwp$ label the Weyl chamber $\cC_{w'}$ by the identity element, hence these are the same induced orientation at infinity.  Therefore $\gamma':\fa \rightsquigarrow (g(b))^{w'}\fa$ is positively folded with respect to the orientation $-\phipartialwp = \phipartialwowp$, and satisfies $\dim_{-\phi_{w'}}(\gamma') = \dim_{-\phi_w}(\gamma) = \dim X_x(b)$.

Now by Theorem~\ref{ADLVvsGal2}, the existence of such a $\gamma'$ shows that $X_{g(x)}(g(b)) \neq \emptyset$ and that $\dim X_{g(x)}(g(b)) \geq \dim X_x(b)$.  To show the converse suppose that $X_{g(x)}(g(b))\neq\emptyset$ and argue as above using $g^{-1}$ instead of $g$. 
\end{proof}

\begin{figure}[ht]
\begin{center}
\begin{overpic}[width=0.8\textwidth]{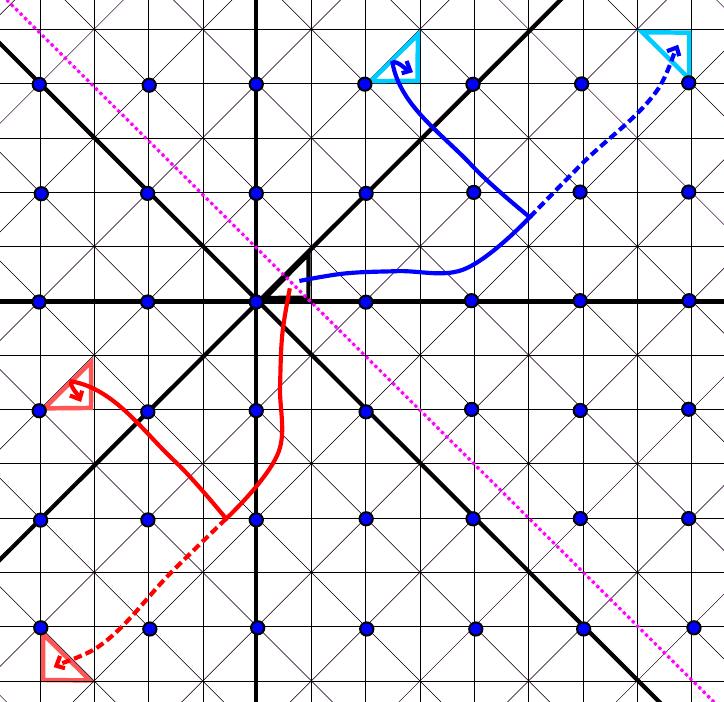}
\put(91,92){$\color{blue} \x$}
\put(55,89){$\color{blue} \bb$}
\put(66,75){$\color{blue} \sigma$}
\put(0,4){$\color{red} g(\x)$}
\put(7,37){$\color{red} g(\bb)$}
\put(18,30){$\color{red} g(\sigma)$}
\end{overpic}
\caption[shrunken]{This example in type $\tilde C_2$ illustrates the symmetry of the pattern of nonemptiness and dimension for affine Deligne--Lusztig varieties induced by the diagram automorphism $g$ switching nodes $s_0$ and $s_1$.  The induced symmetry of $\App$ is the reflection along the dotted (pink) line. Compare Example~\ref{ex:diagram symmetry}.  }
\label{fig:C2symmetry}
\end{center}
\end{figure}
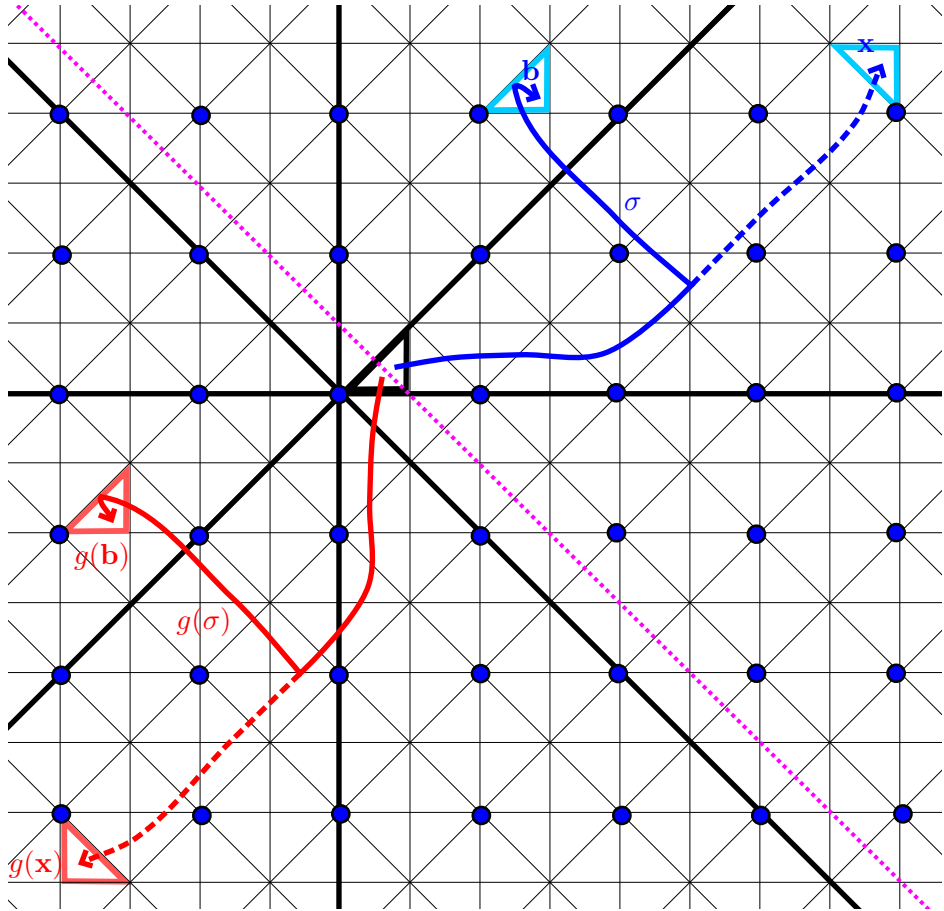

\begin{example}\label{ex:diagram symmetry}
In type $\tilde A_2$, the diagram automorphism which interchanges the two simple roots induces a simplicial isomorphism of the standard apartment which is given by reflection across the line through the origin and $k\rho$, for any $0 \neq k \in \Z$.  Therefore, Theorem~\ref{DiagramSymmetry} proves that the nonemptiness and dimension pattern for varieties $X_x(b)$ with $b = t^{2k\rho}$ has a vertical symmetry along this line.  

In type $\tilde C_2$ there is a symmetry which is induced by a diagram automorphism swapping $s_0$ and $s_1$, and which geometrically is a reflection along the line $H_{\alpha_1^{\vee}+\alpha_2^{\vee},\frac{1}{2}}$.  This line of symmetry is illustrated (in pink) in Figure~\ref{fig:C2symmetry}. In the figure the gallery $\sigma$ is a DL-gallery for $X_x(b)$ and $g(\sigma)$ a DL-gallery for $X_{g(x)}(g(b))$. 

\end{example}

\begin{remark}  Symmetries induced by diagram automorphisms can also be used to prove equalities among certain partially ordered sets of Newton polygons and codimensions of Newton strata, as discussed in the first author's work in \cite{Be1}.
\end{remark}

\section{Applications to affine Hecke algebras and affine reflection length}\label{sec:Applications}

In this section we give two immediate applications of our main results on affine Deligne--Lusztig varieties.  In Section~\ref{sec:Class} we use our results together with the work of He~\cite{HeAnnals} to compute degrees of class polynomials for affine Hecke algebras, and in Section~\ref{sec:Reflection} we consider reflection length in affine Weyl groups.

\subsection{Class polynomials of the affine Hecke algebra}\label{sec:Class}

In this section we closely follow the terminology in \cite{HeAnnals}, in which the connection between affine Deligne--Lusztig varieties and class polynomials for affine Hecke algebras was discovered.  In~\cite{HeAnnals} the setting is the extended affine Weyl group $\eW$, but since our results are for varieties $X_x(b)$ where both $x$ and $b$ are in the affine Weyl group $\aW$, we restrict to $\aW$.

For $x \in \aW$ in the affine Weyl group, we write $[x]$ for the $\sigma$-conjugacy class of $x$.  Recall from \cite{HeAnnals} that $b \in \aW$ is called \emph{$\sigma$-straight} if $\ell(b) = \langle 2\rho, \nu_b\rangle$, where $\nu_b$ is the Newton point of $b$.
Now let $b = t^\mu$ be a pure translation in $\aW$. Then $b$ is $\sigma$-straight since $\ell(b) = \langle 2\rho, \nu_b \rangle = \langle 2\rho, \mu^+ \rangle$.  Here, $\nu_b$ equals the unique dominant element in the $\sW$-orbit of $\mu \in R^{\vee}$, which we have denoted by $\mu^+$.  Since $b$ is $\sigma$-straight,  we say that the the $\sigma$-conjugacy class $[b]$ is $\sigma$-straight.  Recall also that $\lambda \in R^{\vee}$ is \emph{regular} if the stabilizer $\operatorname{Stab}_{\sW}(\lambda)$ is trivial; i.e., $\lambda$ does not lie along the walls of any of the Weyl chambers.  If $b= t^{\mu}$ is a pure translation such that $\mu$ is regular, then $b$ is  \emph{superstraight}.
 
Following 2.2 and 2.3 of~\cite{HeAnnals}, let $A$ be the algebra $\Z[v,v^{-1}]$.  The Hecke algebra associated with the affine Weyl group $\aW$ is the $A$-algebra $H$ with basis $\{ T_x : x\in \aW\}$, and these basis elements satisfying certain relations.  There is an $A$-algebra homomorphism induced by $T_x \mapsto T_{\sigma(x)}$, which is also denoted $\sigma$.  For $h, h' \in H$ the \emph{$\sigma$-commutator} of $h$ and $h'$ is $[h,h']_\sigma = hh' - h'\sigma(h)$.  Write $[H,H]_\sigma$ for the $A$-submodule of $H$ generated by all $\sigma$-commutators.  For any $x \in \aW$ and any $\sigma$-conjugacy class $[y]$, there is a unique $f_{x,[y]} \in A$ such that the following equation holds:
\begin{equation}\label{E:class}
T_x \equiv \sum_{[y]} f_{x,[y]} T_{[y]_{\min}} \mod [H,H]_\sigma,
\end{equation}
where $[y]_{\min}$ denotes a minimal length representative of the $\sigma$-conjugacy class $[y]$. The $f_{x,[y]}$ are polynomials in $\Z[v - v^{-1}]$ and are called the \emph{class polynomials} of the affine Hecke algebra $H$. 

Assume that $b = t^\mu$ is a pure translation with $\mu \in R^{\vee}$ regular.  Then for all $x \in \aW$, if $X_x(b)$ is nonempty, by Corollary 6.2 in~\cite{HeAnnals} we have that $f_{x,[b]} \neq 0$, and by Corollary 6.3 in~\cite{HeAnnals} we have that
\begin{equation}\label{E:He} \dim X_x(b) = \frac{1}{2}\left( \ell(x) + \deg(f_{x,[b]}) \right) - \langle \rho, \mu^+ \rangle. \end{equation}

We first combine this with the dimension equality in our Theorem~\ref{Shrunken} to compute degrees of class polynomials $f_{x,[b]}$,  where $x$ is in the shrunken dominant Weyl chamber $\Cfs$, in the following result.

\begin{thm}\label{T:ClassShrunken}  Let $x = t^\lambda w$ be such that every alcove at the vertex $\lambda$ lies in the shrunken dominant Weyl chamber $\Cfs$.  Let $b = t^\mu$ be a regular dominant pure translation.  Assume that $w$ satisfies Reuman's condition in Theorem \ref{idADLV}, $\bb \in \conv(\fa, \x)$, $t^{-\mu}\x$ lies in $\Cfs$, and $\mu$ lies in the negative cone based at $\lambda - 2\rho$.  Then $f_{x,[b]} \neq 0$ and 
\begin{equation} \deg(f_{x,[b]}) = \ell(w). \end{equation}
\end{thm}

\begin{proof}  Since $w$ satisfies Reuman's condition we have $X_x(1) \neq \emptyset$ by Theorem~\ref{idADLV}.  Thus by Theorem~\ref{ArbitraryNonemptyDim}, the variety $X_x(b)$ is nonempty, and so by Corollary 6.2 of~\cite{HeAnnals}, the polynomial $f_{x,[b]}$ is nonzero.  

Note that as $x$ is in $\Cfs$ its length satisfies
\begin{equation}\label{E:length}
\ell(x) = \ell(t^\lambda) - \ell(w) = \langle 2\rho, \lambda \rangle - \ell(w).
\end{equation} 
Since $X_x(b)$ is nonempty, by Equation~\eqref{E:ShrunkenDim} from Theorem~\ref{Shrunken}, we have that
\begin{equation}\label{E:bDomDimClass} \dim X_x(b) = \langle \rho, \lambda - \mu \rangle. \end{equation} Now $\mu^+ = \mu$ since $b$ is dominant.  Hence using the formula for $\ell(x)$ given by~\eqref{E:length}, together with He's result~\eqref{E:He} above, we get
\begin{equation}
\dim X_x(b) = \langle \rho, \lambda \rangle - \frac{1}{2}\ell(w) + \frac{1}{2} \deg(f_{x,[b]}) - \langle \rho, \mu \rangle.
\end{equation}
Putting this equality together with~\eqref{E:bDomDimClass} yields
\begin{equation}
0 = - \frac{1}{2}\ell(w) + \frac{1}{2} \deg(f_{x,[b]})
\end{equation}
from which we obtain the desired result.
\end{proof}

We also obtain information about class polynomials using the statements concerning dimensions of affine Deligne--Lusztig varieties in our other main results, Theorems~\ref{ForwardShift},~\ref{NonemptyDim}, and~\ref{OtherWeyl}, as follows.  

\begin{thm}\label{T:ClassNonShrunken}  Let $x = t^\lambda w \in \aW$ and let $b = t^\mu$ be a regular pure translation.
\begin{enumerate}
\item\label{forward-shift} Suppose $\bb  \in \conv(\fa,t^\mu \x)$.  If $X_x(1) \neq \emptyset$ then $f_{t^\mu x,[t^\mu]} \neq 0$ and 
\begin{equation}
\deg(f_{t^\mu x,[t^\mu]}) \geq \ell(w).
\end{equation}
\item\label{shrunken} Suppose $\bb \in \conv(\fa,\x)$ and the alcoves $\x$ and $t^{-\mu}\x$ lie in the same shrunken Weyl chamber $\Cus$, where $u \in \sW$.  If $X_x(1) \neq \emptyset$ then $f_{x,[b]} \neq 0$ and 
\begin{equation}
\deg(f_{x,[b]}) \geq \ell(u^{-1}wu) - 2\langle \rho_{B^-}, \mu + \mu_{B^-} \rangle
\end{equation}
with $\deg(f_{x,[b]}) \geq \ell(u^{-1}wu)$ if $\mu$ is dominant.
\item\label{strips} Suppose $\bb \in \conv(\fa,\x)$ and the alcoves $\x$ and $t^{-\mu}\x$ both lie in $\Cu \setminus \Cus$, where $u \in \sW$.  If $X_x(1) \neq \emptyset$ then $f_{x,[b]} \neq 0$ and 
\begin{equation}
\deg(f_{x,[b]}) \geq \ell(\eta_\sigma(x)) - 2\langle \rho_{B^-}, \mu + \mu_{B^-} \rangle
\end{equation}
with $\deg(f_{x,[b]}) \geq \ell(\eta_\sigma(x))$ if $\mu$ is dominant.
\item\label{conj} Suppose every alcove at the vertex $\lambda$ lies in the shrunken dominant Weyl chamber $\Cfs$ and $X_x(b) \neq \emptyset$.  Then for every $u \in \sW$ we have $f_{u^{-1}xu,[b]} \neq 0$ and 
\begin{equation}
\deg(f_{u^{-1}xu,[b]}) \geq \deg(f_x,[b]).
\end{equation}
\end{enumerate}
\end{thm}

\begin{proof}  For parts~\eqref{forward-shift},~\eqref{shrunken}, and~\eqref{strips}, we need to know the dimension of the variety $X_x(1)$, and we use the same results as in the proof of Proposition~\ref{Xx1Xy1}.  In summary, if $\x$ is in the shrunken Weyl chamber $\Cus$, where $u \in \sW$, then 
\begin{equation} \dim X_x(1) = \frac{1}{2}(\ell(x) + \ell(u^{-1}wu)),  \end{equation}
and if $\x$ is in the lowest two-sided cell then
\begin{equation} \dim X_x(1) = \frac{1}{2}(\ell(x) + \ell(\eta_\sigma(x))). \end{equation}

The proofs are then immediate upon combining He's result~\eqref{E:He} above with Theorem~\ref{ForwardShift} for parts \eqref{forward-shift}, Theorem~\ref{NonemptyDim} for parts \eqref{shrunken} and \eqref{strips}, and Theorem~\ref{OtherWeyl} for part \eqref{conj}.
\end{proof}

\subsection{Reflection length in affine Weyl groups}\label{sec:Reflection}
 
In this section we prove the following theorem concerning reflection length in affine Weyl groups.  The shrunken Weyl chambers in the statement of part (1) below are defined at~\ref{def:shrunken}.    In part (2) below we assume for simplicity that the affine Deligne--Lusztig variety $X_x(1)$ is nonempty.  The conditions under which $X_x(1)$ is nonempty were established in~\cite{GHN} and we recall these conditions in the proof of Proposition~\ref{Xx1Xy1}.   The definition of the map $\eta_\sigma:\aW \to \sW$ in part (2) is recalled in the proof of Proposition~\ref{Xx1Xy1} as well.

\begin{thm}\label{CoxeterElements}  Let $\aW$ be an irreducible affine Weyl group of rank $n$, with set of reflections~$\tilde R$, and let $\sW$ be the associated finite Weyl group, with set of reflections $R$.  Let $x = t^\lambda w \in \aW$, where $\lambda \in R^\vee$ and $w \in \sW$.  Suppose that the alcove $\x = x\fa$ lies in the Weyl chamber $\Cu$ where $u \in \sW$. 
\begin{enumerate}
\item If $\x$ lies in the shrunken Weyl chamber $\Cus$, assume that 
\begin{equation}\label{E:w}
u^{-1}wu \in \sW \backslash \bigcup\limits_{T \subsetneq S} W_T.
\end{equation}
Then for all $g \in \Aut(\aW)$ 
\begin{equation} \ell_R(w) \leq \ell_{\tilde R}(g(x)) \leq \ell(u^{-1}wu). \end{equation}
Moreover if $w$ is a Coxeter element of $\sW$, then for all $g \in \Aut(W)$ \begin{equation} \ell_{\tilde R}(g(x)) = n. \end{equation} 
\item If $\x$ lies in $\Cu \setminus \Cus$, assume that $X_x(1)$ is nonempty.  Then for all $g \in \Aut(\aW)$ 
\begin{equation} \ell_R(w) \leq \ell_{\tilde R}(g(x)) \leq \ell(\eta_\sigma(x)). \end{equation}
\end{enumerate}
\end{thm}

\begin{proof}  The automorphism group $\Aut(\aW)$ splits as the semidirect product of its subgroups of inner and diagram automorphisms; see, for instance, Proposition 4.5 of Franzsen~\cite{Franzsen}.  Since both inner and diagram automorphisms preserve the set of reflections~$\tilde R$, it follows that reflection length is $\Aut(\aW)$-invariant.  Hence it suffices to prove the claims for $g(x) = x$.  The lower bound $\ell_R(w) \leq \ell_{\tilde R}(x)$ is Proposition 2.4 of McCammond and Petersen~\cite{McCammondPetersen}.

We now consider the affine Deligne--Lusztig variety $X_x(1)$.  In part (1), since $\x$ lies in the shrunken Weyl chamber $\Cus$ and \eqref{E:w} holds, Theorem~\ref{idADLV} tells us that $X_x(1) \neq \emptyset$.  In part (2) we assumed that $X_x(1) \neq \emptyset$.  The key result is then the following.

\begin{prop}\label{Folds}  Let $x$ be as in the statement of Theorem~\ref{CoxeterElements}.
\begin{enumerate}
\item If $\x$ is in $\Cus$ then every DL-gallery for $X_x(1)$ has $\ell(u^{-1}wu)$ folds.
\item If $\x$ is in $\Cu \setminus \Cus$ then every DL-gallery for $X_x(1)$ has $\eta_\sigma(x)$ folds.
\end{enumerate}
\end{prop}

\begin{proof}  We first note that by Lemma~\ref{lem:ConstantFolds}, every DL-gallery for $X_x(1)$ has the same number of folds.  So let  $\sigma$ be a DL-gallery for $X_x(1)$ which is positively folded with respect to $\phipartial$.  That is, $\dim X_x(1) = \dim_\phi(\sigma)$.  We will compute $F_\phi(\sigma)$, the number of folds in $\sigma$.

The gallery $\sigma$ begins and end at $\fa$, so by Lemma~\ref{lem:gallery dim}
$$
\dim_{\phi}(\sigma)=\frac{1}{2}\left(\ell(x)+F_{\phi}(\sigma)\right).
$$
Now if $\x$ is in $\Cus$ then by Theorem~\ref{idADLV} we have 
$$
\dim X_x(1) = \frac{1}{2}\left(\ell(x) + \ell(u^{-1}wu)\right).
$$
Thus in this case $F_\phi(\sigma) = \ell(u^{-1}wu)$ as required.  If $\x$ is in $\Cu \setminus \Cus$ then by Equation~\eqref{E:virtualdim} in the proof of Proposition~\ref{Xx1Xy1}, 
$$
\dim X_x(1) = \frac{1}{2}\left(\ell(x) + \ell(\eta_\sigma(x))\right).
$$
Hence in this case $F_\phi(\sigma) = \ell(\eta_\sigma(x))$. 
\end{proof}

To finish the proof of Theorem~\ref{CoxeterElements}, let $\sigma$ be a DL-gallery for $X_x(1)$, which is positively folded with respect to the orientation at infinity $\phipartial$.   Lemma~\ref{FoldsLower} gives the following lower bound on the number of folds: 
\[ \ell_{\tilde R}(x) \leq F_\phi(\sigma). \]

Assume now that $\x$ is in $\Cus$.  Then by Proposition~\ref{Folds} we have $F_\phi(\sigma) = \ell(u^{-1}wu)$, and so
\[ \ell_{\tilde R}(x) \leq \ell(u^{-1}wu)  \]
as desired.   Now if $w$ is a Coxeter element of $\sW$, it has reflection length $\ell_R(w) = n$.  Since the conjugate $u^{-1}wu$ is also a Coxeter element of $\sW$, it satisfies $\ell(u^{-1}wu) = n$.  Thus 
\[ n = \ell_R(w) \leq \ell_{\tilde R}(x) \leq \ell(u^{-1}wu) = n\]
and so $\ell_{\tilde R}(x) = n$ in this case.  

If $\x$ is in $\Cu \setminus \Cus$ then by Proposition~\ref{Folds} we have $F_\phi(\sigma) = \ell(\eta_\sigma(x))$, and the result follows by similar arguments.
\end{proof}

\renewcommand{\refname}{Bibliography}
\bibliography{bibliography}
\bibliographystyle{alphaurl}

\end{document}